\documentclass[11pt,a4paper]{amsart}

\usepackage[T1]{fontenc}

\usepackage{mathtools}
\usepackage{amssymb}
\usepackage{amsthm}
\usepackage{mathrsfs}
\usepackage{tikz-cd}
\usepackage{enumitem}
\usepackage{stackrel}
\usepackage{csquotes}
\usepackage[bb=boondox]{mathalfa}

\usepackage[margin=2cm]{geometry}
\usepackage[english]{babel}

\usetikzlibrary{babel}

\newcommand{\N}{\ensuremath{\mathbb N}}

\renewcommand{\epsilon}{\varepsilon}

\newcommand{\id}{\mathsf{id}}

\newcommand{\lax}{\mathsf{lax}}
\newcommand{\opl}{\mathsf{opl}}
\newcommand{\fc}{\mathfrak{F}}
\newcommand{\op}{\mathsf{op}}
\newcommand{\fin}{\mathsf{fin}}

\newcommand{\trm}{\mathbb{1}}
\newcommand{\ini}{\mathbb{0}}

\newcommand{\comma}{\downarrow}
\newcommand{\cat}[1]{\ensuremath{\mathcal{#1}}}
\newcommand{\bicat}[1]{\ensuremath{\mathbb{#1}}}

\newcommand{\intcat}[1]{\mathscr{#1}}
\newcommand{\catname}[1]{\mathsf{#1}}
\newcommand{\Set}{\catname{Set}}
\newcommand{\FinSet}{\catname{FinSet}}

\newcommand{\Grph}{\catname{Grph}}

\newcommand{\Cat}{\catname{Cat}}
\newcommand{\CAT}{\catname{CAT}}

\newcommand{\PsDbCat}{\catname{PsDbCat}}

\newcommand{\VDbCat}{\catname{VDbCat}}
\DeclareMathOperator{\Span}{\catname{Span}}
\newcommand{\SpanB}{\Span(\cat B)}

\newcommand{\Fam}{\catname{Fam}}
\newcommand{\Mat}{\catname{Mat}}
\newcommand{\Top}{\catname{Top}}

\newcommand{\Vect}{\catname{Vect}}
\newcommand{\Mnd}{\catname{Mnd}}
\newcommand{\Mod}{\catname{Mod}}

\newcommand{\FamB}{\Fam(\cat B)}
\newcommand{\FamC}{\Fam(\cat C)}

\newcommand{\VMat}{\cat V\dash\Mat}
\newcommand{\WMat}{\cat W\dash\Mat}

\newcommand{\VCat}{\cat V\dash\Cat}
\newcommand{\SVCat}{(S,\,\cat V)\dash\Cat}
\newcommand{\TVCat}{(T,\,\cat V)\dash\Cat}
\newcommand{\TWCat}{(T,\,\cat W)\dash\Cat}
\newcommand{\TtVCat}{(\overline{T},\,\cat V)\dash\Cat}
\newcommand{\SpanV}{\Span(\cat V)}
\newcommand{\CatV}{\Cat(\cat V)}
\newcommand{\CatTB}{\Cat(T,\,\cat B)}
\newcommand{\CatTV}{\Cat(T,\,\cat V)}
\newcommand{\CatSV}{\Cat(S,\,\cat V)}

\newcommand{\iso}{\cong}
\newcommand{\eqv}{\simeq}

\newcommand{\adj}{\dashv}
\newcommand{\dash}{\text{-}}

\newcommand{\lb}{\langle}
\newcommand{\rb}{\rangle}

\newcommand{\heps}{\hat \epsilon}
\newcommand{\heta}{\hat \eta}

\newcommand{\Tt}{\overline T}
\newcommand{\Ft}{\overline F}

\newcommand{\pt}{\ast}
\newcommand{\ovl}{\overline}

\newcommand{\HKl}[1]{\mathbb{H}\Kl(#1)}

\DeclareMathOperator{\Lax}{\mathsf{Lax}}
\DeclareMathOperator{\Lan}{\mathsf{Lan}}
\DeclareMathOperator{\Alg}{\mathsf{Alg}}
\DeclareMathOperator{\Kl}{\mathsf{Kl}}
\DeclareMathOperator{\Ps}{\mathsf{Ps}}

\DeclareMathOperator{\PsTAlg}{\Ps-\mathit{T}-\Alg}
\DeclareMathOperator{\PsfcAlg}{\Ps-\fc-\Alg}
\DeclareMathOperator{\LaxfcAlg}{\Lax-\fc-\Alg}
\DeclareMathOperator{\LaxidAlg}{\Lax-\id-\Alg}
\DeclareMathOperator{\LaxTAlg}{\Lax-\mathit{T}-\Alg}
\DeclareMathOperator{\LaxModTAlg}{\Lax-\Mod(\mathit{T})-\Alg}
\DeclareMathOperator{\LaxSAlg}{\Lax-\mathit{S}-\Alg}
\DeclareMathOperator{\LaxTHAlg}{\mathbb{H}\LaxTAlg}
\DeclareMathOperator{\LaxModTHAlg}{\mathbb{H}\LaxModTAlg}
\DeclareMathOperator{\LaxSHAlg}{\mathbb{H}\LaxSAlg}
\DeclareMathOperator{\LaxidHAlg}{\mathbb{H}\LaxidAlg}
\DeclareMathOperator{\LaxHAlg}{\mathbb{H}\Lax-(--)-\Alg}
\DeclareMathOperator{\HKlT}{\HKl{\mathit{T}}}
\DeclareMathOperator{\HKlTx}{\HKl{\mathit{T},\textit{x}}}

\DeclareMathOperator{\Desc}{\catname{Desc}}

\DeclareMathOperator{\dom}{\mathsf{dom}}
\DeclareMathOperator{\cod}{\mathsf{cod}}
\DeclareMathOperator{\ob}{\catname{ob}}

\DeclareMathOperator{\mult}{\mathsf{m}}
\DeclareMathOperator{\unit}{\mathsf{e}}
\DeclareMathOperator{\nat}{\mathsf{n}}
\newcommand{\m}[1]{\mult^{#1}}
\newcommand{\tm}[1]{{\tilde\mult}^{#1}}
\newcommand{\e}[1]{\unit^{#1}}
\newcommand{\n}[1]{\nat^{#1}}
\newcommand{\tn}[1]{{\tilde\nat}^{#1}}

\newcommand{\relto}{\nrightarrow}

\newtheorem{lemma}{Lemma}[section]
\newtheorem{proposition}[lemma]{Proposition}
\newtheorem{corollary}[lemma]{Corollary}
\newtheorem{theorem}[lemma]{Theorem}

\theoremstyle{definition}

\newtheorem{remark}[lemma]{Remark}

\usepackage[backend=biber,style=trad-plain,sorting=nyt]{biblatex}
\DefineBibliographyStrings{english}{pages={pp.}}

\numberwithin{equation}{section}
\begin{document}

  \title[\textsc{Generalized multicategories}]{\textbf{\textsc{Generalized
multicategories: change-of-base, embedding, and descent}}}

\author[R. Prezado]{Rui Prezado}
\author[F. Lucatelli Nunes]{Fernando Lucatelli Nunes}

\address[1,2]{University of Coimbra, CMUC, Department of Mathematics,
Portugal}
\address[2]{Utrecht University, The Netherlands}

\email[1]{ruiprezado@gmail.com}
\email[2]{f.lucatellinunes@uu.nl}

\keywords{double category, equipment, lax algebra, generalized multicategory,
effective descent morphisms, Beck-Chevalley condition, virtual equipment,
internal category, enriched category, Grothendieck descent theory, extensive
category, higher category theory}

\subjclass{18B10, 18B15, 18B50, 18D65, 18N10, 18N15}

\date{}

  \begin{abstract}
  Via the adjunction $ - \ast \mathbb{1} \dashv \mathcal V(\mathbb{1},-)
  \colon \mathsf{Span}(\mathcal V) \to \mathcal V \text{-} \mathsf{Mat} $ and
  a cartesian monad $ T $ on an extensive category $ \mathcal V $ with finite
  limits, we construct an adjunction  $ - \ast \mathbb{1} \dashv \mathcal
  V(\mathbb{1},-) \colon \mathsf{Cat}(T,\mathcal V) \to (\overline T, \mathcal
  V)\text{-}\mathsf{Cat} $ between categories of generalized enriched
  multicategories and generalized internal multicategories, provided the monad
  $ T $ satisfies a suitable condition, which is satisfied by several
  examples.
  
  We verify, moreover, that the left adjoint is fully faithful, and preserves
  pullbacks, provided that the copower functor \mbox{$ - \ast \mathbb{1}
  \colon \mathsf{Set} \to \mathcal V $} is fully faithful. We also apply this
  result to study descent theory of generalized enriched multicategorical
  structures.

  These results are built upon the study of base-change for generalized
  multicategories, which, in turn, was carried out in the context of
  categories of horizontal lax algebras arising out of a monad in a suitable
  2-category of pseudodouble categories.
\end{abstract}

  \maketitle

  \tableofcontents

  \section*{Acknowledgements}
    We would like to thank Maria Manuel Clementino for her helpful remarks during
the elaboration of this work. 

We are also grateful to Marino Gran and Tim Van der Linden for hosting us at
UCLouvain, where this work was concluded.

Finally, we thank the anonymous referee for their suggestions and comments.

The authors acknowledge financial support by \textit{Centro de Matemática da
Universidade de Coimbra} (CMUC), funded by the Portuguese Government through
FCT/MCTES, DOI 10.54499/UIDB/00324/2020. The first author acknowledges
financial support by the grant PD/BD/150461/2019 funded by \textit{Fundação
para a Ciência e Tecnologia} (FCT). The second author acknowledges financial
support by a Fields Research Fellowship, provided by the \textit{Fields
Institute for Research in Mathematical Sciences}.

This work was also supported through the programme ``Oberwolfach Leibniz
Fellows'' by the Mathematisches Forschungsinstitut Oberwolfach in 2022.

  \section*{Introduction}
    The systematic study of the dichotomy between enriched categories and internal
categories can be traced as far back as \cite[Section 2.2]{Ver92}, where the
author studied the functor \( \VCat \to \CatV \) for presheaf categories \(
\cat V \). It was shown in \cite[Theorem 9.10]{Luc18a} that for a suitable
base category \( \cat V \), the category \( \VCat \) of enriched \( \cat V
\)-categories can be fully embedded into the category \( \CatV \) of
categories internal to \( \cat V \), enabling us to view enriched \( \cat V
\)-categories as categories with \textit{discrete} object-of-objects internal
to \( \cat V \)\footnote{This embedding result was later studied in more
detail in \cite{CFP17}.}. This observation is, for example, employed in the
study of descent theory of enriched categories (see \cite[Theorem
9.11]{Luc18a} and \cite{Pre23}). The aim of this work is to construct such an
embedding in the setting of generalized multicategories, which we recall
below.

Multicategories, defined in \cite[p. 103]{Lam69}, are structures that
generalize categories, by allowing the domains of morphisms to consist of a
finite list of objects. The most quintessential example is the multicategory
\( \Vect \), whose objects are vector spaces, and whose morphisms are
multilinear maps.  Their ``multicomposition'' and the description of the
analogous notions of associativity and identity can succinctly be described
via the free monoid monad on \( \Set \). More precisely, multicategories can
be formalized by considering the equipment \( \Span(\Set) \) of spans in \(
\Set \) (see \cite[p. 22]{Bén67}), and extending the free monoid monad to a
suitable monad \( (-)^* \) on \( \Span(\Set) \) (see \cite[Corollary
A.4]{Her00}).

\textit{Generalized multicategories} have since been developed in various
contexts, abstracting the notion of ordinary multicategories by replacing the
monad \( (-)^* \) on \( \Span(\Set) \) by a \textit{suitable notion of monad
on a pseudodouble category}. 

\textit{Enriched \(T\)-categories} were first introduced in \cite{CT03} with
the terminology \textit{\( (T, \cat V) \)-categories}. In this setting, the
category of \( (T, \cat V) \)-categories is obtained out of the so-called
\textit{lax extension} of a monad on \(\Set\) to a suitable monad on \( \VMat
\), the ubiquitous equipment of \( \cat V \)-matrices (see
\cite[Section~2]{CT03}). For instance, when \( \cat V \) is a suitable
quantale, the ultrafilter monad \( \mathfrak U \) on \(\Set\) admits a lax
extension \( \overline{\mathfrak U} \) to \( \VMat \) \cite[Section 8]{CT03}.
In particular, when \( \cat V = 2 \), we have an equivalence \( \Top \eqv
(\overline{\mathfrak U}, 2) \dash \Cat \) (first observed in \cite{Bar70}) and
\( (\overline{\mathfrak U},[0,\infty]) \dash \Cat \) is equivalent to the
category of approach spaces.

\textit{Internal \(T\)-categories} were introduced in \cite{Bur71} and
\cite{Her00}. For a category \( \cat B \) with pullbacks, the former defines
\(T\)-categories for any monad \(T\) on \( \cat B \), while the latter
considers \(T\) to be a \textit{cartesian} monad on \( \cat B \).  A cartesian
monad \(T\) on \( \cat B \) induces a strong monad on the equipment \( \SpanB
\) of spans in \( \cat B \). In this setting, we can obtain the category \(
\CatTB \) of \(T\)-categories internal to \( \cat B \). As examples, we
recover the category of ordinary multicategories by considering \(
\Cat((-)^*,\Set) \), and letting \( \fc \) be the free category monad on \(
\Grph \), we obtain the category \( \VDbCat = \Cat(\fc, \Grph) \) of virtual
double categories.

The main goal of this paper is to construct an embedding \( \TtVCat \to \CatTV
\) from a category \( \cat V \), and a monad \( T \) on \( \cat V \),
satisfying suitable properties. This endeavor is motivated by generalizing the
study of the descent theoretical results in \cite[Section 9]{Luc18a} to the
multicategorical setting; namely, the reflection of \textit{effective descent
morphisms} (see \cite{JT94}, \cite[Section 3]{Luc21}) along such an embedding
\( \TtVCat \to \CatTV \). Paired with the results of \cite{PL23}, we obtain
sufficient conditions for effective descent morphisms in various classes of
categories of enriched multicategories \cite{CT03}. This expands our
understanding of effective descent functors between categorical structures --
a survey on this topic can be found in \cite{Pre24x}.

To this end, it is desirable to work in a general setting where these various
notions of generalized multicategories can be uniformly studied and compared
with one another. This was, in part, accomplished by the work of \cite{CS10},
where the notion of \textit{\(T\)-monoids} was introduced, unifying the
several approaches to the theory of generalized multicategories.  To be
precise, these \(T\)-monoids are the \textit{horizontal lax algebras} induced
by a monad \( T = (\bicat E, T, e, m) \) in the 2-category \( \VDbCat \) of
virtual double categories, lax functors and vertical transformations.  These
objects have a natural structure of a virtual double category, which we denote
here by \( \LaxTHAlg \). 

This general setting ought to provide us an ``internalization'' functor \(
\TtVCat \to \CatTV \) obtained from the comparison \( \VMat \to \SpanV \), and
the induced monad \(T\) on \(\SpanV \). However, \cite{CS10} does not provide
a notion of \textit{change-of-base} induced by an appropriate notion of
morphism \( S \to T \) of monads, where \( S = (\bicat D, S, e, m) \) is
another monad in \( \VDbCat \). We remark this was left as future work in
\cite[{}4.4]{CS10}.

It should be noted that the study of change-of-base functors has been studied
in each specific setting of generalized multicategories. In \cite[Section
6.7]{Lei04}, the author provides such constructions for the internal case, and
\cite[Sections 5, 6]{CT03} treats two particular families of monad morphisms
for the enriched case. To establish a relationship between the enriched and
internal structures, we expand on the work of these authors, with the goal of
providing a convenient environment to produce and study such a functor \(
\TtVCat \to \CatTV \) from simpler tools.

We must mention that our approach diverges from the techniques and tools
developed in \cite{CS10}. Firstly, we must restrict our scope from virtual
double categories to pseudodouble categories, as we need to work with
\textit{(op)lax horizontal transformations}, which require horizontal
composition to be defined. Secondly, instead of using (op)cartesian 2-cells
and their universal properties, we opted to use a ``mate theory'' of conjoints
and companions to prove our results, mostly to obtain explicit formulas.
Lastly, this is far from the full scope of the project mentioned in
\cite[{}4.4]{CS10}, as we merely study the underlying categories and functors
of the 2-dimensional structures formed by these horizontal lax algebras.
Instead, we fall back on an \textit{ad-hoc} approach for the natural
transformations between the functors induced by monad (op)lax morphisms,
leaving a treatment of the complete story for future work. 

\subsection*{Outline of the paper:}

We begin by reviewing the notion of \textit{pseudodouble category} in
Section~\ref{sect:psdbcats}, first introduced in \cite{GP99}, and the two
dimensional structures formed by these. For pseudodouble categories \( \bicat
D, \bicat E \), the structures consisting of
\begin{itemize}[label=--]
  \item
    \textit{lax functors} \( \bicat D \to \bicat E \) as 0-cells,
  \item
    \textit{vertical transformations} as vertical\footnote{In accordance with
    \cite{Lei04, CS10}, we take the vertical composition and identities to be
    strictly associative and unital.} 1-cells,
  \item
    \textit{(op)lax horizontal transformations} as horizontal 1-cells,
  \item
    \textit{generalized modifications} as 2-cells.
\end{itemize}
are, by Theorem~\ref{thm:psdbcat.closed}, pseudodouble categories \(
\Lax_\lax(\bicat D,\bicat E) \) (\( \Lax_\opl(\bicat D, \bicat E) \)).  We
also have a third double category \( \PsDbCat \)
(Proposition~\ref{thm:dbcat.psdbcat}) consisting of
\begin{itemize}[label=--]
  \item
    \textit{pseudodouble categories} as 0-cells,

  \item
    \textit{(op)lax} functors as (vertical) horizontal 1-cells,

  \item
    \textit{generalized vertical transformations} as 2-cells.
\end{itemize}
The pseudodouble categories that concern our study are the following:
\begin{itemize}[label=--]
  \item
    the pseudodobule category \( \VMat \) of \( \cat V \)-matrices for
    suitable monoidal categories \( \cat V \), 
  \item
    the pseudodouble category \( \SpanB \) of spans in \( \cat B \), for
    categories \( \cat B \) with pullbacks,
  \item
    the double category of lax \(T\)-algebras, for \( T \) a pseudomonad on a
    2-category \( \bicat B \).
\end{itemize}
We will furthermore review the double categorical structure of the last item.

Let \( \cat V \) be a distributive category with finite limits.
Section~\ref{sect:spanmat} is devoted to studying the pseudodouble categories
\( \VMat \) and \( \SpanV \), and the (op)lax functors induced by the
adjunction \( - \pt \trm \adj \cat V(\trm,-) \colon \cat V \to \Set \). We
confirm these functors induce \(\Cat\)-graph morphisms
\begin{equation*}
  - \pt \trm \colon \VMat \to \SpanV \quad \text{and} \quad
  \cat V(\trm,-) \colon \SpanV \to \VMat, 
\end{equation*}
which give us an adjunction \( - \pt \trm \adj \cat V(\trm,-) \) in the
2-category \( \Grph(\Cat) \) (Lemma~\ref{lem:grph.cat.adj}). We also prove
that \( - \pt \trm \colon \VMat \to \SpanV \) defines an oplax functor of
pseudodouble categories (Proposition~\ref{thm:copower.oplax}). Using
techniques from the following couple of sections, we obtain the following
\begin{equation*}
  \begin{tikzcd}
    \VMat \ar[r,bend left,"- \pt \trm"{name=A}]
    & \SpanV \ar[l,bend left,"\cat V(\trm{,}-)"{name=B,below}]
    \ar[from=A,to=B,phantom,"\adj" {anchor=center, rotate=-90}]
  \end{tikzcd}
\end{equation*}
which is a generalized notion of adjunction -- a \textit{conjunction} -- in
the double category \( \PsDbCat \).

Section~\ref{sect:conj.comp} aims to recall the notions of ``adjoint'' in
pseudodouble categories: \textit{conjoints} and \textit{companions}. These
were first introduced in \cite{GP04}, under different terminology. We provide
an explicit description of ``mate theory'' for these objects (also studied in
\cite{GP04, Shu11, GGV24}), analogous to the mate theory for adjunctions. We
also take the opportunity to work out some known results for three reasons:
first, to fix technical notation for subsequent sections; second, to serve as
examples on their use; and finally, to keep this work self-contained.  This
Section culminates in our first contribution, crucial to construct functor
between categories of lax horizontal algebras,
Theorem~\ref{thm:conjoint.closed}. It states that, if \( \bicat E \) is
\textit{conjoint closed}, then so is \( \Lax_\lax(\bicat D, \bicat E) \). 

In Section~\ref{sect:pstalg}, we explicitly establish an equivalence
(Proposition~\ref{thm:psdbcat.as.ps-t-alg}) between the double category \(
\PsDbCat \) and the double category of pseudo-algebras for the free internal
\( \Cat \)-category 2-monad on the 2-category \( \Grph \), with the goal of
making the tools of two-dimensional algebra \cite{BKP89, Kel74, Luc16}
available to the theory of pseudodouble categories. In particular, via
doctrinal adjunction \cite[Theorem 1.4.11]{LucTh}, we conclude that \( \cat
V(\trm,-) \colon \SpanV \to \VMat \) is a lax functor, and is the conjoint of
\( - \pt \trm \colon \VMat \to \SpanV \) in the double category \( \PsDbCat \).

After recalling the notion of horizontal lax algebra from \cite{CS10}, in
Section~\ref{sect:base.change} we prove Theorem~\ref{thm:base.change}; it
states that any monad lax morphism \( (G,\psi) \colon T \to S \) in \(
\PsDbCat_\lax \) induces a \textit{change-of-base} functor \( G_! \colon
\LaxTHAlg \to \LaxSHAlg \), and any monad oplax morphism \( (F,\phi) \colon S
\to T \) satisfying a suitable conditon also induces a change-of-base functor
\( F_! \colon \LaxSHAlg \to \LaxTHAlg \). We close this section by comparing
our constructions with the change-of-base functors for generalized
multicategories considered in \cite{Lei04} and \cite{CT03}.

In Section~\ref{sect:induced.adjunction}, we consider a conjunction 
\begin{equation*}
  \begin{tikzcd}
    S \ar[r,bend left,"{(F,\,\phi)}"{name=A}]
    & T \ar[l,bend left,"{(G,\,\psi)}"{name=B,below}]
    \ar[from=A,to=B,phantom,"\adj" {anchor=center, rotate=-90}]
  \end{tikzcd}
\end{equation*}
in the double category \( \Mnd(\PsDbCat_\lax) \), and we proceed to the
existence of an adjunction
\begin{equation*}
  \begin{tikzcd}
    \LaxSHAlg \ar[r,bend left,"F_!"{name=A}]
    & \LaxTHAlg \ar[l,bend left,"G_!"{name=B,below}]
    \ar[from=A,to=B,phantom,"\adj" {anchor=center, rotate=-90}]
  \end{tikzcd}
\end{equation*}
between the induced change-of-base functors; this is
Theorem~\ref{thm:lax.alg.adj}. We also study the conditions for invertibility
of unit and counit of such an adjunction, stated in
Lemma~\ref{lem:counit.iso.when} and Corollary~\ref{cor:unit.iso.when}.
Finally, after instanciating these results to the settings considered both in
\cite{Lei04} and \cite{CT03, HST14}, we take the opportunity to point out some
of obstacles to the double pseudofunctoriality of \( \LaxHAlg \).

We devote Section~\ref{sect:ext.cats} to the study of extensive categories
\cite{CLW93}.  When \( \cat C \) is a lextensive category, we provide a
description of its \textit{free coproduct completion} \cite{CLW93, BJ01,
Web07, Pre23, LV24}, denoted \( \FamC \), via Artin glueing
(Lemma~\ref{lem:famc.comma}), from which we deduce that the coproduct functor
\( \sum \colon \FamC \to \cat C \) preserves finite limits. Studying limits of
fibered categories \cite{Gra66}, we obtain Theorem~\ref{thm:pb.index.coprod}:
it confirms that, in a lextensive category, the coproduct of a
``pullback-indexed'' family of pullback diagrams is itself a pullback diagram.
This result is extensively employed, as illustrated in the remaining results
of this section as well as subsequent ones.

The final groundwork is laid down in Section~\ref{sect:disc.morph}. Via a
``structure transfer''-type of result (Proposition~\ref{prop:induced.monad}),
we are able to construct a monad \( \Tt \) on \( \VMat \) from a monad \(T\)
on \( \SpanV \), which is, in turn, induced by a cartesian monad \(T\) on a
lextensive category \( \cat V \) \cite{Her00}. In fact, we obtain a conjunction 
\begin{equation}
  \label{intro:vmat.adj}
  \begin{tikzcd}
    (\Tt, \VMat)
      \ar[r,bend left,"{(- \pt \trm,\, \heps_{T(-\pt \trm)})}"{name=A}]
    & (T, \SpanV)
      \ar[l,bend left,
        "{(\cat V(\trm,-),\, \cat V(\trm,T\heps))}"{name=B,below}]
    \ar[from=A,to=B,phantom,"\adj" {anchor=center, rotate=-90}]
  \end{tikzcd}
\end{equation}
in the double category \( \Mnd(\PsDbCat_\lax) \). However, only under a
suitable condition does this induce an adjunction
\begin{equation}
  \label{intro:tvcat.adj}
  \begin{tikzcd}
    \TtVCat \ar[r,bend left,"- \pt \trm"{name=A}]
    & \CatTV. \ar[l,bend left,"\cat V(\trm{,}-)"{name=B,below}]
    \ar[from=A,to=B,phantom,"\adj" {anchor=center, rotate=-90}]
  \end{tikzcd}
\end{equation}
The goal of this Section is to study this extra condition. In the case \( -
\pt \trm \colon \Set \to \cat V \) is fully faithful, we obtain
Theorem~\ref{thm:heps.strong.conj}, characterizing this condition in terms of
a notion of \textit{fibrewise discreteness} of a monad. Finally, we check that
most of the commonly studied cartesian monads on lextensive categories~\( \cat
V \) are fibrewise discrete, provided \( - \pt \trm \colon \Set \to \cat V \)
is fully faithful.

Section~\ref{sect:embedding} contains our main results. Let \( \cat V \) be a
lextensive category such that \( - \pt \trm \colon \Set \to \cat V \) is fully
faithful, and let \(T\) be a fibrewise discrete, cartesian monad on \( \cat V
\). We also denote the induced monad on \( \SpanV \) by \(T\). Via
Theorem~\ref{thm:lax.alg.adj}, we obtain the (ordinary)
adjunction~\eqref{intro:tvcat.adj} from the conjunction~\eqref{intro:vmat.adj}
in \( \PsDbCat_\lax \) (Theorem~\ref{thm:main.result}).  

We then apply Theorem~\ref{thm:main.result} to study \textit{effective descent
morphisms} in categories of enriched categorical structures in
Section~\ref{sect:descent}. Under an additional technical condition (satisfied
by most of the examples we provided), we confirm that \( \TtVCat \) is
precisely the full subcategory of \( \CatTV \) with a \textit{discrete}
object-of-objects (Theorem~\ref{thm:tvcat.disc.objs}), generalizing \cite[9.10
Theorem]{Luc18a} and \cite[Corollary 4.5]{CFP17}. Via this description, we
confirm that \( - \pt \trm \colon \TtVCat \to \CatTV \) reflects effective
descent morphisms (Lemma~\ref{lem:refl.eff.desc}), and, with the results of
\cite{PL23} pertaining to effective descent morphisms in internal categorical
structures, we provide criteria for an enriched \( (\Tt, \cat V) \)-functor to
be effective for descent (Theorem~\ref{thm:eff.desc.criteria}). We finalize
the section by studying the above examples.

Finally, we leave some final remarks in Section~\ref{sect:epilogue}. Namely,
we sketch the construction of a change-of-base functor \( \LaxTHAlg \to
\LaxModTHAlg \) induced by the embedding \( \bicat D \to \Mod(\bicat D) \)
\cite[Section 5.3]{Lei04} and a normal lax monad \( T \) on \( \bicat D \),
and various avenues for future work, among which we have the problem of
monadicity of categories of horizontal lax algebras, some comments on the
categorical properties of the categories of generalized multicategories,
topics about double fibrations \cite{CLPS22} and descent theory, and finally
further comments on other notions of change-of-base.

  \section{Structure of double categories}
    \label{sect:psdbcats}
    Double categories were first defined in~\cite{Ehr63}, and pseudodouble
categories were later introduced in~\cite[Section 7]{GP99}, which generalize
double categories by requiring the associativity and identity axioms for
vertical morphisms to hold only up to coherent isomorphisms. Other closely
related double-dimensional structures are: \cite{Woo82, Woo85}, which introduced
the notion of \textit{(proarrow) equipment}, which is an important class of
pseudodouble categories; \cite[Section~5.2]{Lei04} for the \textit{unbiased}
analogue of pseudodouble categories; \cite{Mar06}, studying
\textit{pseudocategories} internal to a 2-category; \cite{Fio07}, which
inspired the 2-cell notation present in this work; \cite{CS10}, which studies
the more general notion of \textit{virtual double category}; \cite{LS12},
which studies a notion of \( \mathscr{F} \)-categories. 

The original convention of \cite{GP99} is transposed in \cite{Lei04, Fio07,
Shu08, CS10, GGV24}, among others; that is, the roles of the vertical and
horizontal morphisms were interchanged in the latter references. For this
reason, we follow the convention that
\begin{itemize}[label=--]
  \item
    \textit{vertical 1-cells} satisfy associativity and unity axioms
    \textit{strictly}, and
  \item
    \textit{horizontal 1-cells} satisfy associativity and unity axioms
    \textit{up to coherent 2-cells}.
\end{itemize}

After reviewing the notion of \textit{pseudodouble category}, we also recall
the notion of \textit{(op)lax functor} between pseudodouble categories, and
the notion of \textit{generalized vertical transformation} \cite[{}2.2]{GP04}
(therein called ``\textit{roughly speaking -- horizontal transformations}''),
of which \textit{icons} \cite{Lac10} between lax functors of bicategories are
a special case. 

We also introduce the notion of \textit{(op)lax horizontal transformations},
and the corresponding notion of modifications, which is the natural
generalization of lax natural transformations and modifications from
bicategory theory \cite{Gra69, Gra74} to the setting of pseudodouble
categories. Furthermore, we study the (pseudodouble) categorical structures
formed by these objects, which form our basic vocabulary for the remainder of
this work.

\subsection{Pseudodouble categories:}

A pseudodouble category~\(\bicat D\) consists of:
\begin{itemize}[label=--]
  \item
    A category~\( \bicat D_0 \), whose objects are \textit{0-cells}, 
    whose morphisms are \textit{vertical 1-cells}, whose composition is
    denoted by~\( \circ \), and whose identities at 0-cells~\(x\) are denoted
    by~\( \id_x \). Composition and identities of vertical 1-cells in \(
    \bicat D_0 \) are said to be \textit{vertical} as well.

  \item
    A category~\( \bicat D_1 \), whose objects are \textit{horizontal
    1-cells}, whose morphisms are \textit{2-cells}, whose composition is
    denoted by~\( \circ \), and whose identities at horizontal 1-cells~\( r \)
    are denoted by~\( \id_r \). Composition and identities of 2-cells in \(
    \bicat D_1 \) are said to be \textit{vertical} as well.

  \item
    \textit{Vertical} domain and codomain functors~\( \dom,\, \cod \colon
    \bicat D_1 \to \bicat D_0 \).

  \item
    A \textit{unit} functor~\( 1 \colon \bicat D_0 \to \bicat D_1 \). For each
    0-cell~\(x\), the unit horizontal 1-cell is denoted~\( 1_x \), and
    likewise for vertical 1-cells and their respective unit 2-cells.

  \item
    A \textit{horizontal composition} functor~\( - \cdot - \colon \bicat D_2
    \to \bicat D_1 \), where~\( \bicat D_2 \), given by pullback of~\( \dom \)
    and~\( \cod \), is the category of \textit{composable pairs} of horizontal
    1-cells and 2-cells.
\end{itemize}
This data must satisfy~\( \dom(1_x) = \cod(1_x) = x \) for 0-cells~\(x\), and
analogously for vertical 1-cells, and~\( \dom(s \cdot r) = \dom(r) \),~\(
\cod(s \cdot r) = \cod(s) \) for each composable pair~\( r, s \) of horizontal
1-cells, and analogously for composable pairs of 2-cells. 

If~\( \phi \colon r \to s \) is a 2-cell (a morphism in~\( \bicat D_1 \))
such that~\( f = \dom(\phi) \colon w \to y \) and~\( g = \cod(\phi) \colon
x \to z \), we depict~\( \phi \) as in Diagram \eqref{eq:2-cell.diag}.
\begin{equation}
  \label{eq:2-cell.diag}
  \begin{tikzcd}
    w \ar[d,"f",swap] \ar[r,"r"{name=A}]
      & x \ar[d,"g"] \\
    y \ar[r,"s"{name=B},swap] & z
    \ar[from=A,to=B,phantom,"\phi"]
  \end{tikzcd}
\end{equation}
We say~\( \phi \) is a \textit{globular} 2-cell if~\( f \) and~\( g \) are
identities. Given a category~\( \bicat A \), we say a natural
transformation~\( \phi \colon F \to G \) of functors~\( \bicat A \to \bicat
D_1 \) is a \textit{natural 2-cell}. Furthermore, we say it is
\textit{globular} if~\( \phi_a \) is a globular 2-cell for every object~\(a\)
in~\( \bicat A \), and that \( \phi \) is \textit{invertible} if it is a
natural isomorphism.

We also have data given by:
\begin{itemize}[label=--]
  \item
    globular, natural, invertible 2-cells
    \begin{equation*}
      \lambda \colon 1_{\cod(-)} \cdot - \to -,
      \qquad 
      \rho \colon - \cdot 1_{\dom(-)} \to - 
    \end{equation*} 
    of functors~\( \bicat D_1 \to \bicat D_1 \), the \textit{left} and
    \textit{right unitors}, respectively. 

  \item
    and a globular, natural, invertible 2-cell~\( \alpha \colon (-_1 \cdot
    -_2) \cdot -_3 \to -_1 \cdot (-_2 \cdot -_3) \) of functors~\( \bicat D_3
    \to \bicat D_1 \), the \textit{associator}, where~\( \bicat D_3 \) is the
    category of \textit{composable triples}.
\end{itemize}
These must also satisfy the following \textit{coherence conditions}:
\begin{enumerate}[label=(\alph*),noitemsep]
  \item
    \label{enum:unit.unit.red}
    We have~\( \gamma_{1_x} = \id_{1_x \cdot 1_x} \), where we define~\(
    \gamma = \rho^{-1} \circ \lambda \). 

  \item
    \label{enum:left.left.red}
    The following diagram commutes
    \begin{equation*}
      \begin{tikzcd}[column sep=10]
        (1_z \cdot s) \cdot r \ar[rr,"\alpha_{r,s,1_z}"] 
                              \ar[rd,"\lambda_s \cdot \id_r",swap]
        && 1_z \cdot (s \cdot r) \ar[ld,"\lambda_{s \cdot r}"] \\
        & s \cdot r
      \end{tikzcd}
    \end{equation*}
    for each pair of horizontal 1-cells~\(r \colon x \to y \),~\( s \colon y
    \to z \).

  \item
    \label{enum:right.right.red}
    The following diagram commutes
    \begin{equation*}
      \begin{tikzcd}[column sep=10]
       & s \cdot r  \\
       (s\cdot r) \cdot 1_x \ar[rr,"\alpha_{1_x,r,s}",swap] 
                            \ar[ru,"\rho_{s\cdot r}"] 
       && s \cdot (r \cdot 1_x) \ar[lu,"\id_s \cdot \rho_r",swap]
      \end{tikzcd}
    \end{equation*}
    for each pair of horizontal 1-cells~\(r \colon x \to y \), \( s \colon y
    \to z \).

  \item
    \label{enum:unit.coherence}
    The following diagram commutes
    \begin{equation*} 
      \begin{tikzcd}[column sep=10]
        (s\cdot 1_y) \cdot r \ar[rr,"\alpha_{r,1_y,s}"] 
                             \ar[rd,"\rho_s \cdot \id_r",swap]
        && s \cdot (1_y \cdot r) \ar[ld,"\id_s \cdot \lambda_r"] \\
        & s \cdot r
      \end{tikzcd}
    \end{equation*}
    for each pair of horizontal 1-cells~\(r \colon x \to y \), \(s \colon y
    \to z\).

  \item
    \label{enum:assoc.coherence}
    The following diagram commutes
    \begin{equation*}
      \begin{tikzcd}[column sep=30]
        ((t \cdot s) \cdot r) \cdot q \ar[r,"\alpha_{q,r,t\cdot s}"] 
                                      \ar[d,"\alpha_{r,s,t} \cdot \id_q",swap] 
        & (t \cdot s) \cdot (r \cdot q) \ar[r,"\alpha_{r \cdot q,s,t}"] 
        & t \cdot (s \cdot (r \cdot q)) \\
        (t \cdot (s \cdot r)) \cdot q \ar[rr,"\alpha_{q,s\cdot r,t}",swap]
        && t \cdot ((s \cdot r) \cdot q) 
          \ar[u,"\id_t \cdot \alpha_{q,r,s}",swap]
      \end{tikzcd}
    \end{equation*}
    for each quadruple of composable horizontal 1-cells~\( q,r,s,t \).
\end{enumerate}
Unless there is a need for disambiguation, subscripts will be omitted.  If~\(
\lambda, \rho \) and~\( \alpha \) are the identity transformations, we say~\(
\bicat D \) is a \textit{double category}, or \textit{strict double category}
for emphasis.

\begin{remark}
  The coherence conditions described earlier are essentially analogous to
  those given for monoidal categories \cite{Mac63, Kel64}. Thus, the argument
  presented in \cite{Kel64} can be applied here, demonstrating that the
  coherence axioms~\ref{enum:unit.unit.red}--\ref{enum:right.right.red} can be
  deduced from the other two. We state and prove in the proposition that
  follows.
\end{remark}

\begin{proposition}
  \label{redundant.axioms}
  The coherence conditions~\ref{enum:unit.unit.red}, \ref{enum:left.left.red}
  and~\ref{enum:right.right.red} are redundant.
\end{proposition}

\begin{proof}
  First, observe that~\ref{enum:left.left.red} is the horizontal dual
  of~\ref{enum:right.right.red}, so it is sufficient to
  verify~\ref{enum:unit.unit.red} and~\ref{enum:left.left.red}.

  We may obtain~\ref{enum:unit.unit.red} from the remaining conditions: we
  have an equality of 2-cells~\( (1_x \cdot 1_x) \cdot 1_x \to 1_x \)
  \begin{equation*}
    \lambda \circ (\lambda \cdot 1) 
      = \lambda \circ \lambda \circ \alpha 
      = \lambda \circ (1 \cdot \lambda) \circ \alpha
      = \lambda \circ (\rho \cdot 1)
  \end{equation*}
  by~\ref{enum:left.left.red}, naturality of~\( \lambda \),
  and~\ref{enum:unit.coherence}. We deduce that~\( \lambda \cdot 1 = \rho
  \cdot 1 \), and since~\( \rho \) is a natural isomorphism, we conclude
  that~\( \lambda = \rho \).

  To prove~\ref{enum:left.left.red} given only~\ref{enum:unit.coherence}
  and~\ref{enum:assoc.coherence}, we consider the following diagram:
  \begin{equation*}
    \begin{tikzcd}
      & 1_z \cdot (1_z \cdot (s\cdot r))  
      	\ar[d,"\id \cdot \lambda" description] \\
      & 1_z \cdot (s \cdot r)          \\
      1_z \cdot ((1_z \cdot s) \cdot r)   
      	\ar[uur,"\id \cdot \alpha",bend left] 
      	\ar[ur, "\id \cdot (\lambda \cdot \id)" description]
        \ar[d,  "\alpha^{-1}",swap] 
      & (1_z \cdot s) \cdot r           
      	\ar[u,  "\alpha" description] 
      & (1_z \cdot 1_z) \cdot (s \cdot r) 
      	\ar[uul,"\alpha",bend right,swap] 
      	\ar[ul, "\rho \cdot \id" description]     \\
      (1_z \cdot (1_z \cdot s)) \cdot r
      	\ar[rr, "\alpha^{-1} \cdot \id",swap]
      	\ar[ur, "(\id \cdot \lambda) \cdot \id" description] 
      && ((1_z \cdot 1_z) \cdot s) \cdot r 
      	\ar[ul, "(\rho \cdot \id) \cdot \id" description] 
      	\ar[u,  "\alpha",swap]
    \end{tikzcd}
  \end{equation*}
  Except for the top left triangle, every inner polygon commutes either
  by~\ref{enum:unit.coherence} or by naturality of~\( \alpha \). The outer
  pentagon is an instance of~\ref{enum:assoc.coherence}, so we conclude that
  the top left triangle commutes. Since~\(\lambda\) is a natural isomorphism,
  the result follows.
\end{proof}

\subsection{Lax functors:}

Let~\( \bicat D \), \( \bicat E \) be double categories. A lax functor~\( F
\colon \bicat D \to \bicat E \) consists of:
\begin{itemize}[label=--]
  \item
    a pair of functors~\( F_0 \colon \bicat D_0 \to \bicat E_0 \) and~\( F_1
    \colon \bicat D_1 \to \bicat E_1 \), whose subscripts we will often omit, 

  \item
    a globular natural 2-cell~\( \e F_x \colon 1_{Fx} \to F(1_x) \) for each
    0-cell~\(x\) (unit comparison),

  \item
    a globular natural 2-cell~\( \m F_{f,g} \colon Fg \cdot Ff \to F(g\cdot f)
    \) for each composable pair of horizontal 1-cells~\( f,g \) (composition
    comparison),
\end{itemize}
satisfying the following properties:
\begin{itemize}[label=--]
  \item
    \( \dom(Fr) = F(\dom(r)) \) and~\( \cod(Fr) = F(\cod(r)) \) for every
    horizontal 1-cell~\(r\) and likewise for 2-cells,

  \item
    unit comparsion coherence: the following diagrams commute
    \begin{equation*}
      \begin{tikzcd}
        1_{Fy} \cdot Fr \ar[d,"\lambda",swap] \ar[r,"\e F \cdot Fr"]
        & F1_y \cdot Fr \ar[d,"\m F"]  \\
        Fr \ar[r,"F\lambda^{-1}",swap] & F(1_y \cdot r) 
      \end{tikzcd}
      \quad
      \begin{tikzcd}
        Fr \cdot 1_{Fx} \ar[d,"\rho",swap] \ar[r,"\id \cdot \e F"]
        & Fr \cdot F1_x \ar[d,"\m F"] \\
        Fr \ar[r,"F\rho^{-1}",swap] & F(r \cdot 1_x)
      \end{tikzcd}
    \end{equation*}
    for every horizontal 1-cell~\( r \colon x \to y \),

  \item
    composition comparison coherence: the following diagram commutes
    \begin{equation*}
      \begin{tikzcd}
        (Ft \cdot Fs) \cdot Fr \ar[d,"\alpha",swap] \ar[r,"\m F \cdot \id"]
        & F(t \cdot s) \cdot Fr \ar[r,"\m F"] 
        & F((t \cdot s) \cdot r) \ar[d,"F\alpha"] \\
        Ft \cdot (Fs \cdot Fr) \ar[r,"\id \cdot \m F",swap] 
        & Ft \cdot F(s \cdot r) \ar[r,"\m F",swap]
        & F(t \cdot (s \cdot r))
      \end{tikzcd}
    \end{equation*}
    for every composable triple of horizontal 1-cells~\( r, s, t \).
\end{itemize}

Dually, an \textit{oplax functor}~\( F \colon \bicat B \to \bicat C\) is the
horizontally dual notion, obtained by reversing the direction of the natural
2-cells \( \e F \) and \( \m F \). 

If~\( \e F \) is invertible, then we say~\(F\) is \textit{normal}, and if
both~\( \e F, \m F \) are invertible, then we say~\(F\) is a \textit{strong
functor} (which can be seen both as a lax and an oplax functor).  If \( \e F,
\m F \) are identities, we say \(F\) is a \textit{strict functor}.

\begin{proposition}[]
  \label{prop:dbcat.vert}
  Composition of (op)lax functors is well-defined, associative, and has
  identities, and restricts to strong functors, as well as strict functors. 

  In other words, pseudodouble categories are the objects of the categories \(
  \PsDbCat_\lax \) and \( \PsDbCat_\opl \), whose morphisms are lax functors
  and oplax functors, respectively. There are also common subcategories of
  strong (strict) functors.
\end{proposition}

\begin{proof}
  The identity functor \( \id_{\bicat D} \) on a pseudodouble
  category~\(\bicat D\) is given by the pair of identity functors on~\( \bicat
  D_0 \) and~\( \bicat D_1 \), and the coherences are identities as well.
  Thus, the coherence conditions are trivially satisfied, so \( \id_{\bicat D}
  \) is a strong functor.

  For lax functors~\(F \colon \bicat C \to \bicat D\) and~\(G \colon \bicat D
  \to \bicat E\), define the composite lax functor~\(GF\) to be given by 
  \begin{itemize}[label=--]
    \item
      \( (GF)_0 = G_0 F_0 \), 
    \item
      \( (GF)_1 = G_1 F_1 \), 
    \item
      \( \e{GF} = G\e F \circ \e G \), 
    \item
      \( \m{GF} = G\m F \circ \m G \).
  \end{itemize}
  To verify~\(GF\) is a lax functor, first we note that~\( \e{GF} \) and~\(
  \m{GF} \) are globular natural transformations, since~\(G\) is a lax
  functor. Next, we observe that the following diagrams
  \begin{equation*}
    \begin{tikzcd}[column sep=large,row sep=large]
      \cdot          \ar[dd,"\lambda",swap]
                     \ar[r,"\e G \cdot GFf"] 
        & \cdot      \ar[r,"G\e F \cdot GFf"] 
                     \ar[d,"\m G",swap] 
        & \cdot      \ar[d,"\m G"] \\
        & \cdot      \ar[r,"G(\e F \cdot Fr)",swap]
        & \cdot      \ar[d,"G(\m F)"] \\
      \cdot          \ar[ur,"G\lambda^{-1}" description] 
                     \ar[rr,"GF\lambda^{-1}",swap] 
        && \cdot
       \end{tikzcd}
    \qquad
    \begin{tikzcd}[column sep=large,row sep=large]
      \cdot          \ar[d,"\m G",swap]
        & \cdot      \ar[d,"\m G"]
                     \ar[l,"GFf \cdot G\e F",swap] 
        & \cdot      \ar[l,"GFf \cdot \e G",swap] 
                     \ar[dd,"\rho"]            \\
      \cdot          \ar[d,"G(\m F)",swap]
        & \cdot      \ar[l,"G(Fr \cdot \e F)",swap] \\
      \cdot
        && \cdot     \ar[ul,"G\rho^{-1}" description]  
                     \ar[ll,"GF\rho^{-1}"]
    \end{tikzcd}
  \end{equation*}
  \begin{equation*}
    \begin{tikzcd}[column sep=large, row sep=large]
      && \cdot \ar[rd,"G(\m F \cdot Fr)"]      \\
      &   \cdot \ar[rd,"G\m F \cdot GFr",swap]
                 \ar[ru,"\m G"]
      && \cdot \ar[rd,"G\m F"]                 \\
           \cdot \ar[ru,"\m G \cdot GFr"]
                 \ar[d,"\alpha",swap]
      && \cdot \ar[ru,"\m G",swap]
                 \ar[d,"G\alpha"]
      && \cdot \ar[d,"GF\alpha"]              \\
           \cdot \ar[rd,"GFt \cdot \m G",swap]
      && \cdot \ar[rd,"\m G"]
      && \cdot                                \\
      &   \cdot \ar[ru,"GFt \cdot G\m F"]
                 \ar[rd,"\m G",swap]
      && \cdot \ar[ru,"G\m F",swap]            \\
      && \cdot \ar[ru,"G(Ft \cdot \m F)",swap]
    \end{tikzcd}
  \end{equation*}
  commute, since every inner polygon commutes: either by coherence (of
  both~\(F\) and~\(G\)) or by naturality of~\(\m G\). These show that the
  coherence conditions for \( GF \) hold.

  Finally, note that the identity functors are the units for lax functor
  composition, and this operation is also associative. All verifications occur
  in suitable categories: function composition on the category of sets,
  functor composition on~\( \Cat \), and 2-cell composition on the
  hom-categories (plus the composition preservation by the functors between
  them).
\end{proof}

\begin{remark}
  For (op)lax functors \(F \colon \bicat C \to \bicat D\) ,\(G \colon \bicat D
  \to \bicat E \), we note that if ~\(\e F, \e G\) (respectively, \(\m F, \m G
  \)) are both isomorphisms (identities), then so is~\( \e{GF} \)
  (respectively,~\( \m{GF} \)). Thus, we obtain subcategories of \(
  \PsDbCat_\lax \) and \( \PsDbCat_\opl \) with the same objects, and strong
  (strict) functors.
\end{remark}

\subsection{Vertical transformations}

We fix lax functors~\( H \colon \bicat A \to \bicat B \), \( K \colon \bicat C
\to \bicat D \) and oplax functors~\( F \colon \bicat A \to \bicat C \) and~\(
G \colon \bicat B \to \bicat D \).  A \textit{generalized vertical
transformation}~\( \phi \), depicted as
\begin{equation*}
  \begin{tikzcd}
    \bicat A \ar[d,"F",swap] \ar[r,"H"{name=A}]
      & \bicat B \ar[d,"G"] \\
    \bicat C \ar[r,"K"{name=B},swap] & \bicat D
    \ar[from=A,to=B,phantom,"\phi"]
  \end{tikzcd}
\end{equation*}
so that the vertical domain, codomain are~\(F\), ~\(G\) respectively, and
horizontal domain, codomain given by~\( H \),~\( K \) respectively, is given
by  
\begin{itemize}[label=--]
  \item
    a natural transformation~\( \phi_0 \colon G_0H_0 \to K_0F_0 \),

  \item
    and a natural transformation~\( \phi_1 \colon G_1H_1 \to K_1F_1 \),
\end{itemize}
whose subscripts will often be omitted, satisfying~\( \dom(\phi_r) =
\phi_{\dom(r)} \) and~\( \cod(\phi_r) = \phi_{\cod(r)} \) for every horizontal
1-cell \(r\), subject to the following coherence conditions
\begin{equation*}
  \begin{tikzcd}[column sep=10]
    & 1_{GHx} \ar[rr,"1_{\phi_x}"]
      && 1_{KFx} \ar[rd,"\e K_{Fx}"] \\
    G1_{Hx} \ar[ru,"\e G_{Hx}"] \ar[rd,"G\e H_x",swap]
      &&&& K1_{Fx} \\
    & GH1_x \ar[rr,"\phi_{1_x}",swap] 
      && KF1_x \ar[ru,"K\e F_x",swap]
  \end{tikzcd}
  \quad
  \begin{tikzcd}[column sep=7]
    & GHs \cdot GHr \ar[rr,"\phi_s \cdot \phi_r"]
      && KFs \cdot KFr \ar[rd,"\m K_{Fs,Fr}"] \\
    G(Hs \cdot Hr) \ar[ru,"\m G_{Hs,Hr}"] \ar[rd,"G\m H_{s,r}",swap]
      &&&& K(Fs \cdot Fr) \\
    & GH(s \cdot r) \ar[rr,"\phi_{s\cdot r}",swap]
      && KF(s \cdot r) \ar[ru,"K\m F_{s,r}",swap]
  \end{tikzcd}
\end{equation*}
for every 0-cell \(x\) and every composable pair \(r,s\) of horizontal
1-cells.

We say~\( \phi \) is a vertical transformation between lax functors if~\(
F=\id \) and~\( G = \id \), and a vertical transformation between oplax
functors if \( H = \id \) and \( K = \id \). 

\begin{proposition}
  \label{lem:catgrph.psdbcat}
  Lax functors and generalized vertical transformations form a category, and
  the vertical domain, codomain operations define functors to~\( \PsDbCat_\opl
  \).
\end{proposition}

\begin{proof}
  The identity vertical transformation on a lax functor~\(F\) is given by the
  pair~\( \id_{F_0} \) and~\( \id_{F_1} \), which trivially satisifies the
  conditions, and has identity functors as vertical domain and codomain.

  Let \( \phi \), \( \psi \) be the generalized vertical transformations
  depicted below:
  \begin{equation*}
    \begin{tikzcd}
      \bicat A \ar[r,"P"{name=A}] \ar[d,"F",swap]
        & \bicat B \ar[d,"G"] \\
      \bicat C \ar[r,"Q"{name=B},swap] 
        & \bicat D
      \ar[from=A,to=B,phantom,"\phi"]
    \end{tikzcd}
    \qquad
    \begin{tikzcd}
      \bicat C \ar[r,"Q"{name=A}] \ar[d,"H",swap]
        & \bicat D \ar[d,"K"] \\
      \bicat E \ar[r,"R"{name=B},swap] 
        & \bicat F
      \ar[from=A,to=B,phantom,"\psi"]
    \end{tikzcd}
  \end{equation*}
  We define their \textit{vertical composition} \( \psi \circ \phi \) to be
  given by the natural transformation ~\( (\psi \circ \phi)_i = \psi_iF_i
  \circ K_i\phi_i \) for~\( i=0,1 \). Since 
  \begin{equation*}
    \begin{tikzcd}
      && \cdot \ar[rr,"K\phi_s \cdot K\phi_r"]
      && \cdot \ar[rr,"\psi_{Fs} \cdot \psi_{Fr}"]
      && \cdot \ar[rd,"\m R"] \\
      & \cdot \ar[rr,"K(\phi_s \cdot \phi_r)",swap] \ar[ru,"\m K"]
      && \cdot \ar[ru,"\m K",swap] \ar[rd,"K\m Q"]
      &&&& \cdot \\
      \cdot \ar[ru,"K\m G"] \ar[rd,"KG\m P",swap]
      &&&& \cdot \ar[rr,"\psi_{Fs \cdot Fr}"] 
      && \cdot \ar[ru,"R\m H",swap] \\
      & \cdot \ar[rr,"K\phi_{s \cdot r}",swap]
      && \cdot \ar[ru,"KQ\m F"] \ar[rr,"\psi_{F(s\cdot r)}",swap]
      && \cdot \ar[ru,"RH\m F",swap]
    \end{tikzcd}
  \end{equation*}
  and
  \begin{equation*}
    \begin{tikzcd}
      && \cdot \ar[rr,"1_{K\phi}"]
      && \cdot \ar[rr,"1_{\psi_F}"]
      && \cdot \ar[rd,"\e R"] \\
      & \cdot \ar[rr,"K(1_\phi)",swap] \ar[ru,"\e K"]
      && \cdot \ar[ru,"\e K",swap] \ar[rd,"K\e Q"]
      &&&& \cdot \\
      \cdot \ar[ru,"K\e G"] \ar[rd,"KG\e P",swap]
      &&&& \cdot \ar[rr,"\psi_{1_F}"] 
      && \cdot \ar[ru,"R\e H",swap] \\
      & \cdot \ar[rr,"K\phi_1",swap]
      && \cdot \ar[ru,"KQ\e F"] \ar[rr,"\psi_{F1}",swap]
      && \cdot \ar[ru,"RH\e F",swap]
    \end{tikzcd}
  \end{equation*}
  are commutative diagrams, we conclude that~\( \psi \circ \phi \) is a
  generalized vertical transformation, with vertical domain~\( HF \) and
  codomain~\( KG \). Hence, if~\( \phi \) and~\( \psi \) are globular, then so
  is~\( \psi \circ \phi \), so we obtain a subcategory of lax functors and
  vertical transformations.

  Associativity and identity are obtained via componentwise calculation on the
  underlying natural transformations. 
\end{proof}

\begin{proposition}[{\cite[{}2.2]{GP04}}]
  \label{thm:dbcat.psdbcat}
  We have a strict double category~\( \PsDbCat \) with pseudodouble
  categories as 0-cells, lax and oplax functors respectively as horizontal and
  vertical 1-cells, and generalized vertical transformations as 2-cells.
\end{proposition}

\begin{proof}
  The underlying~\( \Cat \)-graph of \( \PsDbCat \) is described in
  Proposition~\ref{lem:catgrph.psdbcat}. 

  For an oplax functor~\(F \colon \bicat A \to \bicat D \), the identity
  natural transformation on~\(F\) defines a generalized vertical
  transformation~\( 1_F \colon 1_{\id} \to 1_{\id} \), whose vertical domain
  and codomain is~\(F\).
  
  If~\( \phi \) and~\( \psi \) are the generalized vertical transformations
  depicted below:
  \begin{equation*}
    \begin{tikzcd}
      \bicat A \ar[r,"P"{name=A}] \ar[d,"F",swap]
        & \bicat B \ar[d,"G"] \\
      \bicat D \ar[r,"Q"{name=B},swap] 
        & \bicat E
      \ar[from=A,to=B,phantom,"\phi"]
    \end{tikzcd}
    \qquad
    \begin{tikzcd}
      \bicat B \ar[r,"R"{name=A}] \ar[d,"G",swap]
        & \bicat C \ar[d,"H"] \\
      \bicat E \ar[r,"S"{name=B},swap] 
        & \bicat F
      \ar[from=A,to=B,phantom,"\psi"]
    \end{tikzcd}
  \end{equation*}
  we denote their \textit{horizontal composition} by~\( \psi \cdot \phi \),
  and is defined by~\( (\psi \cdot \phi)_i = S_i\phi_i \circ \psi_iP_i \) for
  \( i=0,1 \).  Since
  \begin{equation*}
    \begin{tikzcd}
      & \cdot \ar[rr,"\psi_{Ps} \cdot \psi_{Pr}"]
      && \cdot \ar[rr,"S\phi_s \cdot S\phi_r"]
               \ar[rd,"\m S"]
      && \cdot \ar[rd,"\m S"] \\
      \cdot \ar[ru,"\m H"] \ar[rd,"H\m R",swap]
      &&&& \cdot \ar[rr,"S(\phi_s \cdot \phi_r)"] 
      && \cdot \ar[rd,"S\m Q"] \\
      & \cdot \ar[rr,"\psi_{Ps \cdot Pr}"]
              \ar[rd,"HR\m P",swap]
      && \cdot \ar[ru,"S\m G"] \ar[rd,"SG\m P"]
      &&&& \cdot \\
      && \cdot \ar[rr,"\psi_{P(s \cdot r)}",swap]
      && \cdot \ar[rr,"S\phi_{s \cdot r}",swap]
      && \cdot \ar[ru,"SQ\m F",swap]
    \end{tikzcd}
  \end{equation*}
  and
  \begin{equation*}
    \begin{tikzcd}
      & \cdot \ar[rr,"1_{\psi P}"]
      && \cdot \ar[rr,"1_{S \phi}"]
               \ar[rd,"\e S"]
      && \cdot \ar[rd,"\e S"] \\
      \cdot \ar[ru,"\e H"] \ar[rd,"H\e R",swap]
      &&&& \cdot \ar[rr,"S1_\phi"]
      && \cdot \ar[rd,"S\e Q"] \\
      & \cdot \ar[rr,"\psi_{1_P}"]
              \ar[rd,"HR\e P",swap]
      && \cdot \ar[ru,"S\e G"] \ar[rd,"SG\e P"]
      &&&& \cdot \\
      && \cdot \ar[rr,"\psi_{P1}",swap]
      && \cdot \ar[rr,"S\phi_1",swap]
      && \cdot \ar[ru,"SQ\e F",swap]
    \end{tikzcd}
  \end{equation*}
  are commutative diagrams, we conclude~\( \psi \cdot \phi \) is a generalized
  vertical transformation, with vertical domain~\( F \) and codomain~\( H \).
  Associativity and identity conditions hold, via componentwise calculation on
  the underlying natural transformations, so the associator and unitor 2-cells
  are taken to be identites.
\end{proof}

\subsection{Horizontal transformations:}

Let~\(F,\, G \colon \bicat D \to \bicat E\) be lax functors.  A \textit{lax
horizontal transformation}~\( \phi \colon F \to G \) is given by data
\begin{itemize}[label=--]
  \item
    a functor~\( \phi \colon \bicat D_0 \to \bicat E_1 \),

  \item
    a globular natural 2-cell~\( \n \phi_r \colon Gr \cdot \phi_x \to
    \phi_y\cdot Fr \) for each horizontal 1-cell \(r \colon x \to y \),
\end{itemize}
satisfying the following coherence conditions:
\begin{itemize}[label=--]
  \item
    The following diagram commutes
    \begin{equation*}
      \begin{tikzcd}
        1_{Gx} \cdot \phi_x \ar[d,"\e G \cdot \id",swap]
                            \ar[r,"\gamma"] 
        & \phi_x \cdot 1_{Fx} \ar[d, "\id \cdot \e F"] \\
        G1_x \cdot \phi_x \ar[r,"\phi_{1_x}",swap] & \phi_x \cdot F1_x
      \end{tikzcd}
    \end{equation*}
    for all 0-cells~\(x\).

  \item
    The following diagram commutes
    \begin{equation*}
      \label{octagon.transformation.coherence}
      \begin{tikzcd}
        & G(s\cdot r) \cdot \phi_x 
            \ar[r,"\n \phi_{s\cdot r}"] 
        & \phi_z \cdot F(s\cdot r) \\
        (Gs \cdot Gr) \cdot \phi_x 
            \ar[d,"\alpha",swap] 
            \ar[ur,"\m G \cdot \id"] 
        &&& \phi_z \cdot (Fs \cdot Fr) \ar[ul,"\id \cdot \m F",swap] \\
        Gs \cdot (Gr \cdot \phi_x) 
            \ar[rd,"\id \cdot \n \phi_r",swap]
        &&& (\phi_z \cdot Fs) \cdot Fr 
            \ar[u,"\alpha",swap] \\
        & Gs \cdot (\phi_y \cdot Fr) 
            \ar[r,"\alpha^{-1}",swap] 
        & (Gs \cdot \phi_y) \cdot Fr 
            \ar[ur,"\n \phi_s \cdot \id",swap]
      \end{tikzcd}
    \end{equation*}
    for each pair~\(r \colon x \to y \), \( s \colon y \to z \) of horizontal
    1-cells.
\end{itemize}
\textit{Oplax horizontal transformations} \( \phi \) between (op)lax functors
are obtained by reversing the direction of the natural 2-cell \( \n\phi \). An
(op)lax horizontal transformation is said to be a \textit{strong horizontal
transformation} \cite{GP99, Mar06} (we note these references follow the
opposite convention; they are called \textit{vertical} therein) if \( \n\phi
\) is invertible -- hence strong horizontal transformations are lax and
oplax.

\begin{remark}
  Let \( F,\, G \colon \bicat A \to \bicat B \) be lax functors between
  bicategories \( \bicat A, \bicat B \), seen as vertically trivial
  pseudodouble categories. A lax horizontal transformation \( \phi \colon F
  \to G \) corresponds to the notion of lax natural transformations
  \cite{Gra69, Gra74} (therein, called \textit{2-natural transformations}, and
  \textit{quasi-natural transformations}, respectively). In \cite{Str72b}, one
  also finds \textit{oplax} natural transformations.

  Though the notion of (op)lax horizontal transformation for pseudodouble
  categories is not easily found in the literature, the coherence conditions
  for the unit horizontal transformations and for the horizontal composition
  of horizontal transformations are analogous to the (op)lax natural
  transformations for bicategories.
\end{remark}

\begin{proposition}[Unit horizontal 1-cell]
  \label{lem:lht.hu}
  Let~\( F \colon \bicat D \to \bicat E \) be a lax functor. The data
  \begin{enumerate}[label=(\roman*)]
    \item
      \label{enum:unit.zcells}
      \( (1_F)_x = 1_{Fx} \) for each 0-cell~\(x\),
    \item
      \label{enum:unit.vcells}
      \( (1_F)_f = 1_{Ff} \) for each vertical 1-cell~\(f\),
    \item
      \label{enum:unit.globnat}
      \( \n{1_F}_r = \gamma_{Fr} \) for each horizontal 1-cell~\(r\),
  \end{enumerate}
  defines a strong horizontal transformation~\( 1_F \colon F \to F \).
\end{proposition}

\begin{proof}
  The data~\ref{enum:unit.zcells} and~\ref{enum:unit.vcells} tell us that the
  underlying functor of~\( 1_F \) is the composite
  \begin{equation*}
    \begin{tikzcd}
      \bicat D_0 \ar[r,"F_0"] & \bicat E_0 \ar[r,"1"] & \bicat E_1
    \end{tikzcd}
  \end{equation*}
  and the datum~\ref{enum:unit.globnat} tells us that, for each horizontal
  1-cell \( r \colon x \to y \), the underlying 2-cell \( \n{1_F}_r \) is
  given by
  \begin{equation*}
    \begin{tikzcd}
      1_{Fy} \cdot Fr \ar[r,"\lambda"]
        & Fr \ar[r,"\rho^{-1}"]
        & Fr \cdot 1_{Fx}
    \end{tikzcd}
  \end{equation*}
  which is a composite of globular, natural, invertible 2-cells, so the data is
  well-defined. 

  We're left with checking coherence. First, we observe that
  \begin{equation*}
    \begin{tikzcd}
      1_{Fx} \cdot 1_{Fx} \ar[r,"\gamma"] \ar[d,"\id \cdot \e F",swap] 
        & 1_{Fx} \cdot 1_{Fx} \ar[d,"\e F \cdot \id"] \\
      1_{Fx} \cdot F1_x \ar[r,"\gamma",swap]
        & F1_x \cdot 1_{Fx}
    \end{tikzcd}
  \end{equation*}
  commutes by naturality, giving the comparison coherence diagram for the
  unit. Now, note that we have
  \begin{equation}
    \label{eq:coh.assoc.gam}
    \alpha^{-1} \circ (\id \cdot \gamma) 
                \circ \alpha
                \circ (\gamma \cdot \id)
                \circ \alpha^{-1} = \gamma,
  \end{equation}
  by~\ref{enum:left.left.red}, \ref{enum:unit.coherence}
  and~\ref{enum:right.right.red}, so that the following diagram commutes
  \begin{equation*}
    \begin{tikzcd}
      & 1_{Fz} \cdot F(s \cdot r) \ar[r,"\gamma"]
      & F(s \cdot r) \cdot 1_{Fx} \\
      1_{Fz} \cdot (Fs \cdot Fr) \ar[ur,"\id \cdot \m F"]
                                 \ar[d,"\alpha^{-1}",swap]
                                 \ar[rrr,"\gamma"]
      &&& (Fs \cdot Fr) \cdot 1_{Fx} \ar[ul,"\m F \cdot \id",swap] \\
      (1_{Fz} \cdot Fs) \cdot Fr \ar[rd,"\gamma \cdot \id",swap] 
      &&& Fs \cdot (Fr \cdot 1_{Fx}) \ar[u,"\alpha",swap] \\
      & (Fs \cdot 1_{Fy}) \cdot Fr \ar[r,"\alpha",swap] 
      & Fs \cdot (1_{Fy} \cdot Fr) \ar[ru,"\id \cdot \gamma",swap] 
    \end{tikzcd}
  \end{equation*}
  by naturality of~\( \gamma \). Since~\( \n{1_F}_r \) is invertible for each
  horizontal 1-cell \(r\), we conclude that~\( 1_F \colon F \to F \) is a
  strong horizontal transformation.
\end{proof}

\begin{proposition}[Horizontal composition]
  \label{lem:lht.hc}
  Let~\( \phi \colon F \to G \), \( \psi \colon G \to H \) be lax horizontal
  transformations, where~\(F,G,H \colon \bicat D \to \bicat E \) are lax
  functors. The data
  \begin{enumerate}[label=(\roman*)]
    \item
      \label{enum:comp.zcells}
      \( (\psi \cdot \phi)_x = \psi_x \cdot \phi_x \colon Fx \to Hx \) for
      each 0-cell~\(x\),
    \item
      \label{enum:comp.vcells}
      \( (\psi \cdot \phi)_f = \psi_f \cdot \phi_f \) for each vertical
      1-cell~\(f\),
    \item
      \label{enum:comp.globnat}
      \( \n{\psi \cdot \phi}_r = \alpha^{-1} \circ (\id \cdot \n\phi_r) \circ
      \alpha \circ (\n\psi_r \cdot \id) \circ \alpha^{-1} \) for each
      horizontal 1-cell~\(r\)
  \end{enumerate}
  defines a lax horizontal transformation~\( \psi \cdot \phi \colon F \to H
  \).
\end{proposition}

\begin{proof}
  Due to functoriality of horizontal composition and of the underlying
  functors of~\( \phi \) and~\( \psi \), it is enough to point out that~\(
  \phi_x \) and~\( \psi_x \) are a composable pair of horizontal 1-cells to
  make sure the data~\ref{enum:comp.zcells} and~\ref{enum:comp.vcells} define
  a functor~\( \bicat D_0 \to \bicat E_1 \). 

  Furthermore, note that the datum~\ref{enum:comp.globnat} is a composite of
  globular natural transformations, so it is enough to verify the coherence
  conditions are satisfied. 

  We note the following diagram, in which we have supressed the horizontal
  1-cells,
  \begin{equation*}
    \begin{tikzcd}[column sep=huge,row sep=large]
      \cdot     \ar[d,"\e H \cdot \id",swap] 
                \ar[r,"\alpha^{-1}"]
        & \cdot \ar[d,"(\e H \cdot \id) \cdot \id" description]
                \ar[r,"\gamma \cdot \id"]
        & \cdot \ar[d,"(\id \cdot \e G) \cdot \id" description]
                \ar[r,"\alpha"]
        & \cdot \ar[d,"\id \cdot (\e G \cdot \id)" description]
                \ar[r,"\id \cdot \gamma"]
        & \cdot \ar[d,"\id \cdot (\id \cdot \e F)" description]
                \ar[r,"\alpha^{-1}"]
        & \cdot \ar[d,"\id \cdot \e F"] \\
      \cdot     \ar[r,"\alpha^{-1}",swap]
        & \cdot \ar[r,"\n\psi_1 \cdot \id",swap]
        & \cdot \ar[r,"\alpha",swap]
        & \cdot \ar[r,"\id \cdot \n\phi_1",swap]
        & \cdot \ar[r,"\alpha^{-1}",swap]
        & \cdot
    \end{tikzcd}
  \end{equation*} 
  commutes, by naturality of~\( \alpha \) and unit comparison coherence for~\(
  \psi \) and~\( \phi \). By~\ref{eq:coh.assoc.gam}, the top composite is~\(
  \gamma \), so this confirms unit comparsion coherence for~\( \psi \cdot \phi
  \).

  The next diagram\footnote{See also \cite[Theorem 2]{Mar06}, where we find a
  proof that horizonal composition for \textit{strong horizontal
  transformatons} is well-defined, where we find similar diagrams.}, verifies
  composition coherence for~\( \psi \cdot \phi \): it is a pasting of
  composition coherences for~\(\psi\) and~\(\phi\), a naturality square from
  the functoriality of~\( \cdot \), and the remaining diagrams are coherence
  and naturality of~\( \alpha \).

  \noindent\makebox[\textwidth][c]
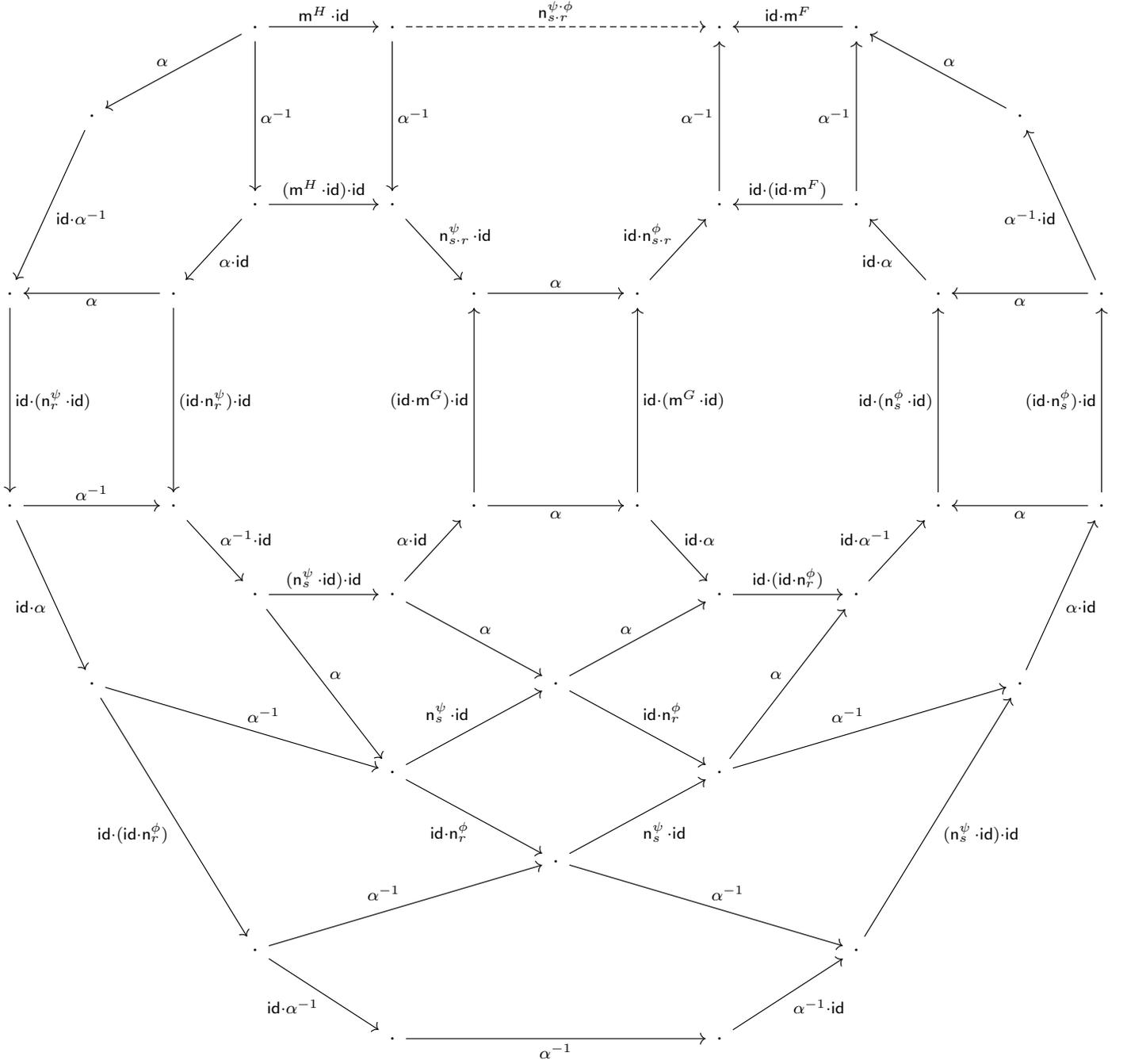
{
    \begin{tikzcd}[ampersand replacement=\&,row sep=large]
      \&\&\&     
        \cdot \ar[rr,"\m H \cdot \id"] 
              \ar[dd,"\alpha^{-1}"] 
              \ar[dll,"\alpha",swap]                          \&\&     
        \cdot \ar[dd,"\alpha^{-1}"]                          
              \ar[rrrr,"\n {\psi \cdot \phi}_{s\cdot r}",dashed] \&\&\&\& 
        \cdot                                                 \&\&     
        \cdot \ar[ll,"\id \cdot \m F",swap]                    \\
      \&
        \cdot \ar[ddl,"\id \cdot \alpha^{-1}"]
      \&\&\&\&\&\&\&\&\&\&\&\& 
        \cdot \ar[ull,"\alpha",swap]                                          \\
      \&\&\&
        \cdot \ar[rr,"(\m H \cdot \id) \cdot \id"]  
              \ar[dl,"\alpha \cdot \id"]                   \&\& 
        \cdot \ar[dr,"\n \psi_{s\cdot r} \cdot \id"]          \&\&\&\&
        \cdot \ar[uu,"\alpha^{-1}"]                           \&\& 
        \cdot \ar[uu,"\alpha^{-1}"]
              \ar[ll,"\id \cdot (\id \cdot \m F)",swap]                  \\ 
        \cdot \ar[ddd,"\id \cdot (\n \psi_r \cdot \id)"]       \&\& 
        \cdot \ar[ll,"\alpha"]   
              \ar[ddd,"(\id \cdot \n \psi_r) \cdot \id"]       \&\&\&\& 
        \cdot \ar[rr,"\alpha"]                                \&\&
        \cdot \ar[ur,"\id \cdot \n \phi_{s\cdot r}"]          \&\&\&\& 
        \cdot \ar[ul,"\id \cdot \alpha"]                   \&\& 
        \cdot \ar[ll,"\alpha"]            
              \ar[uul,"\alpha^{-1} \cdot \id"]                             \\\\\\
        \cdot \ar[rr,"\alpha^{-1}"]             
              \ar[ddr,"\id \cdot \alpha",swap]                      \&\& 
        \cdot \ar[dr,"\alpha^{-1} \cdot \id"]              \&\&\&\& 
        \cdot \ar[rr,"\alpha",swap]  
              \ar[uuu,"(\id \cdot \m G) \cdot \id"]      \&\&    
        \cdot \ar[dr,"\id \cdot \alpha"]   
              \ar[uuu,"\id \cdot (\m G \cdot \id)",swap] \&\&\&\& 
        \cdot \ar[uuu,"\id \cdot (\n \phi_s \cdot \id)"]       \&\& 
        \cdot \ar[ll,"\alpha"]   
              \ar[uuu,"(\id \cdot \n \phi_s) \cdot \id"]                       \\
      \&\&\&
        \cdot \ar[rr,"(\n \psi_s \cdot \id) \cdot \id"]            
              \ar[ddrr,"\alpha"]                              \&\& 
        \cdot \ar[ur,"\alpha \cdot \id"]             
              \ar[drr,"\alpha"]                               \&\&\&\& 
        \cdot \ar[rr,"\id \cdot (\id \cdot \n \phi_r)"]        \&\& 
        \cdot \ar[ur,"\id \cdot \alpha^{-1}"]                              \\
      \&
        \cdot \ar[drrrr,"\alpha^{-1}"] 
              \ar[dddrr,"\id \cdot (\id \cdot \n \phi_r)",swap] \&\&\&\&\&\& 
        \cdot \ar[drr,"\id \cdot \n \phi_r"] 
              \ar[urr,"\alpha"]                               \&\&\&\&\&\& 
        \cdot \ar[uur,"\alpha \cdot \id",swap]                                  \\
      \&\&\&\&\& 
        \cdot \ar[drr,"\id \cdot \n \phi_r",swap] 
              \ar[urr,"\n \psi_s \cdot \id"]       \&\&\&\& 
        \cdot \ar[uurr,"\alpha"]
              \ar[urrrr,"\alpha^{-1}"]                                        \\
      \&\&\&\&\&\&\& 
        \cdot \ar[urr,"\n \psi_s \cdot \id",swap] 
              \ar[drrrr,"\alpha^{-1}"]                                        \\
      \&\&\&
        \cdot \ar[drr,"\id \cdot \alpha^{-1}",swap] 
              \ar[urrrr,"\alpha^{-1}"]                        \&\&\&\&\&\&\&\& 
        \cdot \ar[uuurr,"(\n \psi_s \cdot \id) \cdot \id",swap]                     \\
      \&\&\&\&\& 
        \cdot \ar[rrrr,"\alpha^{-1}",swap]                         \&\&\&\& 
        \cdot \ar[urr,"\alpha^{-1} \cdot \id",swap]                
    \end{tikzcd}
  }
  Hence, we have confirmed composition coherence for~\( \psi \cdot \phi \),
  concluding the proof.
\end{proof}

\subsection{Modifications}

Let~\(F,\,G,\,H,\,K \colon \bicat C \to \bicat D \) be lax functors and let~\(
\zeta \colon F \to H \), \( \xi \colon G \to K \) be oplax horizontal
transformations, and let~\( \phi \colon F \to G \), \( \psi \colon H \to K \)
be vertical transformations. A \textit{modification (for oplax horizontal
transformations)}~\( \Gamma \colon \zeta \to \xi \), depicted as
\begin{equation}
  \begin{tikzcd}
    F \ar[r,"\zeta",""{name=U}] \ar[d,"\phi",swap] 
    & H \ar[d,"\psi"] \\
    G \ar[r,"\xi",""{name=D},swap] & K
    \ar[phantom, from=U, to=D, "\Gamma"]
  \end{tikzcd}
\end{equation}
is a natural transformation~\( \Gamma \colon \zeta \to \xi \) on the
underlying functors~\( \zeta , \xi \colon \bicat D_0 \to \bicat E_1 \)
such that 
\begin{equation}
  \label{eq:modif.cond}
  \begin{tikzcd}
    \zeta_y \cdot Fr \ar[r,"\Gamma_y \cdot \phi_r"]
                      \ar[d,"\n \zeta_r",swap] 
    & \xi_y \cdot Gr  \ar[d,"\n \xi_r"] \\
    Hr \cdot \zeta_x \ar[r,"\psi_r \cdot \Gamma_x",swap]
    & Kr \cdot \xi_x 
  \end{tikzcd}
\end{equation}
commutes for all horizontal 1-cells~\(r \colon x \to y \). We say~\( \phi \)
and~\( \psi \) are respectively the vertical domain and codomain of~\( \Gamma
\).

By reversing the direction of \( \n\zeta \), \( \n\xi \), we obtain the
analogous notion of modification for \textit{lax horizontal transformations}.

\begin{proposition}
  \label{prop:dbcat.horiz}
  We have a category~\( \Lax_\opl(\bicat D, \bicat E) \) with oplax horizontal
  transformations of lax functors~\( \bicat D \to \bicat E \) as objects, and
  the respective modifications as morphisms. Moreover, the vertical domain and
  codomain operations define functors to the category of lax functors and
  vertical transformations.
\end{proposition}

\begin{proof}
  Let~\( \zeta \colon F \to G \) be an oplax transformation of lax functors~\(
  \bicat D \to \bicat E \). We take the identity modification~\( \id_\zeta \)
  on~\( \zeta \) to be given by identity natural transformation on the
  underlying functor of~\( \zeta \), whose vertical domain and codomain are
  taken to be the identity vertical transformations~\( \id_F \) and~\( \id_G
  \), respectively. The instance of the diagram \eqref{eq:modif.cond} for~\(
  \id_\zeta \) is trivially commutative.

  Let~\( \Gamma, \Xi \) be modifications given by
  \begin{equation*}
    \begin{tikzcd}
      F \ar[d,"\phi",swap] \ar[r,"\zeta"{name=A}]
        & G \ar[d,"\psi"] \\
      H \ar[d,"\theta",swap] \ar[r,"\xi" description,""{name=B}]
        & K \ar[d,"\omega"] \\
      L \ar[r,"\chi"{name=C},swap]
        & M 
      \ar[from=A,to=B,"\Gamma" description,phantom]
      \ar[from=B,to=C,"\Xi" description,phantom]
    \end{tikzcd}
  \end{equation*}
  We define the composite~\( \Xi \circ \Gamma \) to be the vertical
  composition of the underlying natural transformations. Since
  \begin{equation*}
    \begin{tikzcd}[row sep=large]
      \zeta_y \cdot Fr \ar[d,"\n \zeta_r",swap] 
                       \ar[r,"\Gamma_y \cdot \phi_r"]
        & \xi_y \cdot Hr \ar[d,"\n \xi_r" description] 
                         \ar[r,"\Xi_y \cdot \theta_r"]
        & \chi_y \cdot Lr \ar[d,"\n \chi_r"] \\
      Gr \cdot \zeta_x \ar[r,"\psi_r \cdot \Gamma_x",swap]
        & Kr \cdot \xi_x \ar[r,"\omega_r \cdot \Xi_x",swap]
        & Mr \cdot \chi_x
    \end{tikzcd}
  \end{equation*}
  commutes for all horizontal 1-cells~\( r \colon x \to y \), we confirm~\(
  \Xi \circ \Gamma \colon \phi \to \chi \) is a modification with vertical
  domain~\( \theta \circ \phi \) and codomain~\( \omega \circ \psi \). 

  Associativity and identity properties are inherited from natural
  transformations, and functoriality of vertical domain and codomain is an
  immediate consequence.
\end{proof}

\begin{theorem}
  \label{thm:psdbcat.closed}
  Let~\( \bicat D \), \( \bicat E \) be double categories.~\( \Lax_\opl(\bicat
  D, \bicat E) \) has the structure of a pseudodouble category, with lax
  functors as 0-cells, vertical transformations as vertical 1-cells, oplax
  horizontal transformations as horizontal 1-cells, and the respective
  modifications as 2-cells. 

  Dually, \( \Lax_\lax(\bicat D, \bicat E) \) is a pseudodouble category with
  the same 0-cells and vertical 1-cell, and lax horizontal transformations as
  horizontal 1-cells and the respective modifications as 2-cells.
\end{theorem}

\begin{proof}
  The underlying categories of cells are provided in
  Propositions~\ref{prop:dbcat.vert} and~\ref{prop:dbcat.horiz}. moreover, the
  latter has provided the vertical domain and codomain functors. 

  We have defined the horizontal unit functor on objects in
  Proposition~\ref{lem:lht.hu}. For a vertical transformation~\( \phi \colon F
  \to G \), we define~\( 1_\phi \) to be the modification with underlying
  natural transformation~\( 1 \cdot \phi_0 \), with vertical domain~\( 1_F \)
  and codomain~\( 1_G \); note that
  \begin{equation*}
    \begin{tikzcd}
      1_F \cdot Fr \ar[r,"1_\phi \cdot \phi_r"] \ar[d,"\gamma",swap]
        & 1_G \cdot Gr \ar[d,"\gamma"] \\
      Fr \cdot 1_F \ar[r,"\phi_r \cdot 1_\phi",swap] 
        & Gr \cdot 1_G
    \end{tikzcd}
  \end{equation*}
  commutes by naturality of~\( \gamma \). Since this is just whiskering
  with~\( 1 \colon \bicat E_0 \to E_1 \), this describes a functor.

  We have defined the horizontal composition functor on objects in
  Proposition~\ref{lem:lht.hc}. For modifications~\( \Gamma \) and~\( \Xi \)
  as depicted below
  \begin{equation*}
    \begin{tikzcd}
      F \ar[d,"\phi",swap] \ar[r,"\zeta"{name=A}]
        & H \ar[d,"\psi"] \ar[r,"\xi"{name=B}]
        & L \ar[d,"\chi"] \\
      G \ar[r,"\theta"{name=C},swap]
        & K \ar[r,"\omega"{name=D},swap]
        & M
      \ar[from=A,to=C,"\Gamma" description,phantom]
      \ar[from=B,to=D,"\Xi" description,phantom]
    \end{tikzcd}
  \end{equation*}
  we define~\( \Xi \cdot \Gamma \) to be the horizontal composition of the
  underlying natural transformations. This is a modification, since the
  following diagram commutes
  \begin{equation}
    \begin{tikzcd}
      (\xi_y \cdot \zeta_y) \cdot Fr \ar[r,"\alpha"] 
          \ar[d,"(\Xi_y \cdot \Gamma_y) \cdot \phi_r" description]
      & \xi_y \cdot (\zeta_y \cdot Fr) \ar[r,"\id \cdot n^\zeta_r"]
          \ar[d,"\Xi_y \cdot (\Gamma_y \cdot \phi_r)" description]
      & \xi_y \cdot (Hr \cdot \zeta_x) \ar[r,"\alpha^{-1}"]
          \ar[d,"\Xi_y \cdot (\psi_r \cdot \Gamma_x)" description]
      & \ldots \\
      (\omega_y \cdot \theta_y) \cdot Gr \ar[r,"\alpha"]
      & \omega_y \cdot (\theta_y \cdot Gr) \ar[r,"\id \cdot n^\theta_r"]
      & \omega_y \cdot (Kr \cdot \theta_x) \ar[r,"\alpha^{-1}"]
      & \ldots \\
      \ldots \ar[r,"\alpha^{-1}"]
      & (\xi_y \cdot Hr) \cdot \zeta x \ar[r,"n^\xi_r \cdot \id"]
          \ar[d,"(\Xi_y \cdot \psi_r) \cdot \Gamma_x" description]
      & (Lr \cdot \xi_x) \cdot \zeta_x \ar[r,"\alpha"]
          \ar[d,"(\chi_r \cdot \Xi_x) \cdot \Gamma_x" description]
      & Lr \cdot (\xi_x \cdot \zeta_x) 
          \ar[d,"\chi_r \cdot (\Xi_x \cdot \Gamma_x)" description] \\
      \ldots \ar[r,"\alpha^{-1}"]
      & (\omega_y \cdot Kr) \cdot \theta_x \ar[r,"n^\omega_r \cdot \id"]
      & (Mr \cdot \omega_x) \cdot \theta_x \ar[r,"\alpha"]
      & Mr \cdot (\omega_x \cdot \theta_x)
    \end{tikzcd}
  \end{equation}
  and has vertical domain~\( \phi \) and codomain~\( \chi \).

  Since horizontal composition in~\( \bicat E \) is functorial, we obtain
  functoriality of horizontal composition of modifications. Moreover, both the
  horizontal unit and horizontal composition have the required behaviour with
  respect to vertical domains and codomains.

  We're left with providing the unitors and associator, and the respective
  proofs that these satisfy the required coherence conditions.
  To do so, we define~\( \lambda_\zeta \colon 1_H \cdot \zeta \to \zeta \) to
  be given by~\( \lambda_{\zeta_x} \colon 1_{Hx} \cdot \zeta_x \to \zeta_x \),
  and~\( \rho_\zeta \) is similarly defined. These are globular modifications,
  as the following diagrams commute
  \begin{equation*}
    \begin{tikzcd}
      (1_{Hy} \cdot \zeta_y) \cdot Fr \ar[d,"\alpha",swap]
        \ar[r,"\lambda_{\zeta_y} \cdot \id"]
      & \zeta_y \cdot Fr \ar[ddddd,"\n\zeta_r"] \\
      1_{Hy} \cdot (\zeta_y \cdot Fr) \ar[d,"\id \cdot \n\zeta_r",swap] \\
      1_{Hy} \cdot (Hr \cdot \zeta_x) \ar[d,"\alpha^{-1}",swap] \\
      (1_{Hy} \cdot Hr) \cdot \zeta_x \ar[d,"\gamma \cdot \id",swap] \\
      (Hr \cdot 1_{Hx}) \cdot \zeta_x \ar[d,"\alpha",swap] \\
      Hr \cdot (1_{Hx} \cdot \zeta_x) 
        \ar[r,"\id \cdot \lambda_{\zeta_x}",swap]
      & Hr \cdot \zeta_x 
    \end{tikzcd}
    \qquad
    \begin{tikzcd}
      (\zeta_y \cdot 1_{Fy}) \cdot Fr \ar[d,"\alpha",swap]
        \ar[r,"\rho_{\zeta_y} \cdot \id"]
      & \zeta_y \cdot Fr \ar[ddddd,"\n\zeta_r"] \\
      \zeta_y \cdot (1_{Fy} \cdot Fr) \ar[d,"\id \cdot \gamma",swap] \\
      \zeta_y \cdot (Fr \cdot 1_{Fx}) \ar[d,"\alpha^{-1}",swap] \\
      (\zeta_y \cdot Fr) \cdot 1_{Fx} \ar[d,"\n\zeta_r \cdot \id",swap] \\
      (Hr \cdot \zeta_x) \cdot 1_{Fx} \ar[d,"\alpha",swap] \\
      Hr \cdot (\zeta_x \cdot 1_{Fx}) \ar[r,"\id \cdot \rho_{\zeta_x}",swap]
      & Hr \cdot \zeta_x 
    \end{tikzcd}
  \end{equation*}
  by naturality of~\( \lambda, \rho \), and coherence.

  Finally, we let~\( \pi \colon L \to P \) be another oplax horizontal
  transformation. We define~\( \alpha \colon (\pi \cdot \xi) \cdot \zeta \to
  \pi \cdot (\xi \cdot \zeta) \) to be given at~\(x\) by~\( \alpha \colon
  (\pi_x \cdot \xi_x) \cdot \zeta_x \to \pi_x \cdot (\xi_x \cdot \zeta_x) \).
  This is also a natural isomorphism, and is a globular modification since the
  following diagram commutes:
  \begin{equation*}
    \begin{tikzcd}
      ((\pi_y \cdot \xi_y) \cdot \zeta_y) \cdot Fr
        \ar[rr,"\alpha \cdot \id"]
        \ar[d,"\alpha",swap] 
      && (\pi_y \cdot (\xi_y \cdot \zeta_y)) \cdot Fr 
        \ar[d,"\alpha"] \\
      (\pi_y \cdot \xi_y) \cdot (\zeta \cdot Fr)
        \ar[d,"\id \cdot \n\zeta_r",swap]
        \ar[rrd,"\alpha"]
      && \pi_y \cdot ((\xi_y \cdot \zeta_y) \cdot Hr)
        \ar[d,"\id \cdot \alpha"] \\
      (\pi_y \cdot \xi_y) \cdot (Hr \cdot \zeta_x)
        \ar[rrd,"\alpha"]
        \ar[d,"\alpha^{-1}",swap]
      && \pi_y \cdot (\xi_y \cdot (\zeta_y \cdot Hr))
        \ar[d,"\id \cdot (\id \cdot \n\zeta_r)"] \\
      ((\pi_y \cdot \xi_y) \cdot Hr) \cdot \zeta_x)
        \ar[d,"\alpha \cdot \id",swap]
      && \pi_y \cdot (\xi \cdot (Hr \cdot \zeta_x))
        \ar[d,"\id \cdot \alpha^{-1}"] \\
      (\pi_y \cdot (\xi_y \cdot Hr)) \cdot \zeta_x 
        \ar[rr,"\alpha"]
        \ar[d,"(\id \cdot \n\xi_r) \cdot \id"]
      && \pi_y \cdot ((\xi_y \cdot Hr) \cdot \zeta_x)
        \ar[d,"\id \cdot (\n\xi_r \cdot \id)"] \\
      (\pi_y \cdot (Lr \cdot \xi_x)) \cdot \zeta_x
        \ar[d,"\alpha^{-1} \cdot \id",swap]
        \ar[rr,"\alpha"]
      && \pi_y \cdot ((Lr \cdot \xi_x) \cdot \zeta_x)
        \ar[d,"\id \cdot \alpha"] \\
      ((\pi_y \cdot Lr) \cdot \xi_x) \cdot \zeta_x
        \ar[d,"(\n\pi_r \cdot \id) \cdot \id",swap]
        \ar[rrd,"\alpha"]
      && \pi_y \cdot (Lr \cdot (\xi_x \cdot \zeta_x))
        \ar[d,"\alpha^{-1}"] \\
      ((Pr \cdot \pi_x) \cdot \xi_x) \cdot \zeta_x
        \ar[d,"\alpha \cdot \id",swap]
        \ar[rrd,"\alpha"]
      && (\pi_y \cdot Lr) \cdot (\xi_x \cdot \zeta_x)
        \ar[d,"\n\pi_r \cdot \id"] \\
      (Pr \cdot (\pi_x \cdot \xi_x)) \cdot \zeta_x
        \ar[d,"\alpha",swap]
      && (Pr \cdot \pi_x) \cdot (\xi_x \cdot \zeta_x)
        \ar[d,"\alpha"] \\
      Pr \cdot ((\pi_x \cdot \xi_x) \cdot \zeta_x)
        \ar[rr,"\id \cdot \alpha",swap]
      && Pr \cdot (\pi_x \cdot (\xi_x \cdot \zeta_x))
    \end{tikzcd}
  \end{equation*}
  which is obtained by pasting coherence pentagons and naturality squares
  of~\( \alpha \).
 
  By checking componentwise, we find that~\( \lambda, \rho \) and~\( \alpha
  \) satisfy the desired coherence conditions.
\end{proof}

\subsection{Examples:}
\label{subsect:ex.psdbcats}

The pseudodouble categories studied in this body of work are:
\begin{itemize}[label=--]
  \item
    The pseudodouble category~\( \VMat \) of~\( \cat V \)-matrices, for
    distributive monoidal categories~\( \cat V \); that is, a monoidal
    category~\( \cat V \) with coproducts, which are preserved by the tensor
    product, \textit{e.g.}, \cite{BCSW83, CT03, CS10}.
    
  \item
    The pseudodouble category~\( \SpanB \) of spans of morphsims in~\( \cat B
    \), for~\( \cat B \) a category with pullbacks, \textit{e.g.},
    \cite{Bén67, Her00, CS10}).

  \item
    The pseudodouble categories~\( \Lax_\lax(\bicat D, \bicat E) \) and~\(
    \Lax_\opl(\bicat D,\bicat E) \) for pseudodouble categories~\( \bicat D
    \), \( \bicat E \) (Theorem~\ref{thm:psdbcat.closed}),

  \item
    The double categories~\( \LaxTAlg \), \( \PsTAlg \) of lax and
    pseudo~\(T\)-algebras, for~\(T\) a pseudomonad on a 2-category~\( \bicat B
    \), \textit{e.g.}, \cite{BKP89, Luc18b, Luc19}.

  \item
    In particular, the double category~\( \Mnd(\bicat B) = \LaxidAlg \) of
    monads in a 2-category~\( \bicat B \), \textit{e.g.} \cite{Str72a},
    \cite[p. 33]{LucTh}.
\end{itemize}

We shall specify the double categorical structure of~\( \LaxTAlg \). First,
recall that we have 2-categories~\( \LaxTAlg_\lax \) and~\( \LaxTAlg_\opl \)
whose 0-cells are lax~\(T\)-algebras, with their (op)lax morphisms and their
respective 2-cells \cite{LucTh}; however, there is a notion of generalized
2-cell which subsumes both structures.

We will be taking the vertical 1-cells to be the oplax morphisms and the
horizontal 1-cells to be the lax morphisms. Let~\( (h,\phi) \colon
(w,a,\eta,\mu) \to \bicat (x,b,\eta,\mu) \),~\( (k,\psi) \colon (y,c,\eta,\mu)
\to (z,d,\eta,\mu) \) be lax~\(T\)-algebra lax morphisms and~\( (f,\zeta)
\colon (w,a,\eta,\mu) \to (y,c,\eta,\mu) \), \((g,\xi) \colon (x,b,\eta,\mu)
\to (z,d,\eta,\mu) \) be lax~\(T\)-algebra oplax morphisms. A
\textit{generalized lax~\(T\)-algebra 2-cell}
\begin{equation*}
  \begin{tikzcd}
    (w,a,\eta,\mu) \ar[d,"{(f,\zeta)}",swap]
                          \ar[r,"{(h,\phi)}"{name=A}]
      & (x,b,\eta,\mu) \ar[d,"{(g,\psi)}"] \\
    (y,c,\eta,\mu) \ar[r,"{(k,\xi)}"{name=B},swap]
      & (z,d,\eta,\mu)
    \ar[from=A,to=B,"\omega",phantom]
  \end{tikzcd}
\end{equation*}
consists of a 2-cell~\( \omega \colon g \cdot h \to k \cdot f \) satisfying
the following coherence condition
\begin{equation*}
  \begin{tikzcd}
    & g \cdot b \cdot Th \ar[ld,"g\cdot \phi",swap] 
                         \ar[rd,"\xi \cdot Th"] \\
    g \cdot h \cdot a \ar[d,"\omega \cdot a",swap]
    && d \cdot Tg \cdot Th \ar[d,"d \cdot \omega^T"] \\
    k \cdot f \cdot a \ar[rd,"k \cdot \zeta",swap]
    && d \cdot Tk \cdot Tf \ar[ld,"\psi \cdot Tf"] \\
    & k \cdot c \cdot Tf
  \end{tikzcd}
\end{equation*}
where we write~\( \omega^T = (\m T)^{-1} \circ T\omega \circ \m T \).
Horizontal and vertical composition is defined as expected: to be explicit,
we consider generalized lax~\(T\)-algebra 2-cells~\( \theta, \sigma \) given
by
\begin{equation*}
  \begin{tikzcd}
    (y,c,\eta,\mu) \ar[r,"{(k,\xi)}"{name=B}]
                   \ar[d,"{(f',\phi')}",swap]
    & (z,d,\eta,\mu)  \ar[d,"{(g',\psi')}"] \\
    (x',b',\eta,\mu) \ar[r,"{(m,\pi)}"{name=A},swap]
      & (z',d',\eta,\mu) 
    \ar[from=B,to=A,"\theta",phantom] 
  \end{tikzcd}
  \quad
  \begin{tikzcd}
    (x,b,\eta,\mu) \ar[d,"{(g,\psi)}",swap] 
                   \ar[r,"{(h',\phi')}"{name=A}]
    & (w',a',\eta,\mu) \ar[d,"{(l,\chi)}"] \\
    (z,d,\eta,\mu) \ar[r,"{(k',\xi')}"{name=B},swap]
    & (y',c',\eta,\mu) 
    \ar[from=A,to=B,"\sigma",phantom]
  \end{tikzcd}
\end{equation*}
and we define~\( \theta \circ \omega = (\theta \cdot f) \circ (g' \cdot
\omega) \) and~\( \sigma \cdot \omega = (k' \cdot \omega) \circ (\sigma \cdot
h) \).  These provide a double categorical structure to~\( \LaxTAlg \),
provided the coherence conditions are satisfied for~\( \theta \circ \omega \)
and~\( \sigma \cdot \omega \), which are given by the commutativity of the
following diagrams:
\begin{equation*}
  \begin{tikzcd}
    & g' \cdot g \cdot b \cdot Th \ar[ld,"g' \cdot g \cdot \phi",swap] 
                                  \ar[rd,"g' \cdot \xi \cdot Th"] \\
    g' \cdot g \cdot h \cdot a \ar[d,"g' \omega \cdot a",swap]
    && g' \cdot d \cdot Tg \cdot Th \ar[d,"g' \cdot d \cdot \omega^T",swap] 
                                    \ar[rd,"\psi' \cdot Tg \cdot Th"] \\
    g' \cdot  k \cdot f \cdot a \ar[rd,"g' \cdot k \cdot \zeta"]
                                \ar[d,"\theta \cdot f \cdot a",swap]
    && g' \cdot d \cdot Tk \cdot Tf \ar[ld,"g' \cdot \psi \cdot Tf"] 
                                    \ar[rd,"\psi'\cdot Tk\cdot Tf",swap]
    & d \cdot Tg' \cdot Tg \cdot Th \ar[d,"d \cdot Tg' \cdot \omega^T"] \\
    m \cdot f' \cdot f \cdot a \ar[rd,"m \cdot f' \cdot\zeta",swap]
    & g' \cdot k \cdot c \cdot Tf \ar[d,"\theta \cdot c \cdot Tf"]
    && d \cdot Tg' \cdot Tk \cdot Tf \ar[d,"d \cdot \theta^T \cdot Tf"] \\
    & m \cdot f' \cdot c \cdot Tf \ar[rd,"m \cdot \phi' \cdot Tf",swap]
    && d \cdot Tm \cdot Tf' \cdot Tf \ar[ld,"\pi \cdot Tf' \cdot Tf"] \\
    && m \cdot b' \cdot Tf' \cdot Tf
  \end{tikzcd}
\end{equation*}
\begin{equation*}
  \begin{tikzcd}
    && l \cdot a' \cdot Th' \cdot Th 
        \ar[ld,"l \cdot \phi' \cdot Th",swap]
        \ar[rd,"\chi \cdot Th' \cdot Th"] \\
    & l \cdot h' \cdot b \cdot Th
        \ar[ld,"l \cdot h' \cdot \phi",swap]
        \ar[d,"\sigma \cdot b \cdot Th"]
    && c' \cdot Tl \cdot Th' \cdot Th
        \ar[d,"c' \cdot \sigma^T \cdot Th"] \\
    l \cdot h' \cdot h \cdot a 
        \ar[d,"\sigma \cdot h \cdot a",swap]
    & k' \cdot g \cdot b \cdot Th 
        \ar[ld,"k'\cdot g \cdot \phi"]
        \ar[rd,"k' \cdot \psi \cdot Th",swap]
    && c' \cdot Tk' \cdot Tg \cdot Th
        \ar[ld,"\xi' \cdot Tg \cdot Th",swap]
        \ar[d,"c' \cdot Tk' \cdot \omega^T"] \\
    k' \cdot g \cdot h \cdot a
        \ar[d,"k' \cdot \omega \cdot a",swap]
    && k' \cdot d \cdot Tg \cdot Th 
        \ar[d,"k' \cdot d \cdot \omega^T",swap]
    & c' \cdot Tk' \cdot Tk \cdot Tf
        \ar[ld,"\xi' \cdot Tk \cdot Tf"] \\
    k' \cdot k \cdot f \cdot a
        \ar[rd,"k' \cdot k \cdot \zeta",swap]
    && k' \cdot d \cdot Tk \cdot Tf \ar[ld,"k' \cdot \psi \cdot Tf"] \\
    & k' \cdot k \cdot c \cdot Tf
  \end{tikzcd}
\end{equation*}
with analogous definitions for~\( \theta^T \) and~\( \sigma^T \), plus a
couple of omitted coherence conditions which confirm that~\( (\theta^T \cdot
Tf) \circ (Tg' \cdot \omega^T) = (\theta \circ \omega)^T \) and~\( (Tk' \cdot
\omega^T) \circ (\sigma^T \cdot Th) = (\sigma \cdot \omega)^T \).

  \section{Spans versus matrices}
    \label{sect:spanmat}
    Let \( \cat V \) be a distributive, cartesian monoidal category with finite
limits, whose terminal object is denoted by \( \trm \). Our starting point is
the adjunction
\begin{equation}
  \label{eq:start.adj}
  \begin{tikzcd}
    \Set \ar[r,bend left,"-\pt \trm"{above,name=A}]
      & \cat V \ar[l,bend left,"{\cat V(\trm,-)}"{below,name=B}]
    \ar[from=A,to=B,phantom,"\adj" {anchor=center, rotate=-90}]
  \end{tikzcd}
\end{equation}
whose unit and counit we denote by \( \heta, \heps \) respectively. Here, \( -
\pt \trm \) is the functor left adjoint to \( \cat V(\trm, -) \), taking each
object \(X\) to its copower by \( \trm \). Explicitly, it is given on objects
by \( X \mapsto X \pt \trm = \sum_{x \in X} \trm \).

After fixing some notation regarding \( \VMat \) and \( \SpanV \), we confirm
that \eqref{eq:start.adj} induces an adjunction in the 2-category \(
\Grph(\Cat) \) of internal \( \Cat \)-graphs
\begin{equation*}
  - \pt \trm \colon \VMat \to \SpanV
  \qquad \cat V(\trm,-) \colon \SpanV \to \VMat
\end{equation*}
and we will furthermore confirm that \( - \pt \trm \) defines an oplax
functor.

Together with the tools and terminology provided in
Sections~\ref{sect:conj.comp}~and~\ref{sect:pstalg}, we will be able to deduce
that \( \cat V(\trm,-) \) is a lax functor, and that we have a conjunction in
the double category \( \PsDbCat \). The unit and counit may be depicted as
follows
\begin{equation*}
  \begin{tikzcd}
    \VMat \ar[r,equal,""{name=A}] \ar[d,"-\pt \trm",swap]
      & \VMat \ar[d,equal] \\
    \SpanV \ar[r,"{\cat V(\trm,-)}"{name=B},swap]
      & \VMat
    \ar[from=A,to=B,phantom,"\heta" description]
  \end{tikzcd}
  \quad
  \begin{tikzcd}
    \SpanV \ar[r,"{\cat V(\trm,-)}"{name=B}]
           \ar[d,equal]
      & \VMat \ar[d,"-\pt \trm"] \\
    \SpanV \ar[r,equal,""{name=A}] & \SpanV
    \ar[from=B,to=A,phantom,"\heps" description]
  \end{tikzcd}
\end{equation*}
alluding to the fact that these are generalized vertical transformations in \(
\PsDbCat \).

\subsection*{Notation for $\SpanV$:} 

The \( \Cat \)-graph \( \SpanV \) is succintly defined as \( [l \leftarrow m
\rightarrow r, \cat V] \rightrightarrows \cat V \), whose underlying functors
are the evaluations at \(l\) and \(r\). Throughout this work, we opt to denote
spans \( p \colon X \relto Y \) in \( \cat V \) as the following diagram
\begin{equation*}
  \begin{tikzcd}
    X & M_p \ar[l,"l_p",swap] \ar[r,"r_p"] & Y
  \end{tikzcd}
\end{equation*}
and a 2-cell \( \theta \) will be denoted as a morphism \( M_p \to M_q \)
making both of the following squares commute:
\begin{equation*}
  \begin{tikzcd}
    X \ar[d,"f",swap]
      & M_p \ar[l,"l_p",swap] \ar[d,"\theta"]
            \ar[r,"r_p"] 
      & Y \ar[d,"g"] \\
    W & M_q \ar[l,"l_q"] \ar[r,"r_q",swap]
      & Z
  \end{tikzcd}
\end{equation*}
The unit span \( 1_X \colon X \relto X \) is defined on objects by \( M_{1_X}
= X \) and \( l_{1_x} = r_{1_x} = \id_X \), and on morphisms \( f \colon X \to
Y \) by \( l_f = r_f = f \).

Let \( q \colon Y \relto Z \) be another span in \( \cat V \). We write the
pullback which defines \( q \cdot p \) as
\begin{equation*}
  \begin{tikzcd}
    M_{q \cdot p} \ar[r,"\pi_0"] \ar[d,"\pi_1",swap] 
                  \ar[rd,"\ulcorner",phantom,very near start]
      & M_q \ar[d,"l_q"] \\
    M_p \ar[r,"r_p",swap] & Y
  \end{tikzcd}
\end{equation*}
so that we have \( l_{q \cdot p} = l_p \circ \pi_1 \) and \( r_{q \cdot p} =
r_q \circ \pi_0 \). By abuse of notation, we may refer to instances of such
pullback diagrams as \( M_{q \cdot p} \).

The unitors \( \lambda \colon 1 \cdot p \to p \) and \( \rho \colon p \cdot 1
\to p \) in \( \SpanV \) are given by the pullback projections \( \pi_1 \colon
M_{1 \cdot p} \to M_p \) and \( \pi_0 \colon M_{p \cdot 1} \to M_p \),
respectively.

Given a third span \( r \colon Z \relto W \), note that the universal
property of the pullback \( M_{q \cdot p} \) guarantees the existence of a
unique map \( \pi_2 \colon M_{(r\cdot q) \cdot p} \to M_{q \cdot p} \) such
that \( \pi_1 \circ \pi_2 = \pi_1 \) and \( \pi_1 \circ \pi_0 = \pi_0 \circ
\pi_2 \):
\begin{equation*}
  \begin{tikzcd}
    M_{(r \cdot q) \cdot p} \ar[rdd,"\pi_1",swap,bend right]
                            \ar[r,"\pi_0"]
                            \ar[rd,dashed,"\pi_2" description]
      & M_{r \cdot q} \ar[rd,"\pi_1"] \\
    & M_{q \cdot p} \ar[r,"\pi_0"] \ar[d,"\pi_1",swap] 
                    \ar[rd,phantom,"\ulcorner",very near start]
    & M_q \ar[d,"l_q"] \\
    & M_p \ar[r,"r_p",swap]
    & Y
  \end{tikzcd}
\end{equation*}
With this, the associator \( \alpha \colon (r\cdot q) \cdot p \to r \cdot (q
\cdot p) \) may be defined as the unique map such that \( \pi_1 \circ \alpha =
\pi_2 \) and \( \pi_0 \circ \alpha = \pi_0 \circ \pi_0 \), via the universal
property of the pullback \( M_{r \cdot (q \cdot p)} \):
\begin{equation*}
  \begin{tikzcd}
    M_{(r \cdot q) \cdot p} \ar[rrd,"\pi_0 \circ \pi_0", bend left=15]
                            \ar[rdd,"\pi_2",bend right=25,swap]
                            \ar[rd,"\alpha" description, dashed] \\
    & M_{r \cdot (q \cdot p)} \ar[r,"\pi_0"] \ar[d,"\pi_1",swap] 
                              \ar[rd,phantom,"\ulcorner",very near start]
    & M_r \ar[d,"l_r"] \\
    & M_{q \cdot p} \ar[r,"r_q \circ \pi_0",swap] & Z
  \end{tikzcd}
\end{equation*}

\subsection*{Notation for $\VMat$:}

Let \( p \colon U \relto V \) be a \( \cat V \)-matrix. We denote by \( p(u,v)
\in \cat V \) the value of \(p\) at the pair \((u,v) \in U\times V \). A
2-cell of \( \cat V \)-matrices 
\begin{equation*}
  \begin{tikzcd}
    U \ar[r,"p"{name=A}] \ar[d,"f",swap]
      & V \ar[d,"g"] \\
    W \ar[r,"q"{name=B},swap] 
      & X
    \ar[from=A,to=B,"\theta",phantom]
  \end{tikzcd}
\end{equation*}
consists of a family of morphisms \( \theta_{u,v} \colon p(u,v) \to q(fu,gv)
\) in \( \cat V \), for \( u \in U \), \( v \in V \) and functions \( f \colon
U \to W \) and \( g \colon V \to X \). 

Given another 2-cell
\begin{equation*}
  \begin{tikzcd}
    W \ar[r,"q"{name=B}] 
      \ar[d,"h",swap]
      & X  \ar[d,"k"] \\
    Y \ar[r,"r"{name=C},swap]
    & Z
    \ar[from=B,to=C,"\omega",phantom]
  \end{tikzcd}
\end{equation*}
the composite \( \omega \circ \theta \) is given at \( u,v \) by the composite
of
\begin{equation*}
  \begin{tikzcd}
    p(u,v) \ar[r,"\theta_{u,v}"]
    & q(fu,gv) \ar[r,"\omega_{fu,gv}"]
    & r(hfu,kgv),
  \end{tikzcd}
\end{equation*}
exhibiting the structure of \( \VMat \) as an internal \( \Cat \)-graph.

Given \( u, u' \in U \), we write \( [u=u'] \) for the set that is a singleton
if \( u=u' \) and empty otherwise. Note that if we have a function \( f \colon
U \to V \), then there is a unique morphism \( [u=u'] \to [fu=fu'] \). With
this, the unit \( \cat V \)-matrix \( 1_U \colon U \relto U \) is defined by
\( 1_U(u,u') = [u = u'] \pt \trm \) for a set \(U\), and \( 1_f \) is given by
\( 1_f(u,u') \colon [u=u'] \pt \trm \to [fu=fu']\pt \trm \) for a function \(
f \colon U \to V \).

Recall that if \( t \colon V \relto W \) is another \( \cat V \)-matrix, we
have
\begin{equation*}
  (t \cdot p)(u,w) = \sum_{v \in V} t(v,w) \times p(u,v) 
\end{equation*}
which is the composition of \( \cat V \)-matrices. The composition for 2-cells
is defined likewise.

The unitors and associators are then given by taking coproducts over the
unitors and associators for the cartesian monoidal structure of \(
\cat V \).


\subsection*{Lifting the adjunction to $\Grph(\Cat)$:}

For a \( \cat V \)-matrix \(p \colon X \relto Y\), we define \( M_{p\pt \trm}
= \sum_{x,y} p(x,y) \), and we define \( p \pt \trm \colon X \pt \trm \relto Y
\pt \trm \) to be the span given by taking the coproduct of \( p(x,y) \to \trm
\) indexed by \( X\times Y \); this gives a morphism \( M_{p\pt \trm} \to X\pt
\trm \times Y \pt \trm \) (see Diagram \eqref{eq:defn.left.adj} below, which
is commutative by the universal property of the coproduct), whose composite
with the projections determine \( l_{p \pt \trm} \) and \( r_{p \pt \trm} \).

\begin{equation}
  \label{eq:defn.left.adj}
  \begin{tikzcd}[column sep=large]
    p(x,y) \ar[r] \ar[d]
      & \trm \ar[d,"{\lb \heta x,\heta y \rb}"] \\
    M_{p \pt \trm} \ar[r,"\lb l_{p \pt \trm}{,}\,r_{p \pt \trm} \rb",swap]
      & X \pt \trm \times Y \pt \trm
  \end{tikzcd}
\end{equation}

We write \( \hat \pi_0 \colon (q\cdot p)\pt \trm \to q \pt \trm \)
(respectively, \( \hat \pi_1 \colon (q \cdot p)\pt \trm \to p \pt \trm\)) for
the coproducts of the projections \( q(v,w) \times p(u,v) \to q(v,w) \)
indexed by \( U\times V \times W \to V \times W \) (respectively, \( q(v,w)
\times p(u,v) \to p(u,v) \) indexed by \( U \times V \times W \to U \times V
\)).

For a span \( s \colon V \relto W \) in \( \cat V \), we define the \( \cat V
\)-matrix \( \cat V(\trm,s) \colon \cat V(\trm,V) \relto \cat V(\trm,W) \) to
be given at \( v,w \) by the following pullback:

\begin{equation}
  \label{eq:defn.right.adj}
  \begin{tikzcd}
    \cat V(\trm,p)(v,w) \ar[rr] \ar[d] 
                     \ar[rrd,phantom,"\ulcorner",very near start]
      && \trm \ar[d,"{\lb v,w \rb}"] \\
    M_s \ar[rr,"{\lb l_s,r_s \rb}",swap] && V \times W
  \end{tikzcd}
\end{equation}
and if we have a 2-cell of spans \( \theta \colon s \to t \):
\begin{equation*}
  \begin{tikzcd}
    V \ar[d,"f",swap]
      & M_s \ar[l,"l_s",swap] \ar[d,"\theta"]
            \ar[r,"r_s"] 
      & W \ar[d,"g"] \\
    X & M_t \ar[l,"l_t"] \ar[r,"r_t",swap]
      & Y
  \end{tikzcd}
\end{equation*}
then \( \cat V(\trm,\theta) \) is the 2-cell uniquely determined by pullback
as follows:
\begin{equation*}
  \begin{tikzcd}[column sep=large]
    \cat V(\trm,s)(v,w) \ar[r,"{\cat V(\trm,\theta)(v,w)}",dashed]
                     \ar[d]
                     \ar[rd,"\ulcorner",phantom,very near start]
      & \cat V(\trm,t)(fv,gw) \ar[r] \ar[d]
                     \ar[rd,"\ulcorner",phantom,very near start]
      & \trm \ar[d,"{\lb fv,gw \rb}"] \\
    M_s \ar[r,"\theta",swap] 
      & M_t \ar[r,"{\lb l_t,r_t \rb}",swap]
      & X \times Y
  \end{tikzcd}
\end{equation*}
We observe that \( l_t \circ \theta = f \circ l_s \) and \( r_t \circ \theta =
f \circ r_s \).

We extend \( \heta \), \( \heps \) to \(\VMat\) and \(\SpanV\): for a \(\cat
V\)-matrix \(p \colon X \relto Y\), we define \( \heta_p \colon p \to \cat
V(\trm,p\pt \trm) \) at \(x,y\) to be given by the dashed arrow
\begin{equation}
  \label{eq:defn.unit}
  \begin{tikzcd}
    p(x,y) \ar[rrd,bend left=15] \ar[ddr,bend right=25] \ar[rd,dashed] \\
      & \cat V(\trm, p\pt \trm)(\heta x,\,\heta y) \ar[r] \ar[d] 
                        \ar[rd,phantom,"\ulcorner",very near start]
      & \trm \ar[d,"{\lb \heta x,\, \heta y \rb}"] \\
      & M_{p \pt \trm} \ar[r,swap,"{\lb l_{p\pt \trm},\,r_{p\pt \trm} \rb}"]
      & X\pt \trm \times Y \pt \trm
  \end{tikzcd}
\end{equation}

For a span \( s \colon V \relto W \), we let \( \heps_s \colon \cat V(\trm,s)
\pt \trm \to s \) to be given by taking the coproduct of
\eqref{eq:defn.right.adj} indexed by 
\begin{equation*}
  \begin{tikzcd}
    \cat V(\trm,V) \times \cat V(\trm,W) \ar[r,equal] \ar[d]
      & \cat V(\trm,V) \times \cat V(\trm,W) \ar[d] \\
    \trm \ar[r,equal] & \trm 
  \end{tikzcd}
\end{equation*}
which yields a commutative square
\begin{equation}
  \label{eq:defn.counit}
  \begin{tikzcd}
    M_{\cat V(\trm,s)\pt \trm} \ar[r] \ar[d,"\heps_s",swap]
      & \cat V(\trm,V) \pt \trm \times \cat V(\trm,W)\pt \trm 
        \ar[d,"{\lb \heps_V,\,\heps_W \rb}"] \\
    M_a \ar[r,swap,"{\lb l_s,\,r_s \rb}"]
      & V \times W
  \end{tikzcd}
\end{equation}

By taking the coproduct of \eqref{eq:defn.unit} over the diagram
\begin{equation*}
  \begin{tikzcd}
    X \times Y \ar[rd,equal] 
               \ar[rrd,bend left=15,equal] 
               \ar[rdd,bend right=25] \\
      & X \times Y \ar[r,equal] \ar[d] & X \times Y \ar[d] \\
      & \trm \ar[r,equal] & \trm
  \end{tikzcd}
\end{equation*}
we conclude that \( \heps_{p \pt \trm} \circ \heta_p \pt \trm = \id_{p \pt
\trm} \) for all \( \cat V \)-matrices \(p\).  Moreover, by considering the
following diagram
\begin{equation*}
  \begin{tikzcd}[column sep=large]
    \cat V(\trm,s)(v,w) \ar[r] \ar[rd,bend right=25]
      & \cat V(\trm,\cat V(\trm,s) \pt \trm)(\heta v, \heta w) 
          \ar[r,"\cat V({\trm,}\heps_s)_{\heta v,\heta w}"]
          \ar[d]
      & \cat V(\trm,s)(v,w) \ar[r] \ar[rd,"\ulcorner",near start,phantom] 
                         \ar[d]
      & \trm \ar[d,"{\lb v,w \rb}"] \\
      & M_{\cat V(\trm,s) \pt \trm} \ar[r,"\heps_s",swap]
      & M_s \ar[r,"{\lb l_s,r_s \rb}",swap]
      & V \times W
  \end{tikzcd}
\end{equation*}
we conclude \( \cat V(\trm,\heps_s) \circ \eta_{\cat V(\trm,s)} = \id_{\cat
V(\trm,s)} \) for all spans \(s\) in \( \cat V \). Hence, we have confirmed
that:
\begin{proposition}
  \label{lem:grph.cat.adj}
  We have an adjunction \( - \pt \trm \adj \cat V(\trm,-) \colon \SpanV \to
  \VMat \) in \( \Grph(\Cat) \).
\end{proposition}

\subsection*{Coherence for $- \pt \trm$:}

Now, we check that \( - \pt \trm \) is a normal oplax functor: for a set
\(X\), consider the pullback diagram
\begin{equation}
  \label{eq:defn.e}
  \begin{tikzcd}
    \left[x=y\right] \ar[r] \ar[d,"\delta_{x,y}",swap]
    & \trm \ar[d,"{\lb x,\,y \rb}"] \\
    X \ar[r,"\Delta",swap]
      & X \times X
  \end{tikzcd}
\end{equation}
so that \( \delta_{x,x} = x \) and \( \delta_{x,y} \) is uniquely determined
when \( x\neq y \). Now, we consider the image of the diagram
\eqref{eq:defn.e} under \(- \pt \trm\), and take its coproduct indexed by:
\begin{equation}
  \begin{tikzcd}
    \label{eq:index.e}
    X \times X \ar[r,equal] \ar[d] & X \times X \ar[d] \\
    \trm \ar[r] & \trm 
  \end{tikzcd}
\end{equation}
This yields us \( \e{-\pt \trm}_X \); since \( [x = y] \pt \trm \iso 0 \) for
\( x \neq y \), we conclude that this 2-cell is invertible.

Moreover, given a function \( f \colon X \to Y \), we observe that the
following diagram
\begin{equation*}
  \begin{tikzcd}
    \left[x = y\right] \pt \trm \ar[r,"1_{f,x,y}"] \ar[d] 
      & \left[fx = fy\right] \pt \trm \ar[d] \\
    X \pt \trm \ar[r,"f\pt \trm"] & Y \pt \trm
  \end{tikzcd}
\end{equation*}
commutes, as it is the image via \( - \pt \trm \) of a commutative diagram in
\(\Set\). Taking the coproduct over \eqref{eq:index.e} confirms naturality of
\( \e{-\pt \trm} \).

For \( \cat V \)-matrices \( p \colon X \relto Y \), \( q \colon Y \relto Z
\), \( \m{-\pt \trm} \) is depicted in Diagram \eqref{eq:defn.m} below by a
dashed arrow, and is uniquely determined by the universal property of the
pullback square:
\begin{equation}
  \label{eq:defn.m}
  \begin{tikzcd}[column sep=small,row sep=small]
    M_{(p \cdot q)\pt \trm}
      \ar[rrrd,bend left=15,"\hat \pi_0"] 
      \ar[dddr,bend right=25,"\hat \pi_1",swap] 
      \ar[rd,dashed,"\m{-\pt \trm}"] \\
    & M_{(p\pt \trm) \cdot (q\pt \trm)} \ar[rr,"\pi_0"] \ar[dd,"\pi_1",swap] 
                               \ar[rrdd,phantom,"\ulcorner",very near start]
      && M_{q \pt \trm} \ar[dd,"l_{q\pt \trm}"] \\ \\
    & M_{p \pt \trm} \ar[rr,swap,"r_{p\pt \trm}"] && Y \pt \trm
  \end{tikzcd}
\end{equation}

We consider the following horizontally composable 2-cells of \( \VMat \):
\begin{equation*}
  \begin{tikzcd}
    U \ar[r,"p_0"{name=A}]
      \ar[d,"f",swap]
      & V \ar[r,"p_1"{name=C}] 
          \ar[d,"g" description]
      & W \ar[d,"h"] \\
    X \ar[r,"q_0"{name=B},swap]
      & Y \ar[r,"q_1"{name=D},swap]
      & Z
    \ar[from=A,to=B,"\zeta_0",phantom]
    \ar[from=C,to=D,"\zeta_1",phantom]
  \end{tikzcd}
\end{equation*}
We have
\begin{align*}
  \pi_j \circ \m{-\pt \trm}_{q_0,q_1} \circ ((\zeta_1 \cdot \zeta_0) \pt \trm)
    &= (\pi_j \pt \trm) \circ ((\zeta_1 \cdot \zeta_0) \pt \trm) \\
    &= (\zeta_j \pt \trm) \circ (\pi_j \pt \trm) \\
    &= (\zeta_j \pt \trm) \circ \pi_j \circ \m{-\pt \trm}_{p_0,p_1} \\
    &= \pi_j \circ ((\zeta_1\pt \trm) \cdot (\zeta_0\pt \trm)) 
             \circ \m{-\pt \trm}_{p_0,p_1}
\end{align*}
for \( j = 0, 1 \). We obtain naturality via the universal property of the
pullback \( M_{(q_1 \pt \trm) \cdot (q_0 \pt \trm)} \).

To verify the unit comparison coherence of \( - \pt \trm \), we let \( p
\colon X \relto Y \) be a span in \( \cat V \), and we consider the composite
\begin{equation*}
  \begin{tikzcd}
    \sum_{x,y,z} 1_Y(y,z) \times p(x,y)
      \ar[r,"\m{-\pt \trm}"]
    & M_{(1_Y \pt \trm) \cdot (p \pt \trm)}
      \ar[r,"\e{-\pt \trm} \cdot \id"]
    & M_{1_{Y \pt \trm} \cdot (p \pt \trm)}
      \ar[r,"\lambda"]
    & \sum_{x,y} p(x,y).
  \end{tikzcd}
\end{equation*}
By definition, \( \lambda \colon M_{1_{Y \pt \trm} \cdot (p \pt \trm)} \to
M_{p \pt 1} \) is simply the pullback projection, thus \( \lambda \circ (\e{-
\pt \trm} \cdot \id) = \pi_1 \) is the pullback projection \( M_{(1_Y \pt
\trm) \cdot (p \pt \trm)} \to M_{p \pt \trm} \), and therefore \( \lambda
\circ (\e{-\pt \trm} \cdot \id) \circ \m{-\pt \trm} = \pi_1 \circ \m{- \pt
\trm} = \hat \pi_1 \) by \eqref{eq:defn.m}. But \( \hat \pi_1 \) itself is the
coproduct of \( 1_Y(y,z) \times p(x,y) \to p(x,y) \) indexed by the projection
\( X \times Y \times Z \to Y \times Z \), which is just \( \lambda \pt \trm
\). A similar argument confirms the right unitor case.

Now, we're left with verifying the composition comparison coherence of \( -
\pt \trm \). For the remainder of the section, we will denote horizontal
composition simply by concatenation. For an easier understanding of the
calculations, we provide the following diagram:
\begin{equation*}
  \begin{tikzcd}[column sep=small]
    ((ts)r) \pt \trm \ar[r] \ar[dddd]
      & ((ts)\pt \trm)(r \pt \trm) \ar[r] \ar[d]
      & ((t\pt \trm)(s\pt \trm))(r\pt \trm) \ar[rr] \ar[rddd] 
      && r \pt \trm \ar[dr]
    \\
      & (ts) \pt \trm \ar[rrdd]
      && (s\pt \trm)(r\pt \trm) \ar[ur] \ar[dr]
      && X \pt \trm
    \\
      &&&& s \pt \trm \ar[ur] \ar[dr]
    \\
      & (sr) \pt \trm \ar[uurr]
      && (t\pt \trm)(s\pt \trm) \ar[ur] \ar[dr]
      && Y \pt \trm
    \\
    (t(sr)) \pt \trm \ar[r]
      & (t\pt \trm)((sr) \pt \trm) \ar[r] \ar[u]
      & (t\pt \trm)((s\pt \trm)(r\pt \trm)) \ar[rr] \ar[uuur,crossing over]
      && t \pt \trm \ar[ur]
    \ar[from=1-3,to=4-4,crossing over]
    \ar[from=1-3,to=4-4,crossing over]
    \ar[from=1-3,to=5-3,crossing over]
  \end{tikzcd}
\end{equation*}

First, we verify that \( \m{-\pt \trm} \circ (\pi_1 \pt \trm) \circ (\alpha
\pt \trm) = \pi_1 \circ \alpha \circ (\m{-\pt \trm} \cdot \id) \circ \m{-\pt
\trm} \) as 2-cells \( ((ts)r)\pt \trm \to (s \pt \trm)(r\pt \trm) \). We have 
\begin{align*}
  \pi_0 \circ \m{-\pt \trm} \circ \hat \pi_1 \circ (\alpha \pt \trm)
    &= \hat \pi_0 \circ \hat \pi_1 \circ (\alpha \pt \trm) \\
    &= \hat \pi_1 \circ \hat \pi_0 \\
    &= \pi_1 \circ \m{-\pt \trm} \circ \pi_0 \circ \m{-\pt \trm} \\
    &= \pi_1 \circ \pi_0 \circ (\m{-\pt \trm} \cdot \id) \circ \m{-\pt \trm} \\
    &= \pi_0 \circ \pi_1 \circ \alpha 
             \circ (\m{-\pt \trm} \cdot \id) \circ \m{-\pt \trm} \\
  \pi_1 \circ \m{-\pt \trm} \circ \hat \pi_1 \circ (\alpha \pt \trm)
    &= \hat \pi_1 \circ \hat \pi_1 \circ (\alpha \pt \trm) \\
    &= \hat \pi_1 \\
    &= \pi_1 \circ \m{-\pt \trm} \\
    &= \pi_1 \circ (\m{-\pt \trm} \cdot \id) \circ \m{-\pt \trm} \\
    &= \pi_1 \circ \pi_1 \circ \alpha \circ (\m{-\pt \trm} \cdot \id) 
             \circ \m{-\pt \trm} ,
\end{align*}
and then we apply the universal property of the pullback \( M_{(s\pt
\trm)(r\pt \trm)} \). With this, we finish our proof: note that
\begin{align*}
  \pi_0 \circ \alpha \circ (\m{-\pt \trm} \cdot \id) \circ \m{-\pt \trm}
    &= \pi_0 \circ \pi_0 \circ (\m{-\pt \trm}\cdot\id) \circ \m{-\pt \trm} \\
    &= \pi_0 \circ \m{-\pt \trm} \circ \pi_0 \circ \m{-\pt \trm} \\
    &= (\pi_0 \pt \trm) \circ (\pi_0 \pt \trm) \\
    &= (\pi_0 \pt \trm) \circ (\alpha \pt \trm) \\
    &= \pi_0 \circ \m{-\pt \trm} \circ (\alpha \pt \trm) \\
    &= \pi_0 \circ (\id \cdot \m{-\pt \trm}) 
             \circ \m{-\pt \trm} \circ (\alpha \pt \trm) \\
  \pi_1 \circ (\id \cdot \m{-\pt \trm}) \circ \m{-\pt \trm} \circ (\alpha \pt \trm) 
    &= \m{-\pt \trm} \circ \pi_1 \circ \m{-\pt \trm} \circ (\alpha \pt \trm) \\
    &= \m{-\pt \trm} \circ (\pi_1 \pt \trm) \circ (\alpha \pt \trm) \\
    &= \pi_1 \circ \alpha \circ (\m{-\pt \trm} \cdot \id) \circ \m{-\pt \trm}
\end{align*}
then we apply the universal property of \( M_{(t\pt \trm)((s \pt \trm)(r \pt \trm))} \).
This concludes the proof of:
\begin{proposition}
  \label{thm:copower.oplax}
  \( - \pt \trm \colon \VMat \to \SpanV \) is a normal oplax functor.
\end{proposition}

  \section{Conjoints and companions}
    \label{sect:conj.comp}
    The notion of ``adjunction'' and the associated mate theory generalizes to two
(dual) internal notions in (pseudo)double category theory, first noted by
\cite{GP04} under the terminology \textit{orthogonal companions} and
\textit{adjoints} for the internalized adjunctions, and \textit{orthogonal
flipping} for the mate theory. 

These were concepts were later studied by various authors, with differing
terminologies. the \textit{holonomies} and \textit{foldings} of \cite{Fio07}
are, respectively, the orthogonal companions/adjoints, and the orthogonal
flipping of \cite{GP04}. In \cite{Shu08, CS10}, we find the use of
\textit{companions} and \textit{conjoints} in place of orthogonal
companions/adjoints, and their mate theory appears as \cite[Corollary
7.21]{CS10}. The work of \cite{DPP10} adapts to the terminology of companions
and conjoints, and considers their mate theory in \cite[Definition
3.22]{DPP10}, under the terminology \textit{Beck-Chevalley cells}. We also
have the mate correspondence for conjoints and companions in \cite[Proposition
5.13]{Shu11}, and more recently, \cite[Proposition 2.9]{GGV24}.

In this section, following the terminology of \cite{CS10, DPP10, Shu11}, we
review the notions of \textit{conjoints} and \textit{companions} in
(pseudo)double categories. We also give an account of their mate theory (Lemma
\ref{lemma:conjoint.mate.theory}), and review their behaviour with lax
functors. Our goal is to state, and prove, Theorem \ref{thm:conjoint.closed},
which was not found in previous works\footnote{A similar result has since
appeared in the preprint \cite{Pat24}, in their study of cartesian structures
of pseudodouble categories.}, which is an important tool for the developments
in Sections~\ref{sect:base.change} and~\ref{sect:induced.adjunction}.

Let \( \bicat D \) be a pseudodouble category, and let \(f \colon a \to b \)
be a vertical 1-cell, \(r \colon b \relto a \) be a horizontal 1-cell.  We say
that \(r\) is the \textit{conjoint} of \(f\) if there exist 2-cells
\begin{equation*}
  \begin{tikzcd}
    a \ar[d,"f",swap] \ar[r,"1"{name=U}] 
      & a \ar[d,equal] \\
    b \ar[r,"r"{name=V},swap] & a
    \ar[from=U,to=V,phantom,"\eta" description]
  \end{tikzcd}
  \qquad
  \begin{tikzcd}
    b \ar[r,"r"{name=U}] \ar[d,equal] 
      & a \ar[d,"f"] \\
    b \ar[r,"1"{name=V},swap] & b
    \ar[from=U,to=V,phantom,"\epsilon" description]
  \end{tikzcd}
\end{equation*}
such that \( \epsilon \circ \eta = 1_f \) and \( \eta \cdot \epsilon =
\rho^{-1} \circ \lambda \). We say \( \eta \), \( \epsilon \) are the
\textit{unit}, \textit{counit} of the conjoint, respectively. Also denote by
\textit{companion} the horizontally dual notion of conjoint; we denote the
\textit{unit} and \textit{counit} 2-cells of a companion as \( \nu \colon 1
\to r \) and \( \delta \colon r \to 1 \), respectively.

In any pseudodouble category \( \bicat D \), the identity vertical 1-cell on
any 0-cell \(x\) always has both a companion and a conjoint; in both cases, it
is given by the horizontal unit \( 1_x \), with unit and counit given by \(
\id_{1_x} = 1_{\id_x} \), which trivially satisfies all four conditions.
Unless otherwise specified, \( 1_x \) will be our fixed choice of
companion/conjoint for \( \id_x \).

We say that \( \bicat D \) is \textit{conjoint (companion) closed} if every
vertical 1-cell of \( \bicat D \) has a conjoint (companion). For instance,
\textit{equipments} may be defined as the pseudodouble categories which are
both conjoint and companion closed (see \cite[Theorem A.2]{Shu08}), of which
\( \SpanV \) and \( \VMat \) are examples.

Let \(T\) be a pseudomonad on a 2-category \( \bicat B \), and consider the
double category of lax \(T\)-algebras as described in Section
\ref{sect:psdbcats}. Our next result, Proposition \ref{prop:doct.adj}, given
in \cite[Theorems 1.4.11 and 1.4.14]{LucTh}, originally stated for strict
\(T\)-algebras in \cite{Kel74}, may be used to characterize conjoints and
companions in \( \LaxTAlg \). Since this is just a restatement of the results
of \cite[Chapter 1]{LucTh}, we omit the argument.

\begin{proposition}[Doctrinal adjunction]
  \label{prop:doct.adj}
  Let \( (f, g, \eta, \epsilon) \) be an adjunction in a 2-category \( \bicat
  B \). There is a bijection between 2-cells \( \zeta \) making \( (f,\zeta)
  \) into an lax \(T\)-algebra oplax morphism and 2-cells \( \xi \) making \(
  (g,\xi) \) into a lax \(T\)-algebra lax morphism.

  Moreover,  \( (f, \zeta) \) is the conjoint of \( (g, \xi) \) in \( \LaxTAlg
  \) if and only if \( \zeta \) and \( \xi \) correspond to each other via the
  aforementioned bijection, and \(f\) has a companion if and only if \( \zeta
  \) is invertible; in which case, its companion is \( (f,\zeta^{-1}) \). 
\end{proposition}

As is the case with ordinary adjunctions in a 2-category, there is also a
notion of \textit{mate theory} for conjoints (and dually, companions), which
we present in Lemma \ref{lemma:conjoint.mate.theory}. Results along these
lines were already present in \cite[{}1.6]{GP04}, as well as \cite[Corollary
7.21]{CS10} and \cite[Propositions 5.13 and 5.19]{Shu11}. We have decided to
provide a slightly different statement and proof: our goal is to provide
explicit formulas as an aid for calculations involving conjoints and
companions, abundant in this work.

\begin{lemma}
  \label{lemma:conjoint.mate.theory}
  Let \( (f,f^*,\eta,\epsilon) \) and \( (g,g^*,\eta,\epsilon) \) be
  conjoints, and consider 2-cells
  \begin{equation}
    \label{eq:pair.of.2.cells}
    \begin{tikzcd}
      u \ar[r,"r"{name=U}] \ar[d,"k \circ f",swap]
      & x \ar[d,"g \circ h"] \\
      w \ar[r,"s"{name=D},swap] & z
      \ar[from=U,to=D,phantom,"\zeta" description]
    \end{tikzcd}
    \qquad
    \begin{tikzcd}
      v \ar[d,"k",swap] \ar[r,"r \cdot f^*"{name=V}] &
      x \ar[d,"h"] \\
      w \ar[r,"g^* \cdot s"{name=E},swap] & y
      \ar[from=V,to=E,phantom,"\xi" description]
    \end{tikzcd}
  \end{equation}
  Then the following are equivalent:
  \begin{enumerate}[label=\emph{(\alph*)}]
    \setlength\itemsep{3mm}
    \item
      \label{enum:form.xi}
      \( \xi = \rho \circ \Big(\big((\eta \circ 1_h) 
                              \cdot \zeta\big) 
                              \cdot (1_k \circ \epsilon)\Big)
                    \circ \alpha^{-1} 
                    \circ \lambda^{-1} \),

    \item
      \label{enum:form.zeta}
      \( \zeta = \lambda \circ (\epsilon \cdot \id_s) \circ \xi
                        \circ (\id_r \cdot \eta) \circ \rho^{-1} \),
    \item
      \label{enum:form.eps}
      \( \lambda \circ (\epsilon \cdot \id_s) \circ \xi
          = \rho \circ \big(\zeta \cdot (1_k \circ \epsilon)\big) \),

    \item
      \label{enum:form.eta}
      \( \xi \circ (\id_r \cdot \eta) \circ \rho^{-1}
          = \big((\eta \circ 1_h) \cdot \zeta\big) \circ \lambda^{-1} \)
  \end{enumerate}
  In particular, the sets of 2-cells as given in \eqref{eq:pair.of.2.cells}
  are in pairwise correspondence, explicitly given by the formulas
  \ref{enum:form.xi} and \ref{enum:form.zeta}. Pairs of such 2-cells are said
  to be \textit{mates} or under \textit{mate correspondence}.
\end{lemma}

\begin{proof}
  We will prove that (c) \( \to_\text{i} \) (b) \( \to_\text{ii} \) (d) \(
  \to_\text{iii} \) (a) \( \to_\text{iv} \) (c).
  \begin{enumerate}[label=(\roman*)]
    \item
      Since \( \epsilon \circ \eta = 1_f \), we have
      \begin{align*}
        \lambda \circ (\epsilon \cdot \id_q) \circ \xi 
                \circ (\id_p \cdot \eta) \circ \rho^{-1}
          &= \rho \circ \big(\zeta \cdot (1 \circ \epsilon)\big) 
                 \circ (\id_p \cdot \eta) \circ \rho^{-1} \\
          &= \rho \circ \big((\zeta \circ \id_p) 
                            \cdot (1 \circ \epsilon \circ \eta)\big)
                 \circ \rho^{-1}
          = \zeta
      \end{align*}

    \item
      Since \( \eta \cdot \epsilon = \rho^{-1} \circ \lambda \), we have
      \begin{align*}
        \big((\eta \circ 1) \cdot \zeta\big) \circ \lambda^{-1}
          &= \Big((\eta \circ 1) \cdot 
                \big(\lambda 
                       \circ (\epsilon \cdot \id_q)
                       \circ \xi
                       \circ (\id_q \cdot \eta)
                       \circ \rho^{-1}) \big)\Big) \circ \lambda^{-1} \\
          &= (\id_s \cdot \lambda) 
              \circ (\eta \cdot (\epsilon \cdot \id_q))
              \circ \big(1 \cdot 
                        (\xi \circ (\id_p \cdot \eta) 
                             \circ \rho^{-1})\big) 
              \circ \lambda^{-1} \\
          &= (\id_s \cdot \lambda) 
                \circ \alpha 
                \circ (\rho^{-1} \cdot \id_q) \circ (\lambda \cdot \id_q)
                \circ \alpha^{-1} \circ \lambda^{-1}
                \circ \xi \circ (\id_p \cdot \eta) \circ \rho^{-1} 
        \end{align*}
        and coherence guarantees 
        \begin{equation*}
          (\id \cdot \lambda) \circ \alpha 
                              \circ (\rho^{-1} \cdot \id)
                              \circ (\lambda \cdot \id)
                              \circ \alpha^{-1} 
                              \circ \lambda^{-1} = \id,
        \end{equation*}
        as desired.

    \item
      Since \( \eta \cdot \epsilon = \rho^{-1} \circ \lambda \), we have
      \begin{align*}
        \big((\eta \circ 1) \cdot \zeta\big) \cdot (1 \circ \epsilon)
          &= \big(\xi \circ (\id_p \cdot \eta) 
                      \circ \rho^{-1} \circ \lambda\big) 
              \cdot \epsilon \\
          &= (\xi \cdot 1) 
              \circ ((\id_p \cdot \eta) \cdot \epsilon)
              \circ ((\rho^{-1} \circ \lambda) \cdot \id_{g^*}) \\
          &= (\xi \cdot 1) \circ \alpha^{-1}
                           \circ (\id_p \cdot (\rho^{-1} \circ \lambda))
                           \circ \alpha
                           \circ ((\rho^{-1} \circ \lambda) \cdot \id_{g^*})
      \end{align*}
      and coherence guarantees
      \begin{equation*}
        \rho \circ \alpha^{-1} 
             \circ (\id \cdot (\rho^{-1} \circ \lambda)) 
             \circ \alpha
             \circ ((\rho^{-1} \circ \lambda) \cdot \id)
             \circ \alpha^{-1}
             \circ \lambda^{-1} = \id,
      \end{equation*}
      as desired.

    \item
      Since \( \epsilon \circ \eta = 1_g \), we have
      \begin{align*}
        (\epsilon \cdot \id_q) \circ \xi
          &= (\epsilon \cdot \id_q) 
              \circ \rho
              \circ \big(((\eta \circ 1) \cdot \zeta) 
                            \cdot (1 \circ \epsilon)\big)
              \circ \alpha^{-1}
              \circ \lambda^{-1} \\
          &= \rho \circ ((\epsilon \cdot \id_q) \cdot \id_1)
                  \circ (((\eta \circ 1) \cdot \zeta) \cdot (1 \circ \epsilon))
                  \circ \alpha^{-1}
                  \circ \lambda^{-1} \\
          &= \rho \circ ((1 \cdot \zeta) \cdot \epsilon)
                  \circ \alpha^{-1}
                  \circ \lambda^{-1} \\
          &= \rho \circ \alpha^{-1} \circ \lambda^{-1}
                  \circ (\zeta \cdot \epsilon)
      \end{align*}
      and coherence guarantees
      \begin{equation*}
        \rho \circ \alpha^{-1} \circ \lambda^{-1} = \lambda^{-1} \circ \rho,
      \end{equation*}
      as desired.
  \end{enumerate}
\end{proof}

\begin{remark}
  \label{rem:mate.corresp}
  Once more, we consider the pair of 2-cells \( \zeta, \xi \) given in
  \eqref{eq:pair.of.2.cells}. We will consider the following specialized
  instances of the mate correspondence:
  \begin{enumerate}[label=(\roman*)]
    \item
      \label{enum:degen.i}
      For \( k = \id \), (c) becomes \( (\epsilon \cdot \id) \circ \xi =
      \gamma^{-1} \circ (\zeta \cdot \epsilon) \).
    \item
      \label{enum:degen.ii}
      For \( h = \id \), (d) becomes \( \xi \circ (\id \cdot \eta) = (\eta
      \cdot \zeta) \circ \gamma^{-1} \).
    \item
      \label{enum:degen.iii}
      For \( s=1 \), (c) becomes \( \epsilon \circ \theta = \rho \circ
      (\zeta \cdot (1_k \circ \epsilon)) \), where \( \theta = \rho \circ \xi
      \).
    \item
      \label{enum:degen.iv}
      For \( r=1 \), (d) becomes \( \theta \circ \eta = ((\eta \circ 1_h)
      \cdot \zeta) \circ \lambda^{-1} \), where \( \theta = \xi \circ
      \lambda^{-1} \).
    \item
      \label{enum:degen.v}
      For \( f=\id \), both (b) and (c) become \( \zeta = \lambda \circ
      (\epsilon \cdot \id) \circ \theta \), where \( \theta = \xi \circ
      \rho^{-1} \).
    \item
      \label{enum:degen.vi}
      For \( g=\id \), both (b) and (d) become \( \zeta = \theta \circ (\id
      \cdot \eta) \circ \rho^{-1} \), where \( \theta = \lambda \circ \xi  \).
  \end{enumerate}
  And by combining these, we may obtain simpler forms. For example,
  \ref{enum:degen.v} and \ref{enum:degen.iii} (respectively,
  \ref{enum:degen.vi} and \ref{enum:degen.iv})) provide the result that the
  counit (unit) of a conjunction is a cartesian (opcartesian) 2-cell in the
  sense of \cite{Shu08, CS10}.

  The combination of \ref{enum:degen.iii} and \ref{enum:degen.iv} is mainly
  used under the hypothesis that we have a commutative square \( k \circ f = g
  \circ h \) of vertical 1-cells, that is, \( \zeta = \id \). In this case,
  the unit \( 1_{g \circ h} = 1_{k \circ f} \) has two mates; they are said to
  be the \textit{mates of the commutative square} \( k \circ f = g \circ h \),
  and are the unique 2-cells \( \theta \), \( \omega \), respectively
  satisfying
  \begin{equation*}
    \epsilon \circ \theta = 1_k \circ \epsilon
    \quad\text{and}\quad
    \theta \circ \eta = \eta \circ 1_h
  \end{equation*}
  \begin{equation*}
    \epsilon \circ \omega = 1_g \circ \epsilon
    \quad\text{and}\quad
    \omega \circ \eta = \eta \circ 1_f
  \end{equation*}
  In practice, we will mention ``the'' mate of a commutative square \( k
  \circ f = g \circ h \), and we let context determine which mate is being
  considered.
\end{remark}

We proceed to review well-known \cite{CS10, GP04, DPP10, Shu08, Shu11}, yet
fundamental results about companions and conjoints. Our aim is to demonstrate
the applications of their mate theory, while fixing notation to use for later
reference in Sections \ref{sect:base.change} and
\ref{sect:induced.adjunction}.

Let \( F \colon \bicat D \to \bicat E \) be a lax functor of conjoint closed
pseudodouble categories,  and let \(f\) is a vertical 1-cell in \(\bicat D\).
We denote the mate of \( F\eta \circ \e F \) obtained via \ref{enum:degen.vi}
by \( \sigma^F_f \colon (Ff)^* \to F(f^*) \):
\begin{equation*}
  \begin{tikzcd}[column sep=large]
    \cdot \ar[r,"1"{name=A}] \ar[d,"Ff",swap] & \cdot \ar[d,equal] \\
    \cdot \ar[r,"(Ff)^*" description, ""{name=B}] \ar[d,equal] 
      & \cdot \ar[d,equal] \\
    \cdot \ar[r,"F(f^*)"{name=C},swap] & \cdot 
    \ar[from=A,to=B,phantom,"\eta" description]
    \ar[from=B,to=C,phantom,"\sigma^F_f" description]
  \end{tikzcd}
  =
  \begin{tikzcd}[column sep=large]
    \cdot \ar[r,"1"{name=A}] \ar[d,equal] & \cdot \ar[d,equal] \\
    \cdot \ar[r,"F1" description,""{name=B}] \ar[d,"Ff",swap]
      & \cdot \ar[d,equal] \\
    \cdot \ar[r,"F(f^*)",""{name=C},swap] & \cdot 
    \ar[from=A,to=B,phantom,"\e F" description]
    \ar[from=B,to=C,phantom,"F\eta" description]
  \end{tikzcd}
\end{equation*}

We say that \(F\) \textit{preserves the conjoint} of \(f\) if \( \sigma^F_f \)
is an invertible 2-cell; we say \(F\) \textit{preserves conjoints} if \(
\sigma^F_f \) is invertible for all vertical 1-cells \(f\). We can show that:

\begin{lemma}
  Let \(F \colon \bicat D \to \bicat E \) be a lax functor of conjoint closed
  pseudodouble categories. The following are equivalent:
  \begin{enumerate}[label=(\alph*)]
    \item
      \label{enum:f.conjoint.id}
      \(F\) preserves conjoints of identities,
    \item
      \label{enum:f.conjoint}
      \(F\) preserves all conjoints,
    \item
      \label{enum:f.normal}
      \(F\) is normal.
  \end{enumerate}
\end{lemma}

\begin{proof}
  We begin by showing that any lax functor \(F\) satisfies the identity \(
  F\epsilon \circ \sigma^F = \e F \circ \epsilon \), for we have
  \begin{equation*}
    F\epsilon \circ \sigma^F \circ \eta = F\epsilon \circ F\eta \circ \e F
      = F1_f \circ \e F = \e F \circ 1_{Ff} = \e F \circ \epsilon \circ \eta,
  \end{equation*}
  so the desired equation follows by \ref{enum:degen.vi}.

  Moreover, whenever \( \sigma^F_f \) is invertible, the following relations
  hold:
  \begin{align*}
    \epsilon \circ (\sigma^F)^{-1} \circ F\eta \circ \e F
      &= \epsilon \circ \eta = 1_f, \\
    \e F \circ \epsilon \circ (\sigma^F)^{-1} \circ F\eta 
      &= F\epsilon \circ F\eta = F1_f.
  \end{align*}
  Hence, if \(\sigma^F_\id \) is invertible for all 0-cells, we conclude that
  \(\e F\) is invertible; this confirms \ref{enum:f.conjoint.id} \( \to \)
  \ref{enum:f.normal}.

  Now, if we assume \(F\) is normal, we let \( \chi^F \) be the unique 2-cell
  such that \( \epsilon \circ \chi^F = (\e F)^{-1} \circ F\epsilon \),
  obtained via \ref{enum:degen.v}. From this, it is clear that \( \chi^F \circ
  \sigma^F = \id \), since
  \begin{equation*}
    \epsilon \circ \chi^F \circ \sigma^F
      = (\e F)^{-1} \circ F\epsilon \circ \sigma^F = \epsilon,
  \end{equation*}
  and
  \begin{align*}
    \rho^{-1} \circ \sigma^F \circ \chi^F \circ \lambda
      &= (\sigma^F \cdot \id) \circ \rho^{-1} 
                              \circ \lambda \circ (\id \cdot \chi^F) \\
      &= (\sigma^F \cdot \id) \circ (\eta \cdot \epsilon)
                              \circ (\id \cdot \chi^F) \\
      &= (\sigma^F \circ \eta) \cdot (\epsilon \circ \chi^F) \\
      &= (F\eta \circ \e F) \cdot ((\e F)^{-1} \circ F\epsilon) \\
      &= (\id \cdot (\e F)^{-1}) \circ (F\eta \cdot F\epsilon)
                                \circ (\e F \cdot \id) \\
      &= \rho^{-1} \circ F\rho \circ \m F \circ (F\eta \cdot F\epsilon)
                   \circ (\e F \cdot \id) \\
      &= \rho^{-1} \circ F\rho \circ F(\eta \cdot \epsilon) \circ \m F
                   \circ (\e F \cdot \id) \\
      &= \rho^{-1} \circ \lambda
  \end{align*}
  confirms that \( \chi^F \) is the inverse of \( \sigma^F \). We have shown
  that \ref{enum:f.normal} \( \to \) \ref{enum:f.conjoint}, and of course,
  \ref{enum:f.conjoint.id} is a consequence of \ref{enum:f.conjoint}.
\end{proof}

For the case of companions, we write \( \tau^F_f \colon (Ff)_! \to F(f_!) \)
for the mate of \( F\nu \circ \e F \), and we say that \(F\) \textit{preserves
the companion} of \(f\) if \( \tau^F \) is invertible. The horizontally dual
result states that \(F\) preserves companions iff \(F\) is normal. Thus, we
obtain the result that these three conditions are equivalent for lax functors
between pseudodouble categories (\cite[Proposition 3.8]{DPP10}).

\begin{lemma}
  \label{lem:conjoint.eq.normal}
  Let \(F \colon \bicat D \to \bicat E \) be a lax functor between conjoint
  closed pseudodouble categories, and let \(r \colon x \relto y\), \(f \colon
  z \to y\) be horizontal, vertical 1-cells respectively. Then the 2-cell
  \begin{equation*}
    \begin{tikzcd}
      (Ff)^* \cdot Fr \ar[r,"\sigma^F \cdot \id"]
      & F(f^*) \cdot Fr \ar[r,"\m F"]
      & F(f^* \cdot r)
    \end{tikzcd}
  \end{equation*}
  is invertible. In particular, \( \m F \colon F(f^*) \cdot Fr \to F(f^* \cdot
  r) \) is invertible for all such \(r,f\) if and only if \(F\) is normal. 
\end{lemma}

\begin{proof}
  We claim the inverse \( l^F \) is given by the mate of \( F\theta \) via 
  \begin{equation*}
    l^F =
    \begin{tikzcd}
      \cdot \ar[rr,"F(f^* \cdot r)"{name=A}] \ar[d,equal]
      && \cdot \ar[d,equal] \\
      \cdot \ar[r,"F(f^* \cdot r)" description,""{name=B}]
            \ar[d,equal]
      & \cdot \ar[r,"1" description,""{name=C}] \ar[d,"Ff" description]
      & \cdot \ar[d,equal] \\
      \cdot \ar[r,"Fr"{name=D},swap]
      & \cdot \ar[r,"(Ff)^*"{name=E},swap] & \cdot
      \ar[from=A,to=2-2,phantom,"\lambda^{-1}" description]
      \ar[from=B,to=D,phantom,"F\theta" description]
      \ar[from=C,to=E,phantom,"\eta" description]
    \end{tikzcd},
    \quad\text{where}\quad
    \theta = 
    \begin{tikzcd}
      \cdot \ar[r,"r"{name=A}] \ar[d,equal]
        & \cdot \ar[d,equal] \ar[r,"f^*"{name=B}]
        & \cdot \ar[d,"f"] \\
      \cdot \ar[r,""{name=C},"r" description] \ar[d,equal]
        & \cdot \ar[r,""{name=D},"1" description] 
        & \cdot \ar[d,equal] \\
      \cdot \ar[rr,"r"{name=E},swap] && \cdot
      \ar[from=A,to=C,phantom,"=" description]
      \ar[from=B,to=D,phantom,"\epsilon" description]
      \ar[from=2-2,to=E,phantom,"\lambda" description]
    \end{tikzcd}.
  \end{equation*}
  Note that \( l^F \) is the mate of \( F\theta \), and \( \theta \) is the
  mate of \( \id_{f^* \cdot r} \), via \ref{enum:degen.iv} and
  \ref{enum:degen.v}, respectively.  Now, note that
  \begin{align*}
    (\epsilon \cdot \id) \circ l^F \circ \m F \circ (\sigma^F \cdot \id)
      &= \lambda^{-1} \circ F\theta \circ \m F \circ (\sigma^F \cdot \id) \\
      &= \lambda^{-1} \circ F\lambda \circ \m F \circ (F\epsilon \cdot \id) 
                      \circ (\sigma^F \cdot \id) \\
      &= \lambda^{-1} \circ F\lambda \circ \m F \circ (\e F \cdot \id)
                      \circ (\epsilon \cdot \id) = \epsilon \cdot \id \\
    \m F \circ (\sigma^F \cdot \id) \circ l^F
      &= \m F \circ (F\eta \cdot F\theta) 
              \circ (\e F \cdot \id) \circ \lambda^{-1} \\
      &= F(\eta \cdot \theta) \circ \m F 
                              \circ (\e F \cdot \id) \circ \lambda^{-1} 
       = \id
  \end{align*}
  So, the result follows by the mate correspondence. 
\end{proof}

In a conjoint closed pseudodouble category \(\bicat D\), let \(f,g\) be
composable vertical 1-cells with conjoints \( f^* \) and \( g^* \), and let \(
\pi \colon f^* \cdot g^* \to (g\circ f)^* \) be the mate of \( 1_g \circ
\epsilon \colon f^* \to 1 \). Via \ref{enum:degen.i} and \ref{enum:degen.iii},
we obtain:
\begin{equation}
  \label{eq:v2h.comp}
  \begin{tikzcd}[column sep=large,row sep=small]
    \cdot \ar[r,"f^* \cdot g^*"{name=A}] \ar[d,equal]
    & \cdot \ar[d,equal] \\
    \cdot \ar[r,""{name=B},"(g\circ f)^*" description] \ar[d,equal]
    & \cdot \ar[d,"g \circ f"] \\
    \cdot \ar[r,"1"{name=C},swap] & \cdot
    \ar[from=A,to=B,phantom,"\pi" description]
    \ar[from=B,to=C,phantom,"\epsilon" description]
  \end{tikzcd}
  =
  \begin{tikzcd}[column sep=large]
    \cdot \ar[r,"g^*"{name=A}] \ar[d,equal]
      & \cdot \ar[r,"f^*"{name=X}] \ar[d,"g" description]
      & \cdot \ar[d,"g \circ f"] \\
    \cdot \ar[r,""{name=B},"1" description] \ar[d,equal] 
      & \cdot \ar[r,""{name=Y}, "1" description] 
      & \cdot \ar[d,equal] \\
    \cdot \ar[rr,"1"{name=Z},swap] && \cdot
    \ar[from=A,to=B,phantom,"\epsilon" description]
    \ar[from=X,to=Y,phantom,"1_g \circ \epsilon" description]
    \ar[from=2-2,to=Z,phantom,"\rho" description]
  \end{tikzcd}
\end{equation}

We can also define a 2-cell \((g \circ f)^* \to f^* \cdot g^* \) as the mate
of \( \eta \circ 1_f \), which can be shown to be the inverse of \(\pi\),
using a method similar to the proof of Lemma \ref{lem:conjoint.eq.normal}; as
this is not needed, we omit the details.

Now that we have fixed the notation we will need for the rest of the paper,
we end the section with Theorem \ref{thm:conjoint.closed}, to justify the
value of conjoint and companion closed pseudodouble categories.

\begin{theorem}
  \label{thm:conjoint.closed}
  Let \( \bicat D, \bicat E \) be pseudodouble categories. If \( \bicat E \)
  is conjoint closed, then so is \( \Lax_\lax(\bicat D,\bicat E) \).  Dually,
  if \( \bicat E \) is companion closed, then so is \( \Lax_\opl(\bicat D,
  \bicat E) \). 
\end{theorem}

\begin{proof}
  Fix a vertical transformation \( \phi \colon F \to G \) where \( F,G \colon
  \bicat D \to \bicat E \) are lax functors. For each 0-cell \(x\), we write
  \begin{equation*}
    \begin{tikzcd}
      Fx \ar[d,"\phi_x",swap] \ar[r,"1"{name=U}] 
        & Fx \ar[d,equal] \\
      Gx \ar[r,"\phi^*_x"{name=V},swap] & Fx
      \ar[from=U,to=V,phantom,"\eta_x" description]
    \end{tikzcd}
    \quad
    \begin{tikzcd}
      Gx \ar[r,"\phi^*_x"{name=U}] \ar[d,equal] 
        & Fx \ar[d,"\phi_x"] \\
      Gx \ar[r,"1"{name=V},swap] & Gx
      \ar[from=U,to=V,phantom,"\epsilon_x" description]
    \end{tikzcd}
  \end{equation*}
  for the 2-cells satisfying \( \epsilon_x \circ \eta_x = 1_{\phi_x} \) and
  \( \eta_x \cdot \epsilon_x = \rho^{-1} \circ \lambda \), so that \( \phi^*_x
  \) is the conjoint of \( \phi_x \) for all \(x\).

  Define \( \phi^*_f \colon \phi^*_x \to \phi^*_y \) to be the mate of \(
  1_{Ff} \) via \ref{enum:degen.ii}, so that \( \phi^*_f \circ \eta_x = \eta_y
  \circ 1_{Ff} \) and \( \epsilon_y \circ \phi^*_f = 1_{Gf} \circ \epsilon_x
  \). Moreover, note that
  \begin{equation*}
    \phi_g^* \circ \phi^*_f \circ \eta_x 
      = \phi^*_g \circ \eta_y \circ 1_{Ff}
      = \eta_z \circ 1_{Gf} \circ 1_{Ff} = \eta_z \circ 1_{G(g\circ f)},
  \end{equation*}
  so we conclude that \( \phi^*_{g\circ f} = \phi_g^* \circ \phi_f^* \) by
  mate correspondence.  Similarly, we have \( \phi^*_{\id_x} = \id_{\phi_x}^* \).

  Next, we consider the map \( r \mapsto Fr \cdot \phi^*_x \), where \( r
  \colon x \relto y \) is a horizontal 1-cell. It is functorial: for 2-cells
  \begin{equation}
    \label{2.cells}
    \begin{tikzcd}
      x \ar[r,"r"{name=U}] \ar[d,"f",swap] & y \ar[d,"g"] \\
      w \ar[r,"s"{name=V},swap] & z
      \ar[from=U,to=V,phantom,"\theta" description]
    \end{tikzcd}
    \qquad
    \begin{tikzcd}
      u \ar[r,"q"{name=U}] \ar[d,"h",swap] & v \ar[d,"k"] \\
      x \ar[r,"r"{name=V},swap] & y
      \ar[from=U,to=V,phantom,"\xi" description]
    \end{tikzcd}
  \end{equation}
  we have 
  \begin{equation*}
    F(\theta \circ \chi) \cdot \phi^*_{h \circ f} 
      = (F\theta \circ F\chi) \cdot (\phi^*_h \circ \phi^*_f) 
      = (F\theta \cdot \phi^*_h) \circ (F\chi \cdot \phi^*_f),
  \end{equation*}
  and \( F(\id) \cdot \phi^*_{\id} = \id \), as desired. Analogously, \( r
  \mapsto \phi^*_y \cdot Gr \) is also functorial.

  We define \( \n{\phi^*}_r \colon Fr \cdot \phi^*_x \to \phi^*_y \cdot Gf \)
  to be the mate of
  \begin{equation*}
    \begin{tikzcd}
      Fx \ar[d,"\phi_x",swap] \ar[r,"Fr"{name=U}] 
        & Fy \ar[d,"\phi_y"] \\
      Gx \ar[r,"Gr"{name=V},swap] & Gy
      \ar[from=U,to=V,phantom,"\phi_r" description]
    \end{tikzcd}
  \end{equation*}
  via \ref{enum:degen.i}. We claim this data makes \( \phi^* \) into a lax
  horizontal transformation \( G \relto F \).

  Given a 2-cell \( \theta \) as in the left diagram of \eqref{2.cells}, we
  have \( \phi_s \circ F\theta = G\theta \circ \phi_r \), since \( \phi \) is
  a vertical transformation. The following pairs
  \begin{align*}
    G\theta \circ \phi_r 
      \quad\text{and}\quad \n{\phi^*}_s \circ (F\theta \cdot \phi^*_f), \\
    \phi_s \circ F\theta 
      \quad\text{and}\quad (\phi^*_g \cdot G\theta) \circ \n{\phi^*}_r 
  \end{align*}
  are mates, so that we have
  \begin{equation*}
    \n{\phi^*}_s \circ (F\theta \cdot \phi^*_f) 
      = (\phi^*_g \cdot G\theta) \circ \n{\phi^*}_r,
  \end{equation*}
  giving naturality. To confirm this,
  \begin{align*}
    \lambda \circ (\epsilon \cdot \id) 
            \circ \n{\phi^*}_s 
            \circ (F\theta \cdot \phi^*_f)
      &= \rho \circ (\phi_s \cdot \epsilon)
              \circ \circ (F\theta \cdot \phi^*_f) \\
      &= \rho \circ \big((\phi_s \circ F\theta) 
                            \cdot (\epsilon \circ \phi^*_f)\big) \\
      &= \rho \circ \big((\phi_s \circ F\theta)
                            \cdot (1 \circ \epsilon)\big), \\
    (\phi^*_g \cdot G\theta) \circ \n{\phi^*}_r
                             \circ (\id \cdot \eta)
                             \circ \rho^{-1} 
      &= (\phi^*_g \cdot G\theta) \circ (\eta \cdot \phi_r) 
                                 \circ \lambda^{-1} \\
      &= \big( (\phi^*_g \circ \eta) \cdot (G\theta \circ \phi_r)\big)
          \circ \lambda^{-1} \\
      &= \big( (\eta \circ 1) \cdot (G\theta \circ \phi_r) \big) 
          \circ \lambda^{-1}
  \end{align*}

  Now, we note that \( \phi_{1_x} \circ \e F_x = \e G_x \circ 1_{\phi_x}
  \) and \( \phi_{s \cdot r} \circ \m F = \m G \circ (\phi_s \cdot \phi_r) \).
  We shall deduce that the coherence diagrams for \(\phi^*\) commute by taking
  the mates of these commutative squares, thereby confirming that \( \phi^* \)
  is a lax horizontal transformation. Via \ref{enum:degen.i}, we will prove
  that the following pairs
  \begin{align*}
    \m G \circ (\phi_s \cdot \phi_r) \quad&\text{and}\quad
    (\id \cdot \m G) \circ \alpha \circ (\phi^*_s \cdot \id)
                    \circ \alpha^{-1} 
                    \circ (\id \cdot \phi^*_r) \circ \alpha \\
    \phi_{s \cdot r} \circ \m F \quad&\text{and}\quad
      \n{\phi^*}_{s \cdot r} \circ (\m F \cdot \id) \\
    \phi_{1_x} \circ \e F \quad&\text{and}\quad
      \n{\phi^*}_{1_x} \circ (\e F \cdot \id) \\
    \e G \circ 1_{\phi_x} \quad&\text{and}\quad
      (\id \cdot \e G) \circ \rho^{-1} \circ \lambda
  \end{align*}
  are under mate correspondence. The last three are one-liners, respectively:
  \begin{align*}
    \lambda \circ (\epsilon \cdot \id)
            \circ \n{\phi^*}_{s \cdot r}
            \circ (\m F \cdot \id)
      &= \rho \circ (\phi_{s \cdot r} \cdot \epsilon)
            \circ (\m F \cdot \id)
       = \rho \circ ((\phi_{s\cdot r} \circ \m F) \cdot \epsilon), \\
    \lambda \circ (\epsilon \cdot \id)
            \circ \n{\phi^*}_{1_x} \circ (\e F \cdot \id)
      &= \rho \circ (\phi_{1_x} \cdot \epsilon)
              \circ (\e F \cdot \id)
       = \rho \circ ((\phi_{1_x} \circ \e F) \cdot \id) \\
    \lambda \circ (\epsilon \cdot \id) 
            \circ (\id \cdot \e G)
            \circ \rho^{-1} \circ \lambda
            \circ (\id \cdot \eta) \circ \rho^{-1}
      &= \e G \circ \lambda \circ (\epsilon \cdot \id)
             \circ \rho^{-1} \circ \eta
       = \e G \circ 1_{\phi_x}.
  \end{align*}
  For the first pair, observe that
  \begin{align*}
    \lambda \circ (\epsilon \cdot \id)
            \circ (\id \cdot \m G)
            \circ \alpha
      &= \lambda \circ (1 \cdot \m G) 
                 \circ (\epsilon \cdot \id)
                 \circ \alpha \\
      &= \m G \circ \lambda
             \circ \alpha
             \circ ((\epsilon \cdot \id) \cdot \id) \\
      &= \m G \circ (\lambda \cdot \id) 
             \circ ((\epsilon \cdot \id) \cdot \id) \\
    ((\epsilon \cdot \id) \cdot \id) \circ (\n{\phi^*}_s \cdot \id)
                                     \circ \alpha^{-1}
      &= (((\epsilon \cdot \id) \circ \n{\phi^*}_s) \cdot \id) 
              \circ \alpha^{-1} \\
      &= ((\gamma^{-1} \circ (\phi_s \cdot \epsilon)) \cdot \id)
            \circ \alpha^{-1} \\
      &= (\gamma^{-1} \cdot \id) 
            \circ ((\phi_s \cdot \epsilon) \cdot \id)
            \circ \alpha^{-1} \\
      &= (\lambda^{-1} \cdot \lambda) 
            \circ (\phi_s \cdot (\epsilon \cdot \id)) \\
    (\phi_s \cdot (\epsilon \cdot \id)) 
            \circ (\id \cdot \n{\phi^*}_r)
            \circ \alpha
      &= ( \phi_s \cdot ((\epsilon \cdot \id) \circ \n{\phi^*}_r))
            \circ \alpha \\
      &= (\phi_s \cdot (\gamma^{-1} \circ (\phi_r \cdot \epsilon)))
            \circ \alpha \\
      &= (\id \cdot \gamma^{-1}) 
            \circ (\phi_s \cdot (\phi_r \cdot \epsilon))
            \circ \alpha \\
      &= (\id \cdot \gamma^{-1}) 
            \circ \alpha 
            \circ ((\phi_s \cdot \phi_r) \cdot \epsilon)
  \end{align*}
  and pasting the expressions above together verifies the claim.

  Finally, note that 
  \begin{align*}
    \n{\phi^*}_r \circ (\id_{Fr} \cdot \eta_x) 
      &= (\eta_y \cdot \phi_r) \circ \gamma^{-1} \\
    (\epsilon_y \cdot \id_{Gr}) \circ \n{\phi^*}_r
      &= \gamma \circ (\phi_r \cdot \epsilon_x)
  \end{align*}
  are immediate consequences of mate correspondence. Thus, \( \eta \) and \(
  \epsilon \) define modifications, and pointwise evaluation confirms that \(
  \phi^* \) is the conjoint of \( \phi \).
\end{proof}

We say that a vertical transformation \( \phi \) \textit{has a strong conjoint
(companion)} if its conjoint (companion) in the appropriate pseudodouble
category is a strong horizontal transfomation; that is, if \( \n{\phi^*} \)
(\( \n{\phi_!} \)) is an invertible natural transformation. The notion of a
vertical transformation \( \phi \) having a strong companion (conjoint) is
present in \cite[A.4]{CS10}; therein, the terminology is
\textit{(co)horizontally strong}.

To provide a class of examples, recall from \cite[A.6]{CS10} that for a
natural transformation \( \phi \colon F \to G \) between pullback-preserving
functors \( F \colon \cat B \to \cat C \) on categories with pullbacks, the
induced vertical transformation \( \hat \phi \colon \hat F \to \hat G \)
between the induced strong functors \( \hat F, \hat G \colon \Span(\cat B) \to
\Span(\cat C) \) has a strong conjoint if and only if it has a strong
companion, if and only if \( \phi \) is a cartesian natural transformation.

We also have the following the result:
\begin{lemma}
  \label{lem:right.conjoint.invert}
  Let \( \phi \colon F \to G \) be a vertical transformation of lax
  functors \( F, G \colon \bicat D \to \bicat E \), let \( H \colon \bicat E
  \to \bicat F \) be another lax functor. We assume \( \bicat E \) is
  conjoint closed, and that \( \phi \) has a strong conjoint.

  \( H\phi \) has a strong conjoint if and only if \( \m H \circ (\id \cdot
  \sigma^H) \colon HFr \cdot (H\phi_x)^* \to H(Fr \cdot \phi^*_x) \) is
  invertible for all \(x\) and all \(r\).
\end{lemma}

\begin{proof}
  We shall verify that
  \begin{equation}
    \label{eq:hphi.2.cell}
    \begin{tikzcd}[column sep=large]
      \cdot \ar[r,"(H\phi_x)_*"]
            \ar[d,equal]
      & \cdot \ar[r,"HFr"] \ar[d,phantom,"\n{(H\phi)^*}_r"] 
      & \cdot \ar[d,equal] \\
      \cdot \ar[r,"HGr" description]
            \ar[d,equal]
      & \cdot \ar[r,"(H\phi_y)^*" description, ""{name=A}]
              \ar[d,equal]
      & \cdot \ar[d,equal] \\
      \cdot \ar[r,"HGr" description]
            \ar[d,equal]
      & \cdot \ar[r,"H(\phi^*_y)" description, ""{name=B}] 
              \ar[d,phantom,"\m H"]
      & \cdot \ar[d,equal] \\
      \cdot \ar[rr,"H(\phi^*_y \cdot Gr)",swap]
      &\,& \cdot
      \ar[from=A,to=B,phantom,"\sigma^H"]
    \end{tikzcd}
    =
    \begin{tikzcd}[column sep=large]
      \cdot \ar[r,"(H\phi_x)_*"{name=A}]
            \ar[d,equal]
      & \cdot \ar[r,"HFr"] \ar[d,equal]
      & \cdot \ar[d,equal] \\
      \cdot \ar[r,"H(\phi^*_x)" description, ""{name=B}]
            \ar[d,equal]
      & \cdot \ar[r,"HFr"] 
      & \cdot \ar[d,equal] \\
      \cdot \ar[rr,"H(Fr \cdot \phi^*_x)" description,""{name=C}] 
            \ar[d,equal]
      & \,
      & \cdot \ar[d,equal] \\
      \cdot \ar[rr,"H(\phi_y \cdot HGr)"{name=D},swap]
      &\,& \cdot
      \ar[from=A,to=B,phantom,"\sigma^H"]
      \ar[from=C,to=D,phantom,"H\n{\phi^*}_r"]
      \ar[from=2-2,to=C,phantom,"\m H"]
    \end{tikzcd}
  \end{equation}
  for all \(r\) and \(x\), from which our result follows as a consequence of
  Lemma \ref{lem:conjoint.eq.normal}. Note that
  \begin{align*}
    \m H \circ (\sigma^H \cdot \id)
         \circ \n{(H\phi)^*}_r
         \circ (\id \cdot \eta)
    &= \m H \circ (\sigma^H \cdot \id)
            \circ (\eta \cdot H\phi_r)
            \circ \gamma^{-1} \\
    &= \m H \circ (H\eta \cdot H\phi_r)
            \circ (\e H \cdot \id)
            \circ \gamma^{-1} \\
    &= H(\eta \cdot \phi_r) \circ \m H
                            \circ (\e H \cdot \id)
                            \circ \gamma^{-1} \\
    &= H(\eta \cdot \phi_r) \circ H\lambda^{-1} \circ \rho \\
    H\n{\phi^*}_r \circ \m H \circ (\id \cdot \sigma^H)
                             \circ (\id \cdot \eta)
    &= H\n{\phi^*}_r \circ \m H \circ (\id \cdot H\eta)
                                \circ (\id \cdot \e H) \\
    &= H\n{\phi^*}_r \circ H(\id \cdot \eta) \circ \m H
                     \circ (\id \cdot \e H) \\
    &= H(\eta \cdot \phi_r) \circ H\gamma^{-1} \circ \m H
                     \circ (\id \cdot \e H) \\
    &= H(\eta \cdot \phi_r) \circ H\lambda^{-1} \circ \rho
  \end{align*}
  so \eqref{eq:hphi.2.cell} holds by mate correspondence.
\end{proof}
This invertibility condition is satisfied, for instance, by strong functors,
and by Barr extensions of monads on \( \Set \); see \cite[{}1.10.2(2)]{HST14}.

  \section{Double categories as pseudo-algebras}
    \label{sect:pstalg}
    We refer to \cite[{}4.4]{GP04} for the notion of equivalence of double
categories. This section is devoted to proving the following result:

\begin{proposition}
  \label{thm:psdbcat.as.ps-t-alg}
  We have an equivalence of double categories \( \PsDbCat
  \eqv \PsfcAlg \), where \(\fc=(\fc,m,e)\) is the free internal category
  2-monad on \( \Grph(\Cat) \), and \( \PsfcAlg \) is the sub-double category
  of \( \LaxfcAlg \) consisting of the pseudo-\(\fc\)-algebras.
\end{proposition}

The proof is laid out as follows:
\begin{enumerate}[label=\textbf{\Roman*}]
  \item
    \label{enum:2.monad}
    We recall the definition of \(\fc\), verifying it is a 2-monad.

  \item
    \label{enum:dbcat2fcalg}
    We provide a construction of a pseudo-\(\fc\)-algebra from a given
    pseudodouble category.

  \item
    \label{enum:functor}
    We provide a construction of (op)lax morphisms of pseudo-\(\fc\)-algebras
    from given (op)lax functors of pseudodouble categories. Moreover, we
    verify this construction defines a functor \( \PsDbCat_\lax \to
    \PsfcAlg_\lax \) (and dually, \( \PsDbCat_\opl \to \PsfcAlg_\opl \)).

  \item
    \label{enum:ff}
    We prove the aforementioned functors are fully faithful. 

  \item
    \label{enum:es}
    We prove the aforementioned functors are essentially surjective.

  \item
    \label{enum:bij}
    Let \( H \colon \bicat A \to \bicat B \) and \( K \colon \bicat C \to
    \bicat D \) be lax functors, and let \( F \colon \bicat A \to \bicat C \)
    and \( G \colon \bicat B \to \bicat D \) be oplax functors, and consider
    the induced pseudo-\(\fc\)-algebra lax and oplax morphisms (as in
    \eqref{eq:fill.2-cells}). Given a 2-cell \( \omega \colon GH \to KF \) of
    internal \( \Cat \)-graphs, we prove that \( \omega \) is a generalized
    vertical transformation if and only if \( \omega \) is a generalized
    2-cell of pseudo \(\fc\)-algebras.
\end{enumerate}

\begin{remark}
  To clarify how the statements \ref{enum:2.monad}--\ref{enum:bij} give our
  desired result, we recall the characterization of equivalence of double
  categories \cite[Theorem 4.5]{GP04}. Namely, this result states that a
  strong functor \( F \colon \bicat D \to \bicat E \) of pseudodouble
  categories is an equivalence if and only if the underlying functor \( F_1
  \colon \bicat D_1 \to \bicat E_1 \) on horizontal 1-cells and 2-cells is an
  equivalence of categories.

  Items \ref{enum:dbcat2fcalg}, \ref{enum:functor} and \ref{enum:bij}
  guarantee that we indeed have a strict functor \( \PsDbCat \to \PsfcAlg \)
  between double categories. The correspondence between generalized vertical
  transformations and generalized 2-cells of pseudo-\( \fc \)-algebras is
  given by \ref{enum:bij}, and is the identity, which is certainly (strictly)
  functorial both horizontally and vertically.

  Indeed, we recall from Proposition \ref{thm:dbcat.psdbcat} that we are
  taking the horizontal 1-cells to be the lax functors. Let 
  \begin{equation*}
    F \colon \bicat A \to \bicat B, \quad  G \colon \bicat C \to \bicat D,
  \end{equation*}
  be lax functors, and write
  \begin{equation*}
    \overline F \colon \overline{\bicat A} \to \overline{\bicat B},
    \quad \overline G \colon \overline{\bicat C} \to \overline{\bicat D},
  \end{equation*}
  for the associated pseudo-\(\fc\)-algebra lax morphisms. Let \( \theta \) be
  a generalized 2-cell of pseudo-\( \fc \)-algebras
  \begin{equation*}
    \begin{tikzcd}
      \overline{\bicat A}
        \ar[r,"\overline F"{name=A}] \ar[d,"P",swap] 
      & \overline{\bicat B} \ar[d,"Q"] \\
      \overline{\bicat C}
        \ar[r,"\overline G"{name=B},swap]
      & \overline{\bicat D} 
      \ar[from=A,to=B,"\theta",phantom]
    \end{tikzcd}
  \end{equation*}
  where \( P, Q \) are the underlying oplax morphisms of pseudo-\( \fc
  \)-algebras. Since \( \PsDbCat_\opl \to \PsfcAlg_\opl \) is fully faithful
  by \ref{enum:ff}, there are unique oplax functors \( H \colon \bicat A \to
  \bicat C,\, K \colon \bicat B \to \bicat D \) such that \(P = \overline H, Q
  = \overline K \), respectively. Now, by \ref{enum:bij}, \( \theta \) is
  identical to a generalized vertical 2-cell
  \begin{equation*}
    \begin{tikzcd}
      \bicat A
        \ar[r,"F"{name=A}] \ar[d,"H",swap] 
      & \bicat B \ar[d,"K"] \\
      \bicat C
        \ar[r,"G"{name=B},swap]
      & \bicat D
      \ar[from=A,to=B,"\theta",phantom]
    \end{tikzcd}
  \end{equation*}
  so the strict functor \( \PsDbCat \to \PsfcAlg \) is full and faithful in
  the terminology of \cite[{}4.4]{GP04}.

  Now, we let \( P \colon \bicat X \to \bicat Y \) be a lax \( \fc \)-morphism
  between lax \( \fc \)-algebras \( \bicat X, \, \bicat Y \). Since the
  functor \( \PsDbCat_\lax \to \PsfcAlg_\lax \) is essentially surjective by
  \ref{enum:es}, there exist pseudodouble categories \( \bicat A \), \( \bicat
  B \) and invertible pseudomorphisms of pseudo-\( \fc \)-algebras \(
  \overline{\bicat A} \iso \bicat X \), \( \bicat Y \iso \overline{\bicat B}
  \), and by taking the composite we recover a lax functor \( F \colon \bicat
  A \to \bicat B \) via \ref{enum:ff}.  There is an invertible generalized
  2-cell \( \overline{F} \iso P \) induced by the isomorphisms \(
  \overline{\bicat A} \iso \bicat X \) and \( \overline{\bicat B} \iso \bicat
  Y \), thereby confirming that the strict double functor \( \PsDbCat \to
  \PsfcAlg \) is essentially surjective in the terminology of
  \cite[{}4.4]{GP04}.
\end{remark}

\subsection*{Step {\ref{enum:2.monad}}:}
We begin by recalling that \( \Grph(\Cat) \) is the functor 2-category \(
[\cdot_1 \rightrightarrows \cdot_0, \Cat] \), whose 2-cells \( \theta \colon F
\to G \) are pairs of natural transformations \( \theta_i \colon F_i \to G_i
\) for \( i = 0,\, 1 \) such that \( d_j \cdot \theta_1 = \theta_0 \cdot d_j
\) for \( j=0,\,1 \).

Since \( \Cat \) is a lextensive category, we can define the free
internal category monad \( \fc = (\fc,m,e) \) on the underlying category of \(
\Grph(\Cat) \); we can obtain an explicit description of \( \fc(\bicat G) \)
for each internal \( \Cat \)-graph
\begin{equation*}
  \bicat G =
  \begin{tikzcd}
    & \bicat G_1 \ar[ld,"d_1",swap] \ar[rd,"d_0"] \\
    \bicat G_0 && \bicat G_0
  \end{tikzcd}
\end{equation*}
by considering the free monoid monad on the monoidal category \( \cat V =
\Span(\Cat)(\bicat G_0,\bicat G_0) \); indeed, we apply \cite[Theorem
23.4]{Kel80}, noting that, by extensivity, the induced tensor product on \(
\cat V \) preserves coproducts on both variables. See also the proof of Lemma
\ref{lem:free.monoid.cartesian} for the expression.

To extend \(\fc\) to a 2-monad, let \( \theta \colon F \to G \) be a 2-cell in
\( \Grph(\Cat) \). We define \( \fc\theta \) by letting \( (\fc\theta)_0 =
\theta_0 \) and \( (\fc\theta)_1 \colon (\fc F)_1 \to (\fc G)_1 \) is given at
a composable string of horizontal arrows \( r_1, \ldots, r_n \) by
\begin{equation*}
  (\fc\theta)_{r_1,\ldots,r_n} = (\theta_{r_1}, \ldots, \theta_{r_n}),
\end{equation*}
which is a horizontally composable string of 2-cells, and is given at the
empty string by \( (\fc\theta)_{()} = \theta_0 \).

We must check \( (\fc\theta)_1 \) is natural; indeed, if \( \phi_i \colon r_i
\to s_i \) is a horizontally composable string of 2-cells, then \(
\theta_{s_i} \circ F\phi_i = G\phi_i \circ \theta_{r_i} \) for all \(i\), so
\begin{align*}
  (\fc\theta)_{s_1, \ldots, s_n} \circ (\fc F)(\phi_1, \ldots, \phi_n)
    &= (\theta_{s_1}, \ldots, \theta_{s_n})
        \circ (F\phi_1, \ldots, F\phi_n) \\
    &= (G\phi_1, \ldots, G\phi_n)
        \circ (\theta_{r_1}, \ldots, \theta_{r_n}) \\
    &= (\fc G)(\phi_1,\ldots,\phi_n) \circ (\fc\theta)_{r_1,\ldots,r_n},
\end{align*} 
and since \( \theta_0 \) is already natural, there is nothing to check for
\(n=0\).

Finally, note that \( d_1 \cdot (\fc\theta)_{r_1,\ldots,r_n} =
d_1(\theta_{r_1}, \ldots, \theta_{r_n}) = d_1(\theta_{r_1}) = \theta_{d_1r_1}
= \theta_{d_1(r_1, \ldots, r_n)} \), and likewise \( d_0 \cdot (\fc\theta)_1 =
\theta_0 \cdot d_0 \). 

To verify \(\fc\) is a 2-functor, we must prove we have strict preservation of
vertical and horizontal composition of 2-cells. Therefore, let \( \omega
\colon G \to H \) and \( \xi \colon K \to L \) be 2-cells, with \( K,\, L \)
composable with \(F,\, G\) respectively. We have \( \fc(\omega \circ \theta)_0
= \fc(\omega)_0 \circ \fc(\theta)_0 \) and \( \fc(\xi \cdot \theta)_0 =
\fc(\xi)_0 \cdot \fc(\theta)_0 \). Moreover, given a composable string of
horizontal arrows \( r_1, \ldots, r_n \), we have
\begin{align*}
  \fc(\omega \circ \theta)_{r_1, \ldots, r_n} 
    &= ((\omega \circ \theta)_{r_1}, \ldots, (\omega \circ \theta)_{r_n}) \\
    &= (\omega_{r_1}, \ldots, \omega_{r_n}) 
          \circ (\theta_{r_1}, \ldots, \theta_{r_n}) \\
    &= \fc(\omega)_{r_1, \ldots, r_n} \circ \fc(\theta)_{r_1, \ldots, r_n}, \\
  \fc(\xi \cdot \theta)_{r_1, \ldots, r_n}
    &= ((\xi \cdot \theta)_{r_1}, \ldots, (\xi \cdot \theta)_{r_n}) \\
    &= (\xi_{r_1}, \ldots, \xi_{r_n})
        \cdot (\theta_{r_1}, \ldots, \theta_{r_n}) \\
    &= \fc(\xi)_{r_1, \ldots, r_n} \cdot \fc(\theta)_{r_1, \ldots, r_n},
\end{align*}
as desired. Nothing needs to be done to verify that \( m, e \) are 2-natural
transformations.

\subsection*{Step {\ref{enum:dbcat2fcalg}}:}

A pseudodouble category consists of a graph of categories \( \bicat D =
(\bicat D_1 \rightrightarrows \bicat D_0) \), with vertical domain and
codomain functors.  The algebra structure \( a \colon \fc\bicat D \to \bicat D
\) is the identity on 0-cells and vertical 1-cells. We define \( a() = 1 \)
(at 0-cells), and if \(a\) is defined for \( \bicat D^{(n)} \), we define 
\begin{equation*}
  a(r_1, \ldots, r_{n+1}) = r_{n+1} \cdot a(r_1, \ldots, r_n).
\end{equation*}
We let \( \eta \colon \id \to a \cdot e \) be the identity on \( \id_{\bicat
D_0} \), and \( \eta_r \colon r \to a(r) \) is given by 
\begin{equation*}
  \rho^{-1} \colon r \to r \cdot 1 = r \cdot a() = a(r).
\end{equation*}

We define \( \mu \colon a \cdot \fc a \to a \cdot m \) to be the identity on
\( \id_{\bicat D_0} \), and on \( \fc\fc\bicat D_1 \to \fc\bicat D_1 \) by
double induction:
\begin{align*}
  \mu_{()} &= \id, \\
  \mu_{k_1,\ldots,k_n,0} 
    &= \mu_{k_1, \ldots, k_n} \circ \lambda, \\
  \mu_{k_1,\ldots,k_{n+1}+1} 
    &= (\id \cdot \mu_{k_1, \ldots, k_{n+1}}) \circ \alpha
\end{align*}
where 
\begin{equation*}
  \mu_{k_1, \ldots, k_n} \colon a(a(r_{1,1}, \ldots, r_{1,k_1}),
                                  \ldots, a(r_{n,1}, \ldots, r_{n,k_n}))
                         \to a(r_{1,1}, \ldots, r_{n,k_n}).
\end{equation*}

To prove that \( (\bicat D,a,\eta,\mu) \) is a pseudo-\(\fc\)-algebra, we must
verify that
\begin{align*}
  \mu_m \circ \eta_{a(r_1, \ldots, r_m)} &= \id, \\
  \mu_{1,\ldots,1} \circ a(\eta_{r_1}, \ldots, \eta_{r_m}) &= \id, \\
  \mu_{j_{1,1}, \ldots, j_{n,k_n}} \circ \mu_{k_1, \ldots, k_n}
    &= \mu_{\hat j_{1,k_1}, \ldots, \hat j_{n,k_n}}
        \circ a(\sigma_1, \ldots, \sigma_n)
\end{align*}
where we use the following abbreviations:
\begin{align*}
  \hat j_{i,k_i} &= \sum_{p=1}^{k_i} j_{i,p} \\
  \sigma_p &= \mu_{j_{p,1}, \ldots, j_{p,k_p}}
\end{align*}
For the first and second, we argue by induction. When \(m=0\), the first
becomes \( \lambda_1 \circ \rho^{-1}_1 = \id \), and the second trivializes.
If we assume the equations hold for some \(m\), then
\begin{equation*}
  \mu_{m+1} \circ \rho^{-1}
    = (\id \cdot \mu_m) \circ \alpha \circ \rho^{-1}
    = (\id \cdot \mu_m) \circ (\id \cdot \rho^{-1}) = \id,
\end{equation*}
and
\begin{align*}
  \mu_{1,\ldots,1,1} \circ a(\rho^{-1}, \ldots, \rho^{-1}, \rho^{-1})
    &= (\id \cdot \mu_{1,\ldots,1}) \circ (\id \cdot \lambda)
      \circ \alpha \circ (\rho^{-1} \cdot \id) 
      \circ (\id \cdot a(\rho^{-1}, \ldots, \rho^{-1})) \\
    &= (\id \cdot \mu_{1,\ldots,1}) 
          \circ (\id \cdot a(\rho^{-1}, \ldots, \rho^{-1})) \\
    &= \id
\end{align*}
For the third, we use triple induction. If \(n=0\), it trivializes, so we
assume it holds for some \(n\). If \(k_{n+1}=0 \), we have
\begin{align*}
  \mu_{j_{1,1},\ldots,j_{n,k_n}}
    \circ \mu_{k_1, \ldots, k_n, k_{n+1}}
    &= \mu_{j_{1,1}, \ldots, j_{n,k_n}}
        \circ \mu_{k_1, \ldots, k_n} \circ \lambda \\
    &= \mu_{\hat j_{1,k_1}, \ldots, \hat j_{n,k_n}}
        \circ a(\sigma_1, \ldots, \sigma_n) \circ \lambda \\
    &= \mu_{\hat j_{1,k_1}, \ldots, \hat j_{n,k_n}, 0}
        \circ a(\sigma_1, \ldots, \sigma_n, \id)  \\
    &= \mu_{\hat j_{1,k_1}, \ldots, \hat j_{n,k_n}, \hat j_{n+1,k_{n+1}}}
        \circ a(\sigma_1, \ldots, \sigma_n, \sigma_{n+1}).
\end{align*}
Now, we assume it holds for some \( k_{n+1} \). If \( j_{n+1,k_{n+1}+1} = 0\),
we have
\begin{align*}
  \mu_{j_{1,1}, \ldots, j_{n+1,k_{n+1}}, 0}
    \circ \mu_{k_1, \ldots, k_n, k_{n+1}+1}
    &= \mu_{j_{1,1}, \ldots, j_{n+1,k_{n+1}}} \circ \lambda 
         \circ (\id \cdot \mu_{k_1, \ldots, k_{n+1}}) \circ \alpha \\
    &= \mu_{j_{1,1}, \ldots, j_{n+1,k_{n+1}}}
         \circ \mu_{k_1, \ldots, k_{n+1}} \circ \lambda \circ \alpha \\
    &= \mu_{\hat j_{1,k_1}, \ldots, \hat j_{n+1,k_{n+1}}}
         \circ a(\sigma_1, \ldots, \sigma_{n+1}) \circ (\lambda \cdot \id) \\
    &= \mu_{\hat j_{1,k_1}, \ldots, \hat j_{n+1,k_{n+1}+1}}
         \circ a(\sigma_1, \ldots, (\sigma_{n+1} \circ \lambda)),
\end{align*}
and finally, if we assume it holds for some \( j_{n+1,k_{n+1}+1} \), then we
have
\begin{align*}
  \mu_{j_{1,1}, \ldots, j_{n+1,k_{n+1}+1}+1}
    \circ \mu_{k_1, \ldots, k_{n+1}+1}
    &= (\id \cdot \mu_{j_{1,1}, \ldots, j_{n+1,k_{n+1}+1}}) \circ \alpha
        \circ (\id \cdot \mu_{k_1, \ldots, k_{n+1}}) \circ \alpha \\
    &= (\id \cdot \mu_{j_{1,1}, \ldots, j_{n+1,k_{n+1}+1}}) 
        \circ (\id \cdot (\id \cdot \mu_{k_1, \ldots, k_{n+1}})) 
        \circ \alpha \circ \alpha \\
    &= (\id \cdot \mu_{j_{1,1}, \ldots, j_{n+1,k_{n+1}+1}}) 
        \circ (\id \cdot (\id \cdot \mu_{k_1, \ldots, k_{n+1}})) 
        \circ (\id \cdot \alpha) \circ \alpha \circ (\alpha \cdot \id) \\
    &= (\id \cdot \mu_{j_{1,1}, \ldots, j_{n+1,k_{n+1}+1}}) 
        \circ (\id \cdot \mu_{k_1, \ldots, k_{n+1}+1}) 
        \circ \alpha \circ (\alpha \cdot \id) \\
    &= (\id \cdot \mu_{\hat j_{1,k_1}, \ldots, \hat j_{n+1,k_{n+1}}})
        \circ (\id \cdot a(\sigma_1, \ldots, \sigma_{n+1}))
        \circ \alpha \circ (\alpha \cdot \id) \\
    &= (\id \cdot \mu_{\hat j_{1,k_1}, \ldots, \hat j_{n+1,k_{n+1}}})
        \circ \alpha 
        \circ ((\id \cdot \sigma_{n+1}) \cdot a(\sigma_1, \ldots, \sigma_n))
        \circ (\alpha \cdot \id) \\
    &= \mu_{\hat j_{1,k_1}, \ldots, \hat j_{n+1,k_{n+1}}+1}
        \circ a(\sigma_1, \ldots, \sigma_n, 
                  (\id \cdot \sigma_{n+1}) \circ \alpha),
\end{align*}
so the result holds by induction. 

\begin{remark}
  It should be noted that the proof (so far) remains unchanged if we consider
  left-biased double categories, in which case \( (\bicat D,a,\eta,\mu) \) is
  a lax \(\fc\)-algebra instead. Respectively, if \( \bicat D \) is
  right-biased double category, \( (\bicat D,a,\eta,\mu) \) is an oplax
  \(\fc\)-algebra.
\end{remark}

\subsection*{Step {\ref{enum:functor}}:}
If we have a lax functor \( F \colon \bicat D \to \bicat E \) between ordinary
double categories, we define a pseudo \(\fc\)-algebra lax morphism \( (F,
\gamma^F) \colon (\bicat D,a,\eta,\mu) \to (\bicat E,b,\eta,\mu) \), taking
\(F\) to be the same underlying graph morphism, and we define \( \gamma^F
\colon b \circ \fc F \to F \circ a \) inductively as follows:
\begin{align*}
  \gamma^F_x &= \e F_x, \\
  \gamma^F_{r_1,\ldots,r_{n+1}} 
             &= \m F \circ (\id \cdot \gamma^F_{r_1, \ldots, r_n}).
\end{align*}
to confirm \( (F,\gamma^F) \) is indeed a lax morphism, we will prove that
\begin{equation*}
  F\eta = \gamma_r \circ \eta
\end{equation*}
and
\begin{equation*}
  \gamma_{r_{1,1},\ldots,r_{n,k_n}} \circ \mu_{k_1,\ldots,k_n}
    = F\mu_{k_1, \ldots, k_n} \circ \gamma_{s_1, \ldots, s_n}
                              \circ b(\sigma_1, \ldots, \sigma_n),
\end{equation*}
where
\begin{align*}
  s_i &= a(r_{i,1}, \ldots, r_{i,k_i}), \\
  \sigma_i &= \gamma_{r_{i,1}, \ldots, r_{i,k_i}}.
\end{align*}
The first is just a restatement of the coherence diagram for the right unitor.
For the second, when \(n=0\), the equation is trivial; \( \e F = \e F \). Now,
we assume the equation holds for some \(n\). If \( k_{n+1} = 0 \), then
\begin{align*}
  \gamma_{r_{1,1},\ldots,r_{n,k_n}}
    \circ \mu_{k_1, \ldots, k_n,0}
    &= \gamma_{r_{1,1},\ldots,r_{n,k_n}}
        \circ \mu_{k_1, \ldots, k_n} \circ \lambda \\
    &= F\mu_{k_1, \ldots, k_n} \circ \gamma_{s_1, \ldots, s_n}
        \circ b(\sigma_1, \ldots, \sigma_n) \circ \lambda \\
    &= \lambda \circ (\id \cdot F\mu_{k_1, \ldots,k_n})
        \circ (\id \cdot \gamma_{s_1, \ldots, s_n})
        \circ (\id \cdot b(\sigma_1,\ldots,\sigma_n)) \\
    &= F\lambda \circ \m F \circ (\e F \cdot \id)
        \circ (\id \cdot F\mu_{k_1, \ldots,k_n})
        \circ (\id \cdot \gamma_{s_1, \ldots, s_n})
        \circ (\id \cdot b(\sigma_1,\ldots,\sigma_n)) \\
    &= F\lambda \circ F(\id \cdot \mu_{k_1,\ldots,k_n})
        \circ \m F \circ (\id \cdot \gamma_{s_1,\ldots,s_n})
        \circ (\e F \cdot b(\sigma_1,\ldots,\sigma_n)) \\
    &= F\mu_{k_1, \ldots, k_n, 0}
        \circ \gamma_{s_1,\ldots,s_n,s_{n+1}}
        \circ b(\sigma_1, \ldots, \sigma_n, \sigma_{n+1}),
\end{align*}
so, we assume the identity holds for some \( k_{n+1} \). We have
\begin{align*}
  &\gamma_{r_{1,1}, \ldots, r_{n+1,k_{n+1}}, r_{n+1,k_{n+1}+1}}
    \circ \mu_{k_1, \ldots, k_n, k_{n+1}+1} \\
    &\quad= \m F \circ (\id \cdot \gamma_{r_{1,1}, \ldots, r_{n+1,k_{n+1}}})
        \circ (\id \cdot \mu_{k_1, \ldots, k_{n+1}}) \circ \alpha \\
    &\quad= \m F \circ (\id \cdot F\mu_{k_1, \ldots, k_{n+1}})
            \circ (\id \cdot \gamma_{s_1, \ldots, s_n, s_{n+1}})
            \circ (\id \cdot b(\sigma_1, \ldots, \sigma_n, \sigma_{n+1}))
            \circ \alpha \\
    &\quad= F(\id \cdot \mu_{k_1, \ldots, k_{n+1}})
            \circ \m F \circ (\id \cdot \m F) 
            \circ (\id \cdot (\id \cdot \gamma_{s_1, \ldots, s_n}))
            \circ \alpha
            \circ ((\id \cdot \sigma_{n+1}) 
                      \cdot b(\sigma_1, \ldots, \sigma_n)) \\
    &\quad= F(\id \cdot \mu_{k_1, \ldots, k_{n+1}})
        \circ F\alpha \circ \m F \circ (\m F \cdot \id)
        \circ (\id \cdot \gamma_{s_1, \ldots, s_n})
        \circ ((\id \cdot \sigma_{n+1}) 
                  \cdot b(\sigma_1, \ldots, \sigma_n)) \\
    &\quad= F\mu_{k_1, \ldots, k_{n+1}+1}
        \circ \gamma_{s_1, \ldots, s_{n+1}}
        \circ b(\sigma_1, \ldots, \sigma_n,
                          \m F \circ (\id \cdot \sigma_{n+1})),
\end{align*}
so, the result follows by induction.

This assignment preserves identities (trivially) as well as composition; that
is, this defines a functor \( \PsDbCat_\lax \to \PsfcAlg_\lax \). To see this,
let \( G \colon \bicat E \to \bicat C \) be another lax functor. We have
\begin{align*}
  (G\gamma^F \circ \gamma^G)_{()} &= G\e F \circ e^G, \\
  (G\gamma^F \circ \gamma^G)_{r_1, \ldots, r_{n+1}}
    &= \m{GF} \circ (\id \cdot G\gamma^F_{r_1, \ldots, r_n})
               \circ (\id \cdot \gamma^G_{Fr_1, \ldots, Fr_n}) \\
    &= G\m{F} \circ \m G
              \circ (\id \cdot G\gamma^F_{r_1, \ldots, r_n})
              \circ (\id \cdot \gamma^G_{Fr_1, \ldots, Fr_n}) \\
    &= G\m{F} \circ G(\id \cdot \gamma^F_{r_1, \ldots, r_n})
              \circ \m G
              \circ (\id \cdot \gamma^G_{Fr_1, \ldots, Fr_n}) \\
    &= G(\gamma^F_{r_1,\ldots,r_{n+1}}) 
        \circ \gamma^G_{Fr_1, \ldots, Fr_{n+1}}.
\end{align*}

Dually, we obtain a functor \( \PsDbCat_\opl \to \PsfcAlg_\opl \).

\subsection*{Step {\ref{enum:ff}}:}
We claim the functor \( \PsDbCat_\lax \to \PsfcAlg_\lax \) is fully faithful;
if \( (F,\gamma^F) \colon (\bicat D,a,\eta,\mu) \to (\bicat E,b,\eta,\mu) \)
is a lax morphism between (the image of) double categories, we define
\begin{align*}
  \e F_x &= \gamma^F_x \\
  \m F_{r,s} &= F(\id \cdot \rho) \circ \gamma^F_{r,s} 
                                  \circ (\id \cdot \rho^{-1})
\end{align*}
We must confirm these satisfy the coherence conditions. First, we observe that
\begin{align*}
  \m F_{1,s} \circ (\id \cdot \e F)
    &= F(\id \cdot \rho) \circ \gamma^F_{1,s} 
                         \circ (\id \cdot \rho^{-1})
                         \circ (\id \cdot \e F) \\
    &= F(\id \cdot \lambda) \circ F(\rho \cdot \id)
                            \circ \gamma^F_{1,a(s)}
                            \circ (F\rho^{-1} \cdot \id)
                            \circ (\id \cdot (\e F \cdot \id))
                            \circ (\id \cdot \lambda^{-1}) \\
    &= F\mu_{0,1} \circ \gamma^F_{1,a(s)}
                  \circ b(\gamma^F_{()}, \gamma^F_s)
                  \circ \mu^{-1}_{0,1} = \gamma_s^F, \\
  \m F_{r,1} \circ (\e F \cdot \id)
    &= F(\id \cdot \rho) \circ \gamma^F_{r,1}
                         \circ (\id \cdot \rho^{-1})
                         \circ (\e F \cdot \id) \\
    &= F(\id \cdot \rho) \circ F(\id \cdot (\rho \cdot \id))
                         \circ \gamma^F_{a(r),1}
                         \circ (\id \cdot (F\rho^{-1} \cdot \id))
                         \circ (\e F \cdot \id)
                         \circ (\id \cdot \rho^{-1}) \\
    &= F(\id \cdot \rho) \circ F(\id \cdot \rho)
                         \circ \gamma^F_{a(r),1}
                         \circ b(\gamma^F_r,\gamma^F_{()})
                         \circ (\id \cdot \rho^{-1})
                         \circ (\id \cdot \rho^{-1}) \\
    &= F(\id \cdot \rho) \circ F(\id \cdot \rho)
                         \circ F\mu_{1,0}^{-1}
                         \circ \gamma^F_r
                         \circ \mu_{1,0}
                         \circ (\id \cdot \rho^{-1})
                         \circ (\id \cdot \rho^{-1}) \\
    &= F\lambda^{-1} \circ F\rho
                     \circ \gamma^F_r
                     \circ \rho^{-1} \circ \lambda \\
    &= F\lambda^{-1} \circ \lambda
\end{align*}
which gives the unit comparsion coherences for \(F\), and after calculating
\begin{align*}
  \mu_{1,2} &= \alpha \circ ((\id \cdot \rho) \cdot (\rho \cdot \id)) \\
  \mu_{2,1} &= (\id \cdot \alpha) 
                  \circ (\id \cdot ((\id \cdot \rho) \cdot \id))
                  \circ (\rho \cdot \id)
\end{align*}
we verify that
\begin{align*}
  F\alpha &\circ \m F_{r,t \cdot s} \circ (\m F_{s,t} \cdot \id) \\
   &= F\alpha \circ F(\id \cdot \rho)
              \circ \gamma^F_{r,t \cdot s}
              \circ (\id \cdot \rho^{-1})
              \circ (F(\id \cdot \rho) \cdot \id)
              \circ (\gamma^F_{s,t} \cdot \id)
              \circ ((\id \cdot \rho^{-1}) \cdot \id) \\
   &= F\alpha \circ F(\id \cdot \rho)
              \circ F((\id \cdot \rho) \cdot (\rho \cdot \id))
              \circ \gamma_{a(r),a(s,t)}^F
              \circ (\gamma^F_{s,t} \cdot (F\rho^{-1} \cdot \id))
              \circ ((\id \cdot \rho^{-1}) \cdot \rho^{-1}) \\
   &= F(\id \cdot (\id \cdot \rho)) \circ F\alpha
              \circ F((\id \cdot \rho) \cdot (\rho \cdot \id))
              \circ \gamma^F_{a(r),a(s,t)}
              \circ b(\gamma^F_r,\gamma^F_{s,t})
              \circ ((\id \cdot \rho^{-1}) \cdot (\rho^{-1} \cdot \id))
              \circ (\id \cdot \rho^{-1}) \\
   &= F(\id \cdot (\id \cdot \rho))
              \circ \gamma^F_{r,s,t} \circ \alpha \circ (\id \cdot \rho^{-1})
\end{align*}
\begin{align*}
  \m F_{s \cdot r, t} &\circ (\id \cdot \m F_{r,s}) \circ \alpha \\
    &= F(\id \cdot \rho) \circ \gamma^F_{s \cdot r, t}
                         \circ (\id \cdot \rho^{-1})
                         \circ (\id \cdot F(\id \cdot \rho))
                         \circ (\id \cdot \gamma^F_{r,s})
                         \circ (\id \cdot (\id \cdot \rho^{-1}))
                         \circ \alpha \\
    &= F(\id \cdot \rho) \circ F(\rho \cdot ((\id \cdot \rho) \cdot \id))
                         \circ \gamma^F_{a(r,s),a(t)}
                         \circ (F\rho^{-1} \cdot (F(\id \cdot \rho^{-1}) 
                                                        \cdot \id))
                         \circ (\id \cdot \rho^{-1}) 
                         \circ (\id \cdot F(\id \cdot \rho)) \\
    &\qquad \quad        \circ (\id \cdot \gamma^F_{r,s})
                         \circ (\id \cdot (\id \cdot \rho^{-1}))
                         \circ \alpha \\
    &= F(\id \cdot (\id \cdot \rho))
                         \circ F(\rho \cdot \rho) 
                         \circ \gamma_{a(r,s),a(t)}^F
                         \circ b(\gamma^F_{r,s},\gamma^F_t)
                         \circ (\rho^{-1} \cdot \rho^{-1})
                         \circ (\id \cdot (\id \cdot \rho^{-1}))
                         \circ \alpha \\
    &= F(\id \cdot (\id \cdot \rho))
                         \circ \gamma^F_{r,s,t}
                         \circ \alpha \circ (\id \cdot \rho^{-1})
\end{align*}
which confirms coherence for the associator comparison. We further verify
that, by induction, \( \mu_{n,1} \circ (\rho^{-1} \cdot \rho^{-1}) = \id \)
(pattern matching), so that
\begin{align*}
  \m F_{a(r_1, \ldots, r_n),r_{n+1}} 
    \circ (\id \cdot \gamma^F_{r_1, \ldots, r_n})
    &= F(\id \cdot \rho) \circ \gamma^F_{a(r_1, \ldots, r_n),r_{n+1}}
                         \circ (\id \cdot \rho^{-1})
                         \circ (\id \cdot \gamma^F_{r_1, \ldots, r_n}) \\
    &= F(\rho \cdot \rho) 
        \circ \gamma^F_{a(r_1, \ldots, r_n),a(r_{n+1})}
        \circ b(\gamma^F_{r_1, \ldots, r_n},\gamma^F_{r_{n+1}})
        \circ (\rho^{-1} \cdot \rho^{-1}) \\
    &= \gamma^F_{r_1, \ldots, r_{n+1}},
\end{align*}
confirming that the functor \( \PsDbCat_\lax \to \PsfcAlg_\lax \) is fully
faithful.

\subsection*{Step {\ref{enum:es}}:}
We claim the functor \( \PsDbCat_\lax \to \PsfcAlg_\lax \) is essentially
surjective; let \( (\bicat D, a, \eta, \mu) \) be a pseudo-\(\fc\)-algebra. We
define
\begin{align*}
  1 &= a(), \\
  s \cdot r &= a(r,s), \\
  \lambda_r &= \eta^{-1}_r \circ \mu_{r,-} \circ a(\eta_r,\id), \\
  \rho_r &= \eta^{-1}_r \circ \mu_{-,r} \circ a(\id,\eta_r), \\
  \alpha_{r,s,t} &= a(\id,\eta^{-1}_t) \circ \mu_{rs,t}^{-1}
                                       \circ \mu_{r,st}
                                       \circ a(\eta_r,\id).
\end{align*}
These endow \( \bicat D \) with the structure of a double category; to see
this, we must verify the coherence conditions hold. First, we have
\begin{align*}
  (\id \cdot \lambda_r) \circ \alpha_{r,1,s} 
    &= a(\eta^{-1}_r,\id) \circ a(\mu_{r,-},\id)
                          \circ a(a(\eta_r,\id),\id)
                          \circ a(\id,\eta^{-1}_s)
                          \circ \mu_{r1,s}^{-1}
                          \circ \mu_{r,1s}
                          \circ a(\eta_r,\id) \\
    &= a(\eta^{-1}_r,\eta^{-1}_s) \circ a(\mu_{r,-},\id)
                                  \circ a(a(\eta_r,\id),\id)
                                  \circ \mu_{r1,s}^{-1}
                                  \circ \mu_{r,1s}
                                  \circ a(\eta_r,\id) \\
    &= \mu_{r,s} \circ a(\mu_{r,-},\id)
                \circ \mu_{a(r)1,s}^{-1}
                \circ a(\eta_r,\id,\id)
                \circ \mu_{r,1s}
                \circ a(\eta_r,\id) \\
    &= \mu_{r,s} \circ a(\mu_{r,-},\mu_s)
                 \circ \mu_{a(r)1,a(s)}^{-1}
                 \circ a(\eta_r,\id,\eta_s)
                 \circ \mu_{r,1s}
                 \circ a(\eta_r,\id) \\
    &= \mu_{r,-,s} \circ \mu_{a(r),1a(s)}
                   \circ a(a(\eta_r),a(\id, \eta_s))
                   \circ a(\eta_r,\id) \\
    &= \mu_{r,s} \circ a(\mu_r,\mu_{-,s})
                 \circ a(a(\eta_r),a(\id,\eta_s))
                 \circ a(\eta_r,\id) \\
    &= a(\eta_r^{-1},\eta_s^{-1}) \circ a(\id,\mu_{-,s})
                                  \circ a(\id,a(\id,\eta_s))
                                  \circ a(\eta_r,\id) \\
    &= a(\id, \eta_s^{-1}) \circ a(\id, \mu_{-,s})
                           \circ a(\id,a(\id, \eta_s))
     = \rho_s \cdot \id.
\end{align*}
and for the associator pentagon, we have, on one hand
\begin{align*}
  (\id \cdot \alpha_{r,s,t}) \circ \alpha_{q,a(r,s),t} 
                             \circ (\alpha_{q,r,s} \cdot \id) \\
    \quad= 
      a(a(\id, \eta^{-1}_s),\id) 
          &\circ a(\mu^{-1}_{qr,s}, \id) 
           \circ a(\mu_{q,rs},\id)
           \circ a(a(\eta_q,\id),\id) \\
          &\circ a(\id,\eta^{-1}_t)
           \circ \mu_{qa(r,s),t}^{-1}
           \circ \mu_{q,a(r,s)t}
           \circ a(\eta_q,\id) \\
          &\circ a(\id, a(\id,\eta^{-1}_t))
           \circ a(\id, \mu_{rs,t}^{-1})
           \circ a(\id, \mu_{r,st})
           \circ a(\id, a(\eta_r,\id)) \\
    \quad= 
      a(a(\id,\eta^{-1}_s), \eta^{-1}_t) 
          &\circ a(\mu^{-1}_{qr,s}, \id)
           \circ a(\mu_{q,rs},\id) \\
          &\circ a(a(\eta_q,\id),\id) 
           \circ \mu_{qa(r,s),t}^{-1}
           \circ \mu_{q,a(r,s)t}
           \circ a(\id, a(\id,\eta^{-1}_t)) \\
          &\circ a(\id, \mu_{rs,t}^{-1})
           \circ a(\id, \mu_{r,st})
           \circ a(\eta_q, a(\eta_r,\id)) \\
    \quad= 
      a(a(\id,\eta^{-1}_s), \eta^{-1}_t) 
          &\circ a(\mu^{-1}_{qr,s}, \id)
           \circ a(\mu_{q,rs},\id) \\
          &\circ a(\id,a(\eta^{-1}_t))
           \circ \mu_{a(q)a(r,s),a(t)}^{-1}
           \circ \mu_{a(q),a(r,s)a(t)}
           \circ a(a(\eta_q),\id) \\
          &\circ a(\id, \mu_{rs,t}^{-1})
           \circ a(\id, \mu_{r,st})
           \circ a(\eta_q, a(\eta_r,\id))  \\
    \quad= 
      a(a(\id,\eta^{-1}_s), \eta^{-1}_t) 
          &\circ a(\mu^{-1}_{qr,s}, \id) 
           \circ a(\mu_{q,rs},\mu_t) \\
          &\circ \mu_{a(q)a(r,s),a(t)}^{-1}
           \circ \mu_{a(q),a(r,s)a(t)} \\
          &\circ a(\mu_q^{-1}, \mu_{rs,t}^{-1}) 
           \circ a(\id, \mu_{r,st})
           \circ a(\eta_q, a(\eta_r,\id)) \\
    \quad= 
      a(a(\id,\eta^{-1}_s), \eta^{-1}_t) 
          &\circ a(\mu^{-1}_{qr,s}, \id) 
           \circ \mu_{qrs,t}^{-1} \\
          &\circ \mu_{q,rst}
           \circ a(\id, \mu_{r,st})
           \circ a(\eta_q, a(\eta_r,\id)),
\end{align*}
while on the other, we have
\begin{align*}
  \alpha_{a(q,r),s,t} \circ \alpha_{q,r,a(s,t)} 
    &= a(\id,\eta^{-1}_t) 
          \circ \mu^{-1}_{a(q,r)s,t} 
          \circ \mu_{a(q,r),st}
          \circ a(\eta_{a(q,r)},\id) \\
          &\quad \circ a(\id,\eta_{a(s,t)}^{-1}) 
                 \circ \mu^{-1}_{qr,a(s,t)}
                 \circ \mu_{q,ra(s,t)}
                 \circ a(\eta_q,\id) \\
    &= a(\id,\eta^{-1}_t) 
           \circ \mu^{-1}_{a(q,r)s,t} 
           \circ a(\id,\eta_s^{-1},\eta_r^{-1}) 
           \circ \mu_{a(q,r),a(s)a(t)} \\
           &\quad \circ a(\id,a(\eta_s,\eta_t)) 
                  \circ a(\id,\eta^{-1}_{a(s,t)}) 
                  \circ a(\eta_{a(q,r)},\id) 
                  \circ a(a(\eta_q,\eta_r),\id) \\
           &\quad \circ \mu^{-1}_{a(q)a(r),a(s,t)}
                  \circ a(\eta_q^{-1},\eta_r^{-1},\id)
                  \circ \mu_{q,ra(s,t)}
                  \circ a(\eta_q,\id) \\
    &= a(\id,\eta^{-1}_t) 
          \circ \mu^{-1}_{a(q,r)s,t} 
          \circ a(\id,\eta_s^{-1},\eta_r^{-1}) \\
          &\quad \circ \mu_{a(q,r),a(s)a(t)} 
                 \circ a(\eta_{a(q,r)},a(\eta_s,\eta_t)) \\
          &\quad \circ a(a(\eta_q^{-1},\eta_r^{-1}),\eta^{-1}_{a(s,t)}) 
                 \circ \mu_{a(q)a(r),a(s,t)}\\
          &\quad \circ a(\eta_q,\eta_r,\id)
                 \circ \mu_{q,ra(s,t)}
                 \circ a(\eta_q,\id) \\
    &= a(\id,\eta^{-1}_t) 
          \circ \mu^{-1}_{a(q,r)s,t} 
          \circ a(\id,\eta_s^{-1},\eta_r^{-1})
          \circ \mu_{qr,s,t}^{-1} \\
          &\quad \circ \mu_{q,r,st} 
                 \circ a(\eta_q,\eta_r,\id)
                 \circ \mu_{q,ra(s,t)}
                 \circ a(\eta_q,\id) \\ 
    &= a(\id,\eta^{-1}_t) 
          \circ \mu^{-1}_{a(q,r)s,t} 
          \circ a(\id,\eta_s^{-1},\eta_r^{-1}) \\
          &\quad \circ a(\mu_{q,r},\mu_s,\mu_t)
                 \circ \mu^{-1}_{a(q)a(r),a(s),a(t)} \\
          &\quad \circ \mu_{a(q),a(r),a(s)a(t)}
                 \circ a(\mu_q^{-1},\mu_r^{-1},\mu_{s,t}^{-1}) \\
          &\quad \circ a(\eta_q,\eta_r,\id)
                 \circ \mu_{q,ra(s,t)}
                 \circ a(\eta_q,\id) \\
    &= a(a(\id,\eta^{-1}_s),\eta^{-1}_t) 
          \circ a(\id,\mu_t)
          \circ \mu^{-1}_{a(q,r)a(s),a(t)} \\
          &\quad \circ a(\mu_{q,r},\mu_s,\mu_t)
                 \circ \mu^{-1}_{a(q)a(r),a(s),a(t)} \\
          &\quad \circ \mu_{a(q),a(r),a(s)a(t)}
                 \circ a(\mu_q^{-1},\mu_r^{-1},\mu_{s,t}^{-1}) \\
          &\quad \circ \mu_{a(q),a(r)a(s,t)}
                 \circ a(\mu_q,\id)
                 \circ a(\eta_q,a(\eta_r,\id)),
\end{align*}
so, our goal is to prove that
\begin{equation}
  \label{eq:final.step}
  \begin{aligned}
    a(\mu^{-1}_{qr,s}, \mu_t^{-1}) 
             \circ \mu_{qrs,t}^{-1} 
             \circ \mu_{q,rst}
             \circ a(\mu_q, \mu_{r,st})
    &= \mu^{-1}_{a(q,r)a(s),a(t)} \circ a(\mu_{q,r},\mu_s,\mu_t) \\
    &\quad          \circ \mu^{-1}_{a(q)a(r),a(s),a(t)} 
                    \circ \mu_{a(q),a(r),a(s)a(t)} \\
    &\quad          \circ a(\mu_q^{-1},\mu_r^{-1},\mu_{s,t}^{-1}) 
                    \circ \mu_{a(q),a(r)a(s,t)}.
  \end{aligned}
\end{equation}
And to do so, we observe that the following diagrams
\begin{equation*}
  \begin{tikzcd}[column sep=huge]
    a(a(a(q)),a(a(r),a(s,t))) 
      \ar[d,swap,"\mu_{a(q),a(r)a(s,t)}"]
      \ar[r,"{a(\mu_q,\mu_{r,st})}"]
    & a(a(q),a(r,s,t)) 
      \ar[d,"\mu_{q,rst}"] \\
    a(a(q),a(r),a(s,t))
      \ar[d,"{a(\mu_q^{-1},\mu_r^{-1},\mu_{s,t}^{-1})}",swap]
      \ar[r,"\mu_{q,r,st}" description]
    & a(q,r,s,t)
      \ar[d,"\mu_{q,r,s,t}^{-1}"] \\
    a(a(a(q)),a(a(r)),a(a(s),a(t)))
      \ar[r,"\mu_{a(q),a(r),a(s)a(t)}",swap] 
    & a(a(q),a(r),a(s),a(t))
  \end{tikzcd}
\end{equation*}
\begin{equation*}
  \begin{tikzcd}[column sep=huge]
    a(a(q,r,s),a(t))
      \ar[r,"{a(\mu_{qr,s}^{-1},\mu_t^{-1})}"]
    & a(a(a(q,r),a(s)),a(a(t)) \\
    a(q,r,s,t) 
      \ar[r,"\mu_{qr,s,t}^{-1}" description]
      \ar[u,"\mu_{qrs,t}^{-1}"]
    & a(a(q,r),a(s),a(t))
      \ar[u,"\mu_{a(q,r)a(s),a(t)}^{-1}",swap] \\
    a(a(q),a(r),a(s),a(t)) 
      \ar[u,"\mu_{q,r,s,t}"]
      \ar[r,"\mu_{a(q)a(r),a(s),a(t)}^{-1}",swap]
    & a(a(a(q),a(r)),a(a(s)),a(a(t)))
      \ar[u,"{a(\mu_{q,r},\mu_s,\mu_t)}",swap]
 \end{tikzcd}
\end{equation*}
are pastings of associativity squares for \( \mu \), and are therefore
commutative. Pasting these diagrams along \( \mu_{q,r,s,t} \) will confirm
\eqref{eq:final.step}, and we conclude that \( \bicat D \) has the structure
of a pseudodouble category.

Now, write \( (\bicat D, \ovl a, \ovl \eta, \ovl \mu) \) for the
pseudo-\(\fc\)-algebra induced by the above pseudodouble category. We define
\( \gamma \colon \ovl a \to a \) to be the natural transformation inductively
given by
\begin{align*}
  \gamma_{()} &= \id, \\
  \gamma_{r_1, \ldots, r_n,r_{n+1}} 
              &= \mu_{r_1 \cdots r_n,r_{n+1}} 
                   \circ a(\gamma_{r_1, \ldots, r_n},\eta_{r_{n+1}})
\end{align*}
We claim that \( (\id,\gamma) \colon (\bicat D,a,\eta,\mu) \to (\bicat D,\ovl
a,\ovl \eta, \ovl \mu) \) is an invertible lax morphism of
pseudo-\(\fc\)-algebras. First, note that
\begin{equation*}
  \gamma_r \circ \ovl \eta_r
    = \mu_{-,r} \circ a(\id,\eta_r) 
                \circ a(\id,\eta_r^{-1})
                \circ \mu_{-,r}^{-1}
                \circ \eta_r
    = \eta_r,
\end{equation*}
and we shall prove that
\begin{equation}
  \label{eq:golpe.final}
  \gamma_{r_{1,1}, \ldots, r_{n,k_n}}
    \circ \ovl \mu_{k_1, \ldots, k_n}
  = \mu_{(r_{1,i})_{i=1}^{k_1}, \ldots, (r_{n,i})_{i=1}^{k_n}}
    \circ \gamma_{a(r_{1,i})_{i=1}^{k_1}, \ldots, (r_{n,i})_{i=1}^{k_n}}
    \circ \ovl a(\gamma_{r_{1,1},\ldots,r_{1,k_1}}, \ldots,
                 \gamma_{r_{n,1},\ldots,r_{n,k_n}})
\end{equation}
by double induction. When \(n=0\), the above reduces to
\begin{equation*}
  \id = \mu_{()} \circ a(\id),
\end{equation*}
which holds, so assume the above is true for some \(n\). If \( k_{n+1} = 0 \),
the left-hand side of \eqref{eq:golpe.final} becomes
\begin{align*}
  \gamma_{r_{1,1}, \ldots, r_{n,k_n}} \circ \ovl \mu_{k_1, \ldots, k_n,0} 
    &= \gamma_{r_{1,1}, \ldots, r_{n,k_n}} 
         \circ \ovl \mu_{k_1, \ldots, k_n}
         \circ \lambda \\
    &= \lambda \circ a(\gamma_{r_{1,1},\ldots,r_{n,k_n}},\id)
               \circ a(\ovl \mu_{k_1,\ldots,k_n},\id)
\end{align*}
while the right-hand side of \eqref{eq:golpe.final} becomes
\begin{align*}
  \mu_{(r_{1,i}),\ldots,(r_{n,i}),()}
    &\circ \gamma_{a(r_{1,i}),\ldots,a(r_{n,i})}
    \circ \ovl a(\gamma_{r_{1,i}}, \ldots, \gamma_{r_{n,i}},\id)  \\
    &= \mu_{(r_{1,i}),\ldots,(r_{n,i}),()} 
        \circ \mu_{a(r_{1,i}) \cdots a(r_{n,i}),a()}
        \circ a(\gamma_{a(r_{1,i}), \ldots, a(r_{n,i})},\eta_{a()})
        \circ a(\ovl a(\gamma_{r_{1,i}}, \ldots, \gamma_{r_{n,i}}),\id) \\
    &= \mu_{r_{1,1} \cdots r_{n,i},()}
        \circ a(\mu_{(r_{1,i}),\ldots,(r_{n,i})}, \mu_{(())})
        \circ a(\gamma_{a(r_{1,i}), \ldots, a(r_{n,i})},\eta_{a()})
        \circ a(\ovl a(\gamma_{r_{1,i}}, \ldots, \gamma_{r_{n,i}}),\id) \\
    &= \mu_{r_{1,1} \cdots r_{n,i},()}
        \circ a(\gamma_{r_{1,1},\ldots,r_{n,k_n}},\id)
        \circ a(\ovl \mu_{k_1,\ldots,k_n},\id)
\end{align*}
so the equality holds by verifying that
\begin{equation*}
  \mu_{r_1 \cdots r_n, ()} = \lambda,
\end{equation*}
which is equivalent to proving that the following diagram commutes
\begin{equation*}
  \begin{tikzcd}[column sep=large]
    a(a(a(r_1, \ldots, r_n)),a()) 
      \ar[d,"a(\mu_{r_1 \cdots r_n}{,}\id)",swap]
      \ar[r,"\mu_{a(r_1, \ldots, r_n),()}"]
      & a(a(r_1, \ldots, r_n)) \ar[d,"\mu_{r_1 \cdots r_n}"] \\
    a(a(r_1, \ldots, r_n),a())
      \ar[r,"\mu_{r_1 \cdots r_n,()}",swap]
      & a(r_1, \ldots, r_n)
  \end{tikzcd}
\end{equation*}
which is given by the coherence condition \( \mu \circ \fc\mu = \mu \circ
\mu_\fc \); recall that \( \mu_{()} = \id \).

Hence, for the final induction step, we suppose \eqref{eq:golpe.final} holds
for some \( k_{n+1} \). For \( k_{n+1}+1 \), the left-hand side of
\eqref{eq:golpe.final} is given by
\begin{align*}
  \gamma_{r_{1,1},\ldots,r_{n+1,k_{n+1}+1}} 
    \circ \ovl \mu_{k_1, \ldots, k_{n+1}+1} \\
    \quad= 
      \mu_{r_{1,1} \cdots r_{n+1,k_{n+1}},r_{n+1,k_{n+1}+1}}
       &\circ a(\gamma_{r_{1,1},\ldots,r_{n+1,k_{n+1}}},
                \eta_{r_{n+1,k_{n+1}+1}}) \\
       &\circ a(\ovl \mu_{k_1,\ldots,k_{n+1}},\id)
        \circ \alpha \\
    \quad= 
      \mu_{r_{1,1} \cdots r_{n+1,k_{n+1}},r_{n+1,k_{n+1}+1}}
       &\circ a(\mu_{(r_{1,i}),\ldots,(r_{n+1,i})},\id) \\
       &\circ a(\gamma_{a(r_{1,i}),\ldots,a(r_{n+1,i})}
                \eta_{r_{n+1,k_{n+1}+1}}) \circ \alpha \\
    \quad= 
      \mu_{r_{1,1} \cdots r_{n+1,k_{n+1}},r_{n+1,k_{n+1}+1}}
       &\circ a(\mu_{(r_{1,i}),\ldots,(r_{n+1,i})},\id) \\
       &\circ a(\mu_{a(r_{1,i}\cdots r_{n,i}),a(r_{n+1,i})},\id) \\
       &\circ a(a(\gamma_{a(r_{1,i}),\ldots,a(r_{n,i})},
                  \eta_{a(r_{n+1,i})}),\id) \\
       &\circ a(a(\ovl a(\gamma_{r_{1,i}},\ldots,\gamma_{r_{n,i}}),
                \gamma_{r_{n+1},i}),\id) \\
       &\circ a(\id, \eta_{r_{n+1,k_{n+1}+1}}) \circ \alpha \\
    \quad=
      \mu_{r_{1,1} \cdots r_{n+1,k_{n+1}},r_{n+1,k_{n+1}+1}}
       &\circ a(\mu_{(r_{1,i}\cdots r_{n,k_n}),(r_{n+1,i})},\id) \\
       &\circ a(a(\mu_{(r_{1,i}),\ldots (r_{n,i})},\mu_{(r_{n+1,i})}),\id)
        \circ \alpha\\
       &\circ a(\gamma_{a(r_{1,i}),\ldots,a(r_{n,i})},
                a(\eta_{a(r_{n+1,i})},\id)) \\
       &\circ a(\ovl a(\gamma_{r_{1,i}},\ldots,\gamma_{r_{n,i}}),
                a(\gamma_{r_{n+1},i},\eta_{r_{n+1,k_{n+1}+1}}))\\
    \quad=
      \mu_{r_{1,1} \cdots r_{n+1,k_{n+1}},r_{n+1,k_{n+1}+1}}
       &\circ a(\mu_{(r_{1,i}\cdots r_{n,k_n}),(r_{n+1,i})},\id) 
        \circ \alpha \\
       &\circ a(\gamma_{r_{1,1},\ldots,r_{n,k_n}},\id) \\
       &\circ a(\ovl \mu_{k_1,\ldots,k_n},\id) \\
       &\circ a(\id,a(\gamma_{r_{n+1},i},\eta_{r_{n+1,k_{n+1}+1}})),
\end{align*}
while the right-hand side is given by
\begin{align*}
  \mu_{(r_{1,i}),\ldots,(r_{n,i}),(r_{n+1,i},r_{n+1,k_{n+1}+1})}
    &\circ \gamma_{a(r_{1,i}),\ldots,a(r_{n,i}),a(r_{n+1,i},r_{n+1,k_{n+1}+1})}
    \circ \ovl a(\gamma_{r_{1,i}}, \ldots, \gamma_{r_{n,i}},
                 \gamma_{r_{n+1,i},r_{n+1,k_{n+1}+1}}) \\
  \quad= \mu_{(r_{1,i}),\ldots,(r_{n,i}),(r_{n+1,i},r_{n+1,k_{n+1}+1})}
    &\circ \mu_{a(r_{1,i}) \cdots a(r_{n,i}),
                a(r_{n+1,i},r_{n+1,k_{n+1}+1})} \\
    &\circ a(\gamma_{a(r_{1,i}),\ldots,a(r_{n,i})},
             \eta_{a(r_{n+1,i},r_{n+1,k_{n+1}+1})}) \\
    &\circ a(\ovl a(\gamma_{r_{1,i}}, \ldots, \gamma_{r_{n,i}}),
             \gamma_{r_{n+1,i},r_{n+1,k_{n+1}+1}}) \\
  \quad= \mu_{(r_{1,1},\ldots,r_{n,k_n}),(r_{n+1,i},r_{n+1,k_{n+1}+1})}
    &\circ a(\mu_{(r_{1,i}),\ldots,(r_{n,i})}, 
             \mu_{(r_{n+1,i}r_{n+1,k_{n+1}+1})}) \\
    &\circ a(\gamma_{a(r_{1,i}),\ldots,a(r_{n,i})},
             \eta_{a(r_{n+1,i},r_{n+1,k_{n+1}+1})}) \\
    &\circ a(\ovl a(\gamma_{r_{1,i}}, \ldots, \gamma_{r_{n,i}}),\id) \\
    &\circ a(\id,\mu_{(r_{n+1,i}),r_{n+1,k_{n+1}+1}}) \\
    &\circ a(\id,a(\gamma_{r_{n+1,i}},\eta_{r_{n+1,k_{n+1}+1}})) \\
  \quad= \mu_{(r_{1,1},\ldots,r_{n,k_n}),(r_{n+1,i},r_{n+1,k_{n+1}+1})}
    &\circ a(\gamma_{r_{1,1}, \ldots, r_{n,k_n}},\id) \\
    &\circ a(\ovl \mu_{k_1, \ldots, k_n}, \id) \\
    &\circ a(\id,\mu_{(r_{n+1,i}),r_{n+1,k_{n+1}+1}}) \\
    &\circ a(\id,a(\gamma_{r_{n+1,i}},\eta_{r_{n+1,k_{n+1}+1}})), 
\end{align*}
and therefore, proving \eqref{eq:golpe.final} reduces to verifying that
\begin{align*}
  &\mu_{r_{1,1} \cdots r_{n+1,k_{n+1}},r_{n+1,k_{n+1}+1}}
        \circ a(\mu_{(r_{1,i}\cdots r_{n,k_n}),(r_{n+1,i})},\id) 
        \circ \alpha \\
  &\quad= \mu_{(r_{1,1},\ldots,r_{n,k_n}),(r_{n+1,i},r_{n+1,k_{n+1}+1})}
        \circ a(\id,\mu_{(r_{n+1,i}),r_{n+1,k_{n+1}+1}}).
\end{align*}
Here, we have
\begin{align*}
  \alpha &= a(\id,\mu_{r_{n+1,k_{n+1}+1}}) 
            \circ \mu^{-1}_{a(r_{1,1}, \ldots, r_{n,k_n})a(r_{n+1,i}), 
                            a(r_{n+1,k_{n+1}+1})} \\
            &\quad\circ \mu_{a(r_{1,1},\ldots,r_{n,k_n}),
                        a(r_{n+1,i})a(r_{n+1,k_{n+1}+1})} \\
            &\quad\circ a(\mu^{-1}_{r_{1,1}\cdots r_{n,k_n}},\id),
\end{align*}
so we just need to verify that 
\begin{align*}
  &\mu_{r_{1,1} \cdots r_{n+1,k_{n+1}},r_{n+1,k_{n+1}+1}}
        \circ a(\mu_{(r_{1,i}\cdots r_{n,k_n}),(r_{n+1,i})},
                \mu_{r_{n+1,k_{n+1}+1}})
        \circ \mu^{-1}_{a(r_{1,1}, \ldots, r_{n,k_n})a(r_{n+1,i}), 
                            a(r_{n+1,k_{n+1}+1})} \\
  &\quad= \mu_{(r_{1,1},\ldots,r_{n,k_n}),(r_{n+1,i},r_{n+1,k_{n+1}+1})}
        \circ a(\mu_{r_{1,1}\cdots r_{n,k_n}}
                \mu_{(r_{n+1,i}),r_{n+1,k_{n+1}+1}})
        \circ \mu^{-1}_{a(r_{1,1},\ldots,r_{n,k_n}),
                        a(r_{n+1,i})a(r_{n+1,k_{n+1}+1})},
\end{align*}
which holds, since both sides of the above expression are equal to
\begin{equation*}
  \mu_{r_{1,1}\cdots r_{n,k_n}, (r_{n+1,i}), r_{n+1,k_{n+1}+1}},
\end{equation*}
confirming that \( (\id,\gamma) \) is an invertible pseudo-morphism of
pseudo-\(\fc\)-algebras. 

\subsection*{Step {\ref{enum:bij}}:}
We consider double categories, lax (horizontal) and oplax (vertical) functors
as in the following diagram, and the respective diagram in the double category
\( \PsfcAlg \).
\begin{equation}
  \label{eq:fill.2-cells}
  \begin{tikzcd}
    \bicat A \ar[d,"F",swap] \ar[r,"H"]
      & \bicat B \ar[d,"G"] \\
    \bicat C \ar[r,"K",swap] & \bicat D
  \end{tikzcd}
  \qquad
  \begin{tikzcd}
    (\bicat A,a,\eta,\mu) \ar[d,"{(F,\delta^F)}",swap]
                          \ar[r,"{(H,\gamma^H)}"]
      & (\bicat B,b,\eta,\mu) \ar[d,"{(G,\delta^G)}"] \\
    (\bicat C,c,\eta,\mu) \ar[r,"{(K,\gamma^K)}",swap]
      & (\bicat D,d,\eta,\mu)
  \end{tikzcd}
\end{equation}

Let \( \omega \) be a 2-cell \( GH \to KF \) of internal \( \Cat \)-graphs.
We claim that \( \omega \) is a generalized vertical transformation if and
only if \( \omega \) is a generalized 2-cell of pseudo-\(\fc\)-algebras.

If \( \omega \) is a generalized vertical transformation, we wish to prove
that the following diagram commutes

\begin{equation*}
  \begin{tikzcd}
    & Gb(Hr_1, \ldots, Hr_n) \ar[ld,"\delta^G_{Hr_1, \ldots, Hr_n}",swap] 
                             \ar[rd,"G\gamma^H_{r_1, \ldots, r_n}"] \\
    d(GHr_1, \ldots, GHr_n) \ar[d,"{d(\omega_{r_1}, \ldots, \omega_{r_n})}",swap]
    && GHa(r_1, \ldots, r_n) \ar[d,"\omega_{a(r_1, \ldots, r_n)}"] \\
    d(KFr_1, \ldots, KFr_n) \ar[rd,"\gamma^K_{Fr_1, \ldots, Fr_n}",swap]
    && KFa(r_1, \ldots, r_n) \ar[ld,"K\delta^F_{r_1, \ldots, r_n}"] \\
    & Kc(Fr_1, \ldots, Fr_n)
  \end{tikzcd}
\end{equation*}
For all \(n\) and all horizontal 1-cells \(r_1, \ldots, r_n\). We proceed by
induction: when \(n=0\), the above is just coherence of \( \omega \) for the
unit comparsion. If true for some \(n\), then
\begin{align*}
  K\delta^F_{r_1, \ldots, r_{n+1}}
    &\circ \omega_{a(r_1, \ldots, r_{n+1})}
     \circ G\gamma^H_{r_1, \ldots, r_{n+1}} \\
    &= K(\id \cdot \delta^F_{r_1, \ldots, r_n})
        \circ K\m F 
        \circ \omega_{r_{n+1} \cdot a(r_1,\ldots,r_n)}
        \circ G\m H 
        \circ G(\id \cdot \gamma^H_{r_1, \ldots, r_n}) \\
    &= K(\id \cdot \delta^F_{r_1, \ldots r_n})
        \circ \m K 
        \circ (\omega_{r_{n+1}} \cdot \omega_{a(r_1,\ldots,r_n)})
        \circ \m G
        \circ G(\id \cdot \gamma^H_{r_1, \ldots, r_n}) \\
    &= \m K 
        \circ (\id \cdot K\delta^F_{r_1, \ldots, r_n})
        \circ (\omega_{r_{n+1}} \cdot \omega_{a(r_1,\ldots,r_n)})
        \circ (\id \cdot G\gamma^H_{r_1, \ldots, r_n}) 
        \circ \m G \\
    &= \m K 
        \circ (\id \cdot \gamma^K_{Fr_1, \ldots, Fr_n})
        \circ (\omega_{r_{n+1}} \cdot d(\omega_{r_1}, \ldots, \omega_{r_n}))
        \circ (\id \cdot \delta^G_{Hr_1, \ldots, Hr_n})
        \circ \m G \\
    &= \gamma^K_{Fr_1, \ldots, Fr_{n+1}}
        \circ d(\omega_{r_1}, \ldots, \omega_{r_{n+1}})
        \circ \delta^G_{Hr_1, \ldots, Hr_{n+1}},
\end{align*}
so \( \omega \) is a pseudo-\(\fc\)-algebra 2-cell as well.

Now, if \( \omega \) is a pseudo-\(\fc\)-algebra 2-cell, the coherence for the
unit comparison holds by definition, and
\begin{align*}
  \m K_{Fr,Fs} &\circ (\omega_s \cdot \omega_r) 
                \circ \m G_{Hr,Hs} \\
    &= K(\id \cdot \rho) \circ \gamma^K_{Fr,Fs}
                         \circ (\id \cdot \rho^{-1})
                         \circ (\omega_s \cdot \omega_r)
                         \circ (\id \cdot \rho)
                         \circ \delta^G_{Hr,Hs}
                         \circ G(\id \cdot \rho^{-1}) \\
    &= K(\id \cdot \rho) \circ \gamma^K_{Fr,Fs}
                         \circ (\omega_s \cdot (\omega_r \cdot \id))
                         \circ \delta^G_{Hr,Hs}
                         \circ G(\id \cdot \rho^{-1}) \\
    &= K(\id \cdot \rho) \circ \gamma^K_{Fr,Fs}
                         \circ d(\omega_r, \omega_s)
                         \circ \delta^G_{Hr,Hs}
                         \circ G(\id \cdot \rho^{-1}) \\
    &= K(\id \cdot \rho) \circ K\delta^F_{r,s}
                         \circ \omega_{a(r,s)}
                         \circ G\gamma^H_{r,s}
                         \circ G(\id \cdot \rho^{-1}) \\
    &= K(\id \cdot \rho) \circ K\delta^F_{r,s}
                         \circ KF(\id \cdot \rho^{-1})
                         \circ \omega_{s \cdot r}
                         \circ GH(\id \cdot \rho)
                         \circ G(\id \cdot \rho^{-1}) \\
    &= K(\m F) \circ \omega_{s \cdot r} \circ G(\m H),
\end{align*}
verifies coherence for composition comparsion, completing our proof.

Now, as promised at the start of Section \ref{sect:spanmat}, we obtain:
\begin{proposition}
  We have a conjunction
  \begin{equation*}
    \begin{tikzcd}
      \VMat \ar[r,bend left,"- \pt \trm"{name=A}]
      & \SpanV \ar[l,bend left,"\cat V(\trm{,}-)"{name=B,below}]
      \ar[from=A,to=B,phantom,"\adj" {anchor=center, rotate=-90}]
    \end{tikzcd}
  \end{equation*}
  in the double category \( \PsDbCat \).
\end{proposition}
\begin{proof}
  Via the equivalence \( \PsDbCat \eqv \PsfcAlg \), we simply apply
  Proposition \ref{prop:doct.adj} to the adjunction \( - \pt \trm \adj \cat
  V(\trm,-) \) in \( \Grph(\Cat) \), with the oplax functor structure of \( -
  \pt \trm \colon \SpanV \to \VMat \), all of which were described in Section
  \ref{sect:spanmat}.
\end{proof}

  \section{Horizontal lax algebras and change of base}
    \label{sect:base.change}
    We will review the notion of categories of \textit{horizontal lax algebras}
introduced in \cite{CS10}, and we define the \textit{change-of-base} functors
between such categories, induced by an appropriate notion of monad morphism.
We begin by fixing monads \( S = (\bicat D, S, m, e) \) and \( T = (\bicat E,
T, m, e) \) in the 2-category \( \PsDbCat_\lax \).

We define the category \( \LaxTHAlg \) of \textit{horizontal lax
$T$-algebras}, as follows:
\begin{itemize}[label=--]
  \item
    Objects are given by 4-tuples \( (x,a,\upsilon,\mu) \) where \(x\) is a
    0-cell, \( a \colon Tx \relto x \) is a horizontal 1-cell, and \( \upsilon,
    \mu \) are 2-cells 
    \begin{equation*}
      \begin{tikzcd}
        x \ar[r,"1"{name=A}] \ar[d,"e",swap] & x \ar[d,equal] \\
        Tx \ar[r,"a"{name=B},swap] & x
        \ar[from=A,to=B,phantom,"\upsilon" description]
      \end{tikzcd}
      \qquad
      \begin{tikzcd}
        TTx \ar[r,"Ta"] \ar[d,"m",swap] & Tx \ar[r,"a"] & x \ar[d,equal]
        \\ Tx \ar[rr,"a"{name=A},swap] && x
        \ar[from=1-2,to=A,phantom,"\mu" description]
      \end{tikzcd}
    \end{equation*}
    satisfying 
    \begin{align*}
      \mu \circ (\upsilon \cdot e_a) &= \lambda \\
      \mu \circ (\id \cdot (T\upsilon \circ \e T)) &= \rho \\
      \mu \circ (\id \cdot (T\mu \circ \m T)) 
        &= \mu \circ (\mu \cdot m_a) \circ \alpha^{-1}
    \end{align*}
  \item
    A morphism \( (x,a,\upsilon,\mu) \to (y,b,\upsilon,\mu) \) is a pair \(
    (f,\zeta) \) where \( f \colon x \to y \) is a vertical 1-cell and \(
    \zeta \) is a 2-cell
    \begin{equation*}
      \begin{tikzcd}
        Tx \ar[d,"Tf",swap] \ar[r,"a"{name=A}] & x \ar[d,"f"] \\
        Ty \ar[r,"b"{name=B},swap] & y 
        \ar[from=A,to=B,phantom,"\zeta" description]
      \end{tikzcd}
    \end{equation*}
    satisfying \( \zeta \circ \upsilon = \upsilon \circ 1_f \) and \( \zeta
    \circ \mu = \mu \circ (\zeta \cdot T\zeta) \).
\end{itemize}
It should be noted that \( \id = (\id,\id) \colon (x,a,\upsilon,\mu) \to
(x,a,\upsilon,\mu) \) is a horizontal lax \(T\)-algebra morphism, and if \(
(f,\zeta) \), \( (g,\xi) \) are composable horizontal lax \(T\)-algebra
morphisms, then so is \( (g,\xi) \circ (f,\zeta) = (g \circ f, \xi \circ
\zeta) \). Associativity and identity properties are inherited from \( \bicat
E_0 \) and \( \bicat E_1 \), making \( \LaxTHAlg \) into a category.

Our work focuses on the cases \( \bicat E = \SpanV \), with \(T\) induced by a
cartesian monad (also denoted by \(T\)) on \( \cat V \), and \( \bicat D =
\VMat \) with \(S\) a lax monad. Then, \( \LaxTHAlg = \CatTV\) is the category
of \textit{internal $T$-categories} of \cite{Her00}, while \( \LaxSHAlg =
\SVCat \) is a generalization of the category of \textit{enriched
$S$-categories} introduced by \cite{CT03}, by not requiring \(S\) to be
normal.\footnote{When $\cat V$ is a quantale, this generalization is already
present in \cite{Sea05}.}

Let \( (F,\phi) \colon S \to T \) be a monad oplax morphism, and we assume \(
\bicat E \) is conjoint closed. By Theorem \ref{thm:conjoint.closed}, \( \phi
\) has a conjoint, given by a lax horizontal transformation \( \phi^* \colon
TF \to FS \). We define a 2-cell \( \e{\phi^*_x} \) for each 0-cell \(x\)
given by
\begin{equation*}
  \begin{tikzcd}
    Fx \ar[r,"1"{name=A}] \ar[d,"e_{Fx}",swap]
    & Fx \ar[d,"Fe_x"] \\
    TFx \ar[r,"\phi^*_x"{name=B},swap]
    & FSx
    \ar[from=A,to=B,phantom,"\e{\phi^*_x}"]
  \end{tikzcd}
\end{equation*}
as the mate of the commutative square \( \phi_x \circ Fe_x = e_{Fx} \circ \id
\), and a 2-cell \( \m{\phi^*_x} \) given by
\begin{equation*}
  \begin{tikzcd}
    TTFx \ar[r,"(T\phi_x)^*"]
         \ar[d,equal]
    & TFSx \ar[r,"\phi^*_{Sx}"]
    & FSSx \ar[d,equal] \\
    TTFx \ar[rr,"(T\phi_x \circ \phi_{Sx})^*" description,""{name=A}]
         \ar[d,"m_{Fx}",swap]
    && FSSx \ar[d,"Fm_x"] \\
    TFx \ar[rr,"\phi^*_x"{name=B},swap]
    && FSx
    \ar[from=1-2,to=A,"\pi",phantom]
    \ar[from=A,to=B,"1^\lor",phantom]
  \end{tikzcd}
\end{equation*}
where \( \pi \) is given as in \eqref{eq:v2h.comp}, and \( 1^{\lor} \) is the
mate of the commutative square \( \phi_x \circ Fm_x = m_{Fx} \circ (T\phi_x
\circ \phi_{Sx}) \). To be explicit, via mate correspondence we have
\begin{equation}
  \label{eq:em.mates.props}
  \epsilon \circ \e{\phi^*_x} = 1_{e_{Fx}},
  \quad\text{and}\quad
  \epsilon \circ \m{\phi^*_x} = 1_{m_{Fx}} \circ \rho \circ ((1_{T\phi_x} \circ
  \epsilon) \cdot \epsilon).
\end{equation}

Analogously, when \( \bicat D \) is companion closed, we define 2-cells 
\( \e{\psi_!}_y \) and \( \m{\psi_!}_y \) for a monad lax morphism \( (G,\psi)
\colon T \to S \).

\begin{lemma}
  \label{lem:h.oplax.monad}
  If \( (F,\phi) \colon S \to T \) is a monad oplax morphism and \( \bicat E
  \) is conjoint closed, then \( \e{\phi^*} \) and \( \m{\phi^*} \) are
  modifications, and the following relations hold:
  \begin{enumerate}[label=(\alph*)]
    \item
      \label{enum:l.id.hmonad}
      \( \m{\phi^*}_x \circ (\e{\phi^*}_{Sx} \cdot e_{\phi_x}^\lor) = \lambda \),
    \item
      \label{enum:r.id.hmonad}
      \( \m{\phi^*}_x \circ (\phi^*_e \cdot \e{(T\phi)^*}_x) = \rho \),
    \item
      \label{enum:assoc.hmonad}
      \( \m{\phi^*}_x \circ (\m{\phi^*}_{Sx} \cdot m_{\phi_x}^\lor) 
          = \m{\phi^*_x} \circ (\phi^*_m \cdot \m{(T\phi)^*}_x) \circ \alpha \),
  \end{enumerate}
  where \( e^\lor_{\phi_x} \) and \( m^\lor_{\phi_x} \) are the mates of the
  naturality squares of \(e\) and \(m\) at \( \phi_x \), and \(
  \e{(T\phi)^*}_x \), \( \m{(T\phi)^*_x} \) are respectively given by the mate
  of the commutative square \( T\phi_x \circ TFe_x = Te_{Fx} \circ \id \), and
  the mate of the commutative square \( T\phi_x \circ TFm_x = Tm_{Fx} \circ
  (TT\phi_x \circ T\phi_{Sx}) \) composed with \( \pi \), satisfying
  properties similar to \eqref{eq:em.mates.props}.
\end{lemma}

\begin{proof}
  Note that \( \pi \) is given as a 2-cell (modification) in \(
  \Lax_\lax(\bicat D,\bicat E) \), and \( \e{\phi^*} \) and \( 1^\lor \) are
  mates of equations of vertical 1-cells. It follows that \( \e{\phi^*} \) and
  \( \m{\phi^*} \) are modifications.

  We have
  \begin{align*}
    \epsilon \circ \m{\phi^*}_x \circ (\e{\phi^*}_{Sx} \cdot e_{\phi_x}^\lor) 
      &= 1_{m_{Fx}} \circ \lambda \circ ((1_{T\phi_x} \circ \epsilon) \cdot \epsilon)
                    \circ (\e{\phi^*}_{Sx} \cdot e_{\phi_x}^\lor) \\
      &= 1_{m_{Fx}} \circ \lambda 
                    \circ (1_{T\phi_x \circ e_{FSx}} 
                            \cdot (1_{e_{TFx}} \circ \epsilon)) \\
      &= 1_{m_{Fx}} \circ 1_{e_{TFx}} \circ \epsilon \circ \lambda 
       = \epsilon \circ \lambda,
  \end{align*}
  \begin{align*}
    \epsilon \circ \m{\phi^*}_x \circ (\phi^*_e \cdot \e{(T\phi)^*}_x)
      &= 1_{m_{Fx}} \circ \rho \circ ((1_{T\phi_x} \circ \epsilon) \cdot \epsilon)
                    \circ (\phi^*_e \cdot \e{(T\phi)^*}_x) \\
      &= 1_{m_{Fx}} \circ \rho 
                    \circ ((1_{T\phi_x \circ TFe_x} \circ \epsilon) 
                            \cdot 1_{Te_{Fx}}) \\
      &= 1_{m_{Fx}} \circ 1_{Te_{Fx}} \circ \epsilon \circ \rho
       = \epsilon \circ \rho,
  \end{align*}
  Now, we note that
  \begin{equation}
    \label{eq:assoc.lhs.1}
    \epsilon \circ \m{\phi^*}_x \circ (\m{\phi^*}_{Sx} \cdot m_{\phi_x}^\lor) 
      = \rho \circ \big( (1_{m_{Fx} \circ T\phi_x} \circ \epsilon 
                                                   \circ \m{\phi^*}_{Sx})
                   \cdot (1_{m_{Fx}} \circ \epsilon \circ m_{\phi_x}^\lor)
                   \big),
  \end{equation}
  and we note that
  \begin{equation*}
    \epsilon \circ \m{\phi^*}_{Sx}
      = \rho \circ \big( (1_{m_{FSx} \circ T\phi_{Sx}} \circ \epsilon)
                   \cdot (1_{m_{FSx}} \circ \epsilon)
                   \big),
  \end{equation*}
  and
  \begin{equation*}
    \epsilon \circ m_{\phi_x}^\lor = 1_{m_{FSx}} \circ \epsilon,
  \end{equation*}
  so that \eqref{eq:assoc.lhs.1} becomes
  \begin{equation*}
    \epsilon \circ \m{\phi^*}_x \circ (\m{\phi^*}_{Sx} \cdot m_{\phi_x}^\lor) 
      = \rho \circ (\rho \cdot \id) \circ ((A \cdot B) \cdot C)
  \end{equation*}
  where \( A = 1_{\hat A} \circ \epsilon \), \( B = 1_{\hat B} \circ \epsilon
  \), \( C = 1_{\hat C} \circ \epsilon \), and
  \begin{itemize}[label=--]
    \item
      \( \hat A = m_{Fx} \circ T\phi_x \circ m_{FSx} \circ T\phi_{Sx} \),
    \item
      \( \hat B = m_{Fx} \circ T\phi_x \circ m_{FSx} \),
    \item
      \( \hat C = m_{Fx} \circ m_{TFx} \).
  \end{itemize}

  On the other hand, we have
  \begin{equation}
    \label{eq:assoc.rhs.1}
    \epsilon \circ \m{\phi^*_x} \circ (\phi^*_m \cdot \m{(T\phi)^*}_x) \circ \alpha
      = \rho \circ \big( (1_{m_{Fx} \circ T\phi_x} \circ \epsilon \circ \phi^*_m)
                   \cdot (1_{m_{Fx}} \circ \epsilon \circ \m{(T\phi)^*}_x)
                   \big),
  \end{equation}
  and we note that
  \begin{equation*}
    \epsilon \circ \phi^*_m = 1_{TFm_x} \circ \epsilon,
  \end{equation*}
  and
  \begin{equation*}
    \epsilon \circ \m{(T\phi)^*}_x
      = \lambda \circ \big( (1_{Tm_{Fx} \circ TT\phi_x} \circ \epsilon)
                   \cdot (1_{Tm_{Fx}} \circ \epsilon) \big),
  \end{equation*}
  so that \eqref{eq:assoc.rhs.1} becomes
  \begin{equation*}
    \epsilon \circ \m{\phi^*_x} \circ (\phi^*_m \cdot \m{(T\phi)^*}_x) \circ \alpha
      = \rho \circ (\id \cdot \lambda) \circ \alpha
             \circ ((X \cdot Y) \cdot Z),
  \end{equation*}
  where \( X = 1_{\hat X} \circ \epsilon \), \( Y = 1_{\hat Y} \circ \epsilon
  \), \( Z = 1_{\hat Z} \circ \epsilon \), and
  \begin{itemize}[label=--]
    \item
      \( \hat X = m_{Fx} \circ T\phi_x \circ TFm_x\),
    \item
      \( \hat Y =  m_{Fx} \circ Tm_{Fx} \circ TT\phi_x \),
    \item
      \( \hat Z =  m_{Fx} \circ Tm_{Fx} \).
  \end{itemize}
  We conclude the proof by observing that \( A = X \), \( B = Y \) and \( C =
  Z \).
\end{proof}

\begin{theorem}
  \label{thm:base.change}
  We suppose that \( (F,\phi) \colon S \to T \) is an monad oplax morphism and
  that \( \bicat E \) is conjoint closed. If \(\phi\) and \(T\phi\) have
  strong conjoints (see Lemma \ref{lem:right.conjoint.invert}), then \(
  (F,\phi) \) induces a functor \( F_! \colon \LaxSHAlg \to \LaxTHAlg \). 

  Analogously, if we suppose that \( (G,\psi) \colon T \to S \) is a monad
  lax morphism and that \( \bicat D \) is companion closed, then \( (G,\psi)
  \) induces a functor \( G_! \colon \LaxTHAlg \to \LaxSHAlg \).
\end{theorem}

\begin{proof}
  The functors \( F_! \) and \(G_!\) are given on objects by 
  \begin{equation*}
    F_!(x,a,\upsilon,\mu) = (Fx, Fa \cdot \phi_x^*, F_!\upsilon, F_!\mu) 
    \quad\text{and}\quad 
    G_!(y,b,\upsilon,\mu) = (Gy, Gb \cdot \psi_{y!}, G_!\upsilon, G_!\mu),
  \end{equation*}
  where \( F_!\upsilon \), \( F_!\mu \) are respectively given by the
  following 2-cells:
  \begin{equation*}
    \begin{tikzcd}
      Fx \ar[r,"1"{name=A}] \ar[dd,"e_{Fx}",swap] 
        & Fx \ar[r,"1"{name=x}] \ar[d,equal]
        & Fx \ar[d,equal] \\
      & Fx \ar[r,""{name=Y},"F1" description] \ar[d,"Fe_x" description]
      & Fx \ar[d,equal] \\
      TFx \ar[r,"\phi_x^*"{name=B},swap]
        & FSx \ar[r,"Fa"{name=Z},swap]
        & Fx
      \ar[from=A,to=B,phantom,"\e{\phi^*}_x" description]
      \ar[from=X,to=Y,phantom,"\e F" description]
      \ar[from=Y,to=Z,phantom,"F\upsilon" description]
    \end{tikzcd}
  \end{equation*}
  \begin{equation*}
    \begin{tikzcd}[column sep=large]
      TTFx \ar[rr,"T(Fa \cdot \phi^*_x)"{name=A}] 
           \ar[d,equal]
        && TFx \ar[r,"\phi^*_x" description] 
               \ar[d,equal]
        & FSx \ar[r,"Fa" description]
              \ar[d,equal]
        & Fx \ar[dd,equal] \\
      TTFx \ar[r,"(T\phi_x)^*" description]
           \ar[d,equal]
        & TFSx \ar[r,"TFa" description]
               \ar[d,equal]
        & TFx \ar[r,"\phi^*_x" description]
        & FSx \ar[d,equal] \\
      TTFx \ar[r,"(T\phi_x)^*" description]
           \ar[dd,"m_{Fx}" description]
        & TFSx \ar[r,"\phi^*_{Sx}" description]
        & FSSx \ar[r,"FSa" description]
               \ar[d,equal]
        & FSx \ar[r,"Fa" description]
        & Fx \ar[d,equal] \\
        && FSSx \ar[rr,"F(a \cdot Sa)" description,""{name=C}]
              \ar[d,"Fm_x" description]
        && Fx \ar[d,equal] \\
      TFx \ar[rr,"\phi^*_x"{name=B},swap]
        && FSx \ar[rr,"Fa"{name=D},swap]
        && Fx
      \ar[from=A,to=2-2,"\theta_\phi^T" description,phantom]
      \ar[from=2-3,to=3-3,"(\n{\phi^*}_a)^{-1}" description,phantom]
      \ar[from=3-2,to=B,"\m{\phi^*}_x" description,phantom]
      \ar[from=3-4,to=C,"\m F" description,phantom]
      \ar[from=C,to=D,"F\mu" description,phantom]
    \end{tikzcd}
  \end{equation*}
  where \( \theta_\phi^T \) is the inverse of \( \m T \circ (\id \cdot
  \sigma^T) \), and \( G_!\upsilon \), \( G_!\mu \) are respectively given by
  the following 2-cells:
  \begin{equation*}
    \begin{tikzcd}
      Gy \ar[r,"1"{name=A}] \ar[dd,"e_{Gy}",swap] 
        & Gy \ar[r,"1"{name=x}] \ar[d,equal]
        & Gy \ar[d,equal] \\
      & Gy \ar[r,""{name=Y},"G1" description] \ar[d,"Ge_y" description]
      & Gy \ar[d,equal] \\
      SGy \ar[r,"\psi_{y!}"{name=B},swap]
        & GTy \ar[r,"Gb"{name=Z},swap]
        & Gy
      \ar[from=A,to=B,phantom,"\e{\psi_!}_y" description]
      \ar[from=X,to=Y,phantom,"\e G" description]
      \ar[from=Y,to=Z,phantom,"G\upsilon" description]
    \end{tikzcd}
  \end{equation*}
  \begin{equation*}
    \begin{tikzcd}[column sep=large]
      SSGy \ar[rr,"S(Gb \cdot \psi_{y!})"{name=A}] 
           \ar[d,equal]
        && SGy \ar[r,"\psi_{y!}" description] 
               \ar[d,equal]
        & SGy \ar[r,"Gb" description]
              \ar[d,equal]
        & Fy \ar[dd,equal] \\
      SSGy \ar[r,"(S\psi_y)_!" description]
           \ar[d,equal]
        & SGTy \ar[r,"SGb" description]
               \ar[d,equal]
        & SGy \ar[r,"\psi_{y!}" description]
        & GTy \ar[d,equal] \\
      SSGy \ar[r,"(S\psi_t)_!" description]
           \ar[dd,"m_{Gx}" description]
        & SGTy \ar[r,"\psi_{Tx!}" description]
        & GTTy \ar[r,"GTb" description]
               \ar[d,equal]
        & GTy \ar[r,"Gb" description]
        & Gy \ar[d,equal] \\
        && GTTy \ar[rr,"G(b \cdot Tb)" description,""{name=C}]
              \ar[d,"Gm_y" description]
        && Fy \ar[d,equal] \\
      SGy \ar[rr,"\psi_{y!}"{name=B},swap]
        && GTy \ar[rr,"Gb"{name=D},swap]
        && Gy
      \ar[from=A,to=2-2,"\theta_\psi^S" description,phantom]
      \ar[from=2-3,to=3-3,"\n{\psi_!}_b" description,phantom]
      \ar[from=3-2,to=B,"\m{\psi_!}_y" description,phantom]
      \ar[from=3-4,to=C,"\m G" description,phantom]
      \ar[from=C,to=D,"G\mu" description,phantom]
    \end{tikzcd}
  \end{equation*}
  where \( \theta^S_\psi \) is the inverse of \( \m S \circ (\id \cdot \tau^S)
  \), given by Lemma \ref{lem:conjoint.eq.normal}.

  If \( (f,\zeta) \colon (w,a,\upsilon,\mu) \to (x,b,\upsilon,\mu) \) is a
  horizontal lax \(S\)-algebra morphism, and if \( (g,\xi) \colon
  (y,c,\upsilon,\mu) \to (z,d,\upsilon,\mu) \) is a horizontal lax
  \(T\)-algebra morphism, then we have
  \begin{equation*}
    F_!(f,\zeta) = (Ff,F\zeta \cdot \phi^*_f)
    \quad\text{and}\quad G_!(g,\xi) = (Gg,G\xi \cdot \psi_{f!}). 
  \end{equation*}

  We observe that \( \phi \) and \( T\phi \) are required to have strong
  conjoints only to guarantee the existence of \( F_!\mu \). All other things
  being equal, we conclude it is enough to verify that one of \(
  F_!(x,a,\upsilon,\mu) \), \( G_!(y,b,\upsilon,\mu) \) is a horizontal lax
  algebra, and likewise for the morphisms.

  Throughout the calculations, we will use the following abbreviations:
  \begin{itemize}[label=--]
    \item
      \( \upsilon^F = F\upsilon \circ \e F \),
    \item
      \( \mu^F = F\mu \circ \m F \),
    \item
      \( \hat \alpha = (\id \cdot \alpha^{-1}) \circ \alpha \),
    \item
      \( \tilde{\theta} = \id \cdot (\theta \cdot \id) \) for a 2-cell \(
      \theta \).
    \item
      \( N^{\omega}_p = \hat \alpha^{-1} \circ (\tn{\omega}_p)^{-1}
      \circ \hat \alpha \) for a strong (lax) horizontal transformation \(
      \omega \).
  \end{itemize}

  We begin by verifying that the following equalities hold:
  \begin{align}
    \label{eq:l.id.theta}
    \theta^T \circ e_{Fa \cdot \phi^*_x}
      &= e_{Fa} \cdot e^\lor_{\phi_x}, \\
    \label{eq:r.id.theta}
    \theta^T \circ T(\upsilon^F \cdot \e{\phi^*}_x) 
             \circ T\rho^{-1} \circ \e T
      &= (\upsilon^{TF} \cdot \e{(T\phi)^*}_x) \circ \rho^{-1},
  \end{align}
  We obtain \eqref{eq:l.id.theta}, via mate correspondence, by noting that
  \begin{align*}
    \m T \circ (\id \cdot \sigma^T)
         \circ (e_{Fa} \cdot e^\lor_{\phi_x})
         \circ (\id \cdot \eta) \circ \rho^{-1}
    &= \m T \circ (\id \cdot \sigma^T)
            \circ (\id \cdot \eta)
            \circ (e_{Fa} \cdot 1_{e_{FSx}}) \circ \rho^{-1} \\
    &= \m T \circ (\id \cdot T\eta)
            \circ (\id \cdot \e T)
            \circ \rho^{-1}
            \circ e_{Fa} \\
    &= T(\id \cdot \eta) \circ \m T
            \circ (\id \cdot \e T)
            \circ \rho^{-1} \circ e_{Fa} \\
    &= T(\id \cdot \eta) \circ T\rho^{-1} \circ e_{Fa}  \\
    &= e_{Fa \cdot \phi^*_x} \circ (\id \cdot \eta) \circ \rho^{-1},
  \end{align*}
  and \eqref{eq:r.id.theta}, directly, since
  \begin{align*}
    \m T \circ (\id \cdot \sigma^T)
         \circ (\upsilon^{TF} \cdot \e{(T\phi)^*}_x) \circ \rho^{-1}
      &= \m T \circ (\id \cdot T\eta) 
              \circ (\id \cdot \e T)
              \circ (\upsilon^{TF} \cdot 1_{TFe_x}) \circ \rho^{-1} \\
      &= T(\id \cdot \eta) \circ \m T
                           \circ (\id \cdot \e T)
                           \circ \rho^{-1} \circ \upsilon^{TF} \\
      &= T(\id \cdot \eta) \circ T\rho^{-1} \circ \upsilon^{TF} \\
      &= T(\id \cdot \eta) \circ T(\upsilon^F \cdot 1_{Fe_x}) 
                           \circ T\rho^{-1} \circ \e T \\
      &= T(\upsilon^F \cdot \e{\phi^*}_x) \circ T\rho^{-1} \circ \e T. 
  \end{align*}

  Furthermore, since \( \e{\phi^*} \) is a modification and \( \n{\phi^*} \)
  is natural, we respectively obtain 
  \begin{align}
    \label{eq:l.id.nat}
    (\n{\phi^*}_a)^{-1} \circ (\e{\phi^*}_x \circ e_{Fa}) 
      &= (Fe_a \cdot \e{\phi^*}_{Sx}) \circ \gamma, \\
    \label{eq:r.id.nat}
    (\n{\phi^*}_a)^{-1} \circ (\id \cdot \upsilon^{TF}) 
      &= (\upsilon^{FS} \cdot \phi^*_e) \circ \gamma.
  \end{align}

  And lastly, we note that the following diagrams commute
  \begin{equation*}
    \begin{tikzcd}
      1 \cdot Fa \ar[rrr,"\lambda"]
                 \ar[rd,"\e F \cdot \id" description]
                 \ar[ddr,bend right,"\upsilon^F \cdot Fe_a",swap]
      &&& Fa \\
      & F1 \cdot Fa \ar[d,"F\upsilon \cdot Fe_a" description]
                    \ar[r,"\m F"]
      & F(1 \cdot a) \ar[d,"F(\upsilon \cdot e_a)" description]
                     \ar[ur,"F\lambda" description] \\
      & Fa \cdot FSa \ar[r,"\m F",swap]
      & F(a \cdot Sa) \ar[ruu,"F\mu",swap,bend right] 
    \end{tikzcd}
  \end{equation*}
  \begin{equation*}
    \begin{tikzcd}
      Fa \cdot 1 \ar[rrr,"\rho"]
                 \ar[rd,"\id \cdot \e F" description]
                 \ar[ddr,bend right,"\id \cdot \upsilon^{FS}",swap]
      &&& Fa \\
      & Fa \cdot F1 \ar[d,"\id \cdot F\upsilon^S" description] 
                    \ar[r,"\m F"]
      & F(a \cdot 1) \ar[d,"F(\id \cdot \upsilon^S)" description]
                     \ar[ur,"F\rho" description] \\
      & Fa \cdot FSa \ar[r,"\m F",swap]
      & F(a \cdot Sa) \ar[ruu,"F\mu",swap,bend right]
    \end{tikzcd}
  \end{equation*}
  which respectively confirm that
  \begin{align}
    \label{eq:l.id.f}
    \mu^F \circ (\upsilon^F \cdot Fe_a) &= \lambda, \\
    \label{eq:r.id.f}
    \mu^F \circ (\id \cdot \upsilon^{FS}) &= \rho.
  \end{align}
    
  By applying \eqref{eq:l.id.theta}, \eqref{eq:l.id.nat}, \eqref{eq:l.id.f},
  and \ref{enum:l.id.hmonad} from Lemma \ref{lem:h.oplax.monad}, we obtain
  \begin{align*}
    (\mu^F \cdot \m{\phi^*}_x)
      &\circ N^{\phi^*}_a \circ (\id \cdot \theta^T)
      \circ \big((\upsilon^F \cdot \e{\phi^*}_x) 
                  \cdot e_{Fa \cdot \phi^*_x}\big)
      \circ (\lambda^{-1} \cdot \id) \\
      &\qquad= (\mu^F \cdot \m{\phi^*}_x)
          \circ N^{\phi^*}_a
          \circ \big((\upsilon^F \cdot \e{\phi^*}_x) 
                      \cdot (e_{Fa} \cdot e^\lor_{\phi_x})\big) 
          \circ (\lambda^{-1} \cdot \id) \\
      &\qquad= (\mu^F \cdot \m{\phi^*}_x)
          \circ \big((\upsilon^F \cdot Fe_a) 
                      \cdot (\e{\phi^*}_{Sx} \cdot e^\lor_{\phi_x})\big) 
          \circ N^1_a \circ (\lambda^{-1} \cdot \id) \\
      &\qquad= (\lambda \cdot \lambda) 
          \circ \hat \alpha^{-1} \circ \tilde{\gamma} \circ \hat \alpha
          \circ (\lambda^{-1} \cdot \id) 
       = \lambda,
  \end{align*}
  verifying the left identity law, and by applying \eqref{eq:r.id.theta},
  \eqref{eq:r.id.nat}, \eqref{eq:r.id.f} and \ref{enum:r.id.hmonad} from Lemma
  \ref{lem:h.oplax.monad}, we obtain
  \begin{align*}
    (\mu^F \cdot \m{\phi^*}_x)
      &\circ N^{\phi^*}_a
       \circ (\id \cdot \theta^T)
       \circ \big( \id \cdot (T(\upsilon^F \cdot \e{\phi^*}) 
                                \circ T\lambda^{-1} \circ \e T)\big) \\
      &\qquad= (\mu^F \cdot \m{\phi^*}_x)
               \circ N^{\phi^*}_a
               \circ \big( \id \cdot (\upsilon^{TF} \cdot \e{(T\phi)^*}_x) \big)
               \circ (\id \cdot \rho^{-1}) \\
      &\qquad= (\mu^F \cdot \m{\phi^*}_x)
               \circ \big( (\id \cdot \upsilon^{FS}) 
                            \cdot (\phi^*_e \cdot \e{(T\phi)^*}_x) \big)
               \circ N^1_a \circ (\id \cdot \rho^{-1}) \\
      &\qquad= (\rho \cdot \rho) 
               \circ \hat \alpha^{-1} \circ \tilde{\gamma} \circ \hat \alpha
               \circ (\id \cdot \rho^{-1}) 
      = \rho,
  \end{align*}
  verifying the right identity law.

  Now, note that
  \begin{equation}
    \label{eq:assoc.theta.1}
    \theta^T \circ m_{Fa \cdot \phi^*_x}
      = (m_{Fa} \cdot m^\lor_{\phi_x}) \circ \theta^{TT}
  \end{equation}
  holds via mate correspondence, since we have
  \begin{align*}
    m_{Fa \cdot \phi^*_x} 
      \circ \m{TT}
      \circ (\id \cdot \sigma^{TT})
      \circ (\id \cdot \eta)
      \circ \rho^{-1}
    &= m_{Fa \cdot \phi^*_x} 
        \circ \m{TT}
        \circ (\id \cdot TT\eta)
        \circ (\id \cdot \e{TT})
        \circ \rho^{-1} \\
    &= m_{Fa \cdot \phi^*_x} 
        \circ TT(\id \cdot \eta)
        \circ \m{TT}
        \circ (\id \cdot \e{TT})
        \circ \rho^{-1} \\
    &= T(\id \cdot \eta)
        \circ m_{Fa \cdot 1} 
        \circ TT\rho^{-1} \\
    &= T(\id \cdot \eta)
        \circ T\rho^{-1}
        \circ m_{Fa}, \\
    \m T \circ (\id \cdot \sigma^T)
         \circ (m_{Fa} \cdot m^\lor_{\phi_x})
         \circ (\id \cdot \eta)
         \circ \rho^{-1}
    &= \m T \circ (\id \cdot \sigma^T)
            \circ (\id \cdot \eta)
            \circ (m_{Fa} \cdot 1)
            \circ \rho^{-1} \\
    &= \m T \circ (\id \cdot T\eta)
            \circ (\id \cdot \e T)
            \circ \rho^{-1} 
            \circ m_{Fa} \\
    &= T(\id \cdot \eta)
            \circ \m T
            \circ (\id \cdot \e T)
            \circ \rho^{-1} 
            \circ m_{Fa} \\ 
    &= T(\id \cdot \eta) \circ T\rho^{-1} \circ m_{Fa},
  \end{align*}
  and since \( \m{\phi^*} \) is a modification, we get
  \begin{equation}
    \label{eq:assoc.modif}
    (\n{\phi^*}_a)^{-1} \circ (\m{\phi^*}_x \cdot m_{Fa})
      = (Fm_a \cdot \m{\phi^*}_{Sx}) \circ 
        (\n{\phi^*_S \cdot (T\phi)^*}_a)^{-1}.
  \end{equation}
  Now, note that the following diagram commutes
  \begin{equation*}
    \begin{tikzcd}
      (Fa \cdot FSa) \cdot FSSa
        \ar[r,"\alpha"]
        \ar[d,"\m F \cdot \id",swap]
      & Fa \cdot (FSa \cdot FSSa)
        \ar[r,"\id \cdot \m F"]
      & Fa \cdot F(Sa \cdot SSa)
        \ar[r,"\id \cdot F\mu^S"]
        \ar[d,"\m F"]
      & Fa \cdot FSa
        \ar[d,"\m F"] \\
      F(a \cdot Sa) \cdot FSSa
        \ar[r,"\m F"]
        \ar[d,"F\mu \cdot Fm_a",swap]
      & F((a \cdot Sa) \cdot SSa) 
        \ar[r,"F\alpha"]
        \ar[d, "F(\mu \cdot m_a)"] 
      & F(a \cdot (Sa \cdot SSa))
        \ar[r,"F(\id \cdot \mu^S)"]
      & F(a \cdot Sa) \ar[d,"F\mu"] \\
      Fa \cdot FSa 
        \ar[r,"\m F",swap]
      & F(a \cdot Sa) 
        \ar[rr,"F\mu",swap]
      && Fa
    \end{tikzcd}
  \end{equation*}
  which confirms
  \begin{equation}
    \label{eq:assoc.f}
    \mu^F \circ (\mu^F \cdot Fm_a) 
      = \mu^F \circ (\id \cdot \mu^{FS}) \circ \alpha.
  \end{equation}

  Our next step is to confirm that
  \begin{equation}
    \label{eq:assoc.coh}
    ((\id \cdot \m{FS}) \cdot \id)
      \circ (\alpha \cdot \alpha)
      \circ N^{\phi^*_S \cdot (T\phi)^*}_a
      \circ (N^{\phi^*}_a \cdot \id)
    = N^{\phi^*}_{a \cdot Sa}
        \circ (\id \cdot (\m{TF} \cdot \id))
        \circ (\id \cdot N^{(T\phi)^*}_a)
        \circ \alpha.
  \end{equation}
  First, we recall that
  \begin{equation*}
    \n{\phi^*_S \cdot (T\phi)^*}_a
      = \alpha^{-1} \circ (\id \cdot \n{(T\phi)^*}_a)
                    \circ \alpha
                    \circ (\n{\phi^*}_{Sa} \cdot \id)
                    \circ \alpha^{-1},
  \end{equation*}
  and
  \begin{equation*}
    \n{\phi^*}_{a \cdot Sa} \circ (\m{FS} \cdot \id)
      = (\id \cdot \m{TF}) \circ \alpha
                           \circ (\n{\phi^*}_a \cdot \id)
                           \circ \alpha^{-1}
                           \circ (\id \cdot \n{\phi^*}_{Sa})
                           \circ \alpha,
  \end{equation*}
  and note that by coherence, we have
  \begin{align*}
    \tilde \alpha \circ \hat \alpha \circ (\hat \alpha^{-1} \cdot \id)
      &= \alpha^{-1} \circ (\id \cdot \hat \alpha) 
                     \circ (\id \cdot \hat \alpha)
                     \circ \alpha, \\
    (\alpha \cdot \alpha) 
                \circ \hat \alpha^{-1} 
                \circ \tilde \alpha
      &= \hat \alpha^{-1} 
                \circ \tilde \alpha^{-1}
                \circ \hat \alpha 
                \circ (\id \cdot \alpha), \\
    \hat \alpha \circ (\id \cdot \alpha) 
                \circ \tilde \alpha^{-1}
                \circ \alpha^{-1}
                \circ (\id \cdot \hat \alpha)
      &= \tilde \alpha \circ (\id \cdot \hat \alpha^{-1}),\\
    \hat \alpha^{-1} \circ \tilde \alpha
                     \circ (\id \cdot \hat \alpha)^{-1}
      &= (\id \cdot \hat \alpha^{-1})
                     \circ \alpha^{-1}
                     \circ (\id \cdot \alpha) \\
    \alpha^{-1} \circ (\id \cdot \alpha) 
               \circ (\id \cdot \hat \alpha)
               \circ \alpha
               \circ (\hat \alpha \cdot \id)
      &= (\id \cdot \hat \alpha) \circ \alpha,
  \end{align*}
  so that
  \begin{align*}
    &((\id \cdot \m{FS}) \cdot \id)
      \circ (\alpha \cdot \alpha)
      \circ N^{\phi^*_S \cdot (T\phi)^*}_a
      \circ (N^{\phi^*}_a \cdot \id) \\
     &= ((\id \cdot \m{FS}) \cdot \id)
        \circ (\alpha \cdot \alpha)
        \circ \hat \alpha^{-1}
        \circ \tilde \alpha
        \circ (\id \cdot ((\n{\phi^*}_{Sa} \cdot \id) \cdot \id))^{-1}
        \circ \tilde \alpha^{-1} \\
      &\quad\circ (\id \cdot ((\id \cdot \n{(T\phi)^*}_a) \cdot \id))^{-1}
        \circ \tilde \alpha
        \circ \hat \alpha
        \circ (\hat \alpha^{-1} \cdot \id)
        \circ ((\id \cdot (\n{\phi^*}_a \cdot \id)) \cdot \id)^{-1}
        \circ (\hat \alpha \cdot \id) \\
    &= \hat \alpha^{-1}
        \circ (\id \cdot ((\m{FS}\cdot \id) \cdot \id))
        \circ \tilde \alpha^{-1}
        \circ (\id \cdot ((\id \cdot \n{\phi^*}_{Sa}) \cdot \id))^{-1} \\
       &\quad\circ \hat \alpha
        \circ (\id \cdot \alpha)
        \circ \tilde \alpha^{-1} 
        \circ \alpha^{-1} 
        \circ (\id \cdot \hat \alpha) 
        \circ (\id \cdot (\n{\phi^*}_a \cdot \id))^{-1} \\
       &\quad\circ (\id \cdot (\id \cdot (\n{(T\phi)^*}_a \cdot \id)))^{-1}
        \circ (\id \cdot \hat \alpha)
        \circ \alpha
        \circ (\hat \alpha \cdot \id) \\
    &= \hat \alpha^{-1}
        \circ (\id \cdot ((\m{FS}\cdot \id) \cdot \id))
        \circ \tilde \alpha^{-1}
        \circ (\id \cdot ((\id \cdot \n{\phi^*}_{Sa}) \cdot \id))^{-1} 
        \circ \tilde \alpha
        \circ (\id \cdot ((\n{\phi^*}_a \cdot \id) \cdot \id))^{-1} \\
       &\quad\circ (\id \cdot \hat \alpha)^{-1}
        \circ (\id \cdot (\id \cdot (\n{(T\phi)^*}_a \cdot \id)))^{-1}
        \circ (\id \cdot \hat \alpha)
        \circ \alpha
        \circ (\hat \alpha \cdot \id) \\
    &= \hat \alpha^{-1}
        \circ (\id \cdot (\n{\phi^*}_{a \cdot Sa} \cdot \id))^{-1}
        \circ (\id \cdot ((\id \cdot \m{TF}) \cdot \id))
        \circ \tilde \alpha
        \circ (\id \cdot N^{(T\phi)^*}_a)
        \circ \alpha 
        \circ (\hat \alpha \cdot \id) \\
    &= N^{\phi^*}_{a \cdot Sa}
        \circ (\id \cdot (\m{TF} \cdot \id))
        \circ \hat \alpha^{-1}
        \circ \tilde \alpha
        \circ (\id \cdot N^{(T\phi)^*}_a)
        \circ \alpha 
        \circ (\hat \alpha \cdot \id) \\
    &= N^{\phi^*}_{a \cdot Sa}
        \circ (\id \cdot (\m{TF} \cdot \id))
        \circ (\id \cdot \hat \alpha^{-1})
        \circ (\id \cdot \tn{(T\phi)^*}_a)^{-1}
        \circ \alpha^{-1}
        \circ (\id \cdot \alpha)
        \circ (\id \cdot \hat \alpha)
        \circ \alpha 
        \circ (\hat \alpha \cdot \id) \\
    &= N^{\phi^*}_{a \cdot Sa}
        \circ (\id \cdot (\m{TF} \cdot \id))
        \circ (\id \cdot N^{(T\phi)^*}_a)
        \circ \alpha.
  \end{align*}
 
  Next, we observe that 
  \begin{equation}
    \label{eq:assoc.nat}
    (FS\mu \cdot \phi^*_m) \circ (\n{\phi^*}_{a \cdot Sa})^{-1}
      = (\n{\phi^*}_a)^{-1} \circ (\id \cdot TF\mu) 
  \end{equation}
  holds by naturality of \( \n{\phi^*} \), and lastly we must confirm that
  \begin{equation}
    \label{eq:assoc.theta.2}
    (\mu^{TF} \cdot \m{(T\phi)^*}_a)
      \circ N^{(T\phi)^*}_a
      \circ (\theta^T \cdot \theta^{TT})
    = \theta^T \circ T(\mu^F \cdot \m{\phi^*}_x)
               \circ TN^{\phi^*}_a
               \circ T(\id \cdot \theta^T)
               \circ \m T,
  \end{equation}
  which we reduce to proving the following relations:
  \begin{equation*}
    TN^{\phi^*}_a
      \circ T(\id \cdot \theta^T)
      \circ \m T
      \circ (\m T \cdot \m{TT})
      \circ ((\id \cdot \sigma^T) \cdot (\id \cdot \sigma^{TT}))
    = \m T \circ (\m T \cdot \m T)
           \circ (\id \cdot (\sigma^T \cdot \sigma^T))
           \circ N^{(T\phi)^*}_a
  \end{equation*}
  \begin{equation*}
    \sigma^T \circ \m{(T\phi)^*}_x
      = T\m{\phi^*}_x \circ \m T \circ (\sigma^T \cdot \sigma^T)
  \end{equation*}

  For the first, we have the commutativity of the following diagram, omitting
  horizontal 1-cells
  \begin{equation*}
    \begin{tikzcd}[column sep=huge]
      \cdot \ar[r,"\id \cdot (\m T \cdot \id)"]
            \ar[d,"\id \cdot \alpha",swap]
      & \cdot \ar[r,"\id \cdot \m T"]
      & \cdot \ar[d,"\id \cdot T\alpha" description]
              \ar[r,"\m T"] 
      & \cdot \ar[d,"T(\id \cdot \alpha)"] \\
      \cdot \ar[r,"\id \cdot (\id \cdot \m T)" description]
            \ar[d,"\alpha^{-1}",swap]
      & \cdot \ar[r,"\id \cdot \m T" description]
              \ar[d,"\alpha^{-1}" description]
      & \cdot \ar[r,"\m T"]
      & \cdot \ar[d,"T\alpha^{-1}"] \\
      \cdot \ar[r,"\id \cdot \m T",swap]
      & \cdot \ar[r,"\m T \cdot \id",swap]
      & \cdot \ar[r,"\m T",swap]
      & \cdot
    \end{tikzcd}
  \end{equation*}
  then we observe that
  \begin{align*}
    T(\id \cdot \theta^T)
      &\circ \m T
      \circ (\m T \cdot \m{TT})
      \circ ((\id \cdot \sigma^T) \cdot (\id \cdot \sigma^{TT})) \\
    &= \m T \circ (\id \cdot T\theta^T)
            \circ (\m T \cdot \m{TT})
            \circ ((\id \cdot \sigma^T) \cdot (\id \cdot \sigma^{TT})) \\
    &= \m T \circ (\m T \cdot \m T) 
         \circ ((\id \cdot \sigma^T) \cdot (\id \cdot \sigma^T)),
  \end{align*}
  and finally recall from \eqref{eq:hphi.2.cell} that
  \begin{equation*}
    \m T \circ (\sigma^T \cdot \id)
         \circ \n{(T\phi)^*}_a
    = T\n{\phi^*}_a \circ \m T \circ (\id \cdot \sigma^T),
  \end{equation*}
  so that we may calculate
  \begin{align*}
    T(\id \cdot \theta^T)
      &\circ \m T
      \circ (\m T \cdot \m{TT})
      \circ ((\id \cdot \sigma^T) \cdot (\id \cdot \sigma^{TT}))
      \circ \hat \alpha^{-1}
      \circ \tn{(T\phi)^*}
      \circ \hat \alpha \\
    &= \m T \circ (\m T \cdot \m T)
            \circ ((\id \cdot \sigma^T) \cdot (\id \cdot \sigma^T))
            \circ \hat \alpha^{-1}
            \circ \tn{(T\phi)^*}
            \circ \hat \alpha \\
    &= \m T \circ (\m T \cdot \m T)
            \circ \hat \alpha^{-1}
            \circ (\id \cdot ((\sigma^T \cdot \id) \cdot \sigma^T))
            \circ \tn{(T\phi)^*}
            \circ \hat \alpha \\
    &= T\hat \alpha^{-1}
            \circ \m T 
            \circ (\id \cdot \m T)
            \circ \tm T
            \circ (\id \cdot ((\sigma^T \cdot \id) \cdot \sigma^T))
            \circ \tn{(T\phi)^*}
            \circ \hat \alpha \\
    &= T\hat \alpha^{-1}
            \circ \m T 
            \circ (\id \cdot \m T)
            \circ (\id \cdot (T\n{\phi^*}_a \cdot \id))
            \circ \tm T
            \circ (\id \cdot ((\id \cdot \sigma^T) \cdot \sigma^T))
            \circ \hat \alpha \\
    &= T\hat \alpha^{-1}
            \circ T\tn{\phi^*}_a
            \circ \m T 
            \circ (\id \cdot \m T)
            \circ \tm T
            \circ \hat \alpha 
            \circ (\id \cdot (\sigma^T \cdot \sigma^T)) \\
    &= T\hat \alpha^{-1}
            \circ T\tn{\phi^*}_a
            \circ T\hat \alpha 
            \circ \m T 
            \circ (\m T \cdot \m T)
            \circ (\id \cdot (\sigma^T \cdot \sigma^T)) 
  \end{align*}

  The second follows by applying the mate correspondence twice: we have
  \begin{align*}
    \sigma^T \circ \m{(T\phi)^*}_x
             \circ (\eta \cdot (\eta \circ 1))
             \circ \rho^{-1}
    &= \sigma^T \circ 1^\lor \circ \eta \\
    &= \sigma^T \circ \eta \circ 1 \\
    &= T\eta \circ \e T \circ 1 \\
    &= T(\eta \circ 1) \circ \e T,
  \end{align*}
  and
  \begin{align*}
    T\m{\phi^*}_x \circ \m T
                  \circ (\sigma^T \cdot \sigma^T)
                  \circ (\eta \cdot (\eta \circ 1))
                  \circ \rho^{-1}
    &= T\m{\phi^*}_x \circ \m T
                     \circ (T\eta \cdot T(\eta \circ 1))
                     \circ (\e T \cdot \e T)
                     \circ \rho^{-1} \\
    &= T\m{\phi^*}_x \circ T(\eta \cdot (\eta \circ 1))
                     \circ \m T
                     \circ (\id \cdot \e T)
                     \circ \rho^{-1} 
                     \circ \e T \\
    &= T(1^\lor \circ \eta) \circ \e T \\
    &= T(\eta \circ 1) \circ \e T.
  \end{align*}
  We obtain \eqref{eq:assoc.theta.2} via the following calculation:
  \begin{align*}
    T(\mu^F \cdot \m{\phi^*}_x)
      &\circ TN\phi^*_a
       \circ T(\id \cdot \theta^T)
       \circ \m T
       \circ (\m T \cdot \m{TT})
       \circ ((\id \cdot \sigma^T) \cdot (\id \cdot \sigma^{TT}))
       \circ (N^{(T\phi)^*}_a)^{-1} \\
      &= T(\mu^F \cdot \m{\phi^*}_x)
           \circ \m T
           \circ (\m T \cdot \m T)
           \circ (\id \cdot (\sigma^T \cdot \sigma^T)) \\
      &= \m T \circ (\mu^{TF} \cdot \m{\phi^*}_x)
              \circ (\id \cdot \m T)
              \circ (\id \cdot (\sigma^T \cdot \sigma^T)) \\
      &= \m T \circ (\id \cdot \sigma^T)
              \circ (\mu^{TF} \cdot \m{(T\phi)^*}_x).
  \end{align*}

  Now, we apply \eqref{eq:assoc.theta.1}, \eqref{eq:assoc.modif},
  \eqref{eq:assoc.f}, \ref{enum:assoc.hmonad} from Lemma
  \ref{lem:h.oplax.monad}, \eqref{eq:assoc.coh}, \eqref{eq:assoc.nat},
  \eqref{eq:assoc.theta.2} in sucession, to obtain
  \begin{align*}
    (\mu^F &\cdot \m{\phi^*}_x)
      \circ N^{\phi^*}_a
      \circ (\id \cdot \theta^T)
      \circ ((\mu^F \cdot \m{\phi^*}_x) \cdot m_{Fa \cdot \phi^*})
      \circ (N^{\phi^*}_a \cdot \id)
      \circ ((\id \cdot \theta^T) \cdot \id) \\
    &= (\mu^F \cdot \m{\phi^*}_x)
        \circ N^{\phi^*}_a
        \circ ((\mu^F \cdot \m{\phi^*}_x) 
                  \cdot (m_{Fa} \cdot m^{\lor}_{\phi^*}))
        \circ (N^{\phi^*}_a \cdot \id)
        \circ ((\id \cdot \theta^T) \cdot \theta^{TT}) \\
    &= (\mu^F \cdot \m{\phi^*}_x)
        \circ ((\mu^F \cdot Fm_a) 
                  \cdot (\m{\phi^*}_{Sx} \cdot m^{\lor}_{\phi^*}))
        \circ N^{\phi^*_S \cdot (T\phi)^*}_a
        \circ (N^{\phi^*}_a \cdot \id)
        \circ ((\id \cdot \theta^T) \cdot \theta^{TT}) \\
    &= (\mu^F \cdot \m{\phi^*}_x)
        \circ ((\id \cdot \mu^{FS}) \cdot (\phi^*_m \cdot m^{\lor}_{\phi^*}))
        \circ (\alpha \cdot \alpha)
        \circ N^{\phi^*_S \cdot (T\phi)^*}_a
        \circ (N^{\phi^*}_a \cdot \id)
        \circ ((\id \cdot \theta^T) \cdot \theta^{TT}) \\
    &= (\mu^F \cdot \m{\phi^*}_x)
        \circ ((\id \cdot FS\mu) \cdot (\phi^*_m \cdot m^{\lor}_{\phi^*}))
        \circ N^{\phi^*}_{a \cdot Sa}
        \circ (\id \cdot (\m{TF} \cdot \id))
        \circ (\id \cdot N^{(T\phi)^*}_a)
        \circ (\id \cdot (\theta^T \cdot \theta^{TT})) 
        \circ \alpha \\
    &= (\mu^F \cdot \m{\phi^*}_x)
        \circ N^{\phi^*}_a
        \circ (\id \cdot (\mu^{TF} \cdot \m{(T\phi)^*}_x))
        \circ (\id \cdot N^{(T\phi)^*}_a)
        \circ (\id \cdot (\theta^T \cdot \theta^{TT})) 
        \circ \alpha \\
    &= (\mu^F \cdot \m{\phi^*}_x)
        \circ N^{\phi^*}_a
        \circ (\id \cdot \theta^T)
        \circ (\id \cdot T(\mu^F \cdot \m{\phi^*}_x))
        \circ (\id \cdot TN^{\phi^*}_a)
        \circ (\id \cdot T(\id \cdot \theta^T)) 
        \circ (\id \cdot \m T)
        \circ \alpha
  \end{align*}
  which confirms the associativity law.

  We conclude the proof by confirming that \( (Ff,F\zeta \cdot \phi^*_f) \) is
  a horizontal lax \(T\)-algebra morphism, for any given horizontal lax
  \(S\)-algebra morphism \( (f,\zeta) \): indeed, note that the following
  diagrams commute
  \begin{equation*}
    \begin{tikzcd}
      1_{Fx} \ar[r,"\lambda"] 
        \ar[d,"1_{Ff}",swap]
      & 1_{Fx} \cdot 1_{Fx} \ar[rr,"\upsilon^F \cdot \e{\phi^*}_x"]
        \ar[d,"1_{Ff} \cdot 1_{Ff}"]
      && Fa \cdot \phi^*_x \ar[d,"F\zeta \cdot \phi^*_f"] \\
      1_{Fy} \ar[r,"\lambda",swap] 
      & 1_{Fy} \cdot 1_{Fy} \ar[rr,"\upsilon^F \cdot \e{\phi^*}_y",swap]
      && Fb \cdot \phi^*_y 
    \end{tikzcd}
  \end{equation*}
  \begin{equation*}
    \begin{tikzcd}[column sep=large]
      (Fa \cdot \phi^*_x) \cdot T(Fa \cdot \phi^*_x)
        \ar[rrr,"(F\zeta \cdot \phi^*_f) \cdot T(F\zeta \cdot \phi^*_f)"]
        \ar[d,"\id \cdot \theta^T",swap]
      &&& (Fb \cdot \phi^*_y) \cdot T(Fb \cdot \phi^*_y) 
        \ar[d,"\id \cdot \theta^T"] \\
      (Fa \cdot \phi^*_x) \cdot (TFa \cdot (T\phi)^*_x)
        \ar[rrr,"(F\zeta \cdot \phi^*_f) 
                  \cdot (TF\zeta \cdot (T\phi)^*_f)" description]
        \ar[d,"N^{\phi^*}_a",swap]
      &&& (Fb \cdot \phi^*_y) \cdot (TFb \cdot (T\phi)^*_y) 
        \ar[d,"N^{\phi^*}_b"] \\
      (Fa \cdot FSa) \cdot (\phi^*_{Sx} \cdot (T\phi)^*_x)
        \ar[rrr,"(F\zeta \cdot FS\zeta) 
                  \cdot (\phi^*_{Sf} \cdot (T\phi)^*_f)" description]
        \ar[d,"\mu^F \cdot \m{\phi^*}_x",swap]
      &&& (Fb \cdot FSb) \cdot (\phi^*_{Sy} \cdot (T\phi)^*_y) 
        \ar[d,"\mu^F \cdot \m{\phi^*}_y"] \\
      Fa \cdot \phi^*_x \ar[rrr,"F\zeta \cdot \phi^*_f",swap] 
      &&& Fb \cdot \phi^*_y
    \end{tikzcd}
  \end{equation*}
  via pasting of naturality and modification squares.

  Functoriality is confirmed componentwise.
\end{proof}

We close this section with a comparative analysis of
Theorem~\ref{thm:base.change} with the notions of change-of-base for internal
\(T\)-categories described in \cite{Lei04}, and with the notions of
change-of-base for enriched \(T\)-categories described in \cite[Sections 5 and
6]{CT03}; we confirm all of these generalize to our setting. The description
of our main object of study, the functor \( \TtVCat \to \CatTV \) induced by
\( - \pt \trm \colon \VMat \to \SpanV \), must be postponed to Section
\ref{sect:disc.morph}. 

\subsection{Internal $T$-categories:} 
\label{subsect:int.t.mcats}

Let \( \cat D, \cat E \) be categories with pullbacks, with respective
cartesian monads \(S\), \(T\) on \( \cat D \), \( \cat E \). We consider the
equipments \( \bicat D = \Span(\cat D) \) and \( \bicat E = \Span(\cat E) \),
and, abusing notation, we denote by \( S\) and \( T \) the induced strong
monads on \( \bicat D \) and \( \bicat E \). 

Using the terminology of \cite{Lei04}, we consider a cartesian monad oplax
morphism \( (P,\phi) \colon S \to T \) and a cartesian monad lax morphism \(
(Q,\psi) \colon T \to S \). The underlying data is given by
\begin{itemize}[label=--]
  \item
    pullback-preserving functors \( P, Q \colon \cat D \to \cat E \),
  \item
    a cartesian natural transformation \( \phi \colon PS \to TP \),
  \item
    a natural transformation \( \psi \colon SQ \to QT \).
\end{itemize}

We note \(P\) and \(Q\) induce strong functors \( \hat P \colon \bicat D \to
\bicat E \), \(\hat Q \colon \bicat E \to \bicat D \), and \( \phi \), \( \psi
\) induce vertical transformations \( \hat \phi \), \( \hat \psi \), which, in
turn, define a monad oplax morphism \( (\hat P,\hat \phi) \) and a monad lax
morphism \( (\hat Q, \hat \psi) \).

We conclude, by Theorem \ref{thm:base.change} that \( (\hat Q, \hat \psi) \)
defines a functor \( \hat Q_! \colon \CatTV \to \CatSV \). Moreover, since \(
\phi \) is cartesian, \( \hat \phi \) has a strong conjoint, and since \( \hat
P \) is strong, \( \hat P \hat \phi \) also has a strong conjoint; therefore
we may also conclude that \( (\hat P, \hat \phi) \) induces a functor \( \hat
P_! \colon \CatSV \to \CatTV \).

In fact, this notion of change-of-base can be plainly extended to include
Burroni's \(T\)-categories \cite{Bur71}. This would require a notion of
horizontal lax \(T\)-algebra for \(T\) an \textit{oplax} monad (which is
possible, merely requiring a couple of adjustments), and a replacement of lax
functors with oplax functors in Theorem \ref{thm:base.change}. We leave a
pursuit of these results and possible applications for future work.

\subsection{Enriched $T$-categories:}

Two instances of change-of-base are constructed in \cite{CT03}; we begin by
fixing a distributive monoidal category \( \cat V \), and let \( \bicat D =
\VMat \). Therein, a \textit{lax extension} of a \(\Set\)-monad \(T\) to \(
\VMat \) is a normal lax monad on \( \bicat D \) with underlying
\(\Set\)-monad \(T\).

First, we suppose we have two normal lax monads \( S \) and \( T \) on \(
\bicat D \), and let \( \phi \colon T \to S \) be a vertical transformation,
so that \( (\id,\phi) \colon S \to T \) defines a monad lax morphism. This is
precisely the data described in \cite[Section 5]{CT03}, restated in a double
categorical language. Theorem \ref{thm:base.change} produces a functor \(
\SVCat \to \TVCat \), which coincides with the \textit{algebraic functor}
construction in the aforementioned work.

Now, let \( \cat W \) be another distributive monoidal category, let \( \bicat
E = \WMat \), and let \( F \colon \cat V \to \cat W \) be a normal lax monoidal
functor, preserving the initial object. \(F\) induces a normal lax functor \(
\hat F \colon \bicat D \to \bicat E \) with \( \hat F_0 = \id_{\Set} \).

We let \(T\) and \(S\) be a lax monads on \( \bicat D \) and on \( \bicat E
\), respectively,  with the same underlying monad on \( \Set \); in other
words, \( S \) and \( T \) are \textit{lax extensions} of the same \( \Set
\)-monad.

Given a vertical transformation \( \phi \colon T\hat F \to \hat FS \), such
that \( \phi_0 \) is the identity and such that \( (\hat F, \phi) \colon S \to
T \) is a monad lax morphism, we may apply Theorem \ref{thm:base.change}
to produce a functor \( \hat F_! \colon \SVCat \to \TWCat \); this is
precisely the functor constructed in \cite[Section 6]{CT03}.

We should highlight that all normality conditions can be omitted, as well as
the preservation of the initial object by \(F\), and still obtain
change-of-base functors. This normality-free setting for both instances of
change-of-base was already studied in \cite[Sections 3.4, 3.5]{HST14}, for
thin categories \( \cat V \).

  \section{Induced adjunction}
    \label{sect:induced.adjunction}
    As observed in \cite[Section 6.7]{Lei04}, a suitable notion of adjunction
between cartesian monads will induce an adjunction on the categories of
internal \(T\)-categories, which has proven fruitful in their study.
Moreover, in \cite[Section 3]{HST14}, several change-of-base adjunctions
between categories of (monad, quantale)-categories are studied as well. Our
aim is to extend these ideas to arbitrary horizontal lax algebras, aiming to
compare the enriched and the internal notions of generalized multicategory.

Throughout this section, our setting is a conjunction
\begin{equation*}
  \label{eq:base.conjunction}
  \begin{tikzcd}
    (S,\bicat D)
      \ar[r,bend left,"{(F,\,\phi)}"{name=A}]
    & (T, \bicat E)
      \ar[l,bend left,"{(G,\,\psi)}"{name=B,below}]
    \ar[from=A,to=B,phantom,"\adj" {anchor=center, rotate=-90}]
  \end{tikzcd}
\end{equation*}
in the double category \( \Mnd(\PsDbCat_\lax) \), such that \( \bicat D \) and
\( \bicat E \) are equipments, and \( \phi \), \( T\phi \) have strong
conjoints. We denote the unit and counit by \( \heta, \heps \), respectively.

We recall that
\begin{itemize}[label=--]
  \item
    \( (F,\phi) \colon (S,\bicat D) \to (T,\bicat E) \) is a monad oplax
    morphism,
  \item
    \( (G,\psi) \colon (T,\bicat E) \to (S,\bicat D) \) is a monad lax
    morphism,
  \item
    we have an adjunction \( F \adj G \) in \( \PsDbCat_\lax \) with unit \(
    \heta \) and counit \( \heps \),
  \item
    and by doctrinal adjunction, \(F\) is strong, and \(\phi\), \( \psi \)
    are mates,
\end{itemize}
so that by Theorem \ref{thm:base.change}, we may construct functors 
\begin{equation*}
  F_! \colon \LaxSHAlg \to \LaxTHAlg \quad\text{and}\quad
  G_! \colon \LaxTHAlg \to \LaxSHAlg
\end{equation*}
induced by \( (F,\phi^*) \) and \( (G,\psi_!) \), respectively.

\begin{theorem}
  \label{thm:lax.alg.adj}
  We have an adjunction \( F_! \adj G_! \).
\end{theorem}

\begin{proof}
  To fix notation, for \( f \colon x \to Gy \) and \( g \colon Fx \to y \), we
  let \( f^\sharp = \heps_y \circ Ff \) and \( g^\flat = Gg \circ \heta_x
  \). This is similarly defined for 2-cells.

  We claim that the hom-isomorphism and its inverse are given by
  \begin{equation}
    \label{hom.iso.formulae}
    (f,\zeta) \mapsto (f^\sharp, \zeta^{\lor \sharp \land}),
    \quad \text{and} \quad 
    (g,\xi) \mapsto (g^\flat, \xi^{\lor \flat \land})
  \end{equation}
  where we use \( (-)^\lor \) and \( (-)^\land \) as short-hand for mates. To
  be explicit, these are respectively given by 
  \begin{align*}
    \zeta^\lor = \rho \circ (\id \cdot \delta) \circ \zeta 
    \quad&\text{and}\quad 
    \xi^\lor = \xi \circ (\id \cdot \eta) \circ \rho^{-1},\\
    \theta^\land = (\theta \cdot (\nu \circ 1)) \circ \rho^{-1}
    \quad&\text{and}\quad
    \chi^\land = \rho \circ (\chi \cdot (1 \circ \epsilon)) 
  \end{align*}
  for suitable 2-cells \( \zeta, \theta \) in \( \bicat D \) and 2-cells \(
  \xi, \chi \) in \( \bicat E \).

  To make sure that the horizontal composition for these mates of 2-cells is
  defined, note that
  \begin{equation*}
    \heps_{Ty} \circ F\psi_y \circ FSf 
      = T\heps_y \circ \phi_{Gy} \circ FSf
      = T\heps_y \circ TFf \circ \phi_x
      = Tf^\sharp \circ \phi_x,
  \end{equation*}
  \begin{equation*}
    GTg \circ G\phi_x \circ \heta_{Sx}
      = GTg \circ \psi_{Gx} \circ S\heta_x
      = \psi_y \circ SFg \circ S\heta_x
      = \psi_y \circ Sg^\flat.
  \end{equation*}

  Since \( \lor \), \( \land \) and \( \flat \), \( \sharp \) are pairs of
  inverse operators, we find that the maps in \eqref{hom.iso.formulae} are
  each others' inverses.

  To check these maps are natural, let \( (k,\omega) \colon (w,d,\upsilon,\mu)
  \to (x,a,\upsilon,\mu) \) be a horizontal lax \(S\) algebra morphism, and \(
  (h, \chi) \colon (y,b,\upsilon,\mu) \to (z,c,\upsilon,\mu) \) a horizontal
  lax \(T\)-algebra morphism. We have
  \begin{equation*}
    \zeta^{\lor \sharp \land} \circ (F\omega \cdot \phi^*_k)
                              \circ (\id \cdot \eta) \circ \rho^{-1}
      = \zeta^{\lor \sharp} \circ F\omega 
      = (\zeta \circ \omega)^{\lor\sharp},
  \end{equation*}
  \begin{equation*}
    ((G\chi \cdot \psi_{!h}) \circ \zeta)^\lor
      = G\chi \circ \zeta^\lor,
    \quad
    (G\chi \circ \zeta^\lor)^\sharp
    = \heps_c \circ FG\chi \circ F\zeta^\lor
    = \chi \circ \zeta^{\lor\sharp}
  \end{equation*}
  and reach our desired conclusion via mate correspondence. 

  So we're left with proving that \( (f^\sharp, \zeta^{\lor \sharp \land}) \)
  and \( (g^\flat, \xi^{\lor \flat \land}) \) are horizontal lax algebra
  morphisms. Since the proofs are similar, we omit the second one.

  We must check the following identities hold:
  \begin{equation*}
    \zeta^{\lor \sharp \land} \circ F_! \upsilon 
      = \upsilon \circ 1_{f^\sharp}
    \quad\text{and}\quad
    \zeta^{\lor \sharp \land} \circ F_! \mu 
      = \mu \circ (\zeta^{\lor \sharp \land} \cdot T\zeta^{\lor \sharp \land})
  \end{equation*}

  Since \( (f,\zeta) \) is a lax \(S\)-algebra morphism, we have
  \begin{equation*}
    \zeta \circ \upsilon
      = G_!\upsilon \circ 1_f,
  \end{equation*}
  so that
  \begin{align*}
    \zeta^{\lor\sharp\land} \circ F_!\upsilon
      &= \zeta^{\lor\sharp} \circ F\upsilon \circ \e F \\
      &= (\zeta \circ \upsilon)^{\lor\sharp} \circ \e F \\
      &= \heps_b \circ FG\upsilon \circ F\e G \circ F1_f \circ \e F \\
      &= \upsilon \circ \heps_1 \circ \e{FG} \circ 1_{Ff} \\
      &= \upsilon \circ 1_{f^\sharp},
  \end{align*}
  which gives the unit law for \( (f^\sharp, \zeta^{\lor \sharp \land}) \).

  For the multiplication law, we first let 
  \begin{equation*}
    Y = (\id \cdot \m T)
            \circ (\id \cdot (\id \cdot \sigma^T))
            \circ \hat \alpha^{-1}
            \circ \tn{\phi^*}_a
            \circ \hat \alpha,
    \quad
    Z = (\id \cdot (\eta \cdot (\eta \circ 1))) 
          \circ (\id \cdot \rho^{-1})
          \circ \rho^{-1}, \\
  \end{equation*}
  and our goal is to confirm that
  \begin{equation*}
    \zeta^{\lor\sharp\land} \circ (\mu^F \cdot \m{\phi^*}_x)
    = \mu \circ (\zeta^{\lor\sharp\land} 
                  \cdot T\zeta^{\lor\sharp\land})
          \circ Y
  \end{equation*}
  holds, via mate correspondence.

  First, we note that
  \begin{align*}
    Y \circ Z
    &= (\id \cdot \m T)
          \circ (\id \cdot (\id \cdot \sigma^T))
          \circ ((\id \cdot \eta) \cdot (\phi_a \cdot (\eta \circ 1)))
          \circ \hat \alpha^{-1}
          \circ (\id \cdot (\gamma^{-1} \cdot \id))
          \circ \hat \alpha
          \circ (\id \cdot \rho^{-1}) 
          \circ \rho^{-1} \\
    &= (\id \cdot \m T)
          \circ (\id \cdot (\id \cdot \sigma^T))
          \circ ((\id \cdot \eta) \cdot (\phi_a \cdot (\eta \circ 1)))
          \circ (\rho^{-1} \cdot \rho^{-1}) \\
    &= (\id \cdot \m T)
          \circ (\id \cdot (\id \cdot \sigma^T))
          \circ (\id \cdot (\id \cdot \eta))
          \circ (\id \cdot \rho^{-1})
          \circ ((\id \cdot \eta) \cdot \phi_a)
          \circ (\rho^{-1} \cdot \id) \\
    &= (\id \cdot \m T)
          \circ (\id \cdot (\id \cdot T\eta))
          \circ (\id \cdot (\id \cdot \e T))
          \circ (\id \cdot \rho^{-1})
          \circ ((\id \cdot \eta) \cdot \phi_a)
          \circ (\rho^{-1} \cdot \id) \\
    &= (\id \cdot T(\id \cdot \eta))
          \circ (\id \cdot \m T)
          \circ (\id \cdot (\id \cdot \e T))
          \circ (\id \cdot \rho^{-1})
          \circ ((\id \cdot \eta) \cdot \phi_a)
          \circ (\rho^{-1} \cdot \id) \\
    &= (\id \cdot T(\id \cdot \eta))
          \circ (\id \cdot T\rho^{-1})
          \circ ((\id \cdot \eta) \cdot \phi_a)
          \circ (\rho^{-1} \cdot \id), 
  \end{align*}
  from which we deduce that
  \begin{align*}
    \mu &\circ (\zeta^{\lor\sharp\land} \cdot T\zeta^{\lor\sharp\land})
        \circ Y \circ Z \\
    &= \mu \circ (\zeta^{\lor\sharp\land} \cdot T\zeta^{\lor\sharp\land})
        \circ (\id \cdot T(\id \cdot \eta))
        \circ (\id \cdot T\rho^{-1})
        \circ ((\id \cdot \eta) \cdot \phi_a))
        \circ (\rho^{-1} \cdot \id) \\
    &= \mu \circ \big(\zeta^{\lor\sharp} \cdot 
                    (T\zeta^{\lor\sharp} \circ \phi_a)\big) \\
    &= \mu \circ \big((\epsilon_b \circ F\zeta^\lor) \cdot 
                      (T\epsilon_b \circ TF\zeta^\lor \circ \phi_a)\big) \\
    &= \mu \circ \big((\epsilon_b \circ F\zeta^\lor) \cdot 
                      (T\epsilon_b \circ \phi_{Gb} \circ FS\zeta^\lor)\big) \\
    &= \mu \circ \big((\epsilon_b \circ F\zeta^\lor) \cdot 
                      (\epsilon_{Tb} \circ F\psi_b \circ FS\zeta^\lor)\big) \\
    &= \mu \circ (\epsilon_b \cdot \epsilon_{Tb})
           \circ \big( F\zeta^\lor \cdot 
                       (F\psi_b \circ FS\zeta^\lor)\big) \\
    &= \epsilon_b \circ F\mu^G \circ \m F
                  \circ \big( F\zeta^\lor \cdot 
                             (F\psi_b \circ FS\zeta^\lor)\big) \\
    &= \Big(\mu^G \circ \big(\zeta^\lor \cdot 
                              (\psi_b \circ S\zeta^\lor)\big)\Big)^\sharp
                  \circ \m F,
  \end{align*}
  and since we also have
  \begin{align*}
    \zeta^{\lor\sharp\land} 
      \circ (\mu^F \cdot \m{\phi^*}_x)
      \circ Z
    &= \zeta^{\lor\sharp\land}
        \circ (\mu^F \cdot (\eta \circ 1)) \circ \rho^{-1} \\
    &= \zeta^{\lor\sharp} \circ \mu^F \\
    &= (\zeta \circ \mu)^{\lor\sharp} \circ \m F,
  \end{align*}
  we conclude it is sufficient to show that 
  \begin{equation*}
    (\zeta \circ \mu)^\lor
      = \mu^G \circ \big(\zeta^\lor \cdot (\psi_b \circ S\zeta^\lor)).
  \end{equation*}
  Indeed, we have
  \begin{align*}
    (\zeta \circ \mu)^\lor
      &= \rho \circ (\id \cdot \delta) 
              \circ G_!\mu
              \circ (\zeta \cdot S\zeta) \\
      &= \rho \circ (\id \cdot \delta)
              \circ (\mu^G \cdot \m{\psi_!}_b)
              \circ \hat \alpha^{-1}
              \circ \tn{\psi_!}_b
              \circ \hat \alpha
              \circ (\id \cdot \chi^S)
              \circ (\zeta \cdot S\zeta) \\
      &= \mu^G \circ \rho
               \circ (\id \cdot \rho)
               \circ \hat \alpha^{-1}
               \circ (\id \cdot (\gamma \cdot \id))
               \circ \hat \alpha
               \circ ((\id \cdot \delta) 
                        \cdot (\psi_b \cdot (1 \circ \delta)))
               \circ (\id \cdot \chi^S)
               \circ (\zeta \cdot S\zeta) \\
      &= \mu^G \circ (\rho \cdot \rho)
               \circ ((\id \cdot \delta) 
                        \cdot (\psi_b \cdot (1 \circ \delta)))
               \circ (\id \cdot \chi^S)
               \circ (\zeta \cdot S\zeta) \\
      &= \mu^G \circ (\rho \cdot \id)
               \circ ((\id \cdot \delta) \cdot \psi_b)
               \circ (\id \cdot \rho)
               \circ (\id \cdot (\id \cdot \delta))
               \circ (\id \cdot \chi^S)
               \circ (\zeta \cdot S\zeta) \\
      &= \mu^G \circ (\rho \cdot \id)
               \circ ((\id \cdot \delta) \cdot \psi_b)
               \circ (\id \cdot S\rho)
               \circ (\id \cdot S(\id \cdot \delta))
               \circ (\zeta \cdot S\zeta) \\
      &= \mu^G \circ \big(\zeta^\lor \cdot (\psi_b \circ S\zeta^\lor)\big),
  \end{align*}
  which concludes the proof.
\end{proof}

For the purposes of applying this adjunction to the study of full faithfulness
of \( F_! \) and subsequent applications to descent theory, it is useful to
establish criteria for the invertibility of the unit and counit of the
adjunction \( F_! \adj G_! \), which are provided by the following results:

\begin{lemma}
  \label{lem:counit.iso.when}
  Let \( (y,b,\upsilon,\mu) \) be a horizontal lax \(T\)-algebra. If \(
  \heps_y \) and \( F\psi_y \) are invertible, then \(
  \heps_{(y,b,\upsilon,\mu)} \) is invertible if and only if \( \n{\heps^*}_b
  \) is invertible.
\end{lemma}

\begin{proof}
  We have \( \heps_{(y,b,\upsilon,\mu)} = (\heps_y, \id^{\lor\sharp\land} )
  \). We first observe that \( \id^{\lor\sharp\land} = \Omega \), where we
  have, with the coherence isomorphisms omitted,
  \begin{equation*}
    \Omega
      = \begin{tikzcd}[column sep=large]
          TFGy \ar[r,"\phi^*_{Gy}"]
               \ar[d,equal]
            & FSGy \ar[rr,"F(Gb \cdot \psi_{y!})"{name=B}]
                    \ar[d,equal]
            && FGy  \ar[d,equal] \\
          TFGy \ar[r,"\phi^*_{Gy}" description]
               \ar[d,"T\heps_y",swap]
            & FSGy \ar[r,"(F\psi_y)_!" description]
            & FGTy \ar[r,"FGb" description]
                   \ar[d,equal]
            & FGy  \ar[d,equal] \\
          Ty \ar[rr,"\heps^*_{Ty}" description,""{name=C}]
             \ar[d,equal]
            && FGTy \ar[r,"FGb" description]
            & FGy \ar[d,equal] \\
          Ty \ar[rr,"b" description,""{name=D}]
             \ar[d,equal]
            && y \ar[r,"\heps^*_y" description,""{name=E}]
                 \ar[d]
            & FGy \ar[d,"\heps_y"] \\
          Ty \ar[rr,"b",swap]
            && y \ar[r,"1"{name=F},swap]
            & y
          \ar[from=B,to=2-3,phantom,"\chi^F"]
          \ar[from=2-2,to=C,phantom,"\omega"]
          \ar[from=C,to=D,phantom,"\n{\heps^*}_b",near end,shift left=1.25cm]
          \ar[from=E,to=F,phantom,"\epsilon"]
        \end{tikzcd}
  \end{equation*}
  and \( \omega = \lambda \circ (\delta \cdot \iota) \) is the mate of \(
  \iota \), which is in turn the mate of \( \heps_{Ty} \circ F\psi_y =
  T\heps_y \circ \phi_{Gx} \).

  Indeed, we note that
  \begin{align*}
    \Omega^\lor 
      &= \lambda 
          \circ (\epsilon \cdot \id)
          \circ \n{\heps^*}_b
          \circ (\id \cdot \omega)
          \circ \alpha 
          \circ (\chi^F \cdot \id)
          \circ (\id \cdot \eta)
          \circ \rho^{-1} \\
      &= \rho \circ (\heps_b \cdot \epsilon)
              \circ (\id \cdot \omega)
              \circ (\id \cdot (\id \cdot \eta))
              \circ \alpha
              \circ (\chi^F \cdot \id)
              \circ \rho^{-1} \\
      &= \rho \circ (\heps_b \cdot \epsilon)
              \circ (\id \cdot \omega)
              \circ (\id \cdot (\id \cdot \eta))
              \circ (\id \cdot \rho^{-1})
              \circ \chi^F, 
  \end{align*}
  and since
  \begin{equation*}
    \epsilon \circ \omega \circ (\id \cdot \eta) \circ \rho^{-1}
    = 1 \circ \delta,
  \end{equation*}
  we obtain
  \begin{equation*}
    \Omega^\lor = \rho \circ (\heps_b \cdot 1)
                       \circ (\id \cdot \delta)
                       \circ \chi^F
                = \heps_b \circ \rho \circ (\id \cdot \delta)
                          \circ \chi^F,
  \end{equation*}
  and of course, \( \rho \circ (\id \cdot \delta) \circ \chi^F \) is
  precisely \( F(\rho \circ (\id \cdot \delta)) \), so we obtain \(
  \Omega^\lor = \id^{\lor\sharp} \), as desired.

  Our claim follows by noting that if \( \heps_y \) and \( F\psi_y \)
  are invertible, then so is \( \iota \), and since \( \delta \colon
  (F\psi_y)_! \to 1 \) is invertible, so is \( \omega \).

  The inverse of \( \iota \) is given by the mate of \( \phi_{Gy} \circ
  (F\psi_y)^{-1} = T\heps_y^{-1} \circ \epsilon_{Ty} \), which we denote by \(
  \theta \). We have
  \begin{equation*}
    \epsilon \circ \iota \circ \theta
      = 1_{T\epsilon_y} \circ \epsilon \circ \theta
      = 1_\id \circ \epsilon = \epsilon
    \quad\text{and}\quad
    \epsilon \circ \theta \circ \iota
      = 1_{T\epsilon_y^{-1}} \circ \epsilon \circ \iota
      = 1_\id \circ \epsilon = \epsilon,
  \end{equation*}
  finishing the proof.
\end{proof}

The analogous characterization for the unit is not quite the dual of Lemma
\ref{lem:counit.iso.when}; it requires one more verification.

\begin{lemma}
  \label{lem:unit.comp.iso.when}
  Let \( (x,a,\upsilon,\mu) \) be a horizontal lax \(S\)-algebra. If \( \eta_x
  \), \( G\phi_x \) and \( \n{(G\phi)^*}_a \) are invertible, then \(
  \heta_{(x,a,\upsilon,\mu)} \) is invertible if and only if \( \n{\eta_!}_a
  \) is invertible.
\end{lemma}

\begin{proof}
  The only missing detail is that, if \( \n{(G\phi)^*}_a \) is invertible,
  then so is
  \begin{equation*}
    \m G \circ (\id \cdot \sigma^G) 
      \colon GFa \cdot (G\phi)^*_x \to G(Fa \cdot \phi^*_x).
  \end{equation*}
  To see this, we take \eqref{eq:hphi.2.cell}, with \( H=G \) and \( r=a \),
  and we recall that \( \phi \) has a strong conjoint, by hypothesis.
\end{proof}

As an immediate corollary, we obtain:

\begin{corollary}
  \label{cor:unit.iso.when}
  \(F_! \colon \LaxSHAlg \to \LaxTHAlg \) is fully faithful whenever \( F
  \colon \bicat D \to \bicat E \) is fully faithful and \( G\phi \) is
  invertible.
\end{corollary}

\begin{proof}
  If \( F \colon \bicat D \to \bicat E \) is fully faithful, then \( \heta \)
  is invertible, and therefore has a strong companion. Likewise, \( G\phi \)
  has a strong conjoint. Thus, \( \eta_x \), \( \n{\eta_!}_a \) and
  \(\n{(G\phi)^*}_a \) are invertible for all \(x\) and all \(a\), so we
  apply Lemma \ref{lem:unit.comp.iso.when}.
\end{proof}

For the remainder of this section, we will compare Theorem
\ref{thm:lax.alg.adj} with \cite[Section 6.7]{Lei04} and \cite[Section
3]{HST14}, confirming we have a common generalization of these results.
Furthermore, we provide some comments comparing our approach with the
pseudofunctoriality ideas stated in \cite[{}4.4]{CS10}.

\subsection{Internal $T$-categories:}

We recall the setting described in Subsection \ref{subsect:int.t.mcats}. If \(
P \adj Q \) and \( \phi \) and \( \psi \) are mates, we can immediately apply
Theorem \ref{thm:lax.alg.adj} to obtain an adjunction \( P_! \adj Q_! \) as
claimed in \cite[Section 6.7]{Lei04}.

Likewise, with a suitable restatement of Theorem \ref{thm:lax.alg.adj} for
oplax monads and functors, we can also obtain adjunctions between categories
of Burroni's \(T\)-categories.

\subsection{Enriched $T$-categories:}
\label{subsect:tvcat.2}

We note that Theorem \ref{thm:lax.alg.adj} is a generalization of
\cite[Proposition 3.5.1]{HST14}, however, we cannot obtain the adjunction
studied in \cite[Subsection 3.4]{HST14}, using our result in the current form.

We will work out the same argument in our more general setting, to emphasize
what goes wrong. Given a monad \( T = (T,m,e) \) in \( \bicat E \), note that
\( e \colon \id \to T \) defines a monad lax morphism \( (\id,e) \colon T \to
\id \), which, by Theorem \ref{thm:base.change}, gives a functor
\begin{equation*}
  e_! \colon \LaxTHAlg \to \LaxidHAlg,
\end{equation*}
meaning every horizontal lax \(T\)-algebra has an underlying horizontal lax
\(\id\)-algebra (a monad!). Moreover, \( e \) also defines a monad oplax
morphism \( (\id, e) \colon \id \to T \), but unless \( e \) and \( Te \) have
strong conjoints, Theorem \ref{thm:base.change} cannot be applied to construct
a functor \( \LaxidHAlg \to \LaxTHAlg \), which would guarantee \( e_! \) has
a left adjoint.

However, it is possible to expand our notion of change-of-base to rectify this
issue: an analogous version of Theorem \ref{thm:base.change} can be obtained
for a monad oplax morphism \( (F,\phi) \colon \id \to T \), without requiring
either \( \phi \) or \(T\phi\) to be strong conjoints, by defining
\(F_!(x,a,\eta,\mu)\) so that \( F_!a = \phi^*_x \cdot TFa \); note that this
is precisely \( a_\sharp \) of \cite[Subsection 3.4]{HST14} when \(F=\id\),
and is isomorphic to the construction of Theorem \ref{thm:base.change} when \(
\phi \) and \( T\phi \) do have strong conjoints.

This would also require an analogous version of Theorem \ref{thm:lax.alg.adj}
for this specialized change-of-base construction, but since such results are
outside of our scope, we leave them for future work.

\subsection{Pseudofunctoriality:}

Theorems \ref{thm:base.change} and \ref{thm:lax.alg.adj} prompt one to view
\( \LaxHAlg \) as a double pseudofunctor \( \bicat M \to \CAT \) (see
\cite[Section 6]{Shu11}), for a suitable sub-double category \( \bicat M \) 
of \(\Mnd(\PsDbCat_\lax) \). Since double pseudofunctors preserve conjoints,
we would obtain the conclusion of Theorem \ref{thm:lax.alg.adj} as an
immediate corollary, for those conjunctions which are in \( \bicat M \).

We haven't pursued this line of reasoning, as obtaining a suitable choice of
\( \bicat M \) which includes our main examples has proved to be elusive, as
we briefly explain below.

We observe that the hypotheses required for Theorem \ref{thm:base.change}
restrict us to a setting where the vertical 1-cells \( (F,\phi) \colon S \to T
\) (monad oplax morphisms) of \( \bicat M \) are those such that \( \phi \)
and \(T\phi\) have strong conjoints. Unfortunately, this property on its own
doesn't determine a sub-double category, as it is not closed under vertical
composition: if \( (G,\psi) \colon T \to U \) is another vertical 1-cell,
there is no reason for \( \omega = \psi_F \circ G\phi \) nor \( U\omega \) to
have strong conjoints, so this property doesn't define a sub-double category.

This obstacle could be overcome, provided we can guarantee that \( G\phi \)
and \( UG\phi \) have strong conjoints. The first condition can be guaranteed
if we require that the underlying functor of every monad oplax morphism \(
(H,\chi) \) is such that
\begin{equation}
  \label{eq:cond.1}
  \begin{tikzcd}
    Hr \cdot (Hf)^*
      \ar[r,"\id \cdot \sigma^H"]
    & Hr \cdot H(f^*)
      \ar[r,"\m H"]
    & H(r \cdot f^*)
  \end{tikzcd}
\end{equation}
is invertible for all horizontal 1-cells \(r\) and vertical 1-cells \(f\);
note that this implies that \( H\phi \) has a strong conjoint whenever \( \phi
\) has a strong conjoint. This property is satisfied, for instance, when \(H\)
is strong. Therefore, this extra requirement is still within the setting of
Theorem \ref{thm:lax.alg.adj}, as the underlying functors of the left adjoints
are necessarily strong.

The problem lies in guaranteeing that \( UG\phi \) has a strong conjoint; we
would need to guarantee that the underlying lax functors of the monads make
\eqref{eq:cond.1} invertible. However, it can be shown that this does not hold
for our applications.

Lacking an alternative method to overcome this obstacle, we opted for the
current \textit{ad-hoc}, yet more general, approach for obtaining an
adjunction of change-of-base functors, instead of going for the more
attractive pseudofunctoriality argument.

  \section{Extensive categories}
    \label{sect:ext.cats}
    Extensivity of \( \cat V \) is a crucial hypothesis to construct and study the
comparison functor \eqref{eq:comp.functor} (see \cite{Luc18a} and \cite{CFP17})
\begin{equation}
  \label{eq:comp.functor}
  - \pt \trm \colon \VCat \to \CatV,
\end{equation}
and therefore we shall devote this section to the study of extensive
categories.

Let \( \cat C \) be a category with coproducts. We say \( \cat C \) is
\textit{extensive} if the functor
\begin{equation}
  \label{eq:fam.coprod.funct}
  \begin{tikzcd}
    \displaystyle\prod_{i \in I} \cat C \comma X_i 
      \ar[r,"\sum"]
      & \cat C \comma \displaystyle\sum_{i \in I} X_i
  \end{tikzcd}
\end{equation}
is an equivalence for all families \( (X_i)_{i\in I} \) of objects in \( \cat
C \). We refer to \cite{CLW93} for a comprehensive introduction to these
categories, and further material on the topic can be found in \cite{BJ01,
LV24, LPV24}. Extensive categories to keep in mind are \( \Set \), \( \Top \),
\( \Cat \), any Grothendieck topos such as \( \Grph \), and any free coproduct
completion \( \FamB \) of a category \( \cat B \).

The following characterization of extensivity in terms of Artin glueing
\cite[p. 465]{SGA-IV} is quite consequential: 

\begin{lemma}
  \label{lem:famc.comma}
  Let \( \cat C \) be a category with coproducts and a terminal object \( \trm
  \). Then Diagram \eqref{eq:comma.diag} 
  \begin{equation}
    \label{eq:comma.diag}
    \begin{tikzcd}
      \FamC \ar[r,"\sum"] \ar[d] 
        & \cat C \ar[d,equal] \ar[ld,Rightarrow,shorten=5mm,"\sigma",swap]\\
      \Set \ar[r,"-\pt \trm",swap] & \cat C
    \end{tikzcd}
  \end{equation}
  is a comma diagram \cite[Section 5]{Str76} if and only if \( \cat C \) is
  extensive, where \( \sigma_{(c_x)_{x \in X}} \colon \sum_{x\in X} c_x \to X
  \pt \trm \) is the coproduct over \(X\) of the morphisms \( c_x \to \trm \).
\end{lemma}

\begin{proof}
  If \( \cat C \) is extensive, then for a morphism \( f \colon c \to X \pt
  \trm \) we consider the family \( (c_x)_{x \in X} \) given by the following
  family of pullbacks:
  \begin{equation*}
    \begin{tikzcd}
      c_x \ar[r] \ar[d] 
        \ar[rd,"\ulcorner"{very near start},phantom]
      & \trm \ar[d,"\heta x"] \\
      c \ar[r,"f",swap] & X \pt \trm
    \end{tikzcd}
  \end{equation*}
  The family is, by definition, indexed over \(X\), and by extensivity, we
  have an isomorphism \( \sum_{x \in X} c_x \iso c \), whose composition with
  \(f\) equals \( \sigma_{(c_x)_{x\in X}} \).

  Let \( (c_x)_{x \in X} \), \( (d_y)_{y \in Y} \) be families of objects, and
  let \( \hat f \colon \sum_{x\in X} c_x \to \sum_{y\in Y} d_y \) be a
  morphism and let \( f \colon X \to Y \) be a function such that \( \sigma
  \circ \hat f = (f\pt \trm) \circ \sigma \). For each \( y \in Y \), we
  consider the following diagram:
  \begin{equation*}
    \begin{tikzcd}
      \sum_{x \in f^*y} c_x \ar[rd] \ar[rrr] \ar[ddd]
        &&& d_y \ar[ld] \ar[ddd] \\
        & f^*y \pt \trm \ar[r] \ar[d] 
          \ar[rd,"\ulcorner"{very near start},phantom]
        & \trm \ar[d,"\heta y"] \\
        & X \pt \trm \ar[r,"f \pt \trm",swap] & Y \pt \trm\\
      \sum_{x \in X} c_x \ar[rrr,"\hat f",swap] \ar[ur]
        &&& \sum_{y \in Y} d_y \ar[ul]
    \end{tikzcd}
  \end{equation*}
  The left, right and inner squares are pullbacks, hence the outer square
  must be a pullback; let \( \hat f|_x \colon c_x \to d_{fx} \) be the top
  morphism composed with the inclusion \( c_x \to \sum_{x \in f^*fx} c_x \),
  and consider the morphism \( (f,\hat f|_x) \colon (c_x)_{x \in X} \to
  (d_y)_{y \in Y} \) in \( \FamC \). It is the unique morphism \( \psi \colon
  (c_x)_{x \in X} \to (d_y)_{y \in Y} \) indexed by \(f\) such that \( \sum
  \psi = \hat f \), by extensivity. With this, we conclude that
  \eqref{eq:comma.diag} is a comma diagram.

  Now, given that \eqref{eq:comma.diag} is a comma diagram, we aim to confirm
  \eqref{eq:fam.coprod.funct} is an equivalence. First, full faithfulness:
  given a commutative triangle in \( \cat C \)
  \begin{equation*}
    \begin{tikzcd}
      \sum_{i \in I} Y_i \ar[rd,"\sum_i f_i",swap]
        \ar[rr,"\phi"]
      && \sum_{i \in I} Z_i \ar[ld,"\sum_i g_i"] \\ 
      & \sum_{i\in I} X_i
    \end{tikzcd}
  \end{equation*}
  we have
  \begin{equation*}
    \sigma_{(X_i)_{i \in I}} \circ \sum_{i\in I} f_i 
      = \sigma_{(Y_i)_{i \in I}} 
    \quad\text{and}\quad
    \sigma_{(X_i)_{i \in I}} \circ \sum_{i\in I} g_i = \sigma_{(Z_i)_{i \in I}},
  \end{equation*}
  from which we obtain \( \sigma_{(Y_i)_{i \in I}} = \sigma_{(Z_i)_{i \in I}}
  \circ \phi \). This implies the unique existence of a morphism \(
  (\id,\phi_i) \colon (Y_i)_{i\in I} \to (Z_i)_{i \in I} \) in \( \FamC \)
  such that \( \sum_i \phi_i = \phi \), by the 2-dimensional universal
  property of comma diagrams.

  Now, if we have a morphism \( \omega \colon S \to \sum_{i \in I} X_i \), we
  consider its composite with \( \sigma_{(X_i)_{i\in I}} \). This yields, by
  the 2-dimensional universal property, a family \( (S_i)_{i \in I} \) and an
  isomorphism \( \nu \colon \sum_{i \in I} S_i \iso S \). From full
  faithfulness above we obtain a family \( \omega_i \colon S_i \to X_i \) such
  that \( \sum_i \omega_i \circ \nu = \omega \).
\end{proof}

If an extensive category \( \cat C \) has all finite limits, we say it is
\textit{lextensive} \cite[{4.4}]{CLW93}. We have the following corollary of
Lemma \ref{lem:famc.comma}:

\begin{theorem}
  \label{thm:lext}
  If \( \cat C \) is a lextensive category, then \( \Fam(\cat C) \) is
  lextensive as well and \( \sum \colon \Fam(\cat C) \to \cat C \) preserves
  finite limits.
\end{theorem}
\begin{proof}
  We recall that \( \Fam(\cat C) \eqv \cat C \comma (- \pt \trm) \) is a
  \textit{weighted 2-limit} \cite{Str76}. By noting that the 2-category \(
  \Cat_{\mathsf{finlim}} \) of categories with finite limits and finite limit
  preserving functors has all weighted 2-limits, the comma diagram
  \eqref{eq:comma.diag} lives in \( \Cat_{\mathsf{finlim}} \), which proves
  our assertions. 
\end{proof}

Moreover, it should be noted that the converse was also shown to hold in
\cite[Section 4.3]{CV04}. See also \cite{LPV24} for a proof, and applications
of Theorem \ref{thm:lext} in the study of lextensive
categories\footnote{Theorem \ref{thm:lext} was instrumental in this extension
of the work of \cite{LV24}.}. In this work, the following instance of limit
preservation is extensively used:

\begin{theorem}
  \label{thm:pb.index.coprod}
  Let \( \cat C \) be an extensive category. If we have a commutative square
  in \( \FamC \)
  \begin{equation}
    \label{ap:given}
    \begin{tikzcd}
      (a_w)_{w\in W} \ar[r,"(f{,}\hat f)"]
                     \ar[d,"(g{,}\hat g)",swap] 
      & (b_x)_{x\in X} \ar[d,"(h{,}\hat h)"] \\
      (c_y)_{y \in Y} \ar[r,"(k{,}\hat k)",swap]
      & (d_z)_{z\in Z}
    \end{tikzcd}
  \end{equation}
  such that
  \begin{equation}
    \label{ap:index}
    \begin{tikzcd}
      W \ar[r,"f"] \ar[d,"g",swap]
        \ar[rd,"\ulcorner"{very near start},phantom]
      & X \ar[d,"h"] \\
      Y \ar[r,"k",swap] & Z
    \end{tikzcd}
  \end{equation}
  is a pullback diagram, as well as 
  \begin{equation}
    \label{ap:fibre}
    \begin{tikzcd}
      a_w \ar[r,"\hat f_w"] 
          \ar[d,"\hat g_w",swap]
          \ar[rd,"\ulcorner"{very near start},phantom]
      & b_{fw} \ar[d,"\hat h_{fw}"] \\
      c_{gw} \ar[r,"\hat k_{gw}",swap]
      & d_z
    \end{tikzcd}
  \end{equation}
  for each \( w \in W \), where \( z = kgw = hfw \). Then
  \begin{equation}
    \label{ap:coprod.pullback}
    \begin{tikzcd}[column sep=large,row sep=large]
      \displaystyle\sum_{w \in W} a_w 
        \ar[r,"\sum_f \hat f"] \ar[d,"\sum_g \hat g",swap]
        \ar[rd,"\ulcorner"{very near start},phantom]
      & \displaystyle\sum_{x \in X} b_x \ar[d,"\sum_h \hat h"] \\
      \displaystyle\sum_{y \in Y} c_y \ar[r,"\sum_k \hat k",swap]
      & \displaystyle\sum_{z \in Z} d_z
    \end{tikzcd}
  \end{equation}
  is a pullback diagram.
\end{theorem}

\begin{proof}
  The hypotheses \eqref{ap:index} and \eqref{ap:fibre} guarantee that
  \eqref{ap:given} is a pullback in \( \FamC \), by \cite[Section 4]{Gra66}
  (see also \cite[Definition 4.7 and Corollary 4.9]{Her99}).  Since \( \sum \)
  preserves limits, \eqref{ap:coprod.pullback} must be a pullback diagram, as
  desired.
\end{proof}

As corollaries, we obtain succinct proofs of a couple of results from
\cite{Luc18a} and \cite{CFP17}, which clarify the role of the extensivity
condition.

\begin{lemma}
  \label{lem:1.strong}
  If \( \cat V \) is extensive with finite limits, \( - \pt \trm \colon \VMat \to
  \SpanV \) is strong.
\end{lemma}

\begin{proof}
  Since \( - \pt \trm \) is normal, it is enough to prove that \( \m{-\pt
  \trm} \) is invertible. Indeed, we have the following pullback diagrams:
  \begin{equation*}
    \begin{tikzcd}
      X \times Y \times Z \ar[r] \ar[d] 
        \ar[rd,"\ulcorner"{very near start},phantom]
      & Y \times Z \ar[d] \\
      X \times Y \ar[r] & Y
    \end{tikzcd}
    \quad
    \begin{tikzcd}
      s(y,z) \times r(x,y) \ar[r] \ar[d] 
        \ar[rd,"\ulcorner"{very near start},phantom]
      & s(y,z) \ar[d] \\
      r(x,y) \ar[r] & \trm 
    \end{tikzcd}
  \end{equation*}
  thus, applying Theorem \ref{thm:pb.index.coprod}, we conclude that the outer
  square of diagram \eqref{eq:defn.m} is a pullback, verifying our claim.
\end{proof}

\begin{remark}
  \label{rem:pt.adj.hom}
  We observe that the above lemma can be restated in terms of a Beck-Chevalley
  condition; see \cite[Definition 1.4.13]{LucTh}. To wit, the lax functor \(
  \cat V(\trm,-) \colon \SpanV \to \VMat \) satisfies the Beck-Chevalley
  condition if \( \cat V \) is extensive. Then, by \cite[Theorem
  1.4.14]{LucTh}, we conclude that \( - \pt \trm \adj \cat V(\trm,-) \) is an
  \textit{adjunction} in the 2-category \( \PsDbCat_\lax \).
\end{remark}

We can also give a short proof that a considerable class of monads are
cartesian:

\begin{lemma}
  \label{lem:free.monoid.cartesian}
  Let \( \cat V \) be a lextensive, monoidal category, whose tensor product \(
  \otimes \) preserves coproducts and pullbacks. Then the free \( \otimes
  \)-monoid monad on \( \cat V \) is cartesian.
\end{lemma}

\begin{proof}
  We let \( X^0 = I \) be the unit object, and \( X^{n+1} = X^n \otimes X \).
  Recall that the underlying functor of the free \( \otimes \)-monoid monad
  may be given by \( X \mapsto X^* = \sum_{n \in N} X^n \) (see, for instance,
  \cite[Theorem 23.4]{Kel80}), and note that since pullbacks are preserved by
  \( \otimes \) (by hypothesis) and by coproducts (as a corollary of Lemma
  \ref{lem:famc.comma}), we conclude the free \(\otimes\)-monoid monad
  preserves pullbacks. Moreover, note that
  \begin{equation*}
    \begin{tikzcd}
      X^n \ar[d,"f^n",swap]
          \ar[r,"\iota_n"]
        \ar[rd,"\ulcorner"{very near start},phantom]
      & X^* \ar[d,"f^*"] \\
      Y^n \ar[r,"\iota_n",swap]
      & Y^*
    \end{tikzcd}
  \end{equation*}
  is a pullback diagram for all \(n \in \N \), due to extensivity. Taking
  \(n=1\) confirms \( \eta \) is a cartesian natural transformation.

  Now, we consider the following pullback diagrams
  \begin{equation*}
    \begin{tikzcd}
      \displaystyle\sum_{k \in \N} \N^k \ar[r,"\textsf{sum}"] 
                           \ar[d,equal]
        & \N \ar[d,equal] \\
      \displaystyle\sum_{k \in \N} \N^k \ar[r,"\textsf{sum}",swap] 
        & \N
    \end{tikzcd}
    \qquad
    \begin{tikzcd}
      X^{n_1} \otimes \ldots \otimes X^{n_k}
       \ar[r,"\iso"]
       \ar[d,"f^{n_1} \otimes \ldots \otimes f^{n_k}",swap]
        & X^{n_1 + \ldots + n_k} \ar[d,"f^{n_1 + \ldots + n_k}"] \\
      Y^{n_1} \otimes \ldots \otimes Y^{n_k}
       \ar[r,"\iso",swap]
      & Y^{n_1 + \ldots + n_k} 
    \end{tikzcd}
  \end{equation*}
  to which we may apply Theorem \ref{thm:pb.index.coprod}, allowing us to
  conclude that \( \mu \) is a cartesian natural transformation as well.
\end{proof}

\subsection*{Connected terminal objects}

We recall that an object \(x\) in a category \( \cat V \) with coproducts is
said to be \textit{connected} \cite[{Definition~6.1.3}]{BJ01} if the
hom-functor \( \cat V(x,-) \) preserves coproducts.

Under the hypothesis that \( \cat V \) is lextensive, understading this
condition turns out to be helpful in our work on the enriched \( \to \)
internal embedding, particularly regarding the study of certain monads on \(
\cat V \); see Lemma~\ref{lem:word.monad.fib.disc}.

\begin{lemma}
  \label{lem:conn.iff.fff}
  If \( \cat V \) is lextensive, then its terminal object \( \trm \) is
  connected if and only if \( - \pt \trm \colon \Set \to \cat V \) is fully
  faithful.
\end{lemma}

\begin{proof}
  Given a morphism \( p \colon \trm \to \sum_{i \in I} X_i \), we consider the
  following diagram
  \begin{equation*}
    \begin{tikzcd}
      u \ar[r] \ar[d] \ar[rd,"\ulcorner",phantom,very near start]
        & X_i \ar[d] \ar[r] \ar[rd,"\ulcorner",phantom,very near start]
        & \trm \ar[d,"i"] \\
      \trm \ar[r,"p",swap]
        & \displaystyle\sum_{i \in I} X_i \ar[r,"\sigma",swap]
        & I \pt \trm
    \end{tikzcd}
  \end{equation*}
  which is a pasting of pullback squares.

  It is clear that if \( - \pt \trm \) is fully faithful, then \( u \iso [\sigma
  \circ p = i] \pt \trm \). Thus, by universality of coproducts, \( p \) is
  uniquely determined by a morphism \( \trm \to X_i \). 

  Conversely, since \( \trm \) is the terminal object we have \( \cat
  V(\trm,X\pt \trm) \iso X \pt \cat V(\trm,\trm) \iso X \), which implies the
  unit of \( - \pt \trm \adj \cat V(\trm,-) \) must be an isomorphism.
\end{proof}

\begin{lemma}
  \label{lem:1.fff}
  If \( \cat V \) is a lextensive category and \( \trm \) is connected, then
  \( - \pt \trm \colon \VMat \to \SpanV \) is fully faithful on 2-cells.
\end{lemma}

\begin{proof}
  The outer square of \eqref{eq:defn.unit} is a pullback, due to extensivity.
  Then, since \( \heta \colon X \to \cat V(\trm,X\pt \trm) \) is an isomorphism for
  all \(X\), the result follows.
\end{proof}

  \section{Fibrewise discrete morphisms}
    \label{sect:disc.morph}
    Let \(T\) be a cartesian monad on a lextensive category \( \cat V \), with
terminal object denoted by \( \trm \). We also denote by \(T\) the induced
strong monad on \( \SpanV \) \cite{Her00, CS10}.  The \( \Set \)-monad \( \Tt
\) under study (as well as its lax extension to \( \VMat \), also denoted by
\( \Tt \)), is constructed via the following consequence of Proposition
\ref{prop:doct.adj}:

\begin{proposition}
  \label{prop:induced.monad}
  Let \( \bicat B \) be a 2-category, let \( (l,r,\eta,\epsilon) \colon b \to
  c \) be an adjunction in \(\bicat B\), and let \( (t,m,e) \) be a monad on
  \( c \). Then \( (rtl, r(m \circ t\epsilon t)l, rel \circ \eta) \) is a
  monad on \( b \), and we have a conjunction 
  \begin{equation*}
    \begin{tikzcd}
      (b,rtl) \ar[r,bend left, "{(l, \epsilon tl)}"{name=A}]
      & (c,t) \ar[l,bend left, "{(r, rt \epsilon)}"{name=B}]
      \ar[from=A,to=B,phantom,"\adj" {anchor=center, rotate=-90}]
    \end{tikzcd}
  \end{equation*}
  in \( \Mnd(\bicat B) \).
\end{proposition}

Indeed, by Remark \ref{rem:pt.adj.hom}, we have an adjunction
\begin{equation}
  \label{eq:step.one}
  \begin{tikzcd}
    \VMat \ar[r,bend left,"- \pt \trm"{name=A}]
    & \SpanV \ar[l,bend left,"\cat V(\trm{,}-)"{name=B,below}]
    \ar[from=A,to=B,phantom,"\adj" {anchor=center, rotate=-90}]
  \end{tikzcd}
\end{equation}
in the 2-category \( \PsDbCat_\lax \), thus, we may apply
Proposition~\ref{prop:induced.monad} to \eqref{eq:step.one} with the monad
\(T\) on \( \SpanV \) to obtain a monad \( \Tt = \cat V(\trm,T(- \pt \trm)) \)
on \( \VMat \), and a conjunction
\begin{equation}
  \label{eq:induced.conjunction}
  \begin{tikzcd}
    (\Tt, \VMat)
      \ar[r,bend left,"{(- \pt \trm,\, \heps_{T(-\pt \trm)})}"{name=A}]
    & (T, \SpanV)
      \ar[l,bend left,
          "{(\cat V(\trm,-),\,\cat V(\trm, T\heps))}"{name=B,below}]
    \ar[from=A,to=B,phantom,"\adj" {anchor=center, rotate=-90}]
  \end{tikzcd}
\end{equation}
where \( (-\pt \trm,\, \heps_{T(-\pt \trm)}) \) is a monad oplax morphism and
\( (\cat V(\trm,-),\, \cat V(\trm,\,T\heps)) \) is a monad lax morphism. 

The only remaining ingredient to place ourselves under the setting of Section
\ref{sect:induced.adjunction} and therefore apply Theorem
\ref{thm:lax.alg.adj} to \eqref{eq:induced.conjunction}, is the hypothesis
that \( \heps_{T(- \pt \trm)} \) has a strong conjoint\footnote{Note that,
since $T$ is strong, it follows that $ T\heps_{T(-\pt \trm)} $ has a strong
conjoint as well, by Lemma \ref{lem:right.conjoint.invert}.}. The study and
characterization of this hypothesis (when \( \trm \) is connected) is the
central purpose of this Section, culminating in Theorem
\ref{thm:heps.strong.conj}. 

Once this characterization is obtained, we get an adjunction (see Lemma
\ref{lem:induced.adjunction})
\begin{equation}
  \label{eq:induced.adjunction}
  \begin{tikzcd}
    \TtVCat
      \ar[r,bend left,"- \pt \trm"{name=A}]
    & \CatTV
      \ar[l,bend left,"{\cat V(\trm,-)}"{name=B,below}]
    \ar[from=A,to=B,phantom,"\adj" {anchor=center, rotate=-90}]
  \end{tikzcd}
\end{equation}
from \eqref{eq:induced.conjunction}, for pairs \( (T,\cat V) \) where \(T\) is
a cartesian monad on a lextensive category \( \cat V \) with \( \trm \)
connected, such that \( \heps_{T(-\pt \trm)} \) has a strong conjoint.

We begin by establishing the following:

\begin{lemma}
  \label{lem:nat.s.conj.eqv}
  Let \( \omega \colon F \to G \) be a vertical transformation of lax functors
  \( F,\, G \colon \bicat D \to \SpanV \). For a horizontal 1-cell \( a \colon
  s \relto t \) in \( \bicat D \), the 2-cell \( \n{\omega^*}_a \) is
  invertible if and only if
  \begin{equation}
    \label{eq:pb.diag.2-cell}
    \begin{tikzcd}
      M_{Fa} \ar[d,"\omega_a",swap] \ar[r,"r_{Fa}"]
        & Ft \ar[d,"\omega_y"] \\
      M_{Ga} \ar[r,"r_{Ga}",swap] & Gt
    \end{tikzcd}
  \end{equation}
  is a pullback diagram.
\end{lemma}

\begin{proof}
  We observe that \( \n{\omega^*}_a \) is uniquely determined by the dashed
  morphism below making the triangles commute
  \begin{equation*}
    \begin{tikzcd}
      && M_{Fa} \ar[ddl,bend right,swap,"\omega_a" description]
                \ar[ddr,bend left,"r_{Fa}" description] 
                \ar[d,dashed] \\
      && \cdot \ar[dd,phantom,"\ulcorner"{anchor=center,
                                          rotate=-45,yshift=-0.5mm},
                              very near start]
               \ar[dr] \ar[dl] \\
      & M_{Ga} \ar[dl,"l_{Ga}" description,swap]  
               \ar[dr,"r_{Ga}" description]
      && Ft \ar[dl,"\omega_y" description]
            \ar[dr,equal] \\
      Gs && Gt && Ft
    \end{tikzcd}
  \end{equation*}
  which is invertible if and only if \eqref{eq:pb.diag.2-cell} is a pullback
  diagram.
\end{proof}

\begin{lemma}
  \label{lem:strong.conj.iff.cart}
  The following are equivalent:
  \begin{enumerate}[label=(\alph*)]
    \item
      \label{enum:nat.s.conj}
      \( \heps_{T(-\pt \trm)} \) has a strong conjoint,
    \item
      \label{enum:nat.cart}
      \( \heps_{T(-\pt \trm)} \) is a cartesian natural transformation.
  \end{enumerate}
\end{lemma}

\begin{proof}
  Instanciating \eqref{eq:pb.diag.2-cell} with \( \omega = \heps_{T(- \pt
  \trm)} \), we find \ref{enum:nat.s.conj} holds if and only if
  \begin{equation}
    \label{eq:nat.s.conj.prop}
    \begin{tikzcd}
      \displaystyle\sum_{\mathfrak x \in \Tt s} 
      \displaystyle\sum_{\mathfrak y \in \Tt t}
        (\Tt a)(\mathfrak x, \mathfrak y)
          \ar[d,"\heps_{T(a\pt \trm)}",swap] 
          \ar[r]
          \ar[rd,"\ulcorner",phantom,very near start]
        & \Tt t \pt \trm \ar[d,"\heps_{T(t\pt \trm)}"] \\
      T \Big( \displaystyle\sum_{x\in s}
              \displaystyle\sum_{y\in t} a(x,y) \Big)
          \ar[r,"r_{T(a\pt \trm)}",swap] 
        & T(t \pt \trm)
    \end{tikzcd}
  \end{equation}
  is a pullback diagram for all \( \cat V \)-matrices \(a\).

  If \eqref{eq:nat.s.conj.prop} is a pullback diagram for all such horizontal
  1-cells, then when we have \( a = f_! \) for a function \( f \colon s \to t
  \), we obtain a pullback diagram
  \begin{equation*}
    \begin{tikzcd}
      \Tt s \pt \trm \ar[r,"\Tt f \pt \trm"]
                           \ar[d,"\heps_{T(s\pt \trm)}",swap]
        \ar[rd,"\ulcorner",phantom,very near start]
        & \Tt t \pt \trm \ar[d,"\heps_{T(t\pt \trm)}"] \\
      T(s \pt \trm) \ar[r,"T(f\pt \trm)",swap] 
        & T(t \pt \trm)
    \end{tikzcd}
  \end{equation*}
  thereby verifying \ref{enum:nat.s.conj} \( \to \) \ref{enum:nat.cart}.

  Now, we assume \ref{enum:nat.cart}. If the outer square of Diagram
  \eqref{eq:big.diag} below commutes:
  \begin{equation}
    \label{eq:big.diag}
    \begin{tikzcd}
      v \ar[ddd,"\omega",swap] 
        \ar[rrr,"r_v"] 
        \ar[rd,dashed,"l_v" description]
      &&& \Tt t \pt \trm 
        \ar[dl]
        \ar[ddd,"\heps_{T(t\pt \trm)}"] \\
      & \Tt s \pt \trm \ar[r] \ar[d,"\heps_{T(s\pt \trm)}",swap]
        \ar[rd,"\ulcorner",phantom,very near start]
      & \Tt \trm \pt \trm \ar[d] \\
      & T(s \pt \trm) \ar[r] 
      & T\trm \\
      T\Big( \displaystyle\sum_{x\in s}
             \displaystyle\sum_{y\in t} a(x,y) \Big)
        \ar[rrr,"r_{T(a\pt \trm)}",swap]
        \ar[ur,"l_{T(a\pt \trm)}" description]
      &&& T(t\pt \trm) \ar[ul]
    \end{tikzcd}
  \end{equation}
  then an immediate calculation shows the diagram commutes. Since the square
  in the middle is a pullback by assumption, there exists a unique morphism \(
  l_v \colon v \to \Tt s \pt \trm \) such that the left and top squares
  commute. 

  Thus, we conclude that the following diagram
  \begin{equation*}
    \begin{tikzcd}[column sep=large]
      v \ar[rr,"{\lb l_v,r_v \rb}"] \ar[d,"\omega",swap]
        && \Tt s\pt \trm \times \Tt t\pt \trm 
          \ar[d,"\heps_{T(s\pt \trm)} \times \heps_{T(t\pt \trm)}"] \\
      T\Big( 
        \displaystyle\sum_{x\in s} \displaystyle\sum_{y\in t} a(x,y) \Big)
        \ar[rr,"{\big\lb l_{T(a\pt \trm)},r_{T(a\pt \trm)}\big\rb}",swap] 
        && T(s\pt \trm) \times T(t\pt \trm) 
    \end{tikzcd}
  \end{equation*}
  commutes. We observe that Diagram~\eqref{eq:defn.counit} is a pullback
  square when \( \cat V \) is extensive, by Theorem~\ref{thm:pb.index.coprod},
  so there exists a unique morphism
  \begin{equation*}
    \begin{tikzcd}
      \omega^\sharp \colon v \ar[r]
      & \displaystyle\sum_{\mathfrak x \in \Tt s}
        \displaystyle\sum_{\mathfrak y \in \Tt t}
          (\Tt a)(\mathfrak x, \mathfrak y)
    \end{tikzcd}
  \end{equation*}
  such that \( \heps_{T(a\pt \trm)} \circ \omega^\sharp = \omega \), \( r_{\Tt
  a \pt \trm} \circ \omega^\sharp = r_v \) and \( l_{\Tt a \pt \trm} \circ
  \omega^\sharp = l_v \), which, in particular, confirms that
  Diagram~\eqref{eq:nat.s.conj.prop} is a pullback square.
\end{proof}

\subsection*{Fibrewise discrete monads}

The search for a more concrete notion of what it means for \( \heps_{T(-\pt
\trm)} \) to have a strong conjoint led us to the notion of fibrewise
discreteness. 

Let \( \cat V \) be a lextensive category, and let \( f \colon x \to y  \) be
a morphism in \( \cat V \). We say \(f\) is \textit{fibrewise discrete} if for
every pullback diagram 
\begin{equation*}
  \begin{tikzcd}  
    f^*p \ar[r] \ar[d] \ar[rd,"\ulcorner",phantom,very near start]
    & x \ar[d,"f"] \\
    \trm \ar[r,"p",swap] & y
  \end{tikzcd}
\end{equation*}
the object \( f^*p \) is \textit{discrete}; that is, \( \heps_{f^*p} \) is an
isomorphism. For instance, in \( \cat V = \Top \), local homeomorphisms are
fibrewise discrete.

We say an endofunctor \( F \) on \( \cat V \) is \textit{fibrewise discrete}
if for all sets \(X\), the morphism \( F! \colon F(X \pt \trm) \to F\trm \) is
fibrewise discrete. 

\begin{lemma}
  \label{lem:fib.disc.iff.cart.nat}
  Let \( F \) be an endofunctor on a lextensive category \( \cat V \). If \(
  \heps_{F(- \pt \trm)} \) is cartesian, then \(F\) is fibrewise discrete. The
  converse holds when the \( \trm \) is connected.
\end{lemma}

\begin{proof}
  Let \( \Ft = \cat V(\trm,F(-\pt \trm)) \). We consider
  Diagram~\eqref{eq:pb.factored}
  \begin{equation}
    \label{eq:pb.factored}
    \begin{tikzcd}
      \Ft X \pt \trm \ar[rrd,bend left=15,"\heps_{F(X\pt \trm)}"]
        \ar[ddr,bend right=25,"\Ft ! \pt \trm",swap] 
        \ar[rd,"\theta_X",dashed] \\
      & \tau_X \ar[d] \ar[r,"\omega_X"] 
               \ar[rd,"\ulcorner",phantom,very near start]
      & F(X \pt \trm) \ar[d,"F!"] \\
      & \Ft \trm \pt \trm \ar[r,"\heps_{F\trm}",swap] & F\trm
    \end{tikzcd}
  \end{equation}
  where the square in the lower right corner is a pullback. The outer square
  commutes by naturality, so there exists a unique \( \theta_X \), depicted by
  a dashed arrow, making both incident triangles commute. Note that \(
  \heps_{F(-\pt \trm)} \) is cartesian if and only if \( \theta_X \) is an
  isomorphism for all \(X\).

  Now, consider the image of \eqref{eq:pb.factored} via \( \cat V(\trm,-) \),
  which preserves pullbacks. Note that since \( \trm \) is connected, \( \cat
  V(\trm,\heps_{F\trm}) = \eta_{\cat V(\trm,F\trm)}^{-1} \) is an isomorphism,
  so we conclude that \( \cat V(\trm,\omega_x) \) is an isomorphism as well. 

  Hence, we consider the following naturality square
  \begin{equation*}
    \begin{tikzcd}[column sep=large]
      \cat V(\trm,\tau_X) \pt \trm 
        \ar[d,"\heps_{\tau_X}",swap] 
        \ar[r,"{\cat V(\trm,\omega_X) \pt \trm}"]
      & \Ft X \pt \trm \ar[d,"\heps_{F(X\pt \trm)}"] \\
      \tau_X \ar[r,"\omega_X",swap] & F(X\pt \trm)
    \end{tikzcd}
  \end{equation*}
  and we observe that \( \theta_X \circ \cat V(\trm,\omega_X) \pt \trm =
  \heps_{\tau_X} \) holds, by the universal property. Thus, \( \theta_X \) is
  invertible if and only if \( \heps_{\tau_x} \) is invertible; that is, if
  and only if \( \tau_x \) is discrete.
\end{proof} 

\begin{remark}
  \label{rem:tiny.tau}
  In Diagram \eqref{eq:pb.factored}, we have a morphism \( \tau_x \to \Ft \trm
  \pt \trm \), which corresponds to a family \( (\tau_{x,p})_{p \in \Ft \trm}
  \), by extensivity (see Theorem \ref{lem:famc.comma}); these are given via
  pullback 
  \begin{equation}
    \label{eq:pb.test}
    \begin{tikzcd}
      \tau_{x,p} \ar[r] \ar[d] 
        \ar[rd,"\ulcorner",phantom,very near start]
      & F(x\pt \trm) \ar[d,"F!"] \\
      \trm \ar[r,"p",swap] & F\trm
    \end{tikzcd}
  \end{equation}
  and we also have \( \sum_{p \in \Ft \trm} \tau_{x,p} \iso \tau_x \). Thus,
  \( \tau_x \) is discrete if and only if \( \tau_{x,p} \) is discrete for all
  \(p \in \Ft \trm \). 
\end{remark}

With this, we obtain the following characterization:

\begin{theorem}
  \label{thm:heps.strong.conj}
  If \( \trm \) is connected, the following are equivalent for an endofunctor
  \( F \colon \cat V \to \cat V \):
  \begin{enumerate}[label=(\roman*)]
    \item
      \label{enum:eps.strong.conj}
      \( \heps_{F(- \pt \trm)} \) has a strong conjoint.
    \item
      \label{enum:eps.cart.nat}
      \( \heps_{F(- \pt \trm)} \) is a cartesian natural transformation.
    \item
      \label{enum:t.fib.disc}
      \( F \) is fibrewise discrete.
    \item
      \label{enum:tiny.tau.disc}
      \( \tau_{x,p} \), as given in \eqref{eq:pb.test}, is discrete for all
      \(x\) and all \( p \in \Ft \trm \).
  \end{enumerate}
\end{theorem}

\begin{proof}
  The equivalence \ref{enum:eps.strong.conj} \( \iff \)
  \ref{enum:eps.cart.nat} is given by Lemma \ref{lem:strong.conj.iff.cart},
  we have \ref{enum:eps.cart.nat} \( \iff \) \ref{enum:t.fib.disc} by 
  Lemma \ref{lem:fib.disc.iff.cart.nat}, and Remark \ref{rem:tiny.tau}
  confirms \ref{enum:t.fib.disc} \( \iff \) \ref{enum:tiny.tau.disc}.
\end{proof}

Naturally, we are most concerned with cartesian monads \( T \) such that \( T
\) is fibrewise discrete, and, armed with Theorem \ref{thm:heps.strong.conj},
we can promptly verify that many familiar examples of cartesian monads are
fibrewise discrete. We begin with the following:

\begin{lemma}
  \label{lem:word.monad.fib.disc}
  Let \( \cat V \) be a distributive monoidal category such that \( \trm \) is
  connected. The free \( \otimes \)-monoid monad on \( \cat V \) is fibrewise
  discrete.
\end{lemma}

\begin{proof}
  Let \(X\) be a set, and let \( p \colon \trm \to (X \pt \trm)^* \) be a
  morphism.  Since \( \trm \) is connected, we may apply Lemma
  \ref{lem:conn.iff.fff} to confirm \(p\) factors uniquely through \( q \colon
  \trm \to (X \pt \trm)^n \) for some \( n \in \N \). Now, note that \( (X\pt
  \trm)^n \iso X^n \pt \trm \) if \( n > 0 \), \( (X\pt \trm)^0 = I \), and
  that we have pullback diagrams
  \begin{equation*}
    \begin{tikzcd}
      \trm \ar[r] \ar[d,equal] 
        \ar[rd,"\ulcorner",phantom,very near start]
      & I \ar[d,equal] \\
      \trm \ar[r] & I
    \end{tikzcd}
    \qquad\text{and, for }n > 0,\qquad
    \begin{tikzcd}
      X^n \pt \trm \ar[d] \ar[r,equal] 
        \ar[rd,"\ulcorner",phantom,very near start]
        & X^n \pt \trm \ar[d] \\
      \trm \ar[r,equal] & \trm
    \end{tikzcd}
  \end{equation*}
  whence \( \tau_{X,p} \iso X^n \pt \trm \) for some \(n \in \N\); this concludes
  the proof, by Theorem \ref{thm:heps.strong.conj}.
\end{proof}

\begin{lemma}
  \label{lem:cart.nat.lifts}
  Let \( S, T \) be endofunctors on \( \cat V \), and let \( \alpha \colon S
  \to T \) be a cartesian natural transformation. If \(T\) is fibrewise
  discrete, then so is \(S\).
\end{lemma}

\begin{proof}
  Consider the following composite of pullbacks:
  \begin{equation*}
    \begin{tikzcd}
      \sigma_{x, p} \ar[r] \ar[d] 
                    \ar[rd,"\ulcorner",phantom,very near start] &
      S(x\pt \trm) \ar[r,"\alpha_{x\pt \trm}"] \ar[d] 
        \ar[rd,"\ulcorner",phantom,very near start]
        & T(x\pt \trm) \ar[d,"T!"] \\
      \trm \ar[r,"p",swap] & S\trm \ar[r,"\alpha_\trm",swap] & T\trm
    \end{tikzcd}
  \end{equation*}
  We have \( \tau_{x,\alpha_\trm \circ p} \iso \sigma_{x,p} \), which is discrete
  for all \(x,p\).
\end{proof}

\subsection{Free monoid monad $\Set \times \Set$:}

We will confirm this monad is not fibrewise discrete. Indeed, we have the
following pullback diagram
\begin{equation*}
  \begin{tikzcd}
    \lb X^m,X^n \rb \ar[d] \ar[r]
      \ar[rd,"\ulcorner",phantom,very near start]
      & \lb X^*, X^* \rb \ar[d] \\
    \lb \trm,\trm \rb \ar[r,"{\lb m,n \rb}",swap] & \lb \N,\N \rb
  \end{tikzcd}
\end{equation*}
for each \(m,n \in \N \) and each set \(X\). However, \( \lb X^m,X^n \rb \) is
not discrete in general, so we cannot obtain a functor \( - \pt \trm \colon
\TtVCat \to \CatTV \) via Theorem \ref{thm:base.change}.

\subsection{Cartesian monads on slice categories:}

If we have a pair \( (T,\cat V) \) where \(T\) is a cartesian monad on a
category \( \cat V \) with finite limits, and \( \intcat C \) is an internal
\(T\)-category, we may construct \cite[Proposition 6.2.1]{Lei04} a cartesian
monad \( T_{\intcat C} \) on \( \cat V \comma \intcat C_0 \), and we obtain an
equivalence \cite[Corollary 6.2.5]{Lei04} of categories
\begin{equation*}
  \Cat(T_{\intcat C}, \cat V \comma \intcat C_0) 
    \eqv \Cat(T,\cat V) \comma \intcat C,
\end{equation*}
which raises the question: can we obtain \eqref{eq:induced.adjunction} for the
pair \( (T_{\intcat C}, \cat V \comma \intcat C_0) \)?

Already when \( T = \id \), \( \cat V = \Set \), we cannot generally guarantee
an affirmative answer. Indeed, let \( \cat C \) be an ordinary small category.
In this case, \( T_{\cat C} \) is the cartesian monad induced by the monadic
adjunction 

\begin{equation*}
  \begin{tikzcd}
    {[\cat C,\Set]}
      \ar[r,bend right,"{\iota^*_{\cat C}}"{name=A,below},swap]
    & {\Set \comma \ob \cat C}
      \ar[l,bend right,"\Lan_{\iota_{\cat C}}"{name=B,above},swap]
    \ar[from=A,to=B,phantom,"\adj" {anchor=center, rotate=-90}]
  \end{tikzcd}
\end{equation*}

However, the terminal object of \( \Set \comma \ob \cat C \iso [(\ob \cat
C)\pt \trm, \Set] \) is connected precisely when \( \ob \cat C \iso \trm \) or
\( \ob \cat C \iso \ini \); in fact, we shall confirm that while it is true
that \( T_{\cat C} \) is fibrewise discrete, \( \heps_{T_{\cat C}(- \pt \trm)}
\) is not a cartesian natural transformation, and thus we cannot obtain
\eqref{eq:induced.adjunction} for general \( \cat C \).

Let \( X = \ob \cat C \). \(T_{\cat C}\) is defined on objects by
\begin{equation*}
  (A_x)_{x \in X} \mapsto 
    \big( \sum_{y \in X} A_y \times \cat C(x,y) \big)_{x \in X},
\end{equation*}
and the terminal object of \( \Set \comma X \) is precisely the constant
family \( \trm = (\trm)_{x \in X} \). In this case, \( T_{\cat C}\trm = (\ob(x
\comma \cat C))_{x \in X} \), while \( \overline{T_{\cat C}}\trm = \prod_{x
\in X} \ob(x \comma \cat C) \).

More generally, for a constant family \( A \pt \trm \iso (A)_{x \in X} \), we
have 
\begin{equation*}
  T_{\cat C}(A \pt \trm) \iso (A \times \ob(x \comma \cat C))_{x \in X},
\end{equation*}
and \( \overline T_{\cat C}(A \pt \trm) = A^X \times \prod_{x \in X} \ob(x
\comma \cat C) \). We have, for each \( x \in X \), a pair of pullback
diagrams
\begin{equation*}
  \begin{tikzcd}
    A \ar[r] \ar[d] \ar[rd,"\ulcorner",phantom,very near start]
      & A \times \ob(x \comma \cat C) \ar[d] \\
    \trm \ar[r] & \ob(x \comma \cat C)
  \end{tikzcd}
  \qquad
  \begin{tikzcd}
    A \times \prod_{x \in X} \ob(x \comma \cat C) \ar[d] \ar[r]
      \ar[rd,"\ulcorner",phantom,very near start]
      & A \times \ob( x \comma \cat C) \ar[d] \\
    \prod_{x \in X} \ob(x \comma \cat C) \ar[r] & \ob(x \comma \cat C)  
  \end{tikzcd}
\end{equation*}
which confirms that \( T_{\cat C} \) is fibrewise discrete, but, since we
cannot guarantee \( A \iso A^X \), we cannot guarantee \( \heps_{T_{\cat C}(-
\pt \trm)} \) to be cartesian as well.

In spite of this, we can obtain the adjunction \eqref{eq:induced.adjunction}
when \( \intcat C_0 \iso \trm \); that is, when \( \intcat C \) is a
\textit{\((T,\cat V)\)-monoid}. We will now treat the case \( T = (-)^* \),
which are denoted \( \cat V \)-operads.

\subsection{\( \cat V \)-operadic monads:}

An important corollary of Lemmas \ref{lem:cart.nat.lifts} and
\ref{lem:word.monad.fib.disc} is that for a cartesian monoidal category \(
\cat V \), \( \cat V \)-operadic monads are fibrewise discrete; note that
these are precisely the cartesian monads on \( \cat V \) with a cartesian
natural transformation to the free \( \times \)-monoid monad.

To be explicit, for a \( \cat V \)-operad \( \mathcal O \) \cite[p. 44]{Lei04}
the monad associated to \( \mathcal O \) is given on objects by
\begin{equation*}
  V \mapsto \sum_{n \in \N} \mathcal O_n \times V^n,
\end{equation*}
and the projections \( \mathcal O_n \times V^n \to V^n \) induce a cartesian
natural transformation to the free \( \times \)-monoid monad.

Thus, any pair \( (T,\cat V) \) where \(T\) an operadic monad over a
lextensive category \( \cat V \) such that \( \trm \) is connected induces an
adjunction \eqref{eq:induced.adjunction}. Of special interest is the case \(T
= (-)^* \) is the free monoid monad on \( \cat V \). In this case, the induced
\( \Set \)-monad \( \Tt \) is precisely the ordinary free monoid monad.

\subsection{Free category monad:}

The free category monad \( \fc \) on \( \Grph \) is fibrewise discrete, since
we have the following pullback of graphs:
\begin{equation*}
  \begin{tikzcd}
    X \pt \trm \ar[rrr] \ar[ddd] \ar[rd,shift right]
                        \ar[rd,shift left]
      &&& (\N \times X) \pt \trm \ar[ld,shift right] 
                              \ar[ld,shift left] \ar[ddd] \\
    & X\pt \trm \ar[r] \ar[d] & X \pt \trm \ar[d] \\
    & \trm \ar[r] & \trm \\
    \trm \ar[rrr] \ar[ur,shift right] \ar[ur,shift left]
    &&& \N \pt \trm \ar[ul,shift right] \ar[ul,shift left]
  \end{tikzcd}
\end{equation*}
so, for the pair \( (\fc,\Grph) \), we also obtain an adjunction
\eqref{eq:induced.adjunction}. We note that \( \overline \fc \) is a lax
extension of the \( \N \times - \) monad on \( \Set \), for the
\textit{multiplicative} structure of \( \N \).

\subsection{Free finite coproduct completion monad:}

For a category \( \cat C \), \( \Fam_\fin(\cat C) \) is the category of
\textit{finite families} of objects of \( \cat C \); it is given on objects by
\( (\ob \cat C)^* \), and a morphism \( \mathfrak x \to \mathfrak y \) is a
pair \( (f,\phi) \), consisting of a function \( f \colon [m] \to [n] \),
where \(m\) and \(n\) are the lengths of \( \mathfrak x \) and \(\mathfrak y
\), respectively, and for each \(i=1,\ldots,n\), a morphism \( \phi_i \colon
\mathfrak x_i \to \mathfrak y_{fi} \).

From \cite[{}5.16]{Web07}, we learn that \( \Fam_\fin \) is a cartesian
(2-)monad on \( \Cat \). We proceed to verify it is fibrewise discrete; first,
observe that \( \ob \Fam_\fin(X \pt \trm) = X^* \), and the hom-sets are given
by
\begin{equation*}
  \Fam_\fin(X\pt \trm)(\mathfrak x,\mathfrak y) 
    = \sum_{f \colon [m] \to [n]} \prod_{i=1}^m [x_i=y_{fi}],
\end{equation*}
where \( m,\, n \) are the lengths of \( \mathfrak x,\, \mathfrak y \)
respectively. Moreover, note that \( \Fam_\fin(\trm) \eqv \FinSet \).  The fiber
of \( \Fam_\fin(X\pt \trm) \to \FinSet \) at (the identity on) \(n\) is given on
objects by the set of families of size \(n\), and on morphisms by \(
[\mathfrak x=\mathfrak y] \iso \prod_{i=1}^n [\mathfrak x_i = \mathfrak y_i]
\), which yields a discrete category; diagrammatically, we have
\begin{equation}
  \label{eq:fam.fin.fib.disc}
  \begin{tikzcd}
    {[x=y]} \ar[d] \ar[r]
            \ar[rd,"\ulcorner",phantom,very near start]
    & \displaystyle\sum_{f \in [n]\to[n]} 
      \displaystyle\prod_{i=1}^n 
        [\mathfrak x_i = \mathfrak y_{fi}] \ar[d] \\
    \trm \ar[r,"\id",swap] & \FinSet([n],[n])
  \end{tikzcd}
\end{equation}
as we desired. Thus, the pair \( (\Fam_\fin, \Cat) \) gives an adjunction
\eqref{eq:induced.adjunction} as well.

We note that \( \overline{\Fam_\fin} \) is a lax extension of the free monoid
monad on \( \Set \).

\subsection{Free finite product completion monad:}

The functor \( (-)^\op \colon \Cat \to \Cat \) taking each category to its
dual is its own adjoint, since we have \( \Cat(\cat C^\op, \cat D) \iso
\Cat(\cat C, \cat D^\op) \), so, via Proposition \ref{prop:induced.monad}, we
can promptly verify that the functor 
\begin{equation*}
  \cat C \mapsto \Fam_\fin^*(\cat C) = \Fam_\fin(\cat C^\op)^\op 
\end{equation*}
is a cartesian monad. For a category \( \cat C \), \( \Fam_\fin^*(\cat C) \)
has the same set of objects as \( \Fam_\fin(\cat C) \), but a morphism \(
\mathfrak x \to \mathfrak y \) is a pair \( (f,\phi) \) consisting of a
function \( f \colon [n] \to [m] \), where \(m,\,n\) is the length of \(
\mathfrak x,\, \mathfrak y \) respectively, and \( \phi_i \colon \mathfrak
x_{fi} \to \mathfrak y \) is a morphism for each \( i=1,\ldots,m \).

This monad is also fibrewise discrete; the only adjustment we need to make to
the pullback diagram \eqref{eq:fam.fin.fib.disc} is to replace \( [\mathfrak
x_i = \mathfrak y_{fi}] \) with \( [\mathfrak x_{fi} = \mathfrak y_i] \), so
the pair \( (\Fam_\fin^*, \Cat) \) induces an adjunction
\eqref{eq:induced.adjunction}.

\subsection{Free symmetric strict monoidal category monad:}

For a category \( \cat C \), we let \( \mathfrak S\cat C \) be a subcategory
of \( \Fam_\fin(\cat C) \) with the same set of objects, and precisely those
morphisms \( (f,\phi) \colon \mathfrak x \to \mathfrak y \) such that \(f\) is
a bijection.

This was shown to be a cartesian monad, for instance, in \cite{Lei04}, or in 
\cite[Example 7.5]{Web07}, where it was shown that we have a cartesian
(2-)natural transformation \( \mathfrak S \to \Fam_\fin \). For this same
reason, it is fibrewise discrete, by Lemma \ref{lem:cart.nat.lifts}, giving us
another example of an adjunction \eqref{eq:induced.adjunction}, with the pair
\( (\mathfrak S,\Cat) \).

Furthermore, note that \( \overline{\mathfrak S} \) is also a lax extension of
the free monoid monad on \( \Set \).

  \section{Embedding}
    \label{sect:embedding}
    Throughout this section, we fix a lextensive category \( \cat V \), with
terminal object \( \trm \) and a cartesian monad \(T = (T,m,e) \) on \( \cat V
\). Following the notation from Section \ref{sect:disc.morph}, we denote by \(
\Tt \) the monad on \( \VMat \) induced by \(T\) on \( \SpanV \).

Via the tools developed throughout the paper, we shall verify that if \( \trm
\) is connected, and \(T\) is fibrewise discrete, then \( - \pt \trm \colon
\TtVCat \to \CatTV \) is a fully faithful, pullback-preserving functor.
Moreover, among the pairs \( (T,\cat V) \) satisfying the hypotheses at the
end of Section \ref{sect:disc.morph}, we shall provide a description of \(
\TtVCat \) and \( \CatTV \).

\begin{lemma}
  \label{lem:induced.adjunction}
  If \( \heps_{T(-\pt \trm)} \) has a strong conjoint, then we have an
  adjunction
  \begin{equation}
    \label{eq:induced.adjunction.2}
    \begin{tikzcd}
      \TtVCat \ar[r,bend left,"- \pt \trm"{name=A}]
      & \CatTV \ar[l,bend left,"\cat V(\trm{,}-)"{name=B,below}]
      \ar[from=A,to=B,phantom,"\adj" {anchor=center, rotate=-90}]
    \end{tikzcd}
  \end{equation}
  whose unit and counit are also denoted by \( \heta \) and \( \heps \),
  respectively.
\end{lemma}

\begin{proof}
  By hypothesis, \(- \pt \trm \) is a strong functor, and \( \heps_{T(-\pt
  \trm)} \) has a strong conjoint. Since \(T\) is a strong functor, we also
  deduce that \( T\heps_{T(-\pt \trm)} \) has a strong conjoint as well. This
  places us in the setting of Section~\ref{sect:induced.adjunction}, hence, we
  obtain \eqref{eq:induced.adjunction.2} by applying
  Theorem~\ref{thm:lax.alg.adj} to the conjunction
  \eqref{eq:induced.conjunction}.
\end{proof}

Henceforth, we shall assume that the \( \trm \) is connected, and that \( T \)
is fibrewise discrete.

\begin{theorem}
  \label{thm:main.result}
  \( - \pt \trm \colon \TtVCat \to \CatTV \) is fully faithful.
\end{theorem}

\begin{proof}
  By Lemma~\ref{lem:1.fff}, we know \( - \pt \trm \colon \VMat \to \SpanV \)
  is fully faithful, and since \( \cat V(\trm,\heps_{T(-\pt \trm)}) \) is a
  natural isomorphism, the result follows by
  Corollary~\ref{cor:unit.iso.when}.
\end{proof}

These results can be immediately applied to the last four examples in Section
\ref{sect:disc.morph}; we shall describe both \( \TtVCat \) and \( \CatTV \)
for each such pair \( (T,\cat V) \).

\subsection{\( \cat V \)-operadic multicategories:}

Let \( T = T_{\mathfrak O} \) be a monad induced by a \( \cat V \)-operad \(
\mathfrak O \). When \( \trm \) is connected, we have shown that \(T\) is
fibrewise discrete, and therefore \( (T,\cat V) \) induces an adjunction
\eqref{eq:induced.adjunction.2}. So, we conclude that \( - \pt \trm \colon
\TtVCat \to \CatTV \) is fully faithful, by Theorem \ref{thm:main.result}.

The induced monad \( \Tt \) on \( \Set \) is given on objects by
\begin{equation*}
  X \mapsto \sum_{n \in \N} \cat V(\trm,\mathfrak O_n) \times X^n,
\end{equation*}
and note that since \( \cat V(\trm,-) \colon \cat V \to \Set \) is a strong
monoidal functor (preserves products), it follows that \( \cat
V(\trm, \mathfrak O) \) a \( \Set \)-monad, so \( \Tt \) is an operadic monad
as well.

Let \( r \colon X \relto Y \) be a \( \cat V \)-matrix, and let \( \sigma \in
\cat V(\trm, \mathfrak O_m) \), \( \mathfrak x \in X^m \), \( \tau \in \cat
V(\trm,\mathfrak O_n) \), \( \mathfrak y \in Y^n \). The \( \cat V \)-matrix
\( \Tt r \) is given at \( (\sigma, \mathfrak x, \tau, \mathfrak Y) \) by
\begin{equation*}
  (\Tt r)(\sigma, \mathfrak x, \tau, \mathfrak y) 
    = \begin{cases}
        0 & \text{if } \sigma \neq \tau, \\
        \prod_{i=1}^n r(x_i,y_i) & \text{otherwise}
      \end{cases}
\end{equation*}
thus, in practice, we just write \( (\Tt r)(\sigma, \mathfrak x, \mathfrak y)
\) for the possibly non-initial values of \( \Tt r \).

The objects of \( \CatTV \) are (internal) operadic \( \cat V \)-categories,
and for this reason, we will consider the objects of \( \TtVCat \) to be the
\textit{enriched} operadic \( \cat V \)-categories.  Such an object
consists of
\begin{itemize}[label=--]
  \item
    a set \(X\) of objects,
  \item
    a \( \cat V \)-matrix \( a \colon \Tt X \times X \to \cat V \),
  \item
    a \( \cat V \)-morphism \( \trm \to a(ex,x) \) for each \(x \in X \),
  \item
    a \( \cat V \)-morphism \( a(\sigma, \mathfrak x, x) \times \Tt a(
    \sigma, (\tau_1,\mathfrak y_1), \ldots, (\tau_n,\mathfrak y_k), \mathfrak
    x) \to a(\sigma(\tau_1,\ldots,\tau_k), \mathfrak y_1 \cdots \mathfrak
    y_k, x) \) for \( \tau_i \in \mathfrak O_{n_i} \), \( \mathfrak y_i \in
    X^{n_i} \), \( \sigma \in \mathfrak O_m \), \( \mathfrak x \in \mathfrak
    X^m \), where \( m = n_1 + \ldots + n_k \).
\end{itemize}
satisfying suitable identity and associativity conditions.

Of particular interest may be the \( T = (-)^* \) free \(\times\)-monoid monad
on \( \cat V \); more generally, monads induced by a \textit{discrete}
operad \( \mathfrak O \) and the \( M \times - \) monad for \( M \) a \( \cat
V \)-monoid.

\subsection{\( (\overline{\fc},\Grph) \)-categories:}

As we have verified in Section \ref{sect:disc.morph}, the pair \( (\fc,\Grph)
\) consists of a fibrewise discrete monad on a lextensive category with \(
\trm \) connected, so \( - \pt \trm \colon (\overline{\fc},\Grph)\dash \Cat
\to \Cat(\fc,\Grph) \) is fully faithful.

The object of the category \( \Cat(\fc,\Grph) \) are precisely the virtual
double categories \cite{Lei04, CS10}. An enriched \( (\overline{\fc},\Grph)
\)-category \(X\) consists of
\begin{itemize}[label=--]
  \item
    A set \( X_0 \) of objects,
  \item
    A graph \( X_1(n,x,y) = (X_{11}(n,x,y) \rightrightarrows X_{10}(n,x,y)) \)
    for each \(n \in \N \) and \( x,y \in X_0 \),
  \item
    A loop \( \trm \to X_1(1,x,x) \) for each \( x \in X_0 \),
  \item
    A graph morphism \( X_1(m,y,z) \times \fc_m X_1(n,x,y) \to X_1(m\cdot
    n,x,z) \) for each \( x,y,z \in X \) and \( m,n \in \N \), where \( \fc_mG
    \) is the graph of \(m\)-chains of \(G\).
\end{itemize}
satisfying suitable identity and associativity conditions. 
      
Via the induced functor \( \Grph(\trm,-) \colon \VDbCat \to (\overline{\fc},
\Grph)\dash \Cat \), we can come across examples of \( (\overline{\fc}, \Grph)
\)-categories. If \( \bicat V \) is a virtual double category, \( \Grph(\trm,
\bicat V) \) consists of
\begin{itemize}[label=--]
  \item
    a set of objects \( \Grph(\trm,\bicat V_0) \), that is, the set of loops in
    \( \bicat V_0 \), 
  \item
    for each \(n \in \N \) and loops \(r,s\) of \(\bicat V_0\), a graph \(
    \bicat V_1(n,r,s) \) has edges \( \theta \colon f \to g \)
    consisting of the 2-cells of the form
    \begin{equation*}
      \begin{tikzcd}
        x \ar[d,"f",swap]
          \ar[r,"r"] & x \ar[r,"r"]
                     & \ldots \ar[r,"r"]
                     & x \ar[r,"r"]
                     & x \ar[d,"f"] \\
        y \ar[rrrr,"s"{name=B,below},swap] &&&& y 
        \ar[from=1-3,to=B,"\theta",phantom]
      \end{tikzcd}
    \end{equation*}
    whose vertical domain has length \(n\); we may simply write
    \begin{equation*}
      \begin{tikzcd}
        x \ar[r,"r^n"{name=A}] \ar[d,"f",swap] & x \ar[d,"g"] \\
        y \ar[r,"s"{name=B},swap] & y
        \ar[from=A,to=B,"\theta",phantom]
      \end{tikzcd}
    \end{equation*}
    as a shorthand.
  \item
    the unit 2-cell at \(r\) in \( \bicat V_1(1,r,r) \) is given by
    \begin{equation*}
      \begin{tikzcd}
        x \ar[d,equal] \ar[r,"r"{name=A}] & x \ar[d,equal] \\
        x \ar[r,"r"{name=B},swap] & x
        \ar[from=A,to=B,"=",phantom]
      \end{tikzcd}
    \end{equation*}
    for each \(r \in \Grph(\trm,\bicat V_0) \), 
  \item
    for 2-cells
    \begin{equation*}
      \begin{tikzcd}
        y \ar[r,"s^m"{name=A}] \ar[d,"g",swap] & y \ar[d,"h"] \\
        z \ar[r,"t"{name=B},swap] & z
        \ar[from=A,to=B,"\omega",phantom]
      \end{tikzcd}
      \quad
      \begin{tikzcd}
        x \ar[r,"r^n"{name=A}] \ar[d,"f_{i-1}",swap] 
        & x \ar[d,"f_i"] \\
        y \ar[r,"s"{name=B},swap] & y
        \ar[from=A,to=B,"\theta_i",phantom]
      \end{tikzcd}
    \end{equation*}
    for \(i=1,\ldots,n\), the composite 2-cell is given by
    \begin{equation*}
      \begin{tikzcd}[column sep=large]
        x \ar[d,"g \circ f_0",swap] \ar[rrr,"r^{m \cdot n}"{name=A}]
        &&& x \ar[d,"h \circ f_n"] \\
        z \ar[rrr,"t"{name=B},swap] &&& z
        \ar[from=A,to=B,"\omega(\theta_1 \cdots \theta_n)",phantom]
      \end{tikzcd}
    \end{equation*}
\end{itemize}

\subsection{Clubs:}

We begin by considering the pair \( (\mathfrak S, \Cat) \). The category \(
\Cat(\mathfrak S,\Cat) \) is the so-called category of \textit{enhanced
symmetric multicategories} in \cite[p. 212]{Lei04}, first defined by \cite{BD98}
(therein, these are called opetopes).  

By analogy, we let
\begin{itemize}[label=--]
  \item
    \( \Cat(\Fam_\fin,\Cat) \) be the category of \textit{enhanced cocartesian
    multicategories}, and
  \item
    \( \Cat(\Fam_\fin^*,\Cat) \) be the category of \textit{enhanced cartesian
    multicategories}.
\end{itemize}

As we shall verify in Section \ref{sect:descent}, in each case \( T =
\Fam_\fin, \Fam_\fin^*, \mathfrak S \), the category \( (\Tt,\Cat)\dash \Cat
\) is the full subcategory of \( \Cat(T,\Cat) \) with discrete categories of
objects. Therefore, 
\begin{itemize}[label=--]
  \item
    \( (\overline{\Fam_\fin},\Cat)\dash \Cat \) is the category of
    \textit{cocartesian multicategories},
  \item
    \( (\overline{\Fam_\fin^*},\Cat) \dash \Cat \) is the category of
    \textit{cartesian multicategories},
  \item
    \( (\overline{\mathfrak S},\Cat) \dash \Cat \) is the category of
    \textit{symmetric multicategories}.
\end{itemize}

  \section{Application to descent theory}
    \label{sect:descent}
    \textit{Effective descent morphisms}, introduced in \cite{Gir64}, (see also
\cite{JT94}, \cite[Section~3]{Luc21}) are the fundamental object of study in
Grothendieck’s descent theory, which has strong connections with various
fields \cite{Moe89, BJ97, PL23}. Paired with their applicability to other fields,
effective descent morphisms are also important in their own right
\cite{Luc22}, since these morphisms formalize the procedure of recovering data
about their codomain from data about their domain, plus additional algebraic
structure. Such algebraic structure is called \textit{descent data} for the
morphism. For further aspects of descent theory, we refer to \cite{JST04,
Luc18a, Luc21, Luc22}.

In this work, we are concerned with effective descent morphisms with respect
to the \textit{basic bifibration} of a category \( \cat C \) with pullbacks,
given at a morphism \( p \colon x \to y \) by the following change-of-base
adjunction:
\begin{equation}
  \label{eq:benabou}
  \begin{tikzcd}
    \cat C/y \ar[r,bend right,"p^*"{name=A,below},swap]
    & \cat C/x. \ar[l,bend right,"p_!"{name=B},swap]
    \ar[from=A,to=B,phantom,"\adj" {anchor=center, rotate=-90}]
  \end{tikzcd}
\end{equation}
By the Bénabou-Roubaud theorem \cite{BR70}\footnote{See also \cite{Luc22} for
the ``moral converse'' of this result.}, we obtain the \textit{category of
descent data} for \(p\), denoted \( \Desc(p) \), as the category of algebras
for the monad induced by \eqref{eq:benabou}, which we denote by \( T^p \) --
that is, we have \( \Desc(p) \eqv T^p\dash\Alg \). 

Hence, we may consider the \textit{Eilenberg-Moore factorization} of \(p^*\)
in the following form:
\begin{equation*}
  \begin{tikzcd}
    \cat C \comma y \ar[rd,"\mathcal K^p",swap] \ar[rr,"p^*"] 
      && \cat C \comma x \\
    & \Desc(p) \ar[ru]
  \end{tikzcd}
\end{equation*}
We say that
\begin{itemize}[label=--]
  \item
    \(p\) is an \textit{effective descent morphism} if \( \mathcal K^p \) is
    an equivalence,
  \item
    \(p\) is a \textit{descent morphism} if \( \mathcal K^p \) is fully
    faithful,
  \item
    \(p\) is an \textit{almost descent morphism} if \( \mathcal K^p \) is
    faithful.
\end{itemize}

For categories \( \cat C \) with finite limits, descent morphisms are
precisely the pullback-stable regular epimorphisms, and almost descent
morphisms are precisely the pullback-stable epimorphisms. If \( \cat C \) is
Barr-exact \cite{Bar71} or locally cartesian closed, then effective descent
morphisms are precisely the descent morphisms. If \( \cat C \) is a topos,
then effective descent morphisms are precisely the epimorphisms \cite{JST04}.
In the context of internal categorical structures, we have the
characterization of \cite{Cre99} for internal categories, as well as the work
of the prequel \cite{PL23}.

Our results concern the effective descent functors between (enriched)
\((T,\cat V)\)-categories. When \( \cat V \) is a quantale, results in this
direction can be traced as far back as the characterization of \cite{RT94} for
effective descent morphisms between topological spaces -- see \cite{CH02,
CH04}, which reformulate the results of Reiterman-Tholen under the perspective
of lax algebras. Further advances are present in \cite{CJ11, CH12, CH17}. 

We shall fix a lextensive, cartesian monoidal category \( \cat V \) with \(
\trm \) connected, and a fibrewise discrete, cartesian monad \( T \) on \(
\cat V \). In this setting, Theorem \ref{thm:main.result} provides us with the
fully faithful embedding \( - \pt \trm \colon \TtVCat \to \CatTV \). Now, we
desire to apply this result to study \textit{effective descent morphisms} in
\( \TtVCat \). We promptly review the fundamental aspects of descent theory
necessary to draw our desired conclusions. Afterwards, under a suitable
hypothesis, we confirm that \( \TtVCat \) is the full subcategory of \( \CatTV
\) with a discrete object-of-objects (Theorem \ref{thm:tvcat.disc.objs}),
which we deduce that \( \TtVCat \to \CatTV \) reflects effective descent
morphisms, generalizing \cite[9.10 Lemma and 9.11 Theorem]{Luc18a} to the
multicategory setting. 

Having fixed the terminology, we begin by recalling the following result of
effective descent morphisms for pseudopullbacks of categories:

\begin{proposition}[{\cite[Theorem 1.6]{Luc18a}}]
  \label{prop:pspb.desc}
  If we have pseudopullback diagram of categories with pullbacks and
  pullback preserving functors
  \begin{equation*}
    \begin{tikzcd}
      \cat A \ar[r,"F"] \ar[d,"G",swap]
        & \cat B \ar[d,"H"] \\
      \cat C \ar[r,"K",swap] & \cat D
    \end{tikzcd}
  \end{equation*}
  and a morphism \( f \) of \( \cat A \) such that
  \begin{itemize}[label=--]
    \item
      \( Ff \) and \( Gf \) are effective descent morphisms, and
    \item
      \( KFf \iso HGf \) is a descent morphism,
  \end{itemize}
  then \( f \) is an effective descent morphism.
\end{proposition}

\begin{lemma}
  \label{lem:key.lemma}
  If \( (X \pt \trm, a, \eta, \mu) \) is an internal \( (T,\cat V)
  \)-category, then \( \heps_a \) is a split epimorphism. Moreover, if \(
  \heps_{T\trm} \) is a monomorphism, then \( \heps_a \) is an isomorphism.
\end{lemma}

\begin{proof}
  We consider the unique morphism \( (X\pt \trm, a, \eta, \mu) \to
  (\trm,e^*_\trm,\eta,\mu) \) to the terminal \( (T,\cat V) \)-category
  \begin{equation*}
    \begin{tikzcd}
      & M_a \ar[ld,"l_a",swap] \ar[rd,"r_a"] \ar[dd] \\
      T(X\pt \trm) \ar[dd,"T!",swap] 
      && X\pt \trm \ar[dd,"!"] \\
      & \trm \ar[ld,"e_\trm",swap] \ar[rd,equal] \\
      T\trm && \trm 
    \end{tikzcd}
  \end{equation*}
  and we note that \( e_\trm = \heps_{T\trm} \circ (\overline e_\trm \pt \trm) \), so
  that there there exists a unique \( \hat l_a \colon M_a \to \Tt X \pt \trm
  \) such that \( \heps_{T(X \pt \trm)} \circ \hat l_a = l_a \) and \( \Tt !
  \pt \trm \circ \hat l_a = (\overline e_\trm \pt \trm) \circ ! \):
  \begin{equation}
    \label{eq:pb.sq.1}
    \begin{tikzcd}
      M_a \ar[rd,dashed,"\hat l_a" description]
          \ar[rrd,bend left=15,"l_a"] \ar[d,"!",swap] \\
      \trm \ar[rd,"\overline e_\trm \pt \trm",swap]
        & \Tt X \pt \trm \ar[r,"\heps_{T(X\pt \trm)}"]
                      \ar[d,"\Tt ! \pt \trm",swap] 
                      \ar[rd,"\ulcorner",phantom,very near start]
        & T(X \pt \trm) \ar[d,"T(! \pt \trm)"] \\
        & \Tt \trm \pt \trm \ar[r,"\heps_{T\trm}",swap]
        & T\trm
    \end{tikzcd}
  \end{equation}
  It follows that there is a unique \( \omega \colon M_a \to M_{\cat V(\trm,a)
  \pt \trm} \) such that \( \heps_a \circ \omega = \id \) and \( \lb \hat l_a,
  r_a \rb = \lb l_a, r_a \rb \circ \omega \),
  \begin{equation}
    \label{eq:pb.sq.2}
    \begin{tikzcd}[column sep=large]
      M_a \ar[rrrd,"{\hat l_a,r_a}",bend left=15]
          \ar[rdd,equal,bend right=25]
          \ar[rd,"\omega" description,dashed] \\
      & M_{\cat V(\trm,a) \pt \trm} \ar[d,"\heps_a",swap]
                              \ar[rr,"{\lb l_{\cat V(\trm,a) \pt \trm},
                                       r_{\cat V(\trm,a) \pt \trm} \rb}"] 
                              \ar[rrd,"\ulcorner",phantom,very near start]
      && \Tt X \pt \trm \times X \pt \trm 
        \ar[d,"\heps_{T(X \pt \trm)} \times \id"] \\
      & M_a \ar[rr,swap,"{\lb l_a,r_a \rb}"] && T(X \pt \trm) \times X \pt \trm
    \end{tikzcd}
  \end{equation}
  thereby confirming \( \heps_a \) is a split epimorphism. 

  Moreover, observe that when \( \heps_{T\trm} \) is a monomorphism, it
  follows by the pullback square in \eqref{eq:pb.sq.1} that \( \heps_{T(X \pt
  \trm)} \) is a monomorphism, and by the pullback square in
  \eqref{eq:pb.sq.2}, we may conclude that \( \heps_a \) is a monomorphism.
  Thus, \( \heps_a \) is an isomorphism.
\end{proof}

As a corollary, we obtain

\begin{theorem}
  \label{thm:tvcat.disc.objs}
  If \( \heps_{T\trm} \) is a monomorphism, then we have a pseudopullback
  diagram
  \begin{equation}
    \label{eq:key.pspb}
    \begin{tikzcd}
      \TtVCat \ar[r,"-\pt \trm"] \ar[d] & \CatTV \ar[d] \\
      \Set \ar[r,"- \pt \trm",swap] & \cat V
    \end{tikzcd}
  \end{equation}
  of categories with pullbacks and pullback-preserving functors.
\end{theorem}

\begin{proof}
  We begin by observing that the objects of the pseudopullback are pairs \(
  (S, (X,a,\eta,\mu), \omega) \) where \(S\) is a set, \( (X,a,\eta,\mu) \) is
  an internal \( (T,\cat V) \)-category, and \( \omega \colon S \pt \trm \to X
  \) is an isomorphism. Naturally, this implies that \( \heps_X \) is an
  isomorphism, since \( \heps_{S \pt \trm} \) is invertible:
  \begin{equation*}
    \begin{tikzcd}
      \cat V(\trm, S\pt \trm)\pt \trm \ar[r,"\heps_{S \pt \trm}"] 
                            \ar[d,"{\cat V(\trm,\omega)} \pt \trm",swap]
      & S \pt \trm \ar[d,"\omega"] \\
      \cat V(\trm,X)\pt \trm \ar[r,"\heps_X",swap]
      & X
    \end{tikzcd}
  \end{equation*}
  and conversely, for any internal \( (T,\cat V) \)-category \( (Y,b,\eta,\mu)
  \) such that \( \heps_Y \) is invertible, the triple 
  \begin{equation*}
    (\cat V(\trm,Y), (Y,b,\eta,\mu),\heps_Y) 
  \end{equation*}
  is an object of the pseudopullback.

  Hence, given a \( (T, \cat V) \)-category \( (X,a,\eta,\mu) \) such that \(
  \heps_X \) is invertible, we have by Lemma \ref{lem:key.lemma} that \(
  \heps_a \) is invertible, since \( \heps_{T\trm} \) is a monomorphism by
  hypothesis.  By Lemma \ref{lem:nat.s.conj.eqv}, it follows that \(
  \n{\heps^*}_a \) is invertible, so that we can apply Lemma
  \ref{lem:counit.iso.when} to conclude that \( (X,a,\eta,\mu) \) is
  isomorphic to an enriched \( (\Tt,\cat V) \)-category, concluding the
  proof.
\end{proof}

From this, we can now apply Proposition \ref{prop:pspb.desc} to conclude that
\begin{lemma}
  \label{lem:refl.eff.desc}
  If \( \heps_{T\trm} \) is a monomorphism, then \( - \pt \trm \colon \TtVCat
  \to \CatTV \) reflects effective descent morphisms.
\end{lemma}

\begin{proof}
  Let \( F \colon \cat C \to \cat D \) be a functor of enriched \( (\Tt,\cat
  V) \)-categories such that \( F \pt \trm \) is an effective descent
  morphism.  Since \( \CatTV \to \cat V \) has fully faithful left and right
  adjoints, we may apply \cite[Lemma 2.3]{Pre23} to conclude that it preserves
  descent morphisms, so that \( (F \pt \trm)_0 = F_0 \pt \trm \) is a descent
  morphism.

  Since \( - \pt \trm \colon \Set \to \cat V \) reflects epimorphisms, we
  conclude that \( F_0 \) is an epimorphism; hence an effective descent
  morphism. Now, we apply Proposition \ref{prop:pspb.desc} with the
  pseudopullback \eqref{eq:key.pspb} to conclude \(F\) is effective for
  descent.
\end{proof}

Via \cite[Theorem 5.3]{PL23}, which provides sufficient conditions for
effective descent morphisms in \( \CatTV \) in terms of effective descent in
\( \cat V \), we can now do the same for \( \TtVCat \):

\begin{theorem}
  \label{thm:eff.desc.criteria}
  Let \( p \colon \cat C \to \cat D \) be a functor of \( (\Tt,\cat V)
  \)-categories. If \( \heps_{T\trm} \) is a monomorphism, and
  \begin{itemize}[label=--]
    \item
      \( (p \pt \trm)_1 \) is an effective descent morphism,
    \item
      \( (p \pt \trm)_2 \) is a descent morphism,
    \item
      \( (p \pt \trm)_3 \) is an almost descent morphism,
  \end{itemize}
  then \(p\) is an effective descent morphism.
\end{theorem}

\begin{proof}
  The three above conditions guarantee that \( p \pt \trm \) is an effective
  descent functor of (internal) \( (T,\cat V) \)-categories. Since \(
  \heps_{T\trm} \) is a monomorphism, we can apply Lemma
  \ref{lem:refl.eff.desc} to obtain the promised conclusion.
\end{proof}

Now, the above work raises (at least) the following two questions:
\begin{itemize}[label=--]
  \item
    For which pairs \( (T, \cat V) \) can we guarantee that \( \heps_{T\trm}
    \) is a monomorphism?
  \item
    Is the requirement that \( \heps_{T\trm} \) be a monomorphism
    ``reasonable''?
\end{itemize}

To answer the first, we note that this holds when
\begin{itemize}[label=--]
  \item
    \( \trm \) is a \textit{separator}; that is, when \( \cat
    V(\trm,-) \) is faithful, which implies \( \heps \) is a componentwise
    monomorphism.  This is the case when \( \cat V = \Set, \Top, \Cat, \) any
    hyperconnected Grothendieck topos \cite[{A4.6}]{Joh02}, but not \( \cat V
    = \Grph \).
  \item
    \( T \) is discrete; that is, when \( \heps_{T\trm} \) is an isomorphism.
    This is the case when \(T\) is the free \(\times\)-monoid monad on \( \cat
    V \), but not when \( T = \fc \).
\end{itemize}

And this, in a sense, answers the second question as well: from a practical
perspective, the above conditions are sufficient for nearly all of our
examples. And while we haven't confirmed whether the condition ``\(
\heps_{T\trm} \) is a monomorphism'' is necessary or not, we can provide a
heuristic argument to convey the intuition that this condition correctly
captures that \( \Tt \trm \pt \trm \) is a ``good'' discretization of \( T\trm
\): a pair which satisfies neither of the above hypotheses is the pair \(
(\fc,\Grph) \), as \( \ob \heps_{\fc \trm} \colon \N \to \trm \); here, \(
\overline{\fc} \trm = \Grph(\trm,\fc \trm) \) has too many points to be a
``reasonable'' discretization.

We now discuss the examples we have worked with so far.

\subsection{\( \cat V \)-operadic \( \cat V \)-categories:}

Let \( \mathfrak O \) be a \( \cat V \)-operad, so that the \( \cat V
\)-operadic monad \( T = T_{\mathfrak O} \) induced by \( \mathfrak O \) is
given by
\begin{equation*}
  X \mapsto \sum_{n \in \N} \mathfrak O_n \times X^n.
\end{equation*}

Since \( \cat V(\trm,-) \) preserves coproducts, we have 
\begin{equation*}
  \cat V(\trm, \sum_{n \in \N} \mathfrak O_n) 
    \iso \sum_{n\in\N} \cat V(\trm, \mathfrak O_n),
\end{equation*}
and therefore \( \heps_{T\trm} \iso \sum_{n\in \N} \heps_{\mathfrak O_n} \). It
is easy to verify that in an extensive category, a coproduct of morphisms is a
monomorphism if and only if every summand is a monomorphism, so we may apply
Theorem \label{thm:eff.desc.ttv} when \( \heps_{\mathfrak O_n} \) is a
monomorphism for all \(n \in \N \). 

Naturally, the result holds if we consider \( \cat V \)-operads for
categories \(\cat V \) such that \( \trm \) is a separator, such as \( \Cat \) or
\( \Top \), or if we consider \textit{discrete} \( \cat V \)-operads; that is,
\( \cat V \)-operads such that \( \heps_{\mathfrak O_n} \) is an isomorphism
for all \(n \in \N \).

However, the above conditions are not necessary, if, for instance, one
considers \( \Grph \)-operads \( \mathfrak O \) such that \( \mathfrak O_n \)
has at most one loop at each vertex for all \(n \in \N\); this is precisely
the case when \( \heps_{\mathfrak O_n} \) is a monomorphism for all \(n\in
\N\).

If \( \mathfrak O \) a discrete \( \cat V \)-operad, we define \( \TtVCat \)
to be the category of enriched \( \mathfrak O \)-categories.

\subsection{\( \cat V \)-multicategories:}

An important instance of the previous case is the case \( \mathfrak O \iso \N
\pt \trm \); that is, when \( T \) is the free \( \times \)-monoid monad on \(
\cat V \). In this case, \( \CatTV \) is the category of
\textit{multicategories internal to} \( \cat V \). We note that \(T\) is a
discrete monad, and the induced \( \Set \)-monad \( \Tt \) is the ordinary
free monoid monad.

Thus, we may define the objects of \( \TtVCat \) to be the \textit{enriched}
\(\cat V\)-multicategories, and the morphisms are the respective
enriched \( \cat V \)-functors. An immediate application of Theorem
\ref{thm:eff.desc.criteria} provides criteria for such an enriched \( \cat V
\)-functor to be effective for descent.

\subsection{Clubs:} 

We consider the pair \( (\mathfrak S, \Cat) \); the free symmetric strict
monoidal category monad \( \mathfrak S \) on \( \Cat \). By Theorem
\ref{thm:tvcat.disc.objs}, we recover the categories of (many-object) clubs
considered in \cite{Kel72a, Kel72b} by taking the fibers of the fibration \(
(\overline{\mathfrak S},\Cat)\dash \Cat \to \Set \) (see \cite[{}4.19]{CS10}).

In fact, this can be carried out for any fibrewise discrete monad \(T\)
on \( \Cat \), as \( \trm \in \Cat \) is a separator.

  \section{Epilogue}
    \label{sect:epilogue}
    We gave a general description of change-of-base functors between horizontal
lax algebras induced by monad (op)lax morphisms on the 2-category \(
\PsDbCat_\lax \), and with this description, we made the dichotomy between
enriched and internal generalized multicategories explicit. As our main
result, we have shown that enriched generalized multicategories are internal
generalized multicategories with a \textit{discrete} object of objects, under
suitable conditions.  Moreover, we applied this result to study the effective
descent morphisms of \( \TtVCat \).

There is still a vast amount of open problems left to settle. For the
remainder of this section, we will state a couple of these problems, sketch a
possible approach to their solution, and highlight possible connections to
other work.

\subsection{Object-discreteness}

In \cite[Section 8]{CS10}, the authors define and study the full subcategories
of \textit{normalized} and \textit{object-discrete} horizontal lax
\(T\)-algebras. Inspired by our Theorem \ref{thm:main.result}, we sketch an
argument, for an instance of \cite[Theorem 8.7]{CS10} for the equipment of
\textit{modules} of a suitable equipment, via change-of-base.

If \( \bicat D \) is an equipment whose hom-categories of the underlying
bicategory have all coequalizers, which are preserved by horizontal
composition, then we have an equipment \( \Mod(\bicat D) \) whose underlying
category of objects is \( \LaxidHAlg \), and horizontal 1-cells are
\textit{modules}; see \cite[Section 5.3]{Lei04}, \cite[Theorem 11.5]{Shu08}.
In fact, \( \Mod \) defines a 2-functor defined on a suitable full
sub-2-category of equipments, hence if \( T \) is a monad on \( \bicat D \),
then \( \Mod(T) \) is a monad on \( \Mod(\bicat D) \). We have an inclusion
\begin{equation*}
  \mathfrak I \colon \bicat D \to \Mod(\bicat D),
\end{equation*}
and the unit comparison 2-cell \( \e T \) induces a monad oplax morphism \(
(\mathfrak I, \e T) \colon T \to \Mod(T) \). When \( T \) is a normal functor,
we may apply Theorem \ref{thm:base.change} to obtain a change-of-base functor
\begin{equation*}
  \mathfrak I_! \colon \LaxTHAlg \to \LaxModTHAlg,
\end{equation*}
which identifies the full subcategory of ``object-discrete'' horizontal lax \(
\Mod(T) \)-algebras as the category of horizontal lax \(T\)-algebras, so we
partially obtain \cite[Theorem 8.7]{CS10}, when \(T\) is normal.

\subsection{Monadicity of horizontal lax algebras}
\label{subsect:monad}

Let \( T = (T,m,e) \) be a monad on an equipment \( \bicat D \) in \(
\PsDbCat_\lax \), and let \( x \) be an object of \( \bicat D \). We define \(
\HKlTx \) to be the category whose objects are horizontal 1-cells \( a \colon
Tx \to x \), and morphisms are the globular 2-cells between them. 

If \( m \) has a strong conjoint, \( \HKlTx \) has a tensor product defined
by
\begin{equation*}
  b \otimes a = b \cdot (Ta \cdot m^*_x),
\end{equation*}
which makes it into a \textit{skew monoidal category} \cite{Szl12}. If we let
\( \LaxTHAlg(x) \) be the category of horizontal lax \(T\)-algebras with
underlying object \(x\), it can be shown that \( \LaxTHAlg(x) \) is the
category of monoids of the skew monoidal category \( \HKlTx \)\footnote{This
construction is analogous to the definition of \( \LaxTHAlg \) in \cite{CS10}
as monoids in \( \HKlT \), adapted to the fixed-object case, and with \(
\bicat D \) an equipment.}. Therefore, we have a forgetful functor 
\begin{equation*}
  \LaxTHAlg(x) \to \HKlTx,
\end{equation*}
and we can study its monadicity, by studying free monoids in skew monoidal
categories, adapting the work of \cite{Dub74, Kel80, Lac10b}.

\subsection{Categorical properties of generalized multicategories}
\label{subsect:categ}

A possible approach to study the problem posed in Subsection
\ref{subsect:monad} is to determine whether we have a fibration
\begin{equation*}
  \LaxTHAlg \to \bicat D_0
\end{equation*}
for a pseudodouble category \( \bicat D \) and a lax monad \(T\) on \( \bicat
D \), carving a path for the application of techniques from fibred category
theory to obtain categorical results about \( \LaxTHAlg \); namely, as
existence of (co)limits, distributivity, extensivity, monoidal closedness, and
so on -- see \cite{LV23}. This approach proved to be fruitful for other sorts
of categorical structures, as studied in \cite{CL23, CLP24, LV24b}.

Such fibrations certainly exist for enriched and internal \(T\)-categories:
\begin{equation}
  \label{eq:fibs}
  \TVCat \to \Set, \qquad \CatTV \to \cat V,
\end{equation}
and several results are obtained via fibrational techniques. For instance,
\cite[Appendix~D]{Lei04} approaches the monadicity of \( \CatTV \to \Kl(T) \)
by studying the problem on fibres over \( \cat V \).

We also have the body of work \cite{HST14}, which explores various categorical
properties of \( (T,\cat V) \)-categories when \( \cat V \) is a quantale, and
the work of \cite{Cle21} further confirms that \( \TVCat \) is an extensive
category for an ample family of monoidal categories \( \cat V \) and suitable
lax monads on \( \VMat \).

\subsection{2-dimensional structure and descent theory}

If we consider the additional 2-dimensional categorical structure on the
categories \( \TVCat \) and \( \CatTV \), we conjecture that we have
\textit{double fibrations} \cite{CLPS22}
\begin{equation*}
  \TVCat \to \VMat, \qquad \CatTV \to \SpanV
\end{equation*}
with underlying fibrations respectively given by the functors \eqref{eq:fibs}. 
More generally, for a suitable pseudodouble category \( \bicat D \) and lax
monad \( T \), we pose the problem of whether we have a double fibration
\begin{equation*}
  \LaxTHAlg \to \bicat D,
\end{equation*}
and whether change-of-base (pseudo)functors are a suitable notion of morphisms
between such double fibrations.

It would also be interesting to inquire on whether there are any intersting
descent-theoretical problems in the context of the double fibrations of
\cite{CLPS22}, akin to the work developed in \cite{Her04b} for 2-fibrations.

\subsection{Discrete fibrations}

Another interesting theme is the study of discrete fibrations in the context
of generalized multicategories, and their behaviour under change-of-base. In
the context of \( (T,\cat V) \)-categories \cite{CT03, CH09}, work in this
direction has been pursued by \cite{BasTh, Bas17}. A pursuit of these ideas in
the generalized multicategory context would require studying the lax
epimorphisms \cite{ABSV01, LPS23} in \( \LaxTHAlg \) with a suitable
2-dimensional structure, and their preservation/reflection via change-of-base
(pseudo)functors. This work has also has strong ties to descent theory,
particularly in regard with effective étale descent, as developed in
\cite{Sob04, LPS23}.

\subsection{Other notions of change-of-base}

We have already mentioned two notions of change-of-base that are not covered
by Theorem \ref{thm:base.change} in Subsections \ref{subsect:int.t.mcats} and
\ref{subsect:tvcat.2}. In fact, with an adequate notion of ``monad morphism''
\( (F,\phi) \colon S \to T \) for \(S\) a lax monad on \( \bicat D \), and
\(T\) an oplax monad on \( \bicat E \), we question if it is possible to
expand the scope of the dichotomy between enriched and internal
multicategories.

  \printbibliography

\end{document}